\documentclass[a4paper,11pt]{amsart}
\usepackage[margin=2.5cm]{geometry}
\usepackage{amssymb,amsmath,mathtools,mathrsfs}
\usepackage[colorlinks,linkcolor=blue,citecolor=blue,urlcolor=blue]{hyperref}
\usepackage{enumitem,xcolor,booktabs,longtable}
\usepackage{moreenum}
\allowdisplaybreaks
\numberwithin{equation}{section}

\newcommand{\R}{\mathbb{R}}\newcommand{\N}{\mathbb{N}}
\newcommand{\E}{\mathbb{E}}\newcommand{\Prob}{\mathbb{P}}\newcommand{\Q}{\mathbb{Q}}
\newcommand{\Law}{\operatorname{Law}}\newcommand{\idx}{\operatorname{id}_x}

\newcommand{\RZ}{\mathbf{Z}}\newcommand{\ZZ}{\mathbb{Z}}
\newcommand{\CZ}{[\mathbf Z]}
\newcommand{\CC}{\mathscr{C}}
\newcommand{\dt}{\,\mathrm{d}t}
\newcommand{\dr}{\,\mathrm{d}r}
\newcommand{\dy}{\,\mathrm{d}y}
\newcommand{\du}{\,\mathrm{d}u}\renewcommand{\d}{\,\mathrm{d}}
\newcommand{\wsigz}{\sigma_{0}}
\newcommand{\bsigo}{\mathcal{B}_{1}}\newcommand{\bsigz}{\mathcal{B}_{0}}
\newcommand{\drifo}{\mathcal{A}_{1}}\newcommand{\drifz}{\mathcal{A}_{0}}
\newcommand{\Dext}{\overline{\mathscr D}}
\newcommand{\tnorm}[1]{|\!|\!|#1|\!|\!|}
\newcommand{\mfn}{\mathfrak n}
\newcommand{\one}{\mathbf 1}
\newcommand{\Frob}[1]{\|#1\|_{\mathrm{fr}}}
\newcommand{\Op}[1]{\|#1\|_{\mathrm{op}}}
\newcommand{\Opb}[1]{\big\|#1\big\|_{\mathrm{op}}}

\makeatletter
\newcommand{\stepnum}[2]{\phantomsection\def\@currentlabel{#1}\label{#2}}
\makeatother

\newtheorem{theorem}{Theorem}[section]
\newtheorem{lemma}[theorem]{Lemma}
\newtheorem{proposition}[theorem]{Proposition}

\newtheorem{openproblem}[theorem]{Open Problem}
\theoremstyle{definition}
\newtheorem{definition}[theorem]{Definition}
\newtheorem{assumption}[theorem]{Assumption}
\newtheorem{example}[theorem]{Example}
\theoremstyle{remark}
\newtheorem{remark}[theorem]{Remark}

\title[Systems of rough SDEs I]{Systems of rough stochastic differential equations I:\\
weak existence and Yamada--Watanabe}
\author{Florian Huber}

\email{dr.huber.florian@gmail.com}
\thanks{Parts of this paper were completed during the author's affiliation with the \'Ecole Polytechnique F\'ed\'erale de Lausanne (EPFL), Switzerland.}
\date{\today}

\begin{document}
\begin{abstract}
For systems of rough stochastic differential equations in the sense of Friz--Hocquet--L\^e, with a
fixed deterministic rough driver, we prove two results.

The first is weak existence with unbounded coefficients. The It\^o data is assumed continuous and of
linear growth, with no H\"older or Lipschitz continuity and no ellipticity. The rough vector field is
a controlled pair of linear growth with bounded derivatives, subject either to a norm-curvature decay
or to diagonality with coordinatewise curvature decay; in the diagonal case a coordinatewise upper
bound on the diffusion matrix is added. The existing results ask the It\^o data and the rough field
to be bounded. We remove that boundedness, and ask in exchange for a curvature hypothesis on the
rough field, bounds on its derivatives, and a narrower index configuration. The conclusion is weaker
on two counts: the bounds that define a solution are imposed in the mean rather than uniformly on the
probability space, and the solution's moments stop at the order the initial law carries.

The second is the Yamada--Watanabe implication for rough equations: weak existence, together with
pathwise uniqueness, yields a strong solution and joint uniqueness in law. A solution here is not
defined by a pathwise identity but by two bounds on a remainder, one of them conditional, and an
arbitrary enlargement of the filtration need not preserve that one. The enlargements the
Yamada--Watanabe construction produces are immersions, however, and along an immersion both bounds
transfer unchanged.

The pathwise-uniqueness input the second result takes as a hypothesis is supplied on an affine
class in the companion paper \cite{HuberII}, and reduced there to one bound outside it.
{\color{blue}Passages in blue belong to this arXiv version only. They carry additional context and side results
that the journal version omits for length. Nothing else depends on them.\color{black}}
\end{abstract}
\maketitle
\tableofcontents

\section{Introduction}\label{sec:intro}

Let $B$ denote a $d$-dimensional Brownian motion, and let $\RZ=(Z,\ZZ)$ denote a rough path of
regularity $\gamma\in(\tfrac13,\tfrac12)$,  fixed for the remainder of this paper. We consider the equation
\[
\d X_{t}=b(t,X_{t})\dt+\sigma(t,X_{t})\,\d B_{t}+(f,f')(t,X_{t})\,\d\RZ_{t},
\qquad X_{0}\sim\nu,\quad X_{t}\in D,
\]
(which is \eqref{eq:rsde} below), interpreted in the sense of Friz, Hocquet and L\^e \cite{FHL}, on a closed state-space
$D\subseteq\R^{n}$. Our motivation arises from interacting particle systems. The picture to keep in mind is a system of particles carrying independent
noise and sitting inside a common exogenous signal that is not assumed to be a semimartingale and is
given rather than random.

The two halves of the equation are built on opposite principles. The It\^o term is integrated by
probabilistic cancellation: its increments are small in the mean because martingale increments
are centred. The rough term admits no such cancellation. $\RZ$
is deterministic, and is integrated by a pathwise Taylor expansion carried to second order,
which is what the postulated second level $\ZZ$ is for and why $\gamma>\tfrac13$ suffices. Neither
technique survives contact with the other, and what reconciles them is the stochastic sewing lemma of
\cite{Le20} in the form \cite{FHL}.

\subsection*{Notion of solution, and where the difficulty arises}

Reconciling the probabilistic and rough approach requires a slightly weaker notion of solution than what we are used to from SDEs. For a classical It\^o
equation, ``$X$ solves the equation'' is a pathwise identity between $X$, $B$, and a stochastic
integral. The solution to the RSDE here is understood via the Davie expansion of the increment
of $X$ over a short interval $[s,t]$. It has four terms: the drift, the It\^o term, and the first two rough Taylor terms, whose coefficients are taken at the left endpoint $s$.
One defines $X$ to be a solution by requiring that the remainder of this expression $J_{s,t}$ (the increment of $X$ minus the four previously mentioned terms) satisfies two bounds,
\[
\big\|\E[J_{s,t}\mid\mathcal F_{s}]\big\|_{L_{m}(\Prob)}=o(|t-s|),
\qquad
\big\|J_{s,t}\big\|_{L_{m}(\Prob)}=o(|t-s|^{1/2}),
\]
which is \eqref{eq:daviemoduli} below. The last two of the four terms are referred to as a rough germ of $X$ on $[s,t]$: a local approximation of the rough part of the increment,
with its coefficients taken at the left endpoint. The second bound directly relates to the size of the remainder. The first bounds the conditional mean, at the strictly better rate
$|t-s|$. Conditioning is what contributes the extra half power.

Two important features of this notion, we want to highlight.
\begin{itemize}
    \item \emph{The bounds rely on a filtration and not a law.} A conditional expectation is taken
with respect to a $\sigma$-field, so ``being a solution'' looks like a property of a stochastic basis rather
than of the joint law of $(X,B)$, which is the case for ordinary SDEs. Here it is what turns the Yamada--Watanabe implication from a formality into
a problem: that construction glues two solutions over a common Brownian motion and thereby
enlarges the filtration, while the free estimate, conditional Jensen, only says that
shrinking the filtration cannot hurt.
\item \emph{The solution rests on a Taylor expansion of the coefficients along the solution.} Unbounded coefficients
force the expansion to be carried relative to a weight, and it leaves one power of the weight
more than the estimates can absorb. A decay hypothesis on the curvature of the rough field is
what yields the missing cancellation, and it is why the weak-existence theorem requires it.
\end{itemize}

Three questions follow. Can weak solutions be built when the It\^o coefficients are unbounded? Does
pathwise uniqueness still yield uniqueness in law and a strong solution, the solution notion being
filtration-dependent? And where is that pathwise uniqueness to come from? We answer the first two here. The third
is the subject of the companion paper \cite{HuberII}, which answers it directly on an affine class
and, outside that class, reduces it to one bound. Composing the three answers gives a single
well-posedness theorem, proved there.

\subsection*{Results}
Two of the three results below are proved here and the third is the companion paper's. Each stands
on its own, and together they give strong well-posedness:

\begin{enumerate}[label=(\Alph*)]
    \item Weak existence at linear growth (Theorem~\ref{thm:weakex}).
     The existing theorem
\cite[Thm.~4.19]{FHL} requires $b,\sigma$ to be bounded. We assume instead that they are continuous and of linear
growth: no H\"older continuity, no Lipschitz continuity, no ellipticity, no diagonality. We assume $(f,f')$
to be a controlled rough vector field of linear growth with bounded derivatives. We require however one of two
curvature hypotheses. The first is norm-curvature decay \eqref{eq:Ecurvnormhyp},
which contains the affine fields. Under the first, we don't ask anything of $\sigma$ beyond continuity and
linear growth. The second is diagonality with coordinatewise
curvature decay \eqref{eq:Ecurvdiag}. Under the second (and only there) we add one coordinatewise
upper bound \eqref{eq:Eacoord} on the diffusion matrix. That bound is strictly stronger than the $\sigma$-half of linear growth.  We require the initial law to have $p>6$ moments
(\eqref{eq:Ep}). Then \eqref{eq:rsde} has an $L_{m}$-conditional solution on $\R^{n}$ for every
$m\in[2,p]$, and the rough martingale problem has a solution for every $m\in[2,p/3]$. We also
obtain an explicit bound on its $p$-th moment. The proof uses
\cite[Thm.~4.6]{FHL} for the approximating solutions. \S\ref{app:FHL46} checks its
hypotheses for the truncated field of Step~\ref{step:Ereg}.

The exchange with \cite[Thm.~4.19]{FHL} is not one-sided. In exchange for lifting the
boundedness of the It\^o data and of the rough field, we ask for, among other things: one of the two curvature
hypotheses of Assumption~\ref{ass:Erough}, the derivative bounds $\|D^{j}f\|_{\infty}<\infty$ for
$j=1,2,3$ and $\|D^{j}f'\|_{\infty}<\infty$ for $j=1,2$, the unweighted time-modulus for $Df$ of
\eqref{eq:Eroughderiv}, and the index configuration
$\tfrac13<\beta\le\gamma$, $\tfrac13<\beta'$ of Assumption~\ref{ass:Erough}, narrower than
\cite[Thm.~4.19]{FHL}'s. We also keep the standing Lipschitz bracket \eqref{eq:bracketLip}, which
restricts the admissible lift. Neither set of hypotheses contains
the other. 

The Lyapunov structure underlying the argument is that of the whole rough generator and not only of the
It\^o part (Remark~\ref{rem:Elyapunov}). One might hope to trade the linear growth of $b$ and $\sigma$
for an abstract Lyapunov condition. Our estimates are estimates of increments, and
they use the quantity $\|b(r,X_{r})\|_{L_{m}}$ itself, whereas a Lyapunov bound controls only the size of
the Lyapunov function along the path and says nothing about that quantity. Dissipative drifts of superlinear growth are therefore
not admitted here, $b(x)=-x|x|^{2}$ being the example. Open
Problem~\ref{op:lyapunov} discusses what would have to replace the bounds \eqref{eq:Vjet} on the
derivatives of the Lyapunov function. 

\item Yamada--Watanabe for rough equations (Theorem~\ref{thm:YWtransfer}).  Weak existence,
together with pathwise uniqueness, yields a strong solution and joint uniqueness in law. To our
knowledge, this implication had not been established for RSDEs. The issue is the
one from above: the gluing enlarges the filtration, and the conditional bounds defining a solution
may not survive an arbitrary enlargement of the filtration.

However, they survive immersions. Along an
immersion the conditional expectations in question are not only comparable but actually equal, the Davie
remainder $J_{s,t}$ being measurable at time $t$. Hence, both bounds actually transfer with the same
moduli (Lemma~\ref{lem:immersion}). 
Lemma~\ref{lem:lawdet} describes the solution set as a constraint on the joint
law alone and not on the basis. On a basis, whose filtration is generated by $(X,B)$, the Davie remainder is a functional of that pair, so both bounds defining a solution depend only on $\Law(X,B)$.
The argument is then set in the framework of
Kurtz \cite{Kurtz14}. Proposition~\ref{prop:MPequiv} identifies that constraint with a rough
martingale problem, at the cost of a factor $3$ in the Davie order, and \S\ref{sec:YWtransfer} uses
it only in the last of its three conclusions. We note that
\cite[Thm.~4.21]{FHL} reaches uniqueness in law by a different route, running the opposite way and
available only under bounded Lipschitz data (Remark~\ref{rem:FHLrelation}(\ref{rem:FHLYW})).
\item The pathwise-uniqueness input (\cite{HuberII}). Neither (A) nor (B) supplies the pathwise
uniqueness that (B) takes as a hypothesis. The companion paper \cite{HuberII} investigates it, and composing the three
results there gives strong well-posedness on the affine class treated there. 
\end{enumerate}

Theorem~\ref{thm:weakex} requires all three of Assumptions~\ref{ass:Eito},
\ref{ass:Erough} and \ref{ass:Einit}, the coordinatewise bound \eqref{eq:Eacoord} only in the
diagonal branch. Theorem~\ref{thm:YWtransfer} requires much less: the standing continuity
\eqref{eq:standingcts}; nothing of Assumption~\ref{ass:Einit}; and its own two hypotheses \ref{yw1}
and \ref{yw2}. Its part \ref{ywc} uses in addition the linear growth of $b$ and
$\sigma$ in Assumption~\ref{ass:Eito} and the linear growth \eqref{eq:MProughgrowth} of the rough
field, which follows from \eqref{eq:Eroughgrowth} of Assumption~\ref{ass:Erough}. \ref{ywa} and
\ref{ywb} use neither.

\subsection*{Relation to the literature}

The equation and the notion of weak solution are adapted from \cite{FHL}. An abstract Yamada--Watanabe theorem
is in \cite{Kurtz14}, and the stochastic sewing lemma in \cite{Le20}. 

\begin{remark}[Relation to {\cite{FHL}}]\label{rem:FHLrelation}
\begin{enumerate}
    \item \label{rem:EFHL} \emph{Comparison with {\cite[Thm.~4.19]{FHL}}}. \cite[Thm.~4.19]{FHL} assumes $b,\sigma$ bounded continuous and $(f,f')$ a deterministic controlled
vector field in $\mathscr D^{\beta,\beta'}_{Z}\mathcal C^{\varkappa}_{b}$, with $\varkappa$ denoting
the field-regularity index, denoted $\gamma-1$ in \cite{FHL}, and proves weak existence with
$L_{m,\infty}$-integrability for every $m\ge2$. That proof mollifies both the It\^o data, to bounded
Lipschitz coefficients, and the rough field, solves each mollified equation by
\cite[Thm.~4.6]{FHL}, reads tightness off the a priori estimate
\cite[Prop.~4.5]{FHL}, which is available precisely because the coefficients are bounded, and
identifies the rough term of the limit through the stability lemma \cite[Lem.~4.20]{FHL}. Steps~\ref{step:Ereg} and \ref{step:Etight} of \S\ref{subsec:Eproof}, and
Theorem~\ref{thm:weakex} replace the appeal to \cite[Prop.~4.5]{FHL} by Step~\ref{step:Emoment}.
That substitution is one headline difference.

Step~\ref{step:Etight} uses Lemma~\ref{lem:roughstab}, which is a variant of \cite[Lem.~4.20]{FHL}. The lemma follows indeed directly from \cite[Lem.~4.20]{FHL}'s proof: the controlled class
is replaced by the hypotheses of \cite[Thm.~3.4]{FHL}. That matters because the continuity condition the
controlled class carries is open on our raw filtration, which we do not assume right-continuous. One filtration per member, which the
step also needs, Skorokhod delivering a jointly converging pair on one space with one Brownian
filtration per member and none common to all $j$, is already allowed in \cite[Lem.~4.20]{FHL}.
Two further choices inside the step are forced by the
growth of the coefficients, and both stay inside the range of the lemma. The integrability index has
to be finite, an essential supremum over $\Omega$ being unavailable once the integrand pair has
linear growth. The H\"older indices have to be taken with a sum no larger than $\beta+\beta'$, the
branch that would give the better order forcing $\|\delta X\|_{L_{2m}}$, that is, the higher moment the proof exists to avoid. Separately, the half of condition~(d) of \cite[Def.~3.1]{FHL} that bounds
$\delta Y'$ itself, and not $\E_{s}\delta Y'$, is not a hypothesis of \cite[Lem.~4.20]{FHL}, but \cite[Prop.~4.3]{FHL} carries it as a standing hypothesis, so
it has to be supplied at the end of the step, where \cite[Prop.~4.5]{FHL} had supplied it for
free. The It\^o half of the identification is not inherited at all,
\cite[Lem.~4.20]{FHL} being the stability lemma for the rough integral and saying nothing about
$\int\sigma\,\d B$; Step~\ref{step:Etight} argues it here.

In exchange the conclusion is weaker in two aspects, both forced by the same removal of boundedness.
First, integrability only up to the exponent $p$ of Assumption~\ref{ass:Einit}, rather than for every
$m\ge2$,
since with unbounded coefficients no moment can exceed what the initial law supplies. Second, and more
consequentially, the solution notion itself: \cite[Thm.~4.19]{FHL} delivers the $L_{m,\infty}$ form
\eqref{eq:daviemoduliinfty}, and it can, since with $b,\sigma$ bounded and $(f,f')$ in
$\mathcal C^{\varkappa}_{b}$ the Davie remainder carries no unbounded state factor, whereas
Theorem~\ref{thm:weakex} delivers only \eqref{eq:daviemoduli} and can't go further. The
counterexample of Remark~\ref{rem:davieLinfty} lies inside its hypotheses and outside
\eqref{eq:daviemoduliinfty}.
\item \label{rem:FHLYW} \emph{Relation to {\cite[\S4.5]{FHL}}}. \cite{FHL} defines weak solutions of RSDEs and proves weak existence for bounded continuous $b,\sigma$
(\cite[Thm.~4.19]{FHL}) and uniqueness in law for bounded Lipschitz $b,\sigma$
(\cite[Thm.~4.21]{FHL}). The second is sometimes interpreted as a Yamada--Watanabe theorem, but it is not.
\cite[Thm.~4.21]{FHL} splits by criticality, where $\gamma$ denotes the
regularity index of the controlled vector field and $\alpha$ the H\"older index of the driver, so that
its $\alpha$ is our $\gamma$ and its $\gamma$ is a smoothness index of the field.
Subcritically, $\gamma>1/\alpha$, it iterates the fixed-point map $\Phi^{T,B}$ of
\cite[Thm.~4.6]{FHL}, whose dependence on the driving Brownian motion is retained, to produce a
measurable functional $\Psi$ with $Y=\Psi(B)$, and uniqueness in law is read off, the law of $\Psi(B)$
not depending on the basis. Critically, $\gamma=1/\alpha$, it mollifies the vector field,
realises the approximating sequence on one common basis by tightness and Skorokhod, and identifies the
limit by the Gy\"ongy--Krylov argument (see e.g.~\cite{GK96}); pathwise uniqueness,
\cite[Thm.~4.10]{FHL}, enters only here, to force two subsequential limits on the same
basis to coincide. Hence the logic runs the opposite way to Yamada--Watanabe: a strong-solution
functional is constructed first and uniqueness in law follows as a corollary. Subcritically no joint
law of two solutions is formed and no filtration enlarged, each of the two given solutions being
identified with $\Psi$ of its own Brownian motion; critically a common basis is used, but for two
subsequential limits of one approximating sequence and not for two given solutions, and again no
enlargement happens. This route is available only where the fixed-point map
converges, i.e.\ under bounded Lipschitz data. Theorem~\ref{thm:YWtransfer} is therefore not a variant
of \cite[Thm.~4.21]{FHL} but the Yamada--Watanabe implication itself, which to our knowledge had not
been established in this setting, and it makes uniqueness in law accessible outside of the Lipschitz class,
where no fixed-point map exists but pathwise uniqueness may still be proved by the modulus argument of
\cite{HuberII}.
\end{enumerate}

\end{remark}

\begin{remark}
    Our Lemma~\ref{lem:roughItoLn} is a variant of
\cite[Thm.~4.13]{FHL} with three hypotheses weakened: the boundedness of $b,\sigma$, which our
setting does not provide, and the controlled class of the integrand pair, which we replace by the
four quantities the proof uses, and the boundedness of $\varphi$, which we replace by bounds
on its derivatives.

At its Taylor step the proof of \cite[Thm.~4.13]{FHL} takes a limit of Riemann sums in $L_{1}$, while
the sewing limit it is matched against is taken in $L_{2}$. Both are limits in probability of one
sequence, so they agree almost surely and the identity of that theorem is almost sure.

{\color{blue}We nevertheless obtain the $L_{2}$ mode itself, in Proposition~\ref{prop:Taylorsew}, by
an argument that avoids the Riemann sum altogether. What it requires in return is that the
integrability index $\mfn $ exceed $2/\gamma$ rather than $4$, the stronger demand since
$\gamma<\tfrac12$; that condition holds at both places.\color{black}}
\end{remark}
\begin{remark}[Which version of {\cite{FHL}} is cited]\label{rem:FHLversion}
\cite{FHL} is not published, stands at its seventh arXiv version, and has renumbered across revisions:
the rough stochastic integral and the continuity of the integration map are Theorem~3.5 and
Corollary~3.6 in version~5 but Theorem~3.4 and Corollary~3.5 from version~6 on, and in the version~3--4 era
the fixed-point theorem was Theorem~4.7, stability Theorem~4.11 and weak solutions \S4.6, against 4.6,
4.9 and \S4.5 now. Versions~6 and~7 carry the same number for every item we cite, version~7
adding one remark at the end of its \S3. 
We cite version~7 throughout.
\end{remark}

\begin{remark}[Two neighbouring papers]\label{rem:YWneighbours}
\cite{Theewis25}, the Yamada--Watanabe--Engelbert theorem for SPDEs in Banach spaces, is the closest
abstract statement to the one used here. It generalises \cite{Kurtz14}'s machine to solutions living in
a Banach space, a contribution on the state space axis and not the axis that matters here. The obstacle Theorem~\ref{thm:YWtransfer} removes is that the solution notion, a family of
conditional bounds \eqref{eq:daviemoduli}, is originally a property of a filtration rather than of a
law, while the Yamada--Watanabe construction enlarges the filtration, and no abstract theorem addresses
that. Indeed, in the applications for which such theorems have previously been given, the notion
refers only to the generated filtrations, hence to the joint law, and the difficulty does not arise. That is so even where
the notion is not a pathwise identity, like for the backward equations treated in \cite{Kurtz14}, whose
solution notion is an identity between conditional expectations on the generated filtrations. What
\cite{Theewis25} would supply, and \cite{Kurtz14} equally provides, is the framework once
Lemmas~\ref{lem:filtdescend}, \ref{lem:lawdet} and \ref{lem:immersion}
are in place. We work in \cite{Kurtz14}'s
because our state space is finite-dimensional and its compatibility notion is the one
Proposition~\ref{prop:gluing} verifies.

\cite{FLZ25}, randomisation of rough stochastic differential equations, bears on the standing
convention of Remark~\ref{rem:fixeddriverI} rather than on the transfer theorem, studying what survives
when the deterministic driver $\RZ$ is replaced by a random rough path. The relevant point for us is negative: the constants of \eqref{eq:Emoment} and \eqref{eq:Epathwise} are of the form
$\exp\{C(1+(1+\|\RZ\|^{1/\gamma}_{\gamma}{{}+\|\CZ\|^{1/\gamma}_{1}})T)\}$ with $1/\gamma\in(2,3)$, hence not integrable against a
Gaussian tail, so none of the quantitative statements of this paper randomises
directly. Nor does Proposition~\ref{prop:gluing}: no property of $\RZ$ is used there, precisely because
$\RZ$ is deterministic and shared by both copies, so that the gluing \eqref{eq:glue} has nothing to
randomise over on the rough side. With a random common driver one would have to glue over
$(x_{0},B,\RZ)$ and verify compatibility of each solution with the pair, and that is the mechanism by
which the McKean--Vlasov setting of \cite{FHLMKV} is harder. We leave open which of these statements survive randomisation.
\end{remark}

\subsection*{Organisation}

\S\ref{sec:setting} fixes the equation and the conventions, the level-two bracket in
\S\ref{subsec:geom} and the second-level coefficient in \S\ref{subsec:gen}. \S\ref{sec:weaksol}
introduces the rough martingale problem and sandwiches it between two pathwise classes;
\S\ref{subsec:MPborel} builds the continuous version of $N^{\varphi}$ that \S\ref{sec:weakex} reads
the problem through, and also shows the solution set Borel and convex; those last two statements are
not quoted later, though the Borel argument is re-run for $\Gamma_{\nu}$ in \S\ref{sec:YWtransfer}.
\S\ref{subsec:graded} refines the problem in graded form.
\S\ref{sec:weakex} proves (A) and \S\ref{sec:YWtransfer} proves (B).
Remark~\ref{rem:YWcompose} composes (A) with (B). In the
appendix, \S\ref{app:Escalefree} and \S\ref{app:MPenvelope} carry two proofs deferred from
\S\ref{sec:weakex} and \S\ref{sec:weaksol}, and \S\ref{app:FHL46} records what an
imported result needs.

The one open problem of the body is stated where it arises; \cite{HuberII} collects those that
concern uniqueness.

\section{Setting}\label{sec:setting}

\subsection{The equation}\label{subsec:equation}

Let us fix $T>0$ and integers $n,d,e\ge1$. Throughout, $\RZ=(Z,\ZZ)\in\CC^{\gamma}([0,T];\R^{e})$ denotes
a deterministic rough path of regularity $\gamma\in(\tfrac13,\tfrac12)$, and $D\subseteq\R^{n}$
denotes a closed set, the state space.

More precisely, we study the following equation,
\begin{equation}\label{eq:rsde}
\d X_{t}=b(t,X_{t})\dt+\sigma(t,X_{t})\,\d B_{t}+(f,f')(t,X_{t})\,\d\RZ_{t},
\qquad X_{0}\sim\nu,\quad X_{t}\in D,
\end{equation}
where $b:[0,T]\times\R^{n}\to\R^{n}$ and $\sigma:[0,T]\times\R^{n}\to\R^{n\times d}$, $B$ denotes a
$d$-dimensional Brownian motion, $(f,f')$ a deterministic controlled vector field, and $\nu$ a Borel
probability measure on $D$.

\emph{Standing regularity.} Throughout the paper, we assume
the following,
\begin{equation}\label{eq:standingcts}
b,\ \sigma,\ f,\ Df\,f+f'\ \text{ are continuous in }(t,x)\ \text{ on }[0,T]\times\R^{n}.
\end{equation}
 Assumption~\ref{ass:Eito} requires continuity of $b$ and $\sigma$ explicitly. For $f$ and
$Df\,f+f'$ it follows from the controlled-path defect of Assumption~\ref{ass:Erough}, which gives
the continuity in $t$, and from \eqref{eq:Eroughgrowth} with the Lipschitz bound \eqref{eq:DGGLip} and the paragraph after it, whose constants are those of Assumption~\ref{ass:Erough} and so bound the untruncated field as well,
which cover both curvature branches and give a
Lipschitz bound in $x$ uniform in $t$. The two together give joint continuity. It enters at step~\ref{stp:imm-b} of
Lemma~\ref{lem:immersion} and in the proof of Lemma~\ref{lem:filtdescend},
where it lets us identify the two stochastic integrals as u.c.p.\ limits of
the same left-point Riemann sums, and in Lemma~\ref{lem:MPborel}, continuity making each Riemann
sum a continuous functional of the canonical path. It is, technically, not necessary for the first two: for a
merely progressively measurable integrand the two integrals still agree, being determined by the predictable
integrand and the integrator, though the left-point Riemann sums no longer converge and a different
argument is needed. Continuity is necessary for the construction in Lemma~\ref{lem:MPborel}.

The following three conventions are used throughout.

\emph{The driver is deterministic and is not part of the stochastic basis.} Only $B$ is random. This is
the setting of \cite{FHL}, in which the rough path is data. Randomising it, that is, replacing $\RZ$ by
the lift of e.g. a fractional Brownian motion, is a separate operation.  
\color{blue} It is performed after every estimate below
and never inside one. Remark~\ref{rem:fixeddriverI} explains why the distinction is important.\color{black}

 When we speak of a system of
$N$ equations we mean $n=N$, or $n=Nd_{0}$ together with a decomposition $X=(X^{1},\dots,X^{N})$ into
blocks. Here, the coefficients may depend on the whole configuration. Every general statement below is
proved on $\R^{n}$ and specialises. 

\emph{The state space is carried as a hypothesis.} $D=\R^{n}$ denotes the plain case, and
$D=\R^{n}_{+}$ a case that a forthcoming paper will address. We do not develop a general
invariance theory here. Still, $D$ is genuinely carried, entering Definitions~\ref{def:davie} and \ref{def:roughMP}
through the requirement that $X$ be $D$-valued and $\nu$ be supported on $D$. The existence
statements below are proved for $D=\R^{n}$, Theorem~\ref{thm:weakex} constructing an $\R^{n}$-valued
solution.

\subsection{Notation and conventions}\label{subsec:rp}\label{subsec:indexconv}

We follow \cite{FH20} for rough paths and \cite{FHL} for rough stochastic
differential equations. The rough path $\RZ=(Z,\ZZ)$ satisfies Chen's relation together with the
finiteness of the homogeneous norm, i.e.
\[
\ZZ_{s,t}-\ZZ_{s,u}-\ZZ_{u,t}=\delta Z_{s,u}\otimes\delta Z_{u,t},
\qquad
\|\RZ\|_{\gamma}:=\|Z\|_{\gamma}+\|\ZZ\|_{2\gamma}<\infty .
\]
The norm is inhomogeneous, so both levels are bounded by its first power,
$|\delta Z_{s,t}|\le\|\RZ\|_{\gamma}|t-s|^{\gamma}$ and
$|\ZZ_{s,t}|\le\|\RZ\|_{\gamma}|t-s|^{2\gamma}$.
Here $|\cdot|$ on $\R^{e}$ and on
$\R^{e\times e}$, denotes the Euclidean
and respectively, the Frobenius norm $\Frob{\cdot}$. Under this convention, a per-index bound contracts against a
driver increment by Cauchy--Schwarz. On a matrix or a multilinear map the two norms we use are written apart. For an
$\R^{m}$-valued $k$-linear map $T$ we denote by
$\Op{T}:=\sup_{|v_{1}|=\dots=|v_{k}|=1}|T(v_{1},\dots,v_{k})|$ the operator norm, which at
$k=2$ and symmetric $T$ is $\sup_{|v|=1}|T(v,v)|$.  Thus $\Op{D^{2}f_{t}(x)}=\sup_{|v|=1}|D^{2}f_{t}(x)(v,v)|$ in
\eqref{eq:Ecurvnormhyp} and likewise in \eqref{eq:Vjet}, while
$\Frob{\sigma}^{2}=\operatorname{tr}(\sigma\sigma^{\top})$ in \eqref{eq:Egrowth}. The
distinction is important. The second bound of \eqref{eq:Vjet} is sharp in the operator norm
and fails in the Frobenius norm for $n\ge2$. A subscript $\infty$ abbreviates $\sup_{x}\Op{A(x)}$ for a
matrix or tensor field $A$, as in $\|D^{2}\varphi\|_{\infty}$, and $\|\cdot\|_{L_{\mfn }}$ applied to a
matrix-valued random variable is the $L_{\mfn }$ norm of its Frobenius norm. A controlled vector
field is a pair
$(f,f')$ with $f:[0,T]\times\R^{n}\to\R^{n\times e}$ and
$f':[0,T]\times\R^{n}\to\R^{n\times e\times e}$ such that, for indices $\beta,\beta'>0$ and a finite
constant $C_{\RZ}$,
\begin{equation}\label{eq:cvfdefect}
\big|f_{t}(x)-f_{s}(x)-f'_{s}(x)\,\delta Z_{s,t}\big|\ \le\ C_{\RZ}\big(1+|x|\big)|t-s|^{\beta+\beta'},
\qquad
\big|f'_{t}(x)-f'_{s}(x)\big|\ \le\ C_{\RZ}\big(1+|x|\big)|t-s|^{\beta'},
\end{equation}
for all $s\le t$ and \emph{all} $x\in\R^{n}$. We write
$(f,f')\in\mathscr D^{\beta,\beta'}_{\RZ}$ for the associated space and denote by $C_{\RZ}$ the
smallest such constant. It is a weighted space, strictly larger than the
deterministic controlled vector fields of \cite[Def.~3.7]{FHL}, whose moduli take an unweighted
supremum over the state: a field whose temporal defect grows linearly in $x$, such as
$f_{t}=t^{\beta+\beta'}g$ with $g$ affine, lies in ours and not in theirs. No
result of \cite{FHL} is applied to $(f,f')$ itself, only to the truncated pair $(f_{R},f'_{R})$ of
Step~\ref{step:Ereg}, whose unweighted constant is finite at each fixed $R$, which is where
\cite{FHL} is applied.

For a path $Y$ and $s\le u\le t$ we denote by $\delta Y_{s,t}:=Y_{t}-Y_{s}$ the increment, and for a
two-parameter family $A$ by
\[
\delta A_{s,u,t}:=A_{s,t}-A_{s,u}-A_{u,t}
\]
the second-order defect. Further $\E_{s}$ is the conditional expectation
$\E[\,\cdot\mid\mathcal F_{s}]$.
By $\mathcal A\perp\!\!\!\perp\mathcal B\mid\mathcal C$ we denote conditional independence of the
$\sigma$-fields $\mathcal A$ and $\mathcal B$ given $\mathcal C$. We denote by $[\![u,v]\!]$ the
stochastic interval $\{(\omega,r):u(\omega)\le r\le v(\omega)\}$. In case the same delimiters carry a subscript, we mean the H\"older seminorm of \cite[Def.~3.7]{FHL}, as in $[\![\delta Df]\!]_{\beta'}$.

For tensors, we write $u\otimes v$ for the outer product, $S\!:\!T:=\sum_{i,j}S_{ij}T_{ij}$ for the full
contraction of two matrices, so that $D^{2}\varphi\!:\!(u\otimes v)=D^{2}\varphi(u,v)$. A dot contracts the single index the two factors share,
as in $f'^{\kappa\lambda}\cdot DV=\sum_{i}f'^{\kappa\lambda}_{i}\partial_{i}V$, and ordinary multiplication
between scalars. Angle brackets $\langle\,\cdot\,,\cdot\,\rangle$ denote the same
pairing written the other way round, except in $\langle M\rangle$ and $\langle M,N\rangle$, which denote
quadratic variation, and in $\langle\nu,\varphi\rangle=\int\varphi\,\d\nu$. We denote by
$\operatorname{Sym}$ the symmetric part of a matrix, and by
$\operatorname{id}$ the identity map, written $\idx$ when its argument is the state variable. For
symmetric matrices, $a\ge0$ and $a\succeq0$ both denote positive semi-definiteness.

 $x\lesssim y$ means $x\le Cy$ with a constant $C$ that is fixed in the surrounding argument, and
$\lesssim_{a}$ means that $C$ may depend on $a$. We write $x\asymp y$ when
$x\lesssim y\lesssim x$.

 The letters $i,j$ and $\kappa,\lambda$ denote coordinates only when carried as a
sub- or superscript on a field, a germ coefficient, an operator or a driver increment, Latin $i,j$ then
running over $1,\dots,n$ and Greek $\kappa,\lambda$ over $1,\dots,e$; elsewhere they are ordinary
symbols. Thus $\kappa$ denotes the exponent $\kappa:=\max\{k,3\}$ of
Proposition~\ref{prop:gradedequiv}. The capital $\Gamma$ is reused: it denotes the sewing constants
$\Gamma_{1},\Gamma_{2},\Gamma_{3}$ of Lemma~\ref{lem:SSL}, the germ coefficients
$\Gamma^{1,\kappa},\Gamma^{2,\kappa\lambda},\Gamma^{3,\kappa\lambda}$ of \eqref{eq:Vgerm}, and the
constraint set $\Gamma_{\nu}$ of Theorem~\ref{thm:YWtransfer} and
\cite{FHL}'s $\Gamma^{\gamma,\beta_{\star};m,\mfn}$, the sum of the four quantities named in
Step~\ref{step:Etight}. The derived field
$\mathsf G_{t}:=Df_{t}f_{t}+f'_{t}$ carries a sans-serif capital, the running quantity $\mathsf S$ of
\eqref{eq:Thetaatoms} is unrelated. Calligraphic $\mathcal G$ is
the generator of \S\ref{subsec:gen} where, together with the traditional SDE generator $\mathcal L$, acts on a test function, or
carries a Greek coordinate, a bracketed level or a power. Elsewhere it is a $\sigma$-field, with
$\mathcal G_{t}$, $\mathcal G^{j}$ and $\mathcal G^{\sharp}$ of \S\ref{sec:weakex}
included. The state dimension is $n$,
$x\in\R^{n}$, throughout. The outer index of the mixed norm $\|\cdot\|_{q,\mfn}$ of
\S\ref{subsec:rp}, and of the classes $\mathscr D^{\beta,\beta'}_{Z}L_{4,\mfn}$ and
$\Dext^{\beta,\beta'}_{Z}L_{4,\mfn}$ of \cite[Def.~3.1/4.12]{FHL}, is the separate letter $\mfn$.
It is $\mfn$ that \eqref{eq:rItoLn}, \eqref{eq:paramsIto} and \eqref{eq:paramsdischarge} constrain. The number of particles is $n$ in
Example~\ref{ex:Adiag}, where it is the state dimension as well. Every sum is
printed: a repeated index with no summation sign before it is free, and so is a Greek index carried by
a coefficient with no driver increment in the same term, the object being one object for each value of
the index or pair. Contracting such a per-index bound against $\delta Z^{\kappa}$,
$\ZZ^{\kappa\lambda}$ or $\CZ^{\kappa\lambda}$ costs, by Cauchy--Schwarz, a factor of at most $\sqrt e$
at level one and $e$ at level two, and costs nothing where the per-index bounds hold in $\ell^{2}$ over
the Greek indices with the same constant, as throughout \cite{HuberII} by that paper's Lipschitz bound. A
germ coefficient written without indices, e.g.~$\Gamma^{1}\delta Z$, denotes the full contraction.

\emph{The order of the pair.} $\kappa$ indexes the first factor of $\RZ\otimes\RZ$, that is,
$\ZZ^{\kappa\lambda}_{s,t}$ denotes the iterated integral of $\d Z^{\kappa}$ against $\d Z^{\lambda}$ in
that order, and against it \eqref{eq:davie} carries $Df^{\lambda}f^{\kappa}+f'^{\kappa\lambda}$. For an affine field $f^{\kappa}(x)=F^{\kappa}x$ this is the matrix
$F^{\lambda}F^{\kappa}+F'^{\kappa\lambda}$, and on functions it is
$\mathcal G^{\kappa}\mathcal G^{\lambda}+\mathcal G'^{\kappa\lambda}$, with $\mathcal G^{\lambda}$
applied to $\varphi$ first; the flow lemma of \cite{HuberII} is the one place where a coefficient reverses the
order. 

\begin{remark}\label{rem:fprimeI}
$f'$ denotes the Gubinelli derivative of the coefficient in time, and constrains nothing regarding
the dependence of $f$ on the state. The condition $f'\equiv0$ says that the rough vector field has no
first-order rough variation in time. $f_{t}(x)=f(x)$ implies it but not
conversely: since $|t^{a}-s^{a}|\le|t-s|^{a}$ for $a\in(0,1]$, a field
$f_{t}(x)=t^{a}g(x)$ with $g$ affine and $a=\beta+\beta'<1$ has $f'\equiv0$ and is not
autonomous.
\end{remark}

\begin{definition}[Davie solution; {\cite[Def.~4.2(c)]{FHL}}]\label{def:davie}
Let $(\Omega,\mathcal A,(\mathcal F_{t}),\Prob)$ be a stochastic basis carrying an
$(\mathcal F_{t})$-Brownian motion $B$ such that $B_{0}=0$, and let $X$ denote a continuous, $(\mathcal F_{t})$-adapted process with
values in $D$. We denote by $J$ the Davie remainder, i.e.
\begin{equation}\label{eq:davie}
J_{s,t}:=\delta X_{s,t}-\int_{s}^{t}\!b_{r}(X_{r})\dr-\int_{s}^{t}\!\sigma_{r}(X_{r})\,\d B_{r}
-f_{s}(X_{s})\,\delta Z_{s,t}-\big(Df_{s}f_{s}+f'_{s}\big)(X_{s})\,\ZZ_{s,t}.
\end{equation}
For $m\ge2$, we call $X$ an \emph{$L_{m}$-conditional solution} of \eqref{eq:rsde} if the two moduli
below hold, together with the integrability of the third clause,
\begin{equation}\label{eq:daviemoduli}
\begin{gathered}
\big\|\E[J_{s,t}\mid\mathcal F_{s}]\big\|_{L_{m}(\Prob)}=o(|t-s|),
\qquad
\big\|\E\big[|J_{s,t}|^{m}\mid\mathcal F_{s}\big]^{1/m}\big\|_{L_{m}(\Prob)}
=\big\|J_{s,t}\big\|_{L_{m}(\Prob)}=o(|t-s|^{1/2}),\\
J_{s,t}\in L_{m}(\Prob)\ \text{ at every pair}.
\end{gathered}
\end{equation}
The two moduli are required uniformly in $s\le t$, and not merely pointwise in $(s,t)$. They
give no bound on $\sup_{s\le t}\|J_{s,t}\|_{L_{m}}$, because the rate constrains the modulus only
near $0$. Replacing $L_{m}(\Prob)$ in the two moduli by $L_{\infty}(\Omega)$,
\begin{equation}\label{eq:daviemoduliinfty}
\big\|\E[J_{s,t}\mid\mathcal F_{s}]\big\|_{L_{\infty}(\Omega)}=o(|t-s|),
\qquad
\big\|\E\big[|J_{s,t}|^{m}\mid\mathcal F_{s}\big]^{1/m}\big\|_{L_{\infty}(\Omega)}=o(|t-s|^{1/2}),
\end{equation}
the third clause of \eqref{eq:daviemoduli} unchanged, defines the strictly stronger \emph{$L_{m,\infty}$-solution}. In the numbering of \cite{FHL},
\eqref{eq:daviemoduli} is condition~(c) of \cite[Def.~4.2]{FHL} at the mixed index $\mfn =m$, the mixed
norm reducing to the $L_{m}$ norm there by \cite[Prop.~2.3(i)]{FHL}, and \eqref{eq:daviemoduliinfty} is
the same condition at $\mfn =\infty$. \cite{FHL} states the two moduli at every pair $(s,t)$; we take them
uniformly in $s\le t$.
\end{definition}

\begin{remark}[The $L_{m,\infty}$ notion is not available at linear growth]\label{rem:davieLinfty}
There is a solution of \eqref{eq:rsde} satisfying every hypothesis of Theorem~\ref{thm:weakex} such that
\eqref{eq:daviemoduliinfty} fails at every pair $s<t$ with $\delta Z_{s,t}\ne0$, whereas
\eqref{eq:daviemoduli} holds at every $m$. Take e.g.~$n=e=d=1$, $b=\sigma=0$, $D=\R$, $f(x)=x$, $f'=0$,
$\nu=N(0,1)$, and any non-constant $\gamma$-H\"older $Z$ such that the lift is geometric,
$\ZZ_{s,t}=\tfrac12\delta Z^{2}_{s,t}$. Assumption~\ref{ass:Eito} holds trivially, $b$ and $\sigma$
vanishing; Assumption~\ref{ass:Erough} holds in branch \ref{Ecurvaffine}, the field being affine, so
that $D^{2}f\equiv0$ and $C_{\mathrm{nrm}}=0$ in \eqref{eq:Ecurvnormhyp}, and $C_{\RZ}=0$; and
Assumption~\ref{ass:Einit} holds at every $p$, $\nu$ being Gaussian. The equation is solved by
$X_{t}=X_{0}\exp\big(\delta Z_{0,t}\big)$, and since $(Df\,f+f')(x)=x$ we obtain the following
expression for the Davie remainder \eqref{eq:davie}
\begin{equation}\label{eq:davieCE}
J_{s,t}\ =\ X_{s}\Big(e^{\delta Z_{s,t}}-1-\delta Z_{s,t}-\tfrac12\delta Z^{2}_{s,t}\Big).
\end{equation}
Here $J_{s,t}$ is $\mathcal F_{s}$-measurable, since $Z$ is deterministic and no Brownian motion
enters the equation, so that $\E[J_{s,t}\mid\mathcal F_{s}]=J_{s,t}$. Both moduli of
\eqref{eq:daviemoduliinfty} then equal $\|X_{s}\|_{L_{\infty}(\Omega)}$ times a nonzero scalar, hence
$+\infty$ at every pair with $\delta Z_{s,t}\ne0$, $X_{s}$ being Gaussian and therefore unbounded on
$\Omega$.

It is indeed the essential supremum and not the rate that fails. A Taylor expansion of the
exponential yields the following for the scalar in \eqref{eq:davieCE}
\[
\Big|e^{\delta Z_{s,t}}-1-\delta Z_{s,t}-\tfrac12\delta Z^{2}_{s,t}\Big|
\ =\ O\big(|\delta Z_{s,t}|^{3}\big)\ =\ O\big(|t-s|^{3\gamma}\big),
\qquad 3\gamma>1 .
\]
Hence, \eqref{eq:daviemoduli} holds at every $m$, at a rate strictly better than what either modulus requires. \color{blue}What fails is that $J$ carries a coefficient factor evaluated along the path, here $X_{s}$ itself,
which is unbounded with respect to the essential supremum over $\Omega$. With the boundedness of $b$ and $\sigma$ and the controlled pair in $\mathcal C^{\varkappa}_{b}$ neither germ nor
remainder carries an unbounded state factor and 
\cite[Thm.~4.19]{FHL} assumes  bounded and \eqref{eq:daviemoduliinfty} is available for \cite[Thm.~4.6]{FHL}.\color{black}
\end{remark}

\begin{remark}[Stopping does not produce a solution of the truncated equation]\label{rem:stoppednosol}
\color{blue} In the following, we need to appeal to results stated for bounded coefficients.\color{black} Truncating the coefficient fields outside a ball $\{|x|\le R\}$ and stopping the
solution at the exit time $\tau_{R}$ is natural, but fails over an interval
straddling $\tau_{R}$: the left-hand side of \eqref{eq:davie} is then the truncated increment
$X_{\tau_{R}}-X_{s}$, while the germ on the right is still built from the full increments
$\delta Z_{s,t}$ and $\delta B_{s,t}$. \color{blue}The mismatch is an increment of the driver over
$[\tau_{R},t]$, of size $|t-\tau_{R}|^{\gamma}$ and not of the size \eqref{eq:daviemoduli} requires of a
remainder.\color{black} We present one case for which it fails, which is enough to not continue in this direction. Let
$\gamma\in(\tfrac13,\tfrac12)$, $T=2$, $n=e=d=1$, $D=\R$, $b=\sigma=0$, $f\equiv1$, $f'=0$, $X_{0}=0$,
and $Z_{t}:=t\ \ (t\le1)$, $Z_{t}:=1+(t-1)^{\gamma}\ \ (t\ge1)$, $\ZZ_{s,t}:=\tfrac12\delta Z^{2}_{s,t}$.
Then $Z$ is $\gamma$-H\"older and $X_{t}=Z_{t}$ solves \eqref{eq:davie} with $J\equiv0$. Truncate at
$R=1$, replacing $f$ by $\chi$ with $\chi=1$ on $[-1,1]$ and $\chi=0$ off $[-2,2]$, and stop at
$\tau_{1}=1$, so that $X^{\tau_{1}}_{t}=t\wedge1$. For $s<1<t$ one has $\chi(X^{\tau_{1}}_{s})=1$ and
$\chi'(X^{\tau_{1}}_{s})=0$, whence $J^{1}_{s,t}=(1-s)-\delta Z_{s,t}=-(t-1)^{\gamma}$. At
$s=1-\varepsilon$, $t=1+\varepsilon$ this gives $|J^{1}_{s,t}|=\varepsilon^{\gamma}$ at $|t-s|=2\varepsilon$, and since
$\gamma<\tfrac12<1$ both quotients $\varepsilon^{\gamma}/(2\varepsilon)^{1/2}$ and
$\varepsilon^{\gamma}/(2\varepsilon)$ diverge as $\varepsilon\downarrow0$; everything being
deterministic, both moduli of \eqref{eq:daviemoduli} are evaluated on this quantity. Hence
$X^{\tau_{1}}$ is not an $L_{m}$-conditional solution of the truncated equation at any $m$.
\end{remark}

We use the following elementary fact repeatedly: we may take an It\^o martingale to have zero Gubinelli derivative
\cite[Rem.~3.2]{FHL}, so it has no first- or second-order coefficient against $\RZ$
and contributes no cross term to a germ against $\RZ$. $\RZ$ is deterministic
throughout, so no independence hypothesis is involved.

\begin{definition}[Weak solution]\label{def:weaksol}
A weak solution of \eqref{eq:rsde} with initial law $\nu$ is a tuple
$\mathfrak X=(\Omega,\mathcal A,(\mathcal F_{t}),\Prob,X,B)$ such that $B$ is an
$(\mathcal F_{t})$-Brownian motion in $\R^{d}$, $X$ is a continuous $(\mathcal F_{t})$-adapted
$D$-valued process with $\Law(X_{0})=\nu$, and $X$ is an $L_{m}$-conditional solution of
\eqref{eq:rsde} for the given deterministic $\RZ$ and some $m\ge2$. We call it $L_{m}$-integrable if in addition
the following moment bound holds,
\[
\big\|\sup_{t\le T}|X_{t}|\big\|_{L_{m}(\Prob)}<\infty .
\]
\end{definition}

Definition~\ref{def:weaksol} is \cite[\S4.5]{FHL}'s notion at mixed index $\mfn =m$ rather than $\mfn =\infty$, and
guarantees only condition~(c) of \cite[Def.~4.2]{FHL}: condition~(a), the almost sure finiteness of
$\int_{0}^{T}|b_{r}(X_{r})|\dr$ and $\int_{0}^{T}|(\sigma\sigma^{\top})_{r}(X_{r})|\dr$, follows from
\eqref{eq:standingcts} and the continuity of $X$. Condition~(b), membership of $\big(f(X),Df\,f+f'(X)\big)$ in
$\mathscr D^{\bar\alpha,\bar\alpha'}_{Z}L_{m,\mfn }$ with $\alpha+(\alpha\wedge\bar\alpha)>\tfrac12$ and
$\alpha+(\alpha\wedge\bar\alpha)+\bar\alpha'>1$, $\alpha=\gamma$ the driver index, is not
guaranteed by Definition~\ref{def:weaksol}. We do not use condition~(b) for a solution in the sense of
Definition~\ref{def:weaksol}. Step~\ref{step:Etight} of the proof of
Theorem~\ref{thm:weakex} records it at the limit, where only the right-continuity in $s$ is left
open, only in order to place our solutions among \cite{FHL}'s.

\begin{remark}\label{rem:fixeddriverI}
Every quantitative estimate below carries a factor
$\exp\{C(1+\Lambda^{1/\gamma}_{\RZ}T)\}$ with
$\Lambda_{\RZ}:=1+\|\RZ\|_{\gamma}+\|\CZ\|_{1}$. Since
$\Lambda^{1/\gamma}_{\RZ}\le3^{1/\gamma-1}\big(1+\|\RZ\|^{1/\gamma}_{\gamma}+\|\CZ\|^{1/\gamma}_{1}\big)$,
a larger $C$ allows the sum in brackets in place of $\Lambda^{1/\gamma}_{\RZ}$. Its last term is
$0$ for a geometric lift. The factor is produced by a partition whose mesh $h$ satisfies $c\,\Lambda_{\RZ}h^{\gamma}\le\tfrac12$, of
which there are $O(1+\Lambda^{1/\gamma}_{\RZ}T)$; see \eqref{eq:Epathwise}.
For a deterministic driver this factor is a constant; after randomisation it need not be integrable. Hence, the randomisation must be
performed afterwards.
\end{remark}

\subsection{Geometricity, and the level-two bracket}\label{subsec:geom}

\begin{assumption}[Driver]\label{ass:driver}
$\RZ\in\CC^{\gamma}([0,T];\R^{e})$ with $\gamma\in(\tfrac13,\tfrac12)$, which enters
everywhere below through the Davie expansion \eqref{eq:davie}, which carries $\ZZ$ explicitly. The bracket is
$[\RZ]_{s,t}:=\delta Z_{s,t}\otimes\delta Z_{s,t}-2\operatorname{Sym}\ZZ_{s,t}$, and we assume
throughout that it is \emph{Lipschitz},
\begin{equation}\label{eq:bracketLip}
\|\CZ\|_{1}:=\sup_{s<t}\frac{\big|\CZ_{s,t}\big|}{|t-s|}\ <\ \infty ,
\qquad\text{that is, the bracket is of H\"older order }1 .
\end{equation}
By Chen's relation $\CZ$ is additive, $\CZ_{s,u}+\CZ_{u,t}=\CZ_{s,t}$, so
\eqref{eq:bracketLip} makes $\CZ$ the increment of a Lipschitz path and every integral against $\CZ$
below a Lebesgue--Stieltjes integral.
\end{assumption}

The Lipschitz hypothesis is no restriction in two cases of interest:
\[
\CZ\equiv0\quad\text{for a geometric lift},
\qquad
\CZ_{s,t}=(t-s)I\quad\text{for the It\^o lift}.
\]
\color{blue}
Note that it makes the bracket term of Definition~\ref{def:roughMP} a drift rather than a second
source of roughness.
\color{black}
Geometricity is nowhere assumed in this paper. Every result below is proved for an arbitrary
level-two lift satisfying Assumption~\ref{ass:driver}, that is, under \eqref{eq:bracketLip} alone.

\color{blue}
\begin{remark}
\eqref{eq:bracketLip} constrains the lift and not the path $Z$, and it leaves the intended
range of drivers untouched.

\begin{enumerate}
    \item Geometric lifts for fractional Brownian motion are trivially included.
    \item Lipschitz corrections of geometric lifts are included: Fix a
geometric $\RZ=(Z,\ZZ)$ and a symmetric matrix $C\in\R^{e\times e}$, and put
$\ZZ^{C}_{s,t}:=\ZZ_{s,t}-\tfrac12(t-s)C$.
\item What is excluded for convenience, but can be adapted rather easily: finite variation instead of Lipschitz. The It\^o lift of a continuous
semimartingale whose quadratic variation is not absolutely continuous with bounded density is of this
kind. For instance $Z:=\int_{0}^{\cdot}r^{-\epsilon}\d W_{r}$ with
$\epsilon\in(0,\tfrac16)$ and $W$ a Brownian motion has
$\CZ_{s,t}=\langle Z\rangle_{t}-\langle Z\rangle_{s}$ with
$\langle Z\rangle_{t}=t^{1-2\epsilon}/(1-2\epsilon)$, increasing and continuous but of unbounded
difference quotient at $0$, while $Z$ is $\gamma$-H\"older for every $\gamma<\tfrac{1-2\epsilon}2$, an
interval meeting $(\tfrac13,\tfrac12)$ exactly when $\epsilon<\tfrac16$. Every $\CZ$-integral remains Lebesgue--Stieltjes, and the arguments go
through with $\|\CZ\|_{1}|t-s|$ replaced by the variation $|\CZ|_{1\text{-var};[s,t]}$ and the uniform
greedy mesh by a greedy partition of that control. We do not carry this out.
\item Excluded: Infinite variation. The $\CZ$-integrals would need sewing, and \cite{HuberII} would lose the semimartingale
structure it reads $\widetilde X$ through. 
\end{enumerate}
\end{remark}
\color{black}
\subsection{The second-level coefficient}\label{subsec:gen}

For \eqref{eq:rsde} with $\RZ\in\CC^{\gamma}([0,T];\R^{e})$ and $(f,f')$ a deterministic controlled
vector field, we define the following operators on $C^{\infty}(\R^{n})$,
\begin{equation}\label{eq:gendef}
\mathcal L_{t}\varphi:=\tfrac12\sum_{i,j}a_{ij}(t,x)\partial^{2}_{ij}\varphi+\sum_{i}b_{i}(t,x)\partial_{i}\varphi,
\quad
\mathcal G^{\kappa}_{t}\varphi:=\sum_{i}f^{\kappa}_{i}(t,x)\partial_{i}\varphi,
\quad
\mathcal G'^{\kappa\lambda}_{t}\varphi:=\sum_{i}f'^{\kappa\lambda}_{i}(t,x)\partial_{i}\varphi,
\end{equation}
for $\kappa,\lambda=1,\dots,e$, where $a:=\sigma\sigma^{\top}$. We call the triple
$(\mathcal L,\mathcal G,\mathcal G')$ the rough generator of \eqref{eq:rsde}.

\begin{lemma}[The second-level coefficient is $\mathcal G^{\kappa}\mathcal G^{\lambda}+\mathcal G'^{\kappa\lambda}$]\label{lem:G2}
Let $X$ be an $L_{m}$-conditional solution of \eqref{eq:rsde} for some $m\ge2$, and
let $f_{t}\in C^{1}$.
For $\varphi\in C^{3}$ the Davie expansion of $\varphi(X)$ has first-level coefficient
$\mathcal G^{\kappa}_{s}\varphi$ against $\delta Z^{\kappa}_{s,t}$ and the following second-level coefficient
\begin{equation}\label{eq:G2}
\Gamma^{(2),\kappa\lambda}_{s}\varphi
=\big(\mathcal G^{\kappa}_{s}\mathcal G^{\lambda}_{s}+\mathcal G'^{\kappa\lambda}_{s}\big)\varphi
\end{equation}
against $\ZZ^{\kappa\lambda}_{s,t}$. For a non-geometric lift, there is one
further second-level term, namely
\begin{equation}\label{eq:G2bracket}
\tfrac12\big(\mathcal Q^{\kappa\lambda}_{s}\varphi\big)(X_{s})\,\CZ^{\kappa\lambda}_{s,t},
\qquad
\mathcal Q^{\kappa\lambda}_{s}\varphi=\textstyle\sum_{i,j}\big(f^{\kappa}_{s}\big)_{i}
\big(f^{\lambda}_{s}\big)_{j}\,\partial^{2}_{ij}\varphi ,
\end{equation}
which is generated by \eqref{eq:brackettaylor}. It vanishes for every $\varphi$
when $\CZ\equiv0$, and for every lift when $D^{2}\varphi\equiv0$. The three terms together are the germ \eqref{eq:MPgerm}.
\end{lemma}

\begin{proof}
Insert $\delta X_{s,t}=\int_{s}^{t}b\dr+\int_{s}^{t}\sigma\d B+\textstyle\sum_{\kappa}f^{\kappa}_{s}(X_{s})\delta Z^{\kappa}_{s,t}
+\textstyle\sum_{\kappa,\lambda}(Df^{\lambda}_{s}f^{\kappa}_{s}+f'^{\kappa\lambda}_{s})(X_{s})\ZZ^{\kappa\lambda}_{s,t}+J_{s,t}$ into the
Taylor expansion of $\varphi(X_{t})-\varphi(X_{s})$ and collect the terms of order $\ZZ$. The
first-order term contributes
$D\varphi\cdot(Df^{\lambda}f^{\kappa}+f'^{\kappa\lambda})$, and the second-order term contributes
$(f^{\kappa}\otimes f^{\lambda}):D^{2}\varphi$. Here we used
$\delta Z^{\kappa}\delta Z^{\lambda}=\ZZ^{\kappa\lambda}+\ZZ^{\lambda\kappa}$ for a geometric lift,
respectively the corresponding bracket-corrected identity \eqref{eq:brackettaylor} otherwise. Their sum
yields the following identity,
\[
\mathcal G^{\kappa}\mathcal G^{\lambda}\varphi+\mathcal G'^{\kappa\lambda}\varphi
=\underbrace{\sum_{i,j}f^{\kappa}_{i}(\partial_{i}f^{\lambda}_{j})\partial_{j}\varphi
+\sum_{i,j}f^{\kappa}_{i}f^{\lambda}_{j}\partial^{2}_{ij}\varphi}_{=\ \sum_{i}f^{\kappa}_{i}\partial_{i}(\sum_{j}f^{\lambda}_{j}\partial_{j}\varphi)}
\ +\ \sum_{i}f'^{\kappa\lambda}_{i}\partial_{i}\varphi .
\qedhere
\]
\end{proof}

\begin{remark}\label{rem:G2special}
In the autonomous case $f'\equiv0$ the coefficient \eqref{eq:G2} is
$\mathcal G^{\kappa}\mathcal G^{\lambda}$, and for $e=1$ simply $\mathcal G^{2}$.
The index order is that of \S\ref{subsec:indexconv}.
\end{remark}
\color{blue}
 There are three objects controlled by $\ZZ$ and each has its own Gubinelli derivative. The first is the map
$t\mapsto f_{t}(x)$, with increment $f'_{s}(x)\delta Z_{s,t}$ at leading order, so its Gubinelli
derivative is $f'$.

The second is the solution and by Definition~\ref{def:davie} the rough part of $\delta X_{s,t}$ is
$f_{s}(X_{s})\delta Z_{s,t}$ at leading order, so the Gubinelli derivative of $X$ is $f(X)$.

The third is the field along the path. The map
$t\mapsto f_{t}(X_{t})$ has Gubinelli derivative $Df\,f+f'$, which is exactly the coefficient of
$\ZZ_{s,t}$ in \eqref{eq:davie}.

\color{black}

Let us restate the central Lemma~\ref{lem:SSL}, and introduce the truncation
Definition~\ref{def:approxfamily}.

\begin{lemma}[Stochastic sewing; {\cite[Thm.~2.1]{Le20}}, {\cite[Thm.~2.9, Lem.~2.11]{FHL}}]\label{lem:SSL}
Let $A$ denote a two-index $L_{m}$-integrable field, $m\ge2$, such that $A_{s,s}=0$ and such that $A$ is
adapted, that is, $A_{s,t}$ is
$\mathcal F_{t}$-measurable. Suppose further that the following two defect bounds hold,
\[
\|\E_{s}\delta A_{s,u,t}\|_{L_{m}}\le\Gamma_{1}|t-s|^{\epsilon_{1}},
\qquad
\|\delta A_{s,u,t}\|_{L_{m}}\le\Gamma_{2}|t-s|^{\epsilon_{2}},
\]
for \emph{deterministic}
$\Gamma_{1},\Gamma_{2}$ and $\epsilon_{1}>1$, $\epsilon_{2}>\tfrac12$. Then there is, up to modification, exactly
one process $\mathcal A$ such that $\mathcal A$ is adapted, $\mathcal A_{0}=0$ and
$\mathcal A_{t}-A_{0,t}$
lies in $L_{m}$ for every $t$, and such that, for all $s\le t$ and for \emph{some} finite constants in place of $C\Gamma_{1}$ and $C\Gamma_{2}$,
\begin{equation}\label{eq:SSLconc}
\big\|\mathcal A_{t}-\mathcal A_{s}-A_{s,t}\big\|_{L_{m}}
\ \le\ C\big(\Gamma_{1}|t-s|^{\epsilon_{1}}+\Gamma_{2}|t-s|^{\epsilon_{2}}\big),
\qquad
\big\|\E_{s}\big[\mathcal A_{t}-\mathcal A_{s}-A_{s,t}\big]\big\|_{L_{m}}
\ \le\ C\,\Gamma_{1}|t-s|^{\epsilon_{1}},
\end{equation}
with $C=C(\epsilon_{1},\epsilon_{2},m)$. The process that exists satisfies the two
bounds with these constants; uniqueness holds in the wider class, where the two constants are only
required to be finite. Moreover, for each fixed $t$ the compensated Riemann
sums $A^{\pi}_{t}:=\sum_{[u,v]\in\pi,\,u\le t}A_{u,v\wedge t}$ converge to $\mathcal A_{t}$ in $L_{m}$ as the
mesh $|\pi|$ tends to $0$.

Suppose moreover that $t\mapsto A_{s,t}$ is almost
surely c\`adl\`ag \textup{(}respectively continuous\textup{)} on $[s,T]$ for each $s$, and that the
following third defect bound holds,
\begin{equation}\label{eq:SSLthird}
\Big\|\ \sup_{r\in[(s+t)/2,\,t]}\big|\delta A_{s,(s+t)/2,r}\big|\ \Big\|_{L_{m}}
\ \le\ \Gamma_{3}\,|t-s|^{\frac1m+\epsilon_{3}}
\qquad(s\le t)
\end{equation}
for finite $\Gamma_{3}\ge0$ and $\epsilon_{3}>0$. Then $\mathcal A$ has a c\`adl\`ag
\textup{(}respectively continuous\textup{)} version, and for that version the convergence just stated
holds along any sequence of partitions of mesh tending to $0$ \emph{uniformly} in $t\in[0,T]$ and at
the rate $|\pi|^{(\epsilon_{1}-1)\wedge(\epsilon_{2}-\frac12)\wedge\epsilon_{3}}$.
\end{lemma}
We ask for integrability of $A$ itself, rather than of $\delta A$
as in \cite[Thm.~2.9]{FHL}. This is a deliberate simplification and holds at every application
below.
The first half of the lemma is \cite[Thm.~2.1]{Le20}. At $m=\mfn $ the mixed norm reduces to the $L_{m}$
norm by \cite[Prop.~2.3(i)]{FHL}, and the lemma is then \cite[Thm.~2.9(i)]{FHL} together with
\cite[Lem.~2.11]{FHL} for the Riemann-sum part.

One point is stated more strongly than in \cite{FHL}. They assert uniqueness among
processes obeying \eqref{eq:SSLconc} with the constants produced by the theorem, while we assert
it among those obeying it at some finite constants. That widening is what
\ref{proof:Eproof:step2:step1:1c} and Step~\ref{stp:rIto4} use, so we give a short proof. Let $\mathcal A$ and
$\widetilde{\mathcal A}$ both meet the conclusion, and put
$D_{s,t}:=(\mathcal A_{t}-\mathcal A_{s})-(\widetilde{\mathcal A}_{t}-\widetilde{\mathcal A}_{s})$,
which is additive and adapted and obeys $\|\E_{s}D_{s,t}\|_{L_{m}}\le c_{1}|t-s|^{\epsilon_{1}}$ and
$\|D_{s,t}\|_{L_{m}}\le c_{2}|t-s|^{\epsilon_{1}\wedge\epsilon_{2}}$, the second exponent being
$\epsilon_{1}\wedge\epsilon_{2}>\tfrac12$ and not $\epsilon_{2}$, since each of the two processes
obeys both bounds of \eqref{eq:SSLconc}. Split $D_{0,T}$ over the uniform mesh
$t_{i}=iT/N$ and each summand as $\E_{t_{i}}D_{t_{i},t_{i+1}}+(1-\E_{t_{i}})D_{t_{i},t_{i+1}}$. The
first parts sum to at most $c_{1}T^{\epsilon_{1}}N^{1-\epsilon_{1}}$ in $L_{m}$. The second parts are
martingale differences, so Burkholder--Davis--Gundy and Minkowski in $L_{m/2}$, at $m\ge2$,
bound their sum by
$C_{m}c_{2}T^{\epsilon_{1}\wedge\epsilon_{2}}N^{\frac12-(\epsilon_{1}\wedge\epsilon_{2})}$. Both
tend to $0$, so
$D_{0,T}=0$ almost surely, and likewise at every pair.

\begin{definition}[Truncation]\label{def:approxfamily}
Let us fix $\Theta_{R}\in C^{\infty}_{c}(\R^{n})$ such that $\Theta_{R}=1$ on $\{|x|\le R\}$,
supported in $\{|x|\le2R\}$,$0\le\Theta_{R}\le1$
and $\|D^{j}\Theta_{R}\|_{\infty}=O(R^{-j})$, and denote by
\[
b_{R}:=\Theta_{R}b,\qquad \sigma_{R}:=\Theta_{R}\sigma
\]
the truncated It\^o data. Step~\ref{step:Ereg} applies the same weight to the
mollified data and denotes it by $b^{k}_{R},\sigma^{k}_{R}$. We write $(f_{R},f'_{R})$ for the truncation of the rough
field, which is the global $\Theta_{R}\,\cdot\,$ under \ref{Ecurvaffine} and a coordinatewise cut-off
under \ref{Ecurvdiag} of Assumption~\ref{ass:Erough}, both with values in $[0,1]$.
Step~\ref{step:Ereg} fixes it and says why the two cases differ. Since the weights lie in $[0,1]$, the truncated coefficients
inherit the growth bound \eqref{eq:Egrowth} and the third bound of \eqref{eq:Eroughgrowth} with the
\emph{same} constant $C_{\mathrm{gr}}$, i.e.~$|b_{R}|+\|\sigma_{R}\|\le|b|+\|\sigma\|\le
C_{\mathrm{gr}}(1+|x|)$ and $|f_{R}|+|f'_{R}|\le|f|+|f'|\le C_{\mathrm{gr}}(1+|x|)$ at every
$(t,x)\in[0,T]\times\R^{n}$. When the It\^o data are only H\"older, we can mollify $b$ and $\sigma$
first.
\end{definition}
\color{blue}
\begin{remark}\label{rem:twotruncations}
Two truncation operators occur in this paper, in the same role and for different reasons. The global truncation $\,\cdot\,\mapsto\Theta_{R}\,\cdot\,$ of
Definition~\ref{def:approxfamily}, subscripted $R$, is the one Step~\ref{step:Ereg} applies to $b$ and $\sigma$ and, under
branch \ref{Ecurvaffine}, to $(f,f')$; it is chosen for compatibility with the flow reduction, $f_{R}$
being compactly supported, so that the transformed problem is posed on a bounded region and
\cite{DFJS26}'s hypotheses are available. It also preserves, uniformly in $R$, the norm curvature
condition \eqref{eq:Ecurvnorm}. The coordinatewise truncation
$q_{[R]}(\xi):=\theta_{R}(\xi)q(\xi)$ of Step~\ref{step:Ereg}, with a bracketed subscript, is the one
used under branch \ref{Ecurvdiag} of Assumption~\ref{ass:Erough}, and it is required there: the
global truncation of a diagonal field is not diagonal, $\Theta_{R}(x)q_{i}(x_{i})$ not being a function
of $x_{i}$ alone, while Lemma~\ref{lem:Escalefree} is applied to the truncated process, so that the
hypothesis would be unavailable exactly where we need it.
\end{remark}
\color{black}

\section{Weak solutions and the rough martingale problem}\label{sec:weaksol}

Definition~\ref{def:weaksol} is a statement about a stochastic basis, whereas uniqueness is a statement
about a law. Lemmas~\ref{lem:filtdescend} and \ref{lem:lawdet} close this gap and 
\S\ref{sec:YWtransfer} uses those two, Lemma~\ref{lem:immersion} and the clause \ref{MP0} of the
rough martingale problem. \ref{MP1}--\ref{MP2}, enters that
section only at Theorem~\ref{thm:YWtransfer}\ref{ywc}, which states the conclusion for it as
well, at the cost of the factor $3$ of Proposition~\ref{prop:MPequiv}.
\color{blue}Two features separate the rough setting from the
classical one:
\begin{enumerate}
    \item \emph{The driver is not part of the basis.} A weak solution is a tuple in which only $B$ is random
and $X$ solves the equation for that same $\RZ$.
\item \emph{The solution notion is conditional, hence filtration-dependent.}
\eqref{eq:daviemoduli} constrains conditional moments; Lemma~\ref{lem:filtdescend} removes the
ambiguity downward and Lemma~\ref{lem:immersion} removes it upward.
\end{enumerate}

\color{black}

\begin{lemma}[The notion descends to a smaller filtration]\label{lem:filtdescend}
Let $X$ denote an $L_{m}$-conditional solution of \eqref{eq:rsde} on the stochastic basis
$(\Omega,\mathcal A,(\mathcal F_{t}),\Prob)$ with Brownian motion $B$, and let $(\mathcal G_{t})$
denote a filtration to which $X$ and $B$ are both adapted and such that
$\mathcal G_{t}\subseteq\mathcal F_{t}$ for every $t$. Then $X$ is an
$L_{m}$-conditional solution on $(\mathcal G_{t})$, with the same or smaller moduli. The same statement
holds identically for the $L_{m,\infty}$ form \eqref{eq:daviemoduliinfty}.
\end{lemma}

\color{blue}
 The lemma is applied below with $(\mathcal G_{t})$ being the
raw filtration $\sigma(X_{r},B_{r}:r\le t)$ generated by the pair, which is what
Definition~\ref{def:roughMP} conditions on: at Proposition~\ref{prop:MPequiv}\ref{mpe1}, at
Step~\ref{step:Etight}(b) of the proof of Theorem~\ref{thm:weakex}, and at the opening of
the proof of Theorem~\ref{thm:YWtransfer}, where it puts the solution of \ref{yw1} into
$\Gamma_{\nu}$. Its remaining mentions are comparisons.
\color{black}
\begin{definition}
    A modulus (of continuity) in \eqref{eq:daviemoduli} is a function
$\omega$ on $[0,T]$ with $\omega(h)/h\to0$ as $h\downarrow0$, respectively
$\omega(h)/h^{1/2}\to0$, bounding the left-hand side there at every pair $s\le t$. We may and do take $\omega$
non-decreasing on $[0,\delta]$ for some $\delta>0$, by replacing it there with
$\sup_{h'\le h}\omega(h')$, which is finite near $0$ and has the same rate. We do not take it
non-decreasing on all of $[0,T]$: that would force $\sup_{s\le t}\|J_{s,t}\|_{L_{m}}<\infty$, which
\eqref{eq:daviemoduli} does not give.
\end{definition}
\begin{proof}
 We say the moduli descend (along the filtrations) when a pair that works on
$(\mathcal F_{t})$ also works on $(\mathcal G_{t})$.
The objects themselves do not change. $B$ remains a $(\mathcal G_{t})$-Brownian motion: it is $(\mathcal G_{t})$-adapted by
hypothesis, and for $s\le t$ the increment $B_{t}-B_{s}$ is independent of $\mathcal F_{s}$, hence of
the smaller field $\mathcal G_{s}$. Both integrals of \eqref{eq:davie} are unchanged. The
$\dr$-integral is defined pathwise and names no filtration. The $\d B$-integral has integrand
$\sigma_{r}(X_{r})$, which is $(\mathcal G_{r})$-measurable at each $r$ and continuous in $r$ by the
standing regularity \eqref{eq:standingcts}, hence $(\mathcal G_{t})$-progressively measurable. Both
constructions therefore realise it as the u.c.p.\ limit of the same left-point Riemann sums, so the two
integrals are indistinguishable. Consequently $J_{s,t}$ is almost surely the same random variable on
either filtration, which is enough for \eqref{eq:daviemoduli}, the conditional
expectations there being defined only up to $\Prob$-null sets; only the conditioning field changes.
Neither completeness of the filtrations nor right-continuity is used in this proof.

 Since
$\mathcal G_{s}\subseteq\mathcal F_{s}$, the tower property applies and reads
\begin{equation*}
\E\big[J_{s,t}\mid\mathcal G_{s}\big]
=\E\big[\,\E[J_{s,t}\mid\mathcal F_{s}]\mid\mathcal G_{s}\big] .
\end{equation*}
Conditional expectation is an $L_{m}(\Prob)$-contraction, by conditional Jensen applied
to $u\mapsto|u|^{m}$ followed by $\E[\E[\,\cdot\mid\mathcal G_{s}]]=\E[\,\cdot\,]$, so applying it to
the outer conditional expectation gives
\begin{equation*}
\big\|\E[J_{s,t}\mid\mathcal G_{s}]\big\|_{L_{m}(\Prob)}
\ \le\ \big\|\E[J_{s,t}\mid\mathcal F_{s}]\big\|_{L_{m}(\Prob)} .
\end{equation*}
The first modulus therefore holds on $(\mathcal G_{t})$ with the same $\omega$. At a
fixed pair the inequality is strict whenever $\E[J_{s,t}\mid\mathcal F_{s}]$ is not almost surely
$\mathcal G_{s}$-measurable, $u\mapsto|u|^{m}$ being strictly convex at the exponents $m\ge2$ of
Definition~\ref{def:davie}.

The second modulus descends as well. Writing the outer norm as an expectation and
using the tower property once,
\begin{equation*}
\begin{aligned}
\big\|\E\big[|J_{s,t}|^{m}\mid\mathcal G_{s}\big]^{1/m}\big\|^{m}_{L_{m}(\Prob)}
&=\E\Big[\Big(\E\big[|J_{s,t}|^{m}\mid\mathcal G_{s}\big]^{1/m}\Big)^{m}\Big]
=\E\Big[\E\big[|J_{s,t}|^{m}\mid\mathcal G_{s}\big]\Big]\\
&=\E\big[|J_{s,t}|^{m}\big]
=\big\|J_{s,t}\big\|^{m}_{L_{m}(\Prob)}.
\end{aligned}
\end{equation*}
The same
computation with $\mathcal F_{s}$ in place of $\mathcal G_{s}$ gives the same number, so
\begin{equation*}
\big\|\E\big[|J_{s,t}|^{m}\mid\mathcal G_{s}\big]^{1/m}\big\|_{L_{m}(\Prob)}
\ =\ \big\|J_{s,t}\big\|_{L_{m}(\Prob)}
\ =\ \big\|\E\big[|J_{s,t}|^{m}\mid\mathcal F_{s}\big]^{1/m}\big\|_{L_{m}(\Prob)} .
\end{equation*}

For the first modulus of
\eqref{eq:daviemoduliinfty} the argument given applies identically, conditional expectation being an
$L_{\infty}(\Omega)$-contraction as well. For the second the cancellation used there for the second modulus is no longer
available, the outer norm now being $L_{\infty}$ and the inner exponent still $m$. Instead, we run the
two steps used for the first modulus on $\E[|J_{s,t}|^{m}\mid\mathcal F_{s}]$, namely the tower property and then
$L_{\infty}$-contractivity of $\E[\,\cdot\mid\mathcal G_{s}]$, using at both ends that
$\|W^{1/m}\|_{L_{\infty}}=\|W\|_{L_{\infty}}^{1/m}$ for $W\ge0$, since $u\mapsto u^{1/m}$ is an
increasing bijection of $[0,\infty)$ and so commutes with the essential supremum:
\begin{equation*}
\begin{aligned}
\big\|\E\big[|J_{s,t}|^{m}\mid\mathcal G_{s}\big]^{1/m}\big\|_{L_{\infty}(\Omega)}
&=\big\|\E\big[\,\E[|J_{s,t}|^{m}\mid\mathcal F_{s}]\mid\mathcal G_{s}\big]\big\|_{L_{\infty}(\Omega)}^{1/m}\\
&\le\ \big\|\E\big[|J_{s,t}|^{m}\mid\mathcal F_{s}\big]\big\|_{L_{\infty}(\Omega)}^{1/m}
=\big\|\E\big[|J_{s,t}|^{m}\mid\mathcal F_{s}\big]^{1/m}\big\|_{L_{\infty}(\Omega)} .
\end{aligned}
\end{equation*}
\end{proof}

\begin{lemma}[It\^o integrals on the raw filtration]\label{lem:itoraw}
Let $X$, with values in $\R^{n}$, and $B$, with values in $\R^{d}$, be continuous on a basis equipped with the (raw) filtration 
$\mathcal G_{t}=\sigma(X_{r},B_{r}:r\le t)$. 
Let $B$ be a $(\mathcal G_{t})$-Brownian motion, and let
$h:[0,T]\times\R^{n}\to\R^{k\times d}$ be continuous, $k\ge1$. Then $\int_{0}^{\cdot}h(r,X_{r})\d B_{r}$ has a version with continuous paths,
unique up to indistinguishability, it is the limit in probability, uniformly on $[0,T]$, of the
left-point Riemann sums $\sum_{i}h(t_{i},X_{t_{i}})\,\delta B_{t_{i},\,t_{i+1}\wedge\,\cdot}$ along
any sequence of partitions of mesh tending to $0$, and
$\Law\big(X,B,\int_{0}^{\cdot}h(r,X_{r})\d B_{r}\big)$ is determined by $\Law(X,B)$.
\end{lemma}

\begin{proof}
The integrand $r\mapsto h(r,X_{r})$ is continuous and $(\mathcal G_{t})$-adapted, so the first two
claims are the construction of the It\^o integral of such an integrand, run on the usual augmentation
of $(\mathcal G_{t})$, which changes neither the integral nor the Brownian property of $B$. For the
third, each Riemann sum is a fixed continuous functional of $(X,B)$, so the joint law of the whole
sequence with $(X,B)$ is determined by $\Law(X,B)$, and a limit in probability is then determined in
law jointly with $(X,B)$.
\end{proof}

\begin{lemma}[The notion is a property of the law]\label{lem:lawdet}
Let $X$ be an $L_{m}$-conditional solution of \eqref{eq:rsde} with Brownian motion $B$ on a basis
whose filtration is $\mathcal G_{t}=\sigma(X_{r},B_{r}:r\le t)$ (raw), with $B$ a
$(\mathcal G_{t})$-Brownian motion. Then $\Law(X,B,J)$ is determined by $\Law(X,B)$ on
$\mathcal C$, and the two quantities of \eqref{eq:daviemoduli} are functions of $\Law(X,B)$ alone.
Hence the moduli hold, with the same rate functions, and $J_{s,t}\in L_{m}(\Prob)$ at every
pair, for every $(\widetilde X,\widetilde B)$ of the same law on a second such basis. If $\Phi_{s,t}$ is, for each $s\le t$, a continuous
functional of $(X,B)$ lying in $L_{m}(\Prob)$, then $\|\Phi_{s,t}\|_{L_{m}}$ and
$\|\E[\Phi_{s,t}\mid\mathcal G_{s}]\|_{L_{m}}$ are likewise functions of $\Law(X,B)$
alone.
\end{lemma}

\begin{proof}
In \eqref{eq:davie} every term other than $M=\int\sigma(X)\d B$ is a continuous functional of the
path, by \eqref{eq:standingcts} and by the continuity of $\delta Z$ and $\ZZ$ in $(s,t)$ that
Assumption~\ref{ass:driver} provides. For $M$, Lemma~\ref{lem:itoraw} at $h=\sigma$ gives that $\Law(X,B,M)$ is determined by
$\Law(X,B)$. Hence, also $\Law(X,B,J)$. Finally $\mathcal G_{s}=\sigma((X,B)|_{[0,s]})$, so for each
pair $s\le t$ there is a measurable $g_{s,t}$ with
$\E[J_{s,t}\mid\mathcal G_{s}]=g_{s,t}\big((X,B)|_{[0,s]}\big)$ almost surely, determined up to
$\Law((X,B)|_{[0,s]})$-null sets by $\Law\big((X,B)|_{[0,s]},J_{s,t}\big)$. Its $L_{m}$-norm is
a function of $\Law(X,B,J_{s,t})$, and likewise for
$\E[|J_{s,t}|^{m}\mid\mathcal G_{s}]$. Both quantities in \eqref{eq:daviemoduli} are thus functions
of $\Law(X,B,J)$, hence of $\Law(X,B)$. The last sentence is the same argument with $\Phi$ in place of $J$: $\Law(X,B,\Phi)$
is determined by $\Law(X,B)$ because $\Phi$ is a path functional, and the representation through
$g_{s,t}$ then applies identically. No moment bound on $X$ is needed: of the solution hypothesis the
proof uses only
continuity and adaptedness, through Lemma~\ref{lem:itoraw}, and integrability of $J_{s,t}$ and of
$|J_{s,t}|^{m}$, which Definition~\ref{def:davie} requires at every pair.
\end{proof}

Together with Lemma~\ref{lem:filtdescend} this closes the gap between basis and law.

\subsection{The rough martingale problem}\label{subsec:MPdef}

\begin{definition}[Rough martingale problem]\label{def:roughMP}
Let $\Prob$ denote a Borel probability measure on $\mathcal C:=C([0,T];D\times\R^{d})$, with canonical
process $(X,B)$ and canonical filtration
$\mathcal G_{t}=\sigma(X_{r},B_{r}:r\le t)$, (uncompleted), and fix $m\ge2$. For
$\varphi\in C^{3}_{b}(\R^{n})$ we
denote by
\begin{equation}\label{eq:seminorm3}
\tnorm\varphi:=\textstyle\sum_{k=1}^{3}\|D^{k}\varphi\|_{\infty}
\end{equation}
the seminorm in the derivatives of order $1$ to $3$, and we put
\begin{equation}\label{eq:roughMP}
\begin{aligned}
D^{\varphi}_{s,t}:={}&\delta\varphi(X)_{s,t}-\!\int_{s}^{t}\!\mathcal L_{r}\varphi(X_{r})\dr
-\delta N^{\varphi}_{s,t}
-\sum_{\kappa}\big(\mathcal G^{\kappa}_{s}\varphi\big)(X_{s})\,\delta Z^{\kappa}_{s,t}\\
&-\sum_{\kappa,\lambda}\big(\mathcal G^{\kappa}_{s}\mathcal G^{\lambda}_{s}+\mathcal G'^{\kappa\lambda}_{s}\big)
\varphi(X_{s})\,\ZZ^{\kappa\lambda}_{s,t}
\ -\ \tfrac12\sum_{\kappa,\lambda}\big(\mathcal Q^{\kappa\lambda}_{s}\varphi\big)(X_{s})\,\CZ^{\kappa\lambda}_{s,t},
\end{aligned}
\end{equation}
with $\mathcal L_{t}$ denoting the It\^o generator of $(b,\sigma)$, $\mathcal G,\mathcal G'$ as in
\S\ref{subsec:gen},
\begin{equation}\label{eq:carre}
\mathcal Q^{\kappa\lambda}_{t}\varphi:=\textstyle\sum_{i,j}f^{\kappa}_{i}(t,x)f^{\lambda}_{j}(t,x)\,\partial^{2}_{ij}\varphi
\end{equation}
the rough carr\'e du champ, $\CZ$ denoting the level-two bracket of Assumption~\ref{ass:driver}, and
$N^{\varphi}_{t}=\int_{0}^{t}\nabla\varphi(X_{r})\cdot\sigma(r,X_{r})\,\d B_{r}$. We say that $\Prob$
solves the rough martingale problem for $(b,\sigma,f,f',\RZ)$ with initial law $\nu$ if the
following hold:
\begin{enumerate}[label=\normalfont(MP\arabic*),start=0,leftmargin=3.2em]
\item\label{MP0} $\Prob\circ X_{0}^{-1}=\nu$, $B$ is a $(\mathcal G_{t})$-Brownian motion such that
$B_{0}=0$, and
\begin{equation*}
\big\|\sup_{t\le T}|X_{t}|\big\|_{L_{m}(\Prob)}<\infty ;
\end{equation*}
\item\label{MP1} \emph{(conditional-mean modulus)} for every $\Lambda>0$ there is a
non-decreasing $\omega_{\Lambda}$ such that $\omega_{\Lambda}(h)/h\to0$ as $h\downarrow0$ and
\begin{equation*}
\big\|\,\E[D^{\varphi}_{s,t}\mid\mathcal G_{s}]\,\big\|_{L_{m}(\Prob)}\ \le\ \omega_{\Lambda}(|t-s|)
\end{equation*}
for \emph{all} $\varphi\in C^{3}_{b}$ with $\tnorm\varphi\le\Lambda$ and all $s\le t$;
\item\label{MP2} \emph{(moment modulus)} likewise with a non-decreasing $\omega'_{\Lambda}$ satisfying
$\omega'_{\Lambda}(h)/h^{1/2}\to0$, and
\begin{equation}\label{eq:MP2form}
\big\|\,\E\big[|D^{\varphi}_{s,t}|^{m}\mid\mathcal G_{s}\big]^{1/m}\,\big\|_{L_{m}(\Prob)}
\ =\ \big\|D^{\varphi}_{s,t}\big\|_{L_{m}(\Prob)}\ \le\ \omega'_{\Lambda}(|t-s|).
\end{equation}
\end{enumerate}
\emph{Weak uniqueness} means: at most one such $\Prob$ for each $\nu$, at the
exponent $m$ of \ref{MP0}--\ref{MP2}.
\end{definition}

\begin{remark}
\begin{enumerate}
\item Every $N^{\varphi}$ here is the continuous version of Lemma~\ref{lem:MPborel}\ref{mpb:version}. \color{blue}The lemma constructs it under
the standing continuity \eqref{eq:standingcts} from \ref{MP0} alone, so that nothing here is circular.\color{black}
\item The equality in \eqref{eq:MP2form} is the tower property, since $\frac{1}{m}$ cancels with the $m$ from the outer norm, so that
\ref{MP2} is an unconditional bound and all the conditional content sits in \ref{MP1}, where the
conditional expectation is taken before the norm. 
\item  The moduli are
calibrated at different exponents, $\omega'_{\Lambda}(h)/h^{1/2}\to0$ against
$\omega_{\Lambda}(h)/h\to0$, so \ref{MP2} does not give \ref{MP1}.
 \ref{MP2} carries a moment above $m$. Take $n=e=d=1$, $b=\sigma=0$, $f(x)=x$ and
$f'=0$. Beyond $\delta\varphi(X)_{s,t}$ the only summands of $D^{\varphi}_{s,t}$ are then
$-X_{s}\varphi'(X_{s})(\delta Z_{s,t}+\ZZ_{s,t})$ and
$-\tfrac12X^{2}_{s}\varphi''(X_{s})\,\delta Z^{2}_{s,t}$, by $\delta Z^{2}=2\ZZ+\CZ$ of
Assumption~\ref{ass:driver}. The first lies in $L_{m}$ by \ref{MP0}. At $\varphi=\sin$ and at
$\varphi=\cos$, and with $|\sin|+|\cos|\ge1$, \ref{MP2} therefore forces
$\|X_{s}\|_{L_{2m}}<\infty$ at every pair with $\delta Z_{s,t}\neq0$. So the definition at exponent
$m$ requires more of $X$ than the moment at $m$ that \ref{MP0} states. Every application below has a
moment at $3m$ or above, so nothing downstream changes.
\end{enumerate}
 
\end{remark}
\begin{remark}[Why the canonical filtration is taken uncompleted]\label{rem:MPraw}
Completion changes none of \ref{MP0}--\ref{MP2}, every set of the augmented field differing from a set
of $\mathcal G_{s}$ by a $\Prob$-null set. The raw field is kept because
Lemma~\ref{lem:MPborel}\ref{mpb:borel} then reduces \ref{MP1} to countably many integrals and path
continuity gives $\mathcal G_{s-}=\mathcal G_{s}$.
Denote by
$\overline{\mathcal G}_{s}:=\mathcal G_{s}\vee\mathcal N$ the augmentation of $\mathcal G_{s}$ by the
$\Prob$-null sets $\mathcal N$ of the Borel field of $\mathcal C$. Every set of
$\overline{\mathcal G}_{s}$ differs from a set of $\mathcal G_{s}$ by a $\Prob$-null set. Hence, for
$\xi$ integrable or non-negative with values in $[0,\infty]$, the variable $\E[\xi\mid\mathcal G_{s}]$
is already a version of $\E[\xi\mid\overline{\mathcal G}_{s}]$: it is measurable for the larger field,
and the two defining integrals agree on every set of that field. With
$\xi=D^{\varphi}_{s,t}$ and $\xi=|D^{\varphi}_{s,t}|^{m}$, this leaves the left-hand sides of
\ref{MP1} and of \ref{MP2} unchanged wherever they are defined. The Brownian requirement of \ref{MP0} is unchanged as well, an increment independent of
$\mathcal G_{s}$ being independent of $\overline{\mathcal G}_{s}$ by the same approximation of sets;
and the other two parts of \ref{MP0}, the initial law and the moment bound, name no filtration at all.

The proof Lemma~\ref{lem:MPborel}\ref{mpb:borel}  cuts the solution set out of the space of Borel probability
measures on $\mathcal C$ by countably many maps $\Prob\mapsto\R$, and each of these has to be one
functional fixed in advance and evaluated at $\Prob$, rather than an object assembled from the $\Prob$ at
which it is stated. The reduction of \ref{MP1} replaces the conditional norm by countably many integrals
$\Prob\mapsto\E_{\Prob}[D^{\varphi}_{s,t}\psi]$ against bounded $\mathcal G_{s}$-measurable $\psi$, and
those $\psi$ are drawn from a countable $\Q$-algebra generated by the maps
$g(\cdot_{r_{1}},\dots,\cdot_{r_{l}})$ at rational $r_{i}\le s$ with $g$ continuous. They are available
because $\mathcal G_{s}$ is generated by evaluations at times $r\le s$, and that generation does not
"move" with $\Prob$. On $\overline{\mathcal G}_{s}$ the same functionals still work, by the
null-set argument above, so the completed version of the argument is not wrong. It is the argument which would have
to be changed slightly, and the raw field lets the proof be written without the detour.
Step~\ref{stp:MPbi} of that proof reaches an irrational
pair $s\le t$ from rational pairs below, and for that it needs $\mathcal G_{s-}=\mathcal G_{s}$ at
$s>0$. This holds because the canonical paths are continuous, so that $X_{s}$ and $B_{s}$ are
measurable for $\bigvee_{r<s}\mathcal G_{r}$. Right-continuity of $(\mathcal G_{t})$ is not needed
here, and \S\ref{sec:YWtransfer} records that it is not needed there either.
\end{remark}

\begin{remark}\label{rem:MPweight}
An $L_{\infty}(\Omega)$-conditional form of \ref{MP1}--\ref{MP2} is blocked twice. 

\emph{Growth:} the
germ defect has degree $2$ in the state and  $V(X_{s})^{-1}$ (coming from the Lyapunov function $V(x)=(1+|x|^2)^{\frac{1}{2}}$) removes one power only,
$V(x)^{-1}(1+|x|)^{2}\asymp1+|x|$. 

\emph{Structure:} it contains $\delta\tilde N^{\varphi}$ of
\eqref{eq:Ntilde}, with no pathwise majorant $\tnorm\varphi\cdot(\varphi\text{-free})$ but only an
$L_{m}$-bound on $\mathcal K^{\varphi}$, \eqref{eq:Ntildebound} and \eqref{eq:PQtildebound}, against
which no $\mathcal G_{s}$-measurable weight is uniform.
\end{remark}

\begin{remark}\label{rem:MPbracket} By the definition of the bracket in Assumption~\ref{ass:driver},
\begin{equation}\label{eq:brackettaylor}
\delta Z^{\kappa}_{s,t}\,\delta Z^{\lambda}_{s,t}
=\ZZ^{\kappa\lambda}_{s,t}+\ZZ^{\lambda\kappa}_{s,t}+\CZ^{\kappa\lambda}_{s,t} .
\end{equation}
We can Taylor-expand $\varphi(X_{t})-\varphi(X_{s})$ to second order and insert the Davie expansion
\eqref{eq:davie} of $\delta X_{s,t}$. The second-order term then contributes
$\tfrac12\sum_{\kappa,\lambda}A_{\kappa\lambda}\,\delta Z^{\kappa}_{s,t}\delta Z^{\lambda}_{s,t}$ with
$A_{\kappa\lambda}:=\sum_{i,j}\partial^{2}_{ij}\varphi\,f^{\kappa}_{i}f^{\lambda}_{j}
=\mathcal Q^{\kappa\lambda}\varphi$, a coefficient symmetric in $(\kappa,\lambda)$. Inserting
\eqref{eq:brackettaylor} and relabelling $\kappa\leftrightarrow\lambda$ in the second summand, we obtain
\begin{equation*}
\tfrac12\sum_{\kappa,\lambda}A_{\kappa\lambda}
\big(\ZZ^{\kappa\lambda}+\ZZ^{\lambda\kappa}+\CZ^{\kappa\lambda}\big)
\ =\ \sum_{\kappa,\lambda}A_{\kappa\lambda}\ZZ^{\kappa\lambda}
\ +\ \tfrac12\sum_{\kappa,\lambda}\big(\mathcal Q^{\kappa\lambda}\varphi\big)\CZ^{\kappa\lambda} ,
\end{equation*}
whose first summand is exactly the
$\textstyle\sum_{i,j}f^{\kappa}_{i}f^{\lambda}_{j}\partial^{2}_{ij}\varphi$ part of
$\mathcal G^{\kappa}\mathcal G^{\lambda}\varphi$, while the second is not absorbed by anything.
 The same term appears on the right-hand side of the rough It\^o formula of
\cite[Thm.~4.13]{FHL},
\[
\varphi(Y_{t})-\varphi(Y_{0})-\int_{0}^{t}\!D\varphi(Y)\sigma\,\d B-\int_{0}^{t}\!(\mathcal L_{r}\varphi)(Y_{r})\dr
=\int_{0}^{t}\!D\varphi(Y)Y'\,\d\RZ+\tfrac12\int_{0}^{t}\!D^{2}\varphi(Y)(Y',Y')\,\d\CZ .
\]
Here $Y$ denotes a controlled path with Gubinelli derivative $Y'$ and second derivative $Y''$, and the
rough integral on the right is taken against the composed controlled pair, which we denote by
\begin{equation}\label{eq:Tphidef}
(\mathcal T\varphi)(Y):=D\varphi(Y)\,Y',
\qquad
(\mathcal T'\varphi)(Y):=D^{2}\varphi(Y)(Y',Y')+D\varphi(Y)\,Y'' ,
\end{equation}
so that the first integral on the right is given by
$\int_{0}^{t}(\mathcal T\varphi,\mathcal T'\varphi)(Y)\,\d\RZ$.
This part is the conclusion quoted in Lemma~\ref{lem:roughItoLn}.

\begin{enumerate}[label=(\roman*)]
    \item With the term present, Definition~\ref{def:roughMP} is genuinely lift-agnostic. Without it, Proposition~\ref{prop:MPequiv}\ref{mpe1} would be false for the It\^o lift $\CZ_{s,t}=(t-s)I$.
The
omitted $\tfrac12\sum_{\kappa}(\mathcal Q^{\kappa\kappa}_{s}\varphi)(X_{s})(t-s)$ has conditional
mean of exact order $(t-s)$. For $n=e=d=1$, $b=\sigma=0$, $f\equiv1$, $f'\equiv0$, $X_{0}=0$,
$\varphi\in C^{3}_{b}$ equal to $x^{2}$ on $[-1,1]$; then $X_{t}=\delta Z_{0,t}$ and on
$\{\sup_{t\le T}|X_{t}|<1\}$ (probability $1$ once $\|\RZ\|_{\gamma}T^{\gamma}<1$)
$\mathcal Q^{11}_{s}\varphi(X_{s})=2$, the omitted term being $(t-s)$ exactly.
\item It changes nothing for the coordinate maps. By \eqref{eq:carre},
$\mathcal Q^{\kappa\lambda}\varphi$ vanishes wherever $D^{2}\varphi$ does, in
particular for the $\varphi_{R}$ of Proposition~\ref{prop:MPequiv}\ref{mpe2} on $\{\tau_{R}>t\}$,
where $D^{2}\varphi_{R}(X_{r})=0$. 
\end{enumerate}
\end{remark}
\begin{remark}\label{rem:MPstrong}
\begin{enumerate}[label=(\roman*)]
    \item The moduli are required
uniformly over $\tnorm\cdot$-bounded families, making $\varphi_{R}$ usable,
$\|\varphi_{R}\|_{\infty}\sim R\to\infty$ while $\sup_{R}\tnorm{\varphi_{R}}<\infty$.
With a $\varphi$-dependent
$o(\cdot)$ the moduli could degenerate as $R\to\infty$ and the limit could not be taken. This is why
\eqref{eq:seminorm3} denotes a seminorm in the derivatives alone.  Every summand of \eqref{eq:roughMP} is linear in $\varphi$, the version of
$N^{\varphi}$ included, it being a limit of Riemann sums that are themselves linear in $\varphi$.
Since $\tnorm\cdot$ is absolutely homogeneous, we obtain
$D^{c\varphi}=c\,D^{\varphi}$ up to indistinguishability, and $\varphi\mapsto\varphi/\Lambda$ maps
$\{\tnorm\varphi\le\Lambda\}$ onto $\{\tnorm\varphi\le1\}$. Hence \ref{MP1}--\ref{MP2} at $\Lambda=1$
already imply them at every $\Lambda>0$, with $\omega_{\Lambda}=\Lambda\,\omega_{1}$ and
$\omega'_{\Lambda}=\Lambda\,\omega'_{1}$, so that we can work with a single modulus pair on the unit
ball, in Definition~\ref{def:roughMP} and in Definition~\ref{def:graded} alike. We keep the quantifier
because the statements that use the definition name the ball they test in, $\Lambda_{0}$ of
\eqref{eq:phiRunif} for one.
\item \ref{MP0} is part of the
definition: it allows the localization and cannot be derived from \ref{MP1}--\ref{MP2}, available
only for bounded test functions until the localization has run.
\end{enumerate}
\end{remark}

\subsection{The rough It\^o formula under moment bounds}\label{subsec:MPito}

\begin{lemma}[Rough It\^o formula under uniform moment bounds: a variant of {\cite[Thm.~4.13]{FHL}}]\label{lem:roughItoLn}
Let Assumption~\ref{ass:driver} hold, so that $\gamma\in(\tfrac13,\tfrac12)$ and
$\CZ$ is Lipschitz. Let $\mfn >4$, let $\varkappa_{\varphi}\in(1/\gamma,3]$ and $0<\beta'\le\beta\le\gamma$, and let
\begin{equation}\label{eq:paramsIto}
\gamma+\beta>\tfrac12 ,
\qquad
\gamma+\beta+\beta''>1 ,
\qquad
\beta'':=\min\Big\{\gamma(\varkappa_{\varphi}-2),\ \gamma\big(\tfrac \mfn 4-1\big),\ \beta'\Big\} .
\end{equation}
Let $b,\sigma,Y',Y''$ be progressively measurable, and let $Y$ be continuous, adapted and a
solution of
$\d Y_{t}=b_{t}\dt+\sigma_{t}\,\d B_{t}+(Y',Y'')_{t}\,\d\RZ_{t}$, the rough integral being that of
\cite[Thm.~3.4]{FHL}. Let $D^{j}\varphi$ be bounded for
$1\le j\le N:=\lceil\varkappa_{\varphi}\rceil-1$, with $D^{N}\varphi$ globally H\"older of order
$\varkappa_{\varphi}-N\in(0,1]$. Assume the uniform moment bound
\begin{equation}\label{eq:rItoLn}
\sup_{t\le T}\Big(\big\|b_{t}\big\|_{L_{\mfn }}+\big\|\sigma_{t}\big\|_{L_{\mfn }}
+\big\|Y'_{t}\big\|_{L_{\mfn }}\Big)\ <\ \infty ,
\end{equation}
and the four quantities that \cite[Def.~4.12]{FHL} requires to be finite of a pair in
$\Dext^{\beta,\beta'}_{Z}L_{4,\mfn }$, here at $(Y',Y'')$ and with $\mathcal F_{s}$ the conditioning
field of the mixed norm $\|\cdot\|_{4,\mfn }$,
\begin{equation}\label{eq:rItoLnsemi}
\begin{gathered}
\sup_{s<t}\frac{\big\|\delta Y'_{s,t}\big\|_{4,\mfn }}{|t-s|^{\beta}}
\ +\ \sup_{t\le T}\big\|Y''_{t}\big\|_{L_{\mfn }}
\ +\ \sup_{s<t}\frac{\big\|\E_{s}R^{Y'}_{s,t}\big\|_{L_{\mfn }}}{|t-s|^{\beta+\beta'}}
\ +\ \sup_{s<t}\frac{\big\|\E_{s}\delta Y''_{s,t}\big\|_{L_{\mfn }}}{|t-s|^{\beta'}}
\ <\ \infty ,
\\
R^{Y'}_{s,t}:=\delta Y'_{s,t}-Y''_{s}\,\delta Z_{s,t} .
\end{gathered}
\end{equation}
Then
\begin{enumerate}[label=\normalfont(\roman*),leftmargin=2.4em]
\item\label{rIto:class} the composed pair $(\mathcal T\varphi,\mathcal T'\varphi)(Y)$ of
\eqref{eq:Tphidef} has those four quantities finite at the indices $(\beta,\beta'')$ and at
$L_{2,2}$, and is progressively measurable. The rough integral against that pair is
therefore defined.
\item The rough It\^o formula stated in Remark~\ref{rem:MPbracket} holds.
\end{enumerate}
\end{lemma}

\begin{remark}[Comparison to the source]\label{rem:rItoLnextra}
\cite[Thm.~4.13]{FHL} asks $b,\sigma$  to be bounded, which is replaced by \eqref{eq:rItoLn}. It also requires the
pair $(Y',Y'')$ to lie in $\Dext^{\beta,\beta'}_{Z}L_{4,\mfn }$, which is replaced by 
\eqref{eq:rItoLnsemi}: the class asks those four quantities together with the measurability
and continuity clauses of \cite[Def.~2.5]{FHL}. The third is the boundedness of $\varphi$,
below. On the lift we assume more: \cite[Thm.~4.13]{FHL} doesn't have any assumptions on the bracket. It proves its Step~1 for a geometric lift and reduces the general case to it in its
Step~2, by passing to the geometric lift $\RZ+(0,\tfrac12\CZ)$. Step~\ref{stp:rIto2} below instead
adds the one defect $J^{5}$ that a general lift creates, and \eqref{eq:bracketLip} is what bounds
it. With $\CZ$ only $2\gamma$-H\"older the order of $J^{5}$ would drop to
$2\gamma+\min\{\gamma\lambda,\beta\}$, at most $3\gamma$, which approaches $1$ as
$\gamma\downarrow\tfrac13$ and then misses the threshold of Step~\ref{stp:rIto4}. We take our route
because Step~\ref{stp:rIto4} identifies the limit through Lemma~\ref{lem:SSL} and so needs the
defect itself. Everything else is taken from \cite{FHL}. The class of
\cite[Def.~4.12]{FHL} replaces condition~(d) of \cite[Def.~3.1]{FHL} by the weaker
condition~(d$'$), so it contains $\mathscr D^{\beta,\beta'}_{Z}L_{4,\mfn }$.

The condition on $\varphi$ is \cite[\S2.1]{FHL}'s $\mathcal C^{\varkappa_{\varphi}}_{b}$, which
gives $N=2$ at the index $\varkappa_{\varphi}=3$ used below. That space asks $\varphi$ itself to be
bounded, and we do not. A footnote inside the proof of \cite[Thm.~4.13]{FHL} records that the test
function enters the composed pair's controlled quantities only through $\|D\varphi\|_{\infty}$,
$\|D^{2}\varphi\|_{\infty}$ and a H\"older seminorm of $D^{2}\varphi$. The other defect bounds there
are stated with the full norm, so we rely on inspection of that proof.
Proposition~\ref{prop:Athreshuses} needs the weakening, applying the lemma at the unbounded $V$ of
\eqref{eq:Vjet}.

The pair $(\beta,\beta')$ is that of the controlled path fed to the formula, and not that of
$(f,f')$: at every place of use below the integrand is a composed pair, whose second index is
the $\bar\beta$ of \eqref{eq:Jbarbeta}, so the third entry of $\beta''$ is $\bar\beta$, or any $\beta_{2}\le\bar\beta$ at which the pair lies in the class, and $\beta'\le\beta\le\gamma$ constrains
that pair. The $\beta'$ of Assumption~\ref{ass:Erough} carries no
upper bound and may exceed $\gamma$.

The input class is used only through progressive
measurability and its seminorms. \cite[Thm.~3.4]{FHL} never explicitly requires the full class: its hypothesis is
progressive measurability and the finiteness of two quantities $\Gamma_{1},\Gamma_{2}$,
which the class dominates, and a footnote there calls this the weakest condition under which
stochastic sewing applies. \cite[Lem.~3.11]{FHL} does ask for the class, but its proof reads the pair
only through those seminorms, and the step that moves a factor through $\E_{s}$ is an identity,
$Df_{s}(Y_{s})$ being $\mathcal F_{s}$-measurable. The continuity clause of \cite[Def.~2.5]{FHL}
that the class also carries enters no step of either. It is supplied anyway at the two places of
Proposition~\ref{prop:Athreshuses}, where the pair is built by \cite[Rem.~4.4]{FHL} and
\cite[Lem.~3.11]{FHL} at fixed $(k,R)$, both of which conclude in the source's own
class. Part~\ref{rIto:class} is cited wherever the integral has to exist before the identity
is used. Lemma~\ref{lem:roughItoLn} keeps the index condition $\mfn >4$ of
\cite[Thm.~4.13]{FHL}. {\color{blue}Proposition~\ref{prop:Taylorsew} asks a second
one, $\mfn >2/\gamma$ in place of $\mfn >4$, which the lemma itself does not use.\color{black}}

The moment bound on $Y'$ is not provided by the controlled class. For
$(Y',Y'')\in\Dext^{\beta,\beta'}_{Z}L_{4,\mfn }$ the class constrains $Y'$ only through $\delta Y'$, and
\cite{FHL} states after \cite[Def.~3.1]{FHL} that no integrability is required of the ground value
$Y'_{0}\in L_{0}(\mathcal F_{0})$. So even $\sup_{t\le T}\|Y'_{t}\|_{L_{4}}<\infty$ has to be
assumed, while the bound
$\|Y'_{s}\delta Z_{s,t}\|_{4,\mfn }\le\|Y'_{s}\|_{L_{\mfn }}|\delta Z_{s,t}|$ of \cite[Rem.~2.2(iii)]{FHL},
$Y'_{s}$ being $\mathcal F_{s}$-measurable, needs it. \cite[Thm.~4.13]{FHL} states it as a
hypothesis of the theorem itself, and the same bound appears in
\cite[Rem.~4.4]{FHL}, which places a solution satisfying $f(Y)\in C^{\beta}L_{4,\mfn }$ in
$\mathscr D^{\gamma,\beta}_{Z}L_{4,\mfn }$ under $\sup_{t}\|g_{t}(Y_{t})\|_{L_{\mfn }}<\infty$ for
$g\in\{b,\sigma\sigma^{\top},f,Df\,f,f'\}$, and with $Y'=f(Y)$ and $Y''=(Df\,f+f')(Y)$ $b$ and $f$ are two summands of
\eqref{eq:rItoLn}, while $Df\,f$ and $f'$ sit in the second entry of \eqref{eq:rItoLnsemi}. 
\eqref{eq:rItoLn} asks  less of the diffusion: $\sigma\in L_{\mfn}$, where \cite[Rem.~4.4]{FHL} requires
$\sigma\sigma^{\top}\in L_{\mfn}$, that is $\sigma\in L_{2\mfn}$.
{\color{blue}The index condition of Proposition~\ref{prop:Taylorsew} holds automatically wherever it
is required: $2/\gamma<4/\gamma-4$ for
$\gamma<\tfrac12$, so $\mfn >2/\gamma$ is implied by \eqref{eq:paramsdischarge}, forced by
\eqref{eq:paramsIto} at every place of use, and $4(\varkappa_{\varphi}-1)>4(1/\gamma-1)>2/\gamma$
there.\color{black}}
\end{remark}

\begin{proof}[Sketch of proof of Lemma~\ref{lem:roughItoLn}]
We follow the arguments of \cite[proof of Thm.~4.13]{FHL}.  Boundedness of
$b,\sigma$ is replaced by \eqref{eq:rItoLn}, which is possible because the proof uses those
coefficients in only two steps, and where a step uses them
pointwise an interpolation replaces that use (Step~\ref{stp:rIto2}), and the class of the pair
is used only through its seminorms, which is all the proof uses (Step~\ref{stp:rIto1}). Separately, $\|\varphi\|_{\infty}$ never appears, neither in the germ estimates nor
in the composition \cite[Lem.~3.11]{FHL} of Step~\ref{stp:rIto1}, whose conclusion uses $\varphi$
only through $D\varphi$ and $D^{2}\varphi$. This is what lets us use the lemma at an unbounded
$\varphi$, which both places of Proposition~\ref{prop:Athreshuses}
do, the germ estimates using
only $[\varphi]_{2}$, $[\varphi]_{\varkappa_{\varphi}}$ and the regularity of
$D\varphi$, never $|\varphi|_{\infty}$, each $\varphi$ standing under a derivative and the formula
being an identity between increments. Two of the steps require a deeper focus, and so do
two summands of the germ estimates. Those four we state and the rest of the germ estimates we
import.

Throughout, $(Y',Y'')$ denotes the controlled pair fed to the rough integral, at the indices
$(\beta,\beta')$ of the statement, and $Y$ the controlled path it drives, with Gubinelli derivative
$Y'$. We use the mixed-norm embedding
\begin{equation}\label{eq:rItoLnembed}
\|\xi\|_{L_{4}}\ \le\ \|\xi\|_{4,\mfn }\ \le\ \|\xi\|_{L_{\mfn }}\qquad(4\le \mfn ),
\end{equation}
which is \cite[Prop.~2.3(ii)]{FHL}, and the collapse $\|\xi\|_{4,\mfn }=\|\xi\|_{L_{\mfn }}$ for
$\mathcal F_{s}$-measurable $\xi$, which is \cite[Rem.~2.2(iii)]{FHL}. The increment splits as
\begin{equation}\label{eq:rItoLnsplit}
\delta Y_{s,t}
=\int_{s}^{t}\!b_{r}\dr+\int_{s}^{t}\!\sigma_{r}\d B_{r}
+Y'_{s}\,\delta Z_{s,t}+Y''_{s}\,\ZZ_{s,t}+E_{s,t} ,
\end{equation}
with $E$ the remainder of the rough integral, estimated by \cite[Thm.~3.4]{FHL} and carrying no
hypothesis on $b$ or $\sigma$.

We assume $\mfn \le4(\varkappa_{\varphi}-1)$, as \cite[proof of Thm.~4.13]{FHL}
does: the hypotheses at $\mfn $ imply them at every smaller index, and $\beta''$ does not change
once $\mfn \ge4(\varkappa_{\varphi}-1)$. The composition of Step~\ref{stp:rIto1} is then applied at
the regularity index $1+\mfn /4\in(2,\varkappa_{\varphi}]$.

\emph{Step 1: the seminorms of the class}. \stepnum{1}{stp:rIto1} \cite[Lem.~3.11]{FHL} asks of $(Y,Y')$ to lie in the class
$\mathscr D^{\gamma,\beta}_{Z}L_{4,\mfn }$. We verify its
seminorms from \eqref{eq:rItoLn} and \eqref{eq:rItoLnsemi}. They are what that lemma's proof uses
(Remark~\ref{rem:rItoLnextra}).\ Minkowski's integral
inequality, and Burkholder--Davis--Gundy followed by Minkowski in $L_{\mfn /2}$, give
\begin{equation}\label{eq:rItoLnbdrift}
\begin{gathered}
\Big\|\int_{s}^{t}\!b_{r}\dr\Big\|_{L_{\mfn }}\ \le\ \sup_{r\le T}\|b_{r}\|_{L_{\mfn }}\,|t-s| ,
\\
\Big\|\int_{s}^{t}\!\sigma_{r}\d B_{r}\Big\|_{L_{\mfn }}
\ \lesssim_{\mfn }\ \Big(\int_{s}^{t}\!\|\sigma_{r}\|^{2}_{L_{\mfn }}\dr\Big)^{1/2}
\ \le\ \sup_{r\le T}\|\sigma_{r}\|_{L_{\mfn }}\,|t-s|^{1/2} ,
\end{gathered}
\end{equation}
the index being $\mfn >4$.
By \eqref{eq:rItoLnembed} both bounds pass to $\|\cdot\|_{4,\mfn }$. The two frozen terms of
\eqref{eq:rItoLnsplit} are $\mathcal F_{s}$-measurable, so the mixed norm collapses on them,
\begin{equation}\label{eq:rItoLnfrozen}
\big\|Y'_{s}\,\delta Z_{s,t}\big\|_{4,\mfn }=\big\|Y'_{s}\big\|_{L_{\mfn }}\big|\delta Z_{s,t}\big| ,
\qquad
\big\|Y''_{s}\,\ZZ_{s,t}\big\|_{4,\mfn }=\big\|Y''_{s}\big\|_{L_{\mfn }}\big|\ZZ_{s,t}\big| ,
\end{equation}
and \cite[Thm.~3.4]{FHL} supplies the two bounds on the remainder,
\begin{equation}\label{eq:rItoLnE}
\big\|E_{s,t}\big\|_{4,\mfn }\ \lesssim\ |t-s|^{\gamma+(\gamma\wedge\beta)} ,
\qquad
\big\|\E_{s}E_{s,t}\big\|_{L_{\mfn }}\ \lesssim\ |t-s|^{\gamma+(\gamma\wedge\beta)+\beta'} .
\end{equation}
The second entry of \eqref{eq:rItoLnsemi} bounds $\sup_{t}\|Y''_{t}\|_{L_{\mfn }}$ and
the third summand of \eqref{eq:rItoLn} bounds $\sup_{t}\|Y'_{t}\|_{L_{\mfn }}$, which are the two
factors \eqref{eq:rItoLnfrozen} needs.

With that bound in hand, \eqref{eq:rItoLnbdrift}, \eqref{eq:rItoLnfrozen} and the first term of
\eqref{eq:rItoLnE} give the five orders $1$, $\tfrac12$, $\gamma$, $2\gamma$ and
$\gamma+(\gamma\wedge\beta)$, each at least $\gamma$ on $|t-s|\le T$ by $\gamma<\tfrac12$, whence
$\delta Y\in C^{\gamma}_{2}L_{4,\mfn }$. Taking the same summands under $\E_{s}$, the It\^o term
vanishes, being a martingale increment, the drift sits at order $1$, the term
$Y''_{s}\ZZ_{s,t}$ at $2\gamma$, and $\E_{s}E_{s,t}$ at $\gamma+\beta+\beta'$ by the second term of
\eqref{eq:rItoLnE} and $\beta\le\gamma$. All three are at least $\gamma+\beta$, again by
$\beta\le\gamma<\tfrac12$, so that
\begin{equation}\label{eq:rItoLnrem}
\big\|\E_{s}R^{Y}_{s,t}\big\|_{L_{\mfn }}\ \lesssim\ |t-s|^{\gamma+\beta},
\qquad R^{Y}_{s,t}:=\delta Y_{s,t}-Y'_{s}\,\delta Z_{s,t} .
\end{equation}
Of the three remaining requirements, the bound
$\sup_{t}\|Y'_{t}\|_{L_{\mfn }}$ is the one just provided, and progressive measurability
and $\delta Y'\in C^{\beta}_{2}L_{4,\mfn }$ are hypotheses on $(Y',Y'')$, not at issue here.

\emph{Step 2: the germ estimates.}\stepnum{2}{stp:rIto2} The source bounds four defects
$J^{1},\dots,J^{4}$ whose sum is, for a geometric lift, the difference between $\mathcal I$ of
Step~\ref{stp:rIto4} and the germ $A$ of that proof, and it bounds their conditional means as well, both in $\|\cdot\|_{L_{2}}$
and never in the mixed norm. In $J^{1}$, in $J^{2}$ and in every summand of $J^{4}$ but one the
coefficients are paired, through the It\^o isometry, Burkholder--Davis--Gundy or Minkowski, with
factors the source controls in $L_{4}$, whether they occur linearly, as in $J^{1}$ and $J^{2}$, or as
a square, as in $\langle D^{2}\varphi(Y_{s}),(\int_{s}^{t}\sigma\d B)^{\otimes2}\rangle$, whose
$L_{2}$-norm is $\|\int_{s}^{t}\sigma\d B\|^{2}_{L_{4}}$; $J^{3}$ carries no coefficient at all,
and $\mfn >4$ supplies the fourth moments. We import those bounds from the source, whose Step~1 proves them for a geometric lift.

Assumption~\ref{ass:driver} allows a general lift, and one further defect is then
added, from two sources. The identity
$2\operatorname{Sym}\ZZ_{s,t}=\delta Z^{\otimes2}_{s,t}-\CZ_{s,t}$ of \eqref{eq:brackettaylor}
replaces the source's $2\operatorname{Sym}\ZZ_{s,t}=\delta Z^{\otimes2}_{s,t}$, and the germ $A$ of
Step~\ref{stp:rIto3} is unchanged while $\mathcal I$ of Step~\ref{stp:rIto4} carries the
$\CZ$-integral. The first source comes from $\mathcal I$,
$\tfrac12\int_{s}^{t}D^{2}\varphi(Y_{r})(Y'_{r},Y'_{r})\d\CZ_{r}$. The identity is the second: it
turns the source's fourth defect into
$J^{4}=J^{4,\mathrm{geo}}-\tfrac12\langle D^{2}\varphi(Y_{s}),Y'^{\otimes2}_{s}\CZ_{s,t}\rangle$,
where $J^{4,\mathrm{geo}}$ denotes the quantity the source bounds. Each of the two is of order $1$,
and at the It\^o lift $\CZ_{s,t}=(t-s)I$ of Remark~\ref{rem:MPbracket} no better, so neither
reaches the threshold $\mu>1$ of Step~\ref{stp:rIto4} alone. We estimate their sum. Since
$\CZ_{s,t}=\int_{s}^{t}\d\CZ_{r}$, that sum is
\[
J^{5}_{s,t}\ :=\ \tfrac12\int_{s}^{t}\Big[D^{2}\varphi(Y_{r})\big(Y'_{r},Y'_{r}\big)
-D^{2}\varphi(Y_{s})\big(Y'_{s},Y'_{s}\big)\Big]\d\CZ_{r} .
\]
By \eqref{eq:bracketLip} the integrator is the increment of a Lipschitz path, so
$|\d\CZ_{r}|\le\|\CZ\|_{1}\dr$ and the integral is a Lebesgue--Stieltjes one. Its integrand splits
into $[D^{2}\varphi(Y_{r})-D^{2}\varphi(Y_{s})](Y'_{r},Y'_{r})$ and two terms carrying
$\delta Y'_{s,r}$. The interpolation below bounds the first at order $\gamma\lambda$, with $Y'_{r}$
in place of $\sigma_{r}$ and at the same exponents; $\|\delta Y'_{s,r}\|_{L_{4}}\|Y'_{r}\|_{L_{4}}$
bounds the other two at order $\beta$, by \eqref{eq:rItoLnsemi} and \eqref{eq:rItoLn}. Both bounds
survive under $\E_{s}$, which is an $L_{2}$-contraction, so
\[
\big\|J^{5}_{s,t}\big\|_{L_{2}}+\big\|\E_{s}J^{5}_{s,t}\big\|_{L_{2}}
\ \lesssim\ \|\CZ\|_{1}\,|t-s|^{1+\min\{\gamma\lambda,\,\beta\}} ,
\]
above both thresholds of Step~\ref{stp:rIto4}.
The exception is the summand
\begin{equation}\label{eq:rItoLnJ4}
\int_{s}^{t}\big\langle a_{r},\ D^{2}\varphi(Y_{r})-D^{2}\varphi(Y_{s})\big\rangle\dr ,
\qquad a:=\tfrac12\sigma\sigma^{\top}{\text{ in this step only}} ,
\end{equation}
where $a$ is paired with the H\"older modulus $|\delta Y_{s,r}|^{\varkappa_{\varphi}-2}$.
With $\sigma$ bounded, that pairing is immediate. \cite{FHL} calls this part of the bound
standard and does not write it out. What follows is our reconstruction under
\eqref{eq:rItoLn}. The paired factor lies in $L_{4}$ and no
better (Step~\ref{stp:rIto3}), so a plain H\"older split would demand 
$a\in L_{4/(4-\varkappa_{\varphi})}$, that is
$\mfn \ge8/(4-\varkappa_{\varphi})$, which at $\varkappa_{\varphi}=3$ is $\mfn \ge8$, above the
$\mfn >4/\gamma-4$ of \eqref{eq:paramsdischarge} at every $\gamma>\tfrac13$. We interpolate the
two bounds on the increment instead, which at the $\mfn =8$ of
Proposition~\ref{prop:Athreshuses} returns the plain split: for $\theta\in(0,1]$ and $\lambda:=\theta(\varkappa_{\varphi}-2)$,
\[
\big|D^{2}\varphi(y)-D^{2}\varphi(y')\big|
\ \le\ \big(2[\varphi]_{2}\big)^{1-\theta}\,[\varphi]^{\theta}_{\varkappa_{\varphi}}\,\big|y-y'\big|^{\lambda} ,
\]
both seminorms being ones this step already keeps. With $\tfrac12=\tfrac\lambda4+\tfrac1q$, H\"older gives
\[
\Big\|\,\big|\delta Y_{s,r}\big|^{\lambda}\,\big|a_{r}\big|\,\Big\|_{L_{2}}
\ \le\ \tfrac12\big\|\delta Y_{s,r}\big\|^{\lambda}_{L_{4}}\,\big\|\sigma_{r}\big\|^{2}_{L_{2q}} ,
\qquad 2q=\frac{8}{2-\lambda} ,
\]
so any $\lambda\le2-8/\mfn$ is paid for by \eqref{eq:rItoLn}, and $\mfn >4$ makes that bound
positive. Fix $\lambda:=\min\{\varkappa_{\varphi}-2,\,2-8/\mfn \}>0$, legitimate since
$\varkappa_{\varphi}>1/\gamma>2$. By Step~\ref{stp:rIto1}, $\|\delta Y_{s,r}\|_{L_{4}}\lesssim|r-s|^{\gamma}$, so
\eqref{eq:rItoLnJ4} is $O(|t-s|^{1+\gamma\lambda})$ in $L_{2}$, and so is its conditional mean. The
exponent exceeds $1$, which is all that
Step~\ref{stp:rIto4} asks of the conditional mean. The unconditional threshold $\tfrac12$ of
Lemma~\ref{lem:SSL}, met at $\gamma+\beta$ by \eqref{eq:paramsIto}, is untouched, since
$1+\gamma\lambda>1$.

\emph{Step 3: the Taylor step.}\stepnum{3}{stp:rIto3}
The germ of \cite[proof of Thm.~4.13]{FHL} is
\[
A_{s,t}=\big\langle D\varphi(Y_{s}),\delta Y_{s,t}\big\rangle
+\big\langle D^{2}\varphi(Y_{s}),Y'^{\otimes2}_{s}\ZZ_{s,t}\big\rangle
+\tfrac12\big\langle D^{2}\varphi(Y_{s}),\delta Y^{\otimes2}_{s,t}
-2Y'^{\otimes2}_{s}\operatorname{Sym}\ZZ_{s,t}\big\rangle .
\]
Since $D^{2}\varphi(Y'^{\kappa}_{s},Y'^{\lambda}_{s})$, the $(\kappa,\lambda)$ entry of that
pairing, is symmetric in $(\kappa,\lambda)$ while $\ZZ_{s,t}-\operatorname{Sym}\ZZ_{s,t}$ is
antisymmetric, the two $\ZZ$-terms cancel, so $A_{s,t}=P_{s}(Y_{t})-\varphi(Y_{s})$,
with $P_{s}$ the second-order Taylor polynomial of $\varphi$ at $Y_{s}$. Denote by
$R_{u,v}:=\varphi(Y_{v})-\varphi(Y_{u})-A_{u,v}$ that Taylor remainder, so that
$|R_{u,v}|\le C_{\varphi}|\delta Y_{u,v}|^{\varkappa_{\varphi}}$. The source sums $R$ over a
partition and takes the limit in $L_{1}$. In $L_{2}$ the estimate would read
\begin{equation}\label{eq:rItoLnL2gap}
\big\|\,|\delta Y_{s,t}|^{\varkappa_{\varphi}}\big\|_{L_{2}}
=\big\|\delta Y_{s,t}\big\|^{\varkappa_{\varphi}}_{L_{2\varkappa_{\varphi}}},
\qquad 2\varkappa_{\varphi}>2/\gamma>4 ,
\end{equation}
while Step~\ref{stp:rIto1} delivers $\delta Y$ only in $L_{4}$: in
$\|\xi\|_{L_{4}}\le\|\xi\|_{4,\mfn }\le\|\xi\|_{L_{\mfn }}$ the right half runs the wrong way, and raising $\mfn $
does not help. In $L_{1}$ everything closes, since
$\varkappa_{\varphi}\le3<4$ gives $\|\cdot\|_{L_{\varkappa_{\varphi}}}\le\|\cdot\|_{L_{4}}$ on a
probability space and hence, for every partition $\pi$ of $[0,t]$,
\begin{equation}\label{eq:rItoLnL1}
\E\Big|\sum_{[u,v]\in\pi}R_{u,v}\Big|
\ \le\ C_{\varphi}\sum_{[u,v]\in\pi}\big\|\delta Y_{u,v}\big\|^{\varkappa_{\varphi}}_{L_{\varkappa_{\varphi}}}
\ \le\ C_{\varphi}\Big(\sup_{s<t}\frac{\|\delta Y_{s,t}\|_{L_{4}}}{|t-s|^{\gamma}}\Big)^{\varkappa_{\varphi}}
t\,|\pi|^{\varkappa_{\varphi}\gamma-1}\ \longrightarrow\ 0 ,
\end{equation}
the rate being positive because $\varkappa_{\varphi}\gamma>1$. Only the inner index $4$ appears in \eqref{eq:rItoLnL1}, so Step~\ref{stp:rIto1} is
needed here at outer index $4$ alone, where the first display of \eqref{eq:rItoLnfrozen} collapses to
$\|Y'_{s}\delta Z_{s,t}\|_{4,4}=\|Y'_{s}\|_{L_{4}}|\delta Z_{s,t}|$ and requires
$\sup_{t}\|Y'_{t}\|_{L_{4}}<\infty$. The germ sums therefore
converge to $\varphi(Y_{\cdot})-\varphi(Y_{0})$ in $L_{1}$, and Step~\ref{stp:rIto4} identifies the limit.

\emph{Step 4: identification.}\stepnum{4}{stp:rIto4} Lemma~\ref{lem:SSL} applies to
the germ $A$ at $m=2$. Write $\mathcal I$ for the right-hand side of the formula $\varphi(Y_{t})-\varphi(Y_{0})=\mathcal I_{t}$, that is, for the sum of the drift and It\^o
integrals and the two rough terms, and
$D_{s,t}:=\mathcal I_{t}-\mathcal I_{s}-A_{s,t}$. Then
$\delta A_{s,u,t}=-D_{s,t}+D_{s,u}+D_{u,t}$, and $\|\E_{s}D_{u,t}\|_{L_{2}}\le\|\E_{u}D_{u,t}\|_{L_{2}}$
by the tower property, so the two estimates on $D$ that Step~\ref{stp:rIto2} assembles, of orders
that \eqref{eq:paramsIto} puts above $\tfrac12$ and above $1$, give the two defect bounds of
Lemma~\ref{lem:SSL}. Those same estimates make $\mathcal I$ the sewing of $A$, by the uniqueness in
that lemma. So the
Riemann sums converge to $\mathcal I$ in $L_{2}$, by the Riemann-sum clause of
Lemma~\ref{lem:SSL}, and to $\varphi(Y_{\cdot})-\varphi(Y_{0})$ in $L_{1}$ by
Step~\ref{stp:rIto3}. Both are limits in probability, so the two agree almost surely at each $t$
and, both being continuous, up to indistinguishability. That is the rough It\^o formula displayed in
Remark~\ref{rem:MPbracket}, and no index condition beyond $\mfn >4$ was used.

{\color{blue}Proposition~\ref{prop:Taylorsew} reaches the same identity without forming a Riemann
sum, and in $L_{2}$ rather than $L_{1}$. It asks $\mfn >2/\gamma$ in return. The present lemma does not
need it.\color{black}}
\end{proof}

\subsection{The two notions agree}\label{subsec:MPequiv}

 Let us directly state the linear growth conditions on the controlled pair:
\begin{equation}\label{eq:MProughgrowth}
\big|f_{t}(x)\big|+\big|(Df_{t}f_{t}+f'_{t})(x)\big|\ \le\ C_{\mathrm{gr}}\big(1+|x|\big),
\qquad (t,x)\in[0,T]\times\R^{n} .
\end{equation}

\begin{proposition}[The two notions agree]\label{prop:MPequiv}
Let $m\ge2$ and let \eqref{eq:Egrowth} and \eqref{eq:MProughgrowth} hold.
\begin{enumerate}[label=\normalfont(\roman*),leftmargin=2.4em]
\item\label{mpe1} \emph{(Forward.)} Let $m_{0}\ge3m$. If $X$ is an $L_{m_{0}}$-conditional solution of
\eqref{eq:rsde} on some stochastic basis such that $\|\sup_{t\le T}|X_{t}|\|_{L_{m'}}<\infty$ for some
$m'\ge3m$, then the law of $(X,B)$ solves the rough martingale problem of
Definition~\ref{def:roughMP} at exponent $m$.
\item\label{mpe2} \emph{(Converse.)} If the law of $(X,B)$ solves Definition~\ref{def:roughMP}, then $X$
is a solution of \eqref{eq:rsde} in the \emph{$L_{m}$-conditional} sense of
Definition~\ref{def:davie}, on the filtration generated by $(X,B)$. We can take the localization along
the exit times of the coordinate process,
\begin{equation}\label{eq:MPlocseq}
\tau_{R}:=\inf\{t\ge0:\ |X_{t}|\ge R\},\qquad R\ge1,
\end{equation}
for which \ref{MP0} gives more than $\tau_{R}\wedge T\uparrow T$ almost surely, namely $\tau_{R}>T$
for all $R$ large enough on almost every path.
\end{enumerate}
\end{proposition}

\begin{remark}[The factor $3$]\label{rem:MPfactor3}
The exponent $3m$ enters \ref{mpe1} at three places. At each of them the $3$ is the order
$\varkappa_{\varphi}$ of the test class.

The first place is the Taylor remainder in \eqref{eq:MPR3}, bounded through
$\||\xi|^{\varkappa_{\varphi}}\|_{L_{m}}=\|\xi\|^{\varkappa_{\varphi}}_{L_{\varkappa_{\varphi} m}}$.
Its rate is $o(|t-s|)$ as soon as $\varkappa_{\varphi}\gamma>1$. The second is the mismatch term
$\tilde Q^{\varphi}$ in \eqref{eq:PQtildebound}, bounded through
$\Op{D^{2}\varphi(X_{r})-D^{2}\varphi(X_{s})}\le\|D^{3}\varphi\|_{\infty}|\delta X_{s,r}|$ and the
H\"older triple $(3,3,3)$. The third is the top degree of the envelope \eqref{eq:Dphibound} in
Step~\ref{stp:MP13}, where generalised H\"older is applied at $N=3$. That degree comes from the
Taylor row and is then inherited: Step~\ref{stp:MP7} records that substituting \eqref{eq:davie} for
$\delta X$ leaves the total degree unchanged, and Step~\ref{stp:MP8} that its own substitutions do
not raise it, the row quadratic in $J$ leaving a power to spare. Any $\varkappa_{\varphi}>1/\gamma$
would be enough at all three.

If we take instead the test class $C^{2,\alpha}_{b}$ with $\alpha>\tfrac1\gamma-2$, valid
 because $\gamma>\tfrac13$, and replace the seminorm \eqref{eq:seminorm3} by
$\|D\varphi\|_{\infty}+\|D^{2}\varphi\|_{\infty}+[D^{2}\varphi]_{\alpha}$. At the first place $3m$
becomes $(2+\alpha)m$. At the second it becomes $(2+\alpha)m$ as well, once the H\"older pair is
rebalanced to $\big(\tfrac{2+\alpha}\alpha,\tfrac{2+\alpha}2\big)$. The rate then drops to
$O(|t-s|^{1+\alpha\gamma})$, which is weaker but still enough. At the third, the envelope has to be
rebuilt at top degree $2+\alpha$, where the clearing exponent of Step~\ref{stp:MP13} becomes
$2+\alpha-j-\#$, non-negative at every $j$ and zero only on the Taylor row.

This would weaken two conditions elsewhere. The $p>6$ of \eqref{eq:Ep} becomes $p>2(2+\alpha)$. The
$p>12(\tfrac1\gamma-1)^{2}$ that the exactly one half of
Theorem~\ref{thm:YWtransfer}\ref{ywc} requires in the proof of the strong well-posedness corollary
of \cite{HuberII} becomes $p>4(2+\alpha)(\tfrac1\gamma-1)^{2}$. Over the admissible $\alpha$ these
two bounds have infima $2/\gamma$ and $\tfrac4\gamma(\tfrac1\gamma-1)^{2}$, neither attained.

The localizers of \ref{mpe2} survive the larger class. Applying
$[u]_{\alpha}\le(2\|u\|_{\infty})^{1-\alpha}\|Du\|^{\alpha}_{\infty}$ to $u=D^{2}\varphi_{R}$ gives
$[D^{2}\varphi_{R}]_{\alpha}\le CR^{-1-\alpha}$, uniformly in $R\ge1$. On this route the factor
cannot reach $2$: $|\delta X_{s,t}|^{\varkappa_{\varphi}}$ carries no conditional centring, so only
the unconditional rate $\varkappa_{\varphi}\gamma>1$ is available. We prove no lower bound.

We have not carried the change through, since the potential gain did not outweigh the further notational load. Appendix~\ref{app:MPenvelope} is written at
$\varkappa_{\varphi}=3$; the separability step of Lemma~\ref{lem:MPborel}\ref{mpb:borel} uses
$C^{3}_{b}$; and the graded seminorm of Proposition~\ref{prop:gradedequiv} would have to be taken at
$\max\{k,2+\alpha\}$. Everything below is stated at $\varkappa_{\varphi}=3$.
\end{remark}

\begin{proof}[Proof of Proposition~\ref{prop:MPequiv}]
Since the proof is rather long, we split it into several steps.
\emph{Forward,} \ref{mpe1}: \ref{MP0} is immediate. Steps~\ref{stp:MP1}--\ref{stp:MP8} produce the pointwise envelope
\eqref{eq:Dphibound}, Step~\ref{stp:MP9} proves the two $L_{m}$-bounds \eqref{eq:Ntildebound} and
\eqref{eq:PQtildebound} that Step~\ref{stp:MP7} states for the three terms \eqref{eq:Dphibound} sets aside,
Step~\ref{stp:MP10} records the exact decomposition \eqref{eq:Dphisplit} that a pointwise bound cannot see,
Step~\ref{stp:MP11} the four facts the estimates rest on, and Steps~\ref{stp:MP12}--\ref{stp:MP13} assemble the norms and read off
\ref{MP1}--\ref{MP2}. \emph{Converse,} \ref{mpe2}: after a localization the two
conditions of \eqref{eq:daviemoduli} are established in reverse order, the unconditional modulus
first and the conditional one second. 

\medskip\noindent\textbf{\ref{mpe1} Forward.}
By Lemma~\ref{lem:filtdescend} we can take the filtration to be the raw generated one of
Definition~\ref{def:roughMP}. \ref{MP0} is the
integrability in Definition~\ref{def:weaksol}. Every term of \eqref{eq:roughMP} other
than $N^{\varphi}$ is a continuous functional of $(X,B)$ by \eqref{eq:standingcts}, and
$N^{\varphi}$ is taken to be the version of
Lemma~\ref{lem:MPborel}\ref{mpb:version}.

For \ref{MP1}--\ref{MP2} we derive \eqref{eq:Dphibound} directly. The ingredients are
\eqref{eq:davie}, the bracket identity
\eqref{eq:brackettaylor}, the operator identity behind Lemma~\ref{lem:G2}, the It\^o isometry with Burkholder--Davis--Gundy, and
the growth bounds \eqref{eq:Egrowth} and \eqref{eq:MProughgrowth}.
So the strategy is: expand the first terms appearing in $D^{\varphi}_{s,t}$, and then see which of the pieces it breaks into are already sitting there with a minus sign in front of them.

\emph{Step~1: the data along the path.}\stepnum{1}{stp:MP1} By \eqref{eq:Egrowth} the sum $|b_{t}(x)|+\|\sigma_{t}(x)\|$ and by
\eqref{eq:MProughgrowth} the sum $|f_{t}(x)|+|(Df\,f+f')_{t}(x)|$ are each at most
$C_{\mathrm{gr}}(1+|x|)$, so that for every $\mfn \le m'$
\begin{equation}\label{eq:MPdata}
\sup_{t\le T}\Big(\big\|b_{t}(X_{t})\big\|_{L_{\mfn }}+\big\|\sigma_{t}(X_{t})\big\|_{L_{\mfn }}
+\big\|f_{t}(X_{t})\big\|_{L_{\mfn }}+\big\|(Df\,f+f')_{t}(X_{t})\big\|_{L_{\mfn }}\Big)
\ \le\ 4C_{\mathrm{gr}}\Big(1+\big\|\sup_{t\le T}|X_{t}|\big\|_{L_{\mfn }}\Big),
\end{equation}
finite by the hypothesis $\|\sup_{t\le T}|X_{t}|\|_{L_{m'}}<\infty$ of \ref{mpe1}. 

\emph{Step~2: the Taylor expansion of $\delta\varphi(X)$.}\stepnum{2}{stp:MP2} The starting point is the Taylor expansion
of $\delta\varphi(X)_{s,t}$ about $X_{s}$ to order $3$,
\begin{equation}\label{eq:MPtaylor}
\delta\varphi(X)_{s,t}
=\nabla\varphi(X_{s})\cdot\delta X_{s,t}
+\tfrac12D^{2}\varphi(X_{s})\!:\!\delta X_{s,t}^{\otimes2}
+\mathcal R^{\varphi,(3)}_{s,t},
\qquad
\big|\mathcal R^{\varphi,(3)}_{s,t}\big|\ \le\ \tfrac16\|D^{3}\varphi\|_{\infty}\big|\delta X_{s,t}\big|^{3},
\end{equation}
into which we substitute the Davie expansion
\eqref{eq:davie} of $\delta X_{s,t}$.

 The germ is subtracted by the definition of $D^{\varphi}$: besides $\int_{s}^{t}\mathcal L_{r}\varphi(X_{r})\dr$
and $\delta N^{\varphi}_{s,t}$, \eqref{eq:roughMP} subtracts the three germ terms, collected as
$\mathfrak G^{\varphi}_{s,t}$ at \eqref{eq:MPgerm} below, and what has to be shown is that the
rewritten $\delta\varphi(X)_{s,t}$ reproduces $\mathfrak G^{\varphi}_{s,t}$ in full. Level one is
\eqref{eq:gendef}; level two needs the bracket identity \eqref{eq:brackettaylor} as well, the
second-order Taylor term producing $\delta Z^{\kappa}_{s,t}\delta Z^{\lambda}_{s,t}$ and not a
level-two increment. Steps~4--6 do this.

In $L_{m}$ the third-order remainder obeys
\begin{equation}\label{eq:MPR3}
\big\|\mathcal R^{\varphi,(3)}_{s,t}\big\|_{L_{m}}
\ \le\ \tfrac16\|D^{3}\varphi\|_{\infty}\,\big\|\delta X_{s,t}\big\|^{3}_{L_{3m}}
\ =\ O\big(|t-s|^{3\gamma}\big)\ =\ o\big(|t-s|\big),
\qquad 3\gamma>1 ,
\end{equation}
by the identity $\big\||\xi|^{3}\big\|_{L_{m}}=\|\xi\|^{3}_{L_{3m}}$, the rate
$\|\delta X_{s,t}\|_{L_{3m}}=O(|t-s|^{\gamma})$ being the one computed in Step~\ref{stp:MP9}\ref{stp:MP9i} at
$2m$ and re-read there at $3m$.

\emph{Step~3: notation for the five Davie summands and for the subtracted germ.}\stepnum{3}{stp:MP3} Throughout this proof we abbreviate
the five summands of the Davie expansion \eqref{eq:davie} at the pair $(s,t)$ by
\begin{equation}\label{eq:MPpieces}
\begin{gathered}
A_{s,t}:=\int_{s}^{t}\!b_{r}(X_{r})\dr,
\qquad
M_{s,t}:=\int_{s}^{t}\!\sigma_{r}(X_{r})\,\d B_{r},
\qquad
G^{1}_{s,t}:=\sum_{\kappa}f^{\kappa}_{s}(X_{s})\,\delta Z^{\kappa}_{s,t},
\\
G^{2}_{s,t}:=\sum_{\kappa,\lambda}\big(Df^{\lambda}_{s}f^{\kappa}_{s}+f'^{\kappa\lambda}_{s}\big)(X_{s})\,
\ZZ^{\kappa\lambda}_{s,t},
\qquad\text{so that}\qquad
\delta X_{s,t}=A_{s,t}+M_{s,t}+G^{1}_{s,t}+G^{2}_{s,t}+J_{s,t} ,
\end{gathered}
\end{equation}
and we denote by
\begin{equation}\label{eq:MPgerm}
\mathfrak G^{\varphi}_{s,t}:=\sum_{\kappa}\big(\mathcal G^{\kappa}_{s}\varphi\big)(X_{s})\,\delta Z^{\kappa}_{s,t}
+\sum_{\kappa,\lambda}\big(\mathcal G^{\kappa}_{s}\mathcal G^{\lambda}_{s}
+\mathcal G'^{\kappa\lambda}_{s}\big)\varphi(X_{s})\,\ZZ^{\kappa\lambda}_{s,t}
+\tfrac12\sum_{\kappa,\lambda}\big(\mathcal Q^{\kappa\lambda}_{s}\varphi\big)(X_{s})\,\CZ^{\kappa\lambda}_{s,t}
\end{equation}
the germ that \eqref{eq:roughMP} subtracts, so that
$D^{\varphi}_{s,t}=\delta\varphi(X)_{s,t}-\int_{s}^{t}\mathcal L_{r}\varphi(X_{r})\dr
-\delta N^{\varphi}_{s,t}-\mathfrak G^{\varphi}_{s,t}$.

\emph{Step~4: the first-order Taylor term.}\stepnum{4}{stp:MP4} Substituting \eqref{eq:MPpieces} into
$\nabla\varphi(X_{s})\cdot\delta X_{s,t}$ and reading
$\sum_{i}f^{\kappa}_{i}\partial_{i}\varphi=\mathcal G^{\kappa}\varphi$ off \eqref{eq:gendef} yields
\begin{equation}\label{eq:MPfirstorder}
\begin{aligned}
\nabla\varphi(X_{s})\cdot\delta X_{s,t}
={}&\nabla\varphi(X_{s})\cdot A_{s,t}+\nabla\varphi(X_{s})\cdot M_{s,t}
+\sum_{\kappa}\big(\mathcal G^{\kappa}_{s}\varphi\big)(X_{s})\,\delta Z^{\kappa}_{s,t}\\
&\phantom{xx}{}+\sum_{\kappa,\lambda}\Big(\sum_{i}\big(Df^{\lambda}_{s}f^{\kappa}_{s}
+f'^{\kappa\lambda}_{s}\big)_{i}\,\partial_{i}\varphi\Big)(X_{s})\,\ZZ^{\kappa\lambda}_{s,t}
+\nabla\varphi(X_{s})\cdot J_{s,t} .
\end{aligned}
\end{equation}

\emph{Step~5: the second-order Taylor term.}\stepnum{5}{stp:MP5} The Hessian is symmetric, so
$D^{2}\varphi(X_{s})\!:\!(u\otimes v)=D^{2}\varphi(X_{s})(u,v)=D^{2}\varphi(X_{s})(v,u)$, and squaring
\eqref{eq:MPpieces} yields
\begin{equation}\label{eq:MPsecondorder}
\begin{aligned}
\tfrac12D^{2}\varphi(X_{s})\!:\!\delta X_{s,t}^{\otimes2}
={}&\tfrac12D^{2}\varphi(X_{s})\!:\!\big(G^{1}_{s,t}\big)^{\otimes2}
+\tfrac12D^{2}\varphi(X_{s})\!:\!M_{s,t}^{\otimes2}\\
&\phantom{xx}{}+D^{2}\varphi(X_{s})\big(G^{1}_{s,t}+G^{2}_{s,t},\,M_{s,t}\big)
+D^{2}\varphi(X_{s})\big(G^{1}_{s,t}+G^{2}_{s,t},\,J_{s,t}\big)\\
&\phantom{xx}{}+D^{2}\varphi(X_{s})\big(G^{1}_{s,t},G^{2}_{s,t}\big)
+\tfrac12D^{2}\varphi(X_{s})\!:\!\big(G^{2}_{s,t}\big)^{\otimes2}\\
&\phantom{xx}{}+D^{2}\varphi(X_{s})\big(M_{s,t},J_{s,t}\big)
+\tfrac12D^{2}\varphi(X_{s})\!:\!J_{s,t}^{\otimes2}\\
&\phantom{xx}{}+D^{2}\varphi(X_{s})\big(A_{s,t},\,\delta X_{s,t}-\tfrac12A_{s,t}\big) .
\end{aligned}
\end{equation}
The last line collects every term carrying $A_{s,t}$, using
$\tfrac12D^{2}\varphi\!:\!\delta X^{\otimes2}
=\tfrac12D^{2}\varphi\!:\!(\delta X-A)^{\otimes2}+D^{2}\varphi(A,\delta X-\tfrac12A)$.

\emph{Step~6: the expansion reproduces the germ, leaving
\eqref{eq:MPleftover}.}\stepnum{6}{stp:MP6}
We match \eqref{eq:MPgerm} against \eqref{eq:MPfirstorder} and \eqref{eq:MPsecondorder}, one level
at a time, the two level-two summands jointly. Components are written
$\big(f^{\kappa}_{s}\big)_{i}$, just as
$\big(Df^{\lambda}_{s}f^{\kappa}_{s}+f'^{\kappa\lambda}_{s}\big)_{i}$ is written in
\eqref{eq:MPfirstorder}, and the argument $(X_{s})$ of a coefficient is suppressed
where no confusion can arise.

\emph{Level one.} By the definition of $\mathcal G^{\kappa}$ in \eqref{eq:gendef},
\begin{equation}\label{eq:MPcancel1}
\nabla\varphi(X_{s})\cdot G^{1}_{s,t}
=\sum_{i}\partial_{i}\varphi(X_{s})\sum_{\kappa}\big(f^{\kappa}_{s}\big)_{i}(X_{s})\,\delta Z^{\kappa}_{s,t}
=\sum_{\kappa}\big(\mathcal G^{\kappa}_{s}\varphi\big)(X_{s})\,\delta Z^{\kappa}_{s,t},
\end{equation}
which is the $\delta Z$-term of \eqref{eq:MPfirstorder} and the first summand of
\eqref{eq:MPgerm}.

\emph{Level two: the part supplied by the first-order Taylor term.} The $\ZZ$-term of
\eqref{eq:MPfirstorder} is
\begin{equation}\label{eq:MPcancel2a}
\nabla\varphi(X_{s})\cdot G^{2}_{s,t}
=\sum_{\kappa,\lambda}\Big(\sum_{i}\big(Df^{\lambda}_{s}f^{\kappa}_{s}+f'^{\kappa\lambda}_{s}\big)_{i}
\,\partial_{i}\varphi\Big)(X_{s})\,\ZZ^{\kappa\lambda}_{s,t} .
\end{equation}
On its own this is not the second summand of \eqref{eq:MPgerm}, whose coefficient is
$(\mathcal G^{\kappa}_{s}\mathcal G^{\lambda}_{s}+\mathcal G'^{\kappa\lambda}_{s})\varphi$. Expanding
that by the product rule out of \eqref{eq:gendef} and reading the carr\'e du champ off
\eqref{eq:carre}, we obtain
\begin{equation}\label{eq:MPcancelG2}
\big(\mathcal G^{\kappa}_{s}\mathcal G^{\lambda}_{s}+\mathcal G'^{\kappa\lambda}_{s}\big)\varphi
=\sum_{i}\big(Df^{\lambda}_{s}f^{\kappa}_{s}+f'^{\kappa\lambda}_{s}\big)_{i}\,\partial_{i}\varphi
\ +\ \mathcal Q^{\kappa\lambda}_{s}\varphi ,
\end{equation}
This is the identity in the proof of Lemma~\ref{lem:G2}, by the product rule and
\eqref{eq:carre} instead of the Davie expansion. The second half of \eqref{eq:MPcancelG2} is missing from
\eqref{eq:MPcancel2a}. That half, and the whole third summand of \eqref{eq:MPgerm}, come from the
second-order Taylor term.

\emph{Level two: the part supplied by the second-order Taylor term.} Writing out the leading summand
of \eqref{eq:MPsecondorder} and reading the carr\'e du champ off \eqref{eq:carre}, we obtain
\begin{equation}\label{eq:MPcancel2b}
\begin{aligned}
\tfrac12D^{2}\varphi(X_{s})\!:\!\big(G^{1}_{s,t}\big)^{\otimes2}
&=\tfrac12\sum_{i,j}\partial^{2}_{ij}\varphi(X_{s})\sum_{\kappa,\lambda}
\big[\big(f^{\kappa}_{s}\big)_{i}\big(f^{\lambda}_{s}\big)_{j}\big](X_{s})\,
\delta Z^{\kappa}_{s,t}\delta Z^{\lambda}_{s,t}\\
&=\tfrac12\sum_{\kappa,\lambda}\big(\mathcal Q^{\kappa\lambda}_{s}\varphi\big)(X_{s})\,
\delta Z^{\kappa}_{s,t}\delta Z^{\lambda}_{s,t} .
\end{aligned}
\end{equation}
What \eqref{eq:MPcancel2b} carries is the product
$\delta Z^{\kappa}_{s,t}\delta Z^{\lambda}_{s,t}$, which is not a level-two increment. The bracket
identity \eqref{eq:brackettaylor} is what supplies one, and it supplies a $\CZ$ alongside:
\begin{equation}\label{eq:MPcancelbr}
\tfrac12\sum_{\kappa,\lambda}\mathcal Q^{\kappa\lambda}_{s}\varphi\,
\delta Z^{\kappa}_{s,t}\delta Z^{\lambda}_{s,t}
=\tfrac12\sum_{\kappa,\lambda}\mathcal Q^{\kappa\lambda}_{s}\varphi\,\ZZ^{\kappa\lambda}_{s,t}
+\tfrac12\sum_{\kappa,\lambda}\mathcal Q^{\kappa\lambda}_{s}\varphi\,\ZZ^{\lambda\kappa}_{s,t}
+\tfrac12\sum_{\kappa,\lambda}\mathcal Q^{\kappa\lambda}_{s}\varphi\,\CZ^{\kappa\lambda}_{s,t} .
\end{equation}
In the middle sum we exchange the names of the two summation indices and then use
\eqref{eq:MPQsym},
\begin{equation}\label{eq:MPQsym}
\mathcal Q^{\lambda\kappa}_{s}\varphi
=\sum_{i,j}\big(f^{\lambda}_{s}\big)_{i}\big(f^{\kappa}_{s}\big)_{j}\,\partial^{2}_{ij}\varphi
=\sum_{i,j}\big(f^{\kappa}_{s}\big)_{i}\big(f^{\lambda}_{s}\big)_{j}\,\partial^{2}_{ji}\varphi
=\mathcal Q^{\kappa\lambda}_{s}\varphi ,
\end{equation}
the middle equality by exchanging the names of $i$ and $j$ and the last by the symmetry of the
Hessian. The middle sum of \eqref{eq:MPcancelbr} therefore equals the first, and the two together
carry $\mathcal Q^{\kappa\lambda}_{s}\varphi$ with coefficient $1$ rather than $\tfrac12$:
\begin{equation}\label{eq:MPcancel2c}
\tfrac12D^{2}\varphi(X_{s})\!:\!\big(G^{1}_{s,t}\big)^{\otimes2}
=\sum_{\kappa,\lambda}\big(\mathcal Q^{\kappa\lambda}_{s}\varphi\big)(X_{s})\,\ZZ^{\kappa\lambda}_{s,t}
+\tfrac12\sum_{\kappa,\lambda}\big(\mathcal Q^{\kappa\lambda}_{s}\varphi\big)(X_{s})\,
\CZ^{\kappa\lambda}_{s,t} .
\end{equation}

\emph{The two parts add to the remaining two summands of \eqref{eq:MPgerm}.} Adding
\eqref{eq:MPcancel2a} and \eqref{eq:MPcancel2c} and applying \eqref{eq:MPcancelG2}, we obtain
\begin{equation}\label{eq:MPcancel2}
\begin{aligned}
\nabla\varphi&(X_{s})\cdot G^{2}_{s,t}
+\tfrac12D^{2}\varphi(X_{s})\!:\!\big(G^{1}_{s,t}\big)^{\otimes2}\\
&=\ \sum_{\kappa,\lambda}\Big(\sum_{i}\big(Df^{\lambda}_{s}f^{\kappa}_{s}
+f'^{\kappa\lambda}_{s}\big)_{i}\,\partial_{i}\varphi
+\mathcal Q^{\kappa\lambda}_{s}\varphi\Big)(X_{s})\,\ZZ^{\kappa\lambda}_{s,t}
+\tfrac12\sum_{\kappa,\lambda}\big(\mathcal Q^{\kappa\lambda}_{s}\varphi\big)(X_{s})\,
\CZ^{\kappa\lambda}_{s,t}\\
&=\ \sum_{\kappa,\lambda}\big[\big(\mathcal G^{\kappa}_{s}\mathcal G^{\lambda}_{s}
+\mathcal G'^{\kappa\lambda}_{s}\big)\varphi\big](X_{s})\,\ZZ^{\kappa\lambda}_{s,t}
+\tfrac12\sum_{\kappa,\lambda}\big(\mathcal Q^{\kappa\lambda}_{s}\varphi\big)(X_{s})\,
\CZ^{\kappa\lambda}_{s,t},
\end{aligned}
\end{equation}
the second and third summands of \eqref{eq:MPgerm}. The $\CZ$-term is what \eqref{eq:brackettaylor} leaves behind,
and nothing else cancels it. That is the reason why \eqref{eq:roughMP} carries a $\CZ$-term at all, and
Remark~\ref{rem:MPbracket} records the same computation in advance.

Adding \eqref{eq:MPcancel1} and \eqref{eq:MPcancel2} gives exactly 
$\mathfrak G^{\varphi}_{s,t}$, so that subtracting the germ removes precisely three summands
of the expansion: the $\delta Z$- and $\ZZ$-terms of \eqref{eq:MPfirstorder} and the leading term of
\eqref{eq:MPsecondorder}. Every other summand of \eqref{eq:MPfirstorder} and
\eqref{eq:MPsecondorder} survives, together with $\mathcal R^{\varphi,(3)}_{s,t}$. The only regrouping
is that $\nabla\varphi(X_{s})\cdot J_{s,t}$ and
$D^{2}\varphi(X_{s})\big(G^{1}_{s,t}+G^{2}_{s,t},J_{s,t}\big)$ are collected into a single factor
against $J_{s,t}$, by the symmetry of the Hessian. This is \eqref{eq:MPleftover}:
\begin{equation}\label{eq:MPleftover}
\begin{aligned}
\delta\varphi(X)_{s,t}-\mathfrak G^{\varphi}_{s,t}
={}&\nabla\varphi(X_{s})\cdot A_{s,t}
\ +\ \nabla\varphi(X_{s})\cdot M_{s,t}
\ +\ \tfrac12D^{2}\varphi(X_{s})\!:\!M^{\otimes2}_{s,t}\\
&\phantom{xx}{}+\Big(\nabla\varphi(X_{s})+D^{2}\varphi(X_{s})
\big(\,\cdot\,,G^{1}_{s,t}+G^{2}_{s,t}\big)\Big)J_{s,t}\\
&\phantom{xx}{}+D^{2}\varphi(X_{s})\big(G^{1}_{s,t}+G^{2}_{s,t},M_{s,t}\big)\\
&\phantom{xx}{}+D^{2}\varphi(X_{s})\big(G^{1}_{s,t},G^{2}_{s,t}\big)
+\tfrac12D^{2}\varphi(X_{s})\!:\!\big(G^{2}_{s,t}\big)^{\otimes2}\\
&\phantom{xx}{}+D^{2}\varphi(X_{s})\big(M_{s,t},J_{s,t}\big)
+\tfrac12D^{2}\varphi(X_{s})\!:\!J^{\otimes2}_{s,t}\\
&\phantom{xx}{}+D^{2}\varphi(X_{s})\big(A_{s,t},\delta X_{s,t}-\tfrac12A_{s,t}\big)
\ +\ \mathcal R^{\varphi,(3)}_{s,t} .
\end{aligned}
\end{equation}
(The coefficient of $J_{s,t}$ on the second line is the factor $W^{\varphi}_{s,t}$ of
\eqref{eq:Dphisplit}.)

Steps~7--13, the construction of the envelope \eqref{eq:Dphibound} and the norms taken
against it, are Appendix~\ref{app:MPenvelope}. They yield \ref{MP1} and \ref{MP2} at exponent $m$,
which is \ref{mpe1}.

\medskip\noindent\textbf{\ref{mpe2} Converse.} What has to be produced is
\eqref{eq:daviemoduli} at exponent $m$ on the generated filtration: first the second condition, the
unconditional modulus at rate $|t-s|^{1/2}$, then the first, the conditional modulus at rate
$|t-s|$.

The coordinate maps are unbounded, so we localize inside a single $\tnorm\cdot$-ball. Fix
$\chi\in C^{\infty}$ such that $\chi=1$ on $B_{1}$, $\chi=0$ off $B_{2}$ and $0\le\chi\le1$, and set
$\varphi_{R}(x):=x_{i}\chi(x/R)$. By Leibniz we obtain, for $1\le k\le3$,
\[
D^{k}\varphi_{R}(x)=x_{i}R^{-k}(D^{k}\chi)(x/R)+k\,e_{i}\!\odot\!R^{-(k-1)}(D^{k-1}\chi)(x/R),
\]
where, for a vector $v\in\R^{n}$ and a symmetric $(k-1)$-tensor $S$, the symbol $v\odot S$ denotes the
symmetrised tensor product,
\[
(v\odot S)_{j_{1}\dots j_{k}}\ :=\ \frac1k\sum_{r=1}^{k}v_{j_{r}}\,S_{j_{1}\dots\widehat{j_{r}}\dots j_{k}} ,
\]
a hat denoting an omitted index. Since $(v\odot S)(u^{\otimes k})=(v\cdot u)\,S(u^{\otimes(k-1)})$, we
obtain $\|v\odot S\|\le|v|\,\|S\|$ in the operator norm, the only property of $\odot$ used below, at
$v=e_{i}$ with $|e_{i}|=1$. Both terms are supported where $|x|\le2R$; there $|x_{i}|\le2R$, so for
$R\ge1$ we obtain
\[
\big\|D^{k}\varphi_{R}\big\|_{\infty}
\ \le\ 2R\cdot R^{-k}\big\|D^{k}\chi\big\|_{\infty}+k\,R^{-(k-1)}\big\|D^{k-1}\chi\big\|_{\infty}
\ \le\ (2+k)\max_{0\le j\le3}\big\|D^{j}\chi\big\|_{\infty}\ \le\ C_{\chi}
\]
uniformly in $R\ge1$, and
\begin{equation}\label{eq:phiRunif}
\Lambda_{0}:=\sup_{R\ge1}\max_{i}\tnorm{\varphi_{R}}<\infty .
\end{equation}
On $\{\tau_{R}>t\}$, with $\tau_{R}$ denoting the exit time \eqref{eq:MPlocseq}, one has $\varphi_{R}(X_{r})=X^{i}_{r}$,
$\nabla\varphi_{R}(X_{r})=e_{i}$ and $D^{2}\varphi_{R}(X_{r})=0$ for $r\le t$. The bracket term of
\eqref{eq:roughMP} therefore vanishes there, $\mathcal Q^{\kappa\lambda}$ carrying two derivatives,
whence we obtain
\[
D^{\varphi_{R}}_{s,t}\one_{\{\tau_{R}>t\}}=J^{i}_{s,t}\one_{\{\tau_{R}>t\}}
\]
with $J^{i}$ denoting the $i$-th component of \eqref{eq:davie}. For the stochastic
integrals, this is locality, since on that event the two integrands agree on $[s,t]$. So the localization is insensitive to
the lift, and part \ref{mpe2} needs no geometricity hypothesis. The event $\{\tau_{R}>t\}$ is
$\mathcal G_{t}$- but not $\mathcal G_{s}$-measurable, so inserting it needs an argument, which is
dominated convergence and uses no exponent above $m$. Let $A_{R}:=\{\tau_{R}\le t\}$ denote the
exit event, so that $\Prob[A_{R}]\to0$ by \ref{MP0}. Indeed, \ref{MP0} gives $\|\sup_{u\le T}|X_{u}|\|_{L_{m}}<\infty$, hence
$S:=\sup_{u\le T}|X_{u}|<\infty$ almost surely, and then $|X_{u}|<R$ for every $u\le T$ whenever
$R>S$, that is $\tau_{R}>T$ for every $R>S$. This is the statement recorded after
\eqref{eq:MPlocseq}.

\emph{An $R$-free dominating function.} The Leibniz formula above gives more than \eqref{eq:phiRunif}:
for $k\ge2$ the derivative $D^{k}\varphi_{R}$ is supported in $\{R\le|x|\le2R\}$ and satisfies
$\|D^{k}\varphi_{R}\|_{\infty}\le C_{\chi}R^{-(k-1)}$. Hence, on that support, using $|x|\le2R$ and
$R\ge1$, we obtain
\[
\big(1+|x|\big)^{k}\big|D^{k}\varphi_{R}(x)\big|
\ \le\ C_{\chi}R^{-(k-1)}\big(1+2R\big)^{k-1}\big(1+|x|\big)
\ \le\ 3^{k-1}C_{\chi}\big(1+|x|\big)
\ \le\ C\big(1+|x|\big),
\]
while $|\nabla\varphi_{R}|\le C_{\chi}$ everywhere.
A coefficient of \eqref{eq:roughMP} paired with $D^{k}\varphi$ grows like $(1+|x|)^{k}$ by
\eqref{eq:Egrowth} \emph{and} \eqref{eq:MProughgrowth}: the $\dt$- and $\d B$-terms by the first, the
$\delta Z$-, $\ZZ$- and $\CZ$-terms by the second, whose coefficients are built from $f$ and from
$Df\,f+f'$ and on which \eqref{eq:Egrowth} says nothing. Every term of $D^{\varphi_{R}}_{s,t}$ other
than $\delta N^{\varphi_{R}}_{s,t}$ is therefore bounded, uniformly in $R\ge1$,
by the quantity
\[
H_{s,t}:=C\Big(\big|\delta X_{s,t}\big|+\big(1+\sup_{r\in[s,t]}|X_{r}|\big)
\big(|t-s|+|\delta Z_{s,t}|+|\ZZ_{s,t}|+|\CZ_{s,t}|\big)\Big),
\]
of degree one in the state and lying in $L_{m}$ by \ref{MP0} at exponent $m$.

\emph{The second condition.} By \eqref{eq:phiRunif} and \ref{MP2} in the form \eqref{eq:MP2form},
$\|D^{\varphi_{R}}_{s,t}\|_{L_{m}}\le\omega'_{\Lambda_{0}}(|t-s|)$ uniformly in $R$, whence we obtain
\[
\big\|J^{i}_{s,t}\one_{A^{c}_{R}}\big\|_{L_{m}}
=\big\|D^{\varphi_{R}}_{s,t}\one_{A^{c}_{R}}\big\|_{L_{m}}
\ \le\ \big\|D^{\varphi_{R}}_{s,t}\big\|_{L_{m}}
\ \le\ \omega'_{\Lambda_{0}}(|t-s|) .
\]
Letting $R\to\infty$ and using Fatou, we obtain
$\|J^{i}_{s,t}\|_{L_{m}}\le\omega'_{\Lambda_{0}}(|t-s|)=o(|t-s|^{1/2})$, and by the tower property this
is the second condition of \eqref{eq:daviemoduli} in $L_{m}$-conditional form.

\emph{The first condition.} Now $J^{i}_{s,t}\in L_{m}$, and
$D^{\varphi_{R}}_{s,t}-J^{i}_{s,t}=(D^{\varphi_{R}}_{s,t}-J^{i}_{s,t})\one_{A_{R}}$, so, denoting
$N^{i}:=\int_{0}^{\cdot}\sigma^{i\cdot}_{r}(X_{r})\d B_{r}$,
\[
\big\|D^{\varphi_{R}}_{s,t}-J^{i}_{s,t}\big\|_{L_{m}}
\ \le\ \big\|\big(H_{s,t}+|\delta N^{i}_{s,t}|+|J^{i}_{s,t}|\big)\one_{A_{R}}\big\|_{L_{m}}
+\big\|\delta N^{\varphi_{R}}_{s,t}-\delta N^{i}_{s,t}\big\|_{L_{m}}\ \longrightarrow\ 0 ,
\]
the first term because $H,\delta N^{i},J^{i}$ do not depend on $R$, lie in $L_{m}$ and are multiplied by
$\one_{A_{R}}\to0$ a.s., the second by Burkholder--Davis--Gundy and dominated convergence,
$\nabla\varphi_{R}(X_{r})\to e_{i}$ boundedly: more precisely,
\[
\big\|\delta N^{\varphi_{R}}_{s,t}-\delta N^{i}_{s,t}\big\|_{L_{m}}
\ \le\ C_{m}\Big(\int_{s}^{t}\big\|\,\big|\nabla\varphi_{R}(X_{r})-e_{i}\big|^{2}
\,\Frob{\sigma_{r}(X_{r})}^{2}\,\big\|_{L_{m/2}}\dr\Big)^{1/2}\ \longrightarrow\ 0 .
\]
Conditional expectation being an $L_{m}$-contraction,
$\E[D^{\varphi_{R}}_{s,t}\mid\mathcal G_{s}]\to\E[J^{i}_{s,t}\mid\mathcal G_{s}]$ in $L_{m}$, and
\ref{MP1} gives $\|\E[J^{i}_{s,t}\mid\mathcal G_{s}]\|_{L_{m}}\le\omega_{\Lambda_{0}}(|t-s|)=o(|t-s|)$.
Doing this for each $i$ gives \eqref{eq:daviemoduli} coordinatewise.
\end{proof}

The remarks rely on objects built in Appendix~\ref{app:MPenvelope}. We introduce them here in a condensed way. A Taylor expansion yields
$D^{\varphi}_{s,t}=-\mathcal K^{\varphi}_{s,t}+R^{\varphi}_{s,t}$ at \eqref{eq:Dphibound}, where
$\mathcal K^{\varphi}$ collects the three terms set aside, $\delta\tilde N^{\varphi}$ and the two
integrals $\tilde P^{\varphi},\tilde Q^{\varphi}$ of \eqref{eq:PQtilde}, and where $R^{\varphi}$
satisfies $|R^{\varphi}_{s,t}|\le\tnorm\varphi\,\Xi_{s,t}$ pointwise for the $\varphi$-free envelope
\[
\Xi_{s,t}=C\Big(|J_{s,t}|+\sum_{j=0}^{3}\big(1+\sup_{r\in[s,t]}|X_{r}|\big)^{3-j}
\big|\delta X_{s,t}\big|^{j}\,\Theta^{(j)}_{s,t}\Big) ,
\]
the factors $\Theta^{(j)}$ of \eqref{eq:Thetadef} being $1$ at $j\ge2$ and built from $|t-s|$,
$|\delta Z_{s,t}|$, $|\ZZ_{s,t}|$ and a normalized It\^o increment at $j\le1$. Sorting the summands
by whether they are conditionally centred then gives
$D^{\varphi}_{s,t}=W^{\varphi}_{s,t}J_{s,t}+\mathcal M^{\varphi}_{s,t}+\Upsilon^{\varphi}_{s,t}$
at \eqref{eq:Dphisplit}, with $W^{\varphi}_{s,t}$ $\mathcal G_{s}$-measurable, $\mathcal M^{\varphi}$
conditionally centred and $\Upsilon^{\varphi}$ the rest. \eqref{eq:Wsplit} is the H\"older split of
the first summand, at the uneven indices $r=\tfrac{mm_{0}}{m_{0}-m}$ and $r'=m_{0}$,
that is $r=\tfrac32m$ and $r'=3m$ at the stated Davie order.

\begin{remark}[The cross term in $\delta Z$ and $J$]\label{rem:MPgrouping}
The cross term $D^{2}\varphi(X_{s})(f_{s}(X_{s})\delta Z_{s,t},J_{s,t})$ is not conditionally centred,
and H\"older at $(2,2)$ bounds it in $L_{m}$ only by
$C\tnorm\varphi\big(1+\|\sup_{t\le T}|X_{t}|\|_{L_{2m}}\big)|\delta Z_{s,t}|\,\|J_{s,t}\|_{L_{2m}}
=o(|t-s|^{1/2+\gamma})$; conditioning first pulls the $\mathcal G_{s}$-measurable
factor $f_{s}(X_{s})\delta Z_{s,t}$ out and leaves $\|\E_{s}J_{s,t}\|_{L_{2m}}$, so the
bound becomes $O(|t-s|^{\gamma})\,o(|t-s|)$, which is $o(|t-s|)$,\color{blue}
\begin{align*}
\big\|\E\big[D^{2}\varphi(X_{s})\big(f_{s}(X_{s})\delta Z_{s,t},J_{s,t}\big)\mid\mathcal G_{s}\big]\big\|_{L_{m}}
\ &=\ \big\|D^{2}\varphi(X_{s})\big(f_{s}(X_{s})\delta Z_{s,t},\E[J_{s,t}\mid\mathcal G_{s}]\big)\big\|_{L_{m}}\\
\ &\le\ C\tnorm\varphi\Big(1+\big\|\sup_{t\le T}|X_{t}|\big\|_{L_{2m}}\Big)
\big|\delta Z_{s,t}\big|\,\big\|\E[J_{s,t}\mid\mathcal G_{s}]\big\|_{L_{2m}}\\
\ &=\ O\big(|t-s|^{\gamma}\big)\cdot o\big(|t-s|\big)
\end{align*}
\color{black}
so it is grouped within $W^{\varphi}$ in \eqref{eq:Dphisplit}.
\end{remark}

\begin{remark}[The factor $|t-s|$]\label{rem:MPtsexception}
The exception is $|t-s|$, and it is not a germ factor: once $\tilde Q^{\varphi}$ of
\eqref{eq:PQtilde} is set aside, the compensated square
$\big(\int_{s}^{t}\sigma\d B\big)^{\otimes2}-\int_{s}^{t}\sigma\sigma^{\top}(X_{r})\dr$ is majorized
pointwise only by $\big|\int_{s}^{t}\sigma\d B\big|^{2}+\int_{s}^{t}\operatorname{tr}a_{r}(X_{r})\dr$,
whose second summand is at most $C_{\mathrm{gr}}^{2}\big(1+\sup_{r\in[s,t]}|X_{r}|\big)^{2}|t-s|$ by
\eqref{eq:Egrowth}; its rate $o(|t-s|^{1/2})$ leaves \ref{MP2} unaffected, and the square is
conditionally centred, carried by fact~(3).
\end{remark}

\begin{remark}[The $W^{\varphi}J$ split]\label{rem:MPWsplit}
Moving a $\mathcal G_{s}$-measurable weight across $J$ with no H\"older loss 
would need
\[
\big\|W^{\varphi}J\big\|_{L_{m}}\le\big\|\E[|J|^{m}\mid\mathcal G_{s}]^{1/m}\big\|_{L_{\infty}}
\big\|W^{\varphi}\big\|_{L_{m}}
\]
and would appeal to
\eqref{eq:daviemoduliinfty}, which is not available under \eqref{eq:Egrowth}. The unevenness $r=\tfrac32m$, $r'=3m$ is not forced: $r=r'=2m$ needs $2m\le m'$ and $2m\le m_{0}$, a
stronger demand on $m'$ than $\tfrac32m\le m'$ and a weaker one on $m_{0}$ than $3m\le m_{0}$. \color{blue} We keep $r'=m_{0}$, because that is where the Davie modulus 
sits and because it leaves the largest slack. Read as a degree count the step is not a new estimate at
all: $W^{\varphi}J$ has total degree $2$, one state factor and one $J$, and \eqref{eq:Wsplit} at
$m_{0}=3m$ sits strictly below the generalised H\"older bound ``$3$ factors in $L_{3m}$'' already run
for the top-degree terms.\color{black}
\end{remark}

\color{blue}
\begin{remark}[What facts~(1) and~(2) rule out]\label{rem:MPfactsfail}
Each of the first two facts of Step~11 excludes a term that would break \ref{MP1} or \ref{MP2}.

\emph{Fact~(1).} A germ factor standing alone would be fatal, since
\[
\big\|1+\sup_{r\in[s,t]}|X_{r}|\big\|^{3}_{L_{3m}}\,\big|\delta Z_{s,t}\big|
\ =\ O\big(|t-s|^{\gamma}\big),
\qquad \gamma<\tfrac12<1 ,
\]
such a term would be $O(|t-s|^{\gamma})$ and neither \ref{MP1} nor \ref{MP2} could hold. No such term
exists. The factor that did stand alone in fact~(1),
$\nabla\varphi(X_{s})\cdot\int_{s}^{t}\sigma_{r}(X_{r})\d B_{r}$, would have stood in $\Theta^{(0)}$ at
rate $|t-s|^{1/2}$, which is not $o(|t-s|^{1/2})$, so \ref{MP2} would fail; setting it aside as
$-\delta\tilde N^{\varphi}$ is what removes it.

\emph{Fact~(2).} The qualifier surviving cannot be dropped, and $\delta Z^{\otimes2}$ is the
reason: it is of rate $|t-s|^{2\gamma}$ with $2\gamma<1$ throughout $\gamma\in(\tfrac13,\tfrac12)$, and
it is not conditionally centred, $\delta Z$ being deterministic, so that \ref{MP1} would fail outright
had it survived.
\end{remark}
\color{black}

\begin{remark}\label{rem:MPdaviesharp}
At the endpoint $m'=3m$ the hypothesis $m_{0}\ge3m$ cannot be lowered by the argument below.
Substituting \eqref{eq:davie} into $|\delta X_{s,t}|^{3}$ leaves three kinds of monomials carrying
$J$, namely $|J|^{3}$, $|J|^{2}U$ and $|J|U_{1}U_{2}$. The monomials independent of $J$ are products of
three factors $\mathsf S(X,s,t)$ times rate factors, and need only $\mathsf S\in L_{3m}$, which
$m'\ge3m$ gives. Splitting $|J|^{3}$ as $|J|^{3}=|J|^{2}\cdot|J|$ with
$\tfrac1m=\tfrac2{m_{0}}+\tfrac1\rho$, the first factor is $o(|t-s|)$ by \eqref{eq:daviemoduli} at
order $m_{0}$ and the second is $O(1)$ once $\rho\le m'$, so that $m_{0}\ge\tfrac{2mm'}{m'-m}$
suffices. The other two kinds need $m_{0}\ge\tfrac{2mm'}{m'-m}$ and
$m_{0}\ge\tfrac{mm'}{m'-2m}$, the second is only meaningful at $m'>2m$, which
\ref{mpe1} assumes, and the first of these two numbers is at least the second exactly when
$m'\ge3m$, with equality at $m'=3m$. \ref{mpe1} assumes $m'\ge3m$. That bound is $3m$ at $m'=3m$ and decreases to $2m$ as $m'\to\infty$. We do not prove a lower bound. The three-factor
split of Steps~10 and~13 is the instance $m_{0}=\rho=3m$.
\end{remark}

\begin{remark}\label{rem:MPconverse}
Proposition~\ref{prop:MPequiv}\ref{mpe2} recovers \eqref{eq:daviemoduli} at Davie order $m$, where \ref{mpe1}
used order $3m$, and with \ref{MP0} at $m$ where \ref{mpe1} used a moment bound at $m'\ge3m$. The
two halves therefore do not produce an equivalence. They result in the following sandwich,
\[
\begin{gathered}
\big\{\text{laws of }L_{3m}\text{-conditional solutions with }
\|\textstyle\sup_{t}|X_{t}|\|_{L_{m'}}<\infty \text{ for some }m'\ge3m\ \big\}\\
\subseteq\ \big\{\text{solutions of Definition~\ref{def:roughMP} at }m\big\}
\ \subseteq\ \big\{\text{laws of }L_{m}\text{-conditional solutions}\big\}.
\end{gathered}
\]
However, the following holds: Let $X$ denote an $L_{m}$-conditional solution for every $m\in[2,p]$ such that
$\|\sup_{t}|X_{t}|\|_{L_{p}}<\infty$. Then its law solves Definition~\ref{def:roughMP} at every
$m\in[2,p/3]$ by \ref{mpe1} with $m_{0}=3m$ and $m'=p$, and conversely every such law is the law of an
$L_{m}$-conditional solution at every $m\in[2,p/3]$ by \ref{mpe2}. These two choices are admissible
exactly on that range, $m_{0}=3m\le p$ and $m'=p\ge3m$ holding together if and only if $m\le p/3$. The two conditions have different sources: $m'\ge3m$ is the degree count of
\eqref{eq:Dphibound}, and $m_{0}\ge3m$ is the H\"older split that places the Davie remainder at that
degree. Both are non-strict, and the range closes at $m=p/3$. Both halves of the sandwich are used at
Theorem~\ref{thm:YWtransfer}\ref{ywc}: \ref{mpe2} for the inclusion of the solution set in
$\Gamma_{\nu}$ and \ref{mpe1} for the converse. Non-emptiness of $\Gamma_{\nu}$ comes instead from
Lemma~\ref{lem:filtdescend}, and the applicability of Definition~\ref{def:pathuniq} from the
definition of $\Gamma_{\nu}$.
\end{remark}

\subsection{Convexity, and measurability of the solution set}\label{subsec:MPborel}

\begin{lemma}[Convexity and mixtures]\label{lem:convexsel}
Let $\nu$ be given. The following hold.
\begin{enumerate}[label=\normalfont(\roman*),leftmargin=2.4em]
\item\label{cs:convex} The set of solutions of Definition~\ref{def:roughMP} for $\nu$ is convex.

\item\label{cs:mixture} Let $\pi$ be a probability measure on a measurable space
$\mathrm S$ and $\nu=\int\nu_{\theta}\,\pi(\d\theta)$, let
$\theta\mapsto\Prob_{\theta}$ be measurable from $\mathrm S$ to $\mathcal P(\mathcal C)$, let $\Prob_{\theta}$ solve the problem for $\nu_{\theta}$ with moduli uniform in $\theta$,
and assume that
\begin{equation}\label{eq:mixUI}
\int\E_{\Prob_{\theta}}\Big[\sup_{t\le T}|X_{t}|^{m}\Big]\pi(\d\theta)\ <\ \infty .
\end{equation}
Then $\int\Prob_{\theta}\,\pi(\d\theta)$ solves the problem for $\nu$. 
\end{enumerate}
\end{lemma}
All three extra hypotheses are used: measurability of $\theta\mapsto\Prob_{\theta}$ is what makes the mixture a measure in the
first place, uniformity of the moduli in $\theta$ is needed because the pointwise
supremum $\sup_{\theta}\omega'_{\theta}$ of moduli need not be a modulus,
and \eqref{eq:mixUI} is needed because the third part of \ref{MP0} is a moment bound
and not a modulus, which a mixture does not inherit from its mixands.
\begin{proof}
We write $\Prob:=\int\Prob_{\theta}\,\pi(\d\theta)$. A convex combination is the case of a $\pi$ carried
by finitely many points, and is not treated separately below.

 The first part is affineness of
$\Prob\mapsto\Prob\circ X^{-1}_{0}$. The second is linearity, the increment law being common to all
mixands, so that for bounded measurable $h$ and bounded $\mathcal G_{s}$-measurable $G$ we obtain
\[
\begin{gathered}
\E_{\Prob_{\theta}}\big[h(\delta B_{s,t})G\big]=\E\big[h(\delta B_{s,t})\big]\,\E_{\Prob_{\theta}}[G]
\quad\text{for every }\theta,\\
\text{hence}\qquad
\E_{\Prob}\big[h(\delta B_{s,t})G\big]=\E\big[h(\delta B_{s,t})\big]\,\E_{\Prob}[G],
\end{gathered}
\]
both sides being linear in
$\Prob_{\theta}$ and hence surviving $\int\cdot\,\pi(\d\theta)$; and $\Prob_{\theta}[B_{0}=0]=1$ for every
$\theta$ gives $\Prob[B_{0}=0]=1$. The third part is \eqref{eq:mixUI}, since
\[
\E_{\Prob}\Big[\sup_{t\le T}|X_{t}|^{m}\Big]
=\int\E_{\Prob_{\theta}}\Big[\sup_{t\le T}|X_{t}|^{m}\Big]\pi(\d\theta)\ <\ \infty .
\]

 Lemma~\ref{lem:MPborel}\ref{mpb:version} constructs $N^{\varphi}$ as a limit in
probability, so it does not present $D^{\varphi}$ as a single Borel functional of the
path, and the random variables called
$D^{\varphi}_{s,t}$ under $\Prob$ and under $\Prob_{\theta}$ have to be identified before any statement
about $\Prob\mapsto\E_{\Prob}[D^{\varphi}_{s,t}\,\cdot\,]$ is meaningful. Fix $\varphi$ and a sequence
of partitions of mesh tending to $0$. Since $\Prob$ satisfies \ref{MP0}, Lemma~\ref{lem:MPborel}\ref{mpb:version} applies to it, and the associated Riemann sums
\eqref{eq:MPriemann} converge to $N^{\varphi}$ uniformly on $[0,T]$ in $\Prob$-probability. Pass to a
subsequence converging $\Prob$-almost surely, let $N^{\varphi}$ denote that almost sure limit and let
$A$ denote the set on which the subsequence fails to converge. Then
\[
0=\Prob(A)=\int\Prob_{\theta}(A)\,\pi(\d\theta),
\qquad\text{hence}\qquad
\Prob_{\theta}(A)=0\ \ \text{for }\pi\text{-almost every }\theta ,
\]
and on $A^{c}$ the same limit is a version of $N^{\varphi}$ under $\Prob_{\theta}$ as well.
Indeed $\Prob_{\theta}$ satisfies \ref{MP0} too, so the same lemma gives convergence of
those Riemann sums to its own $N^{\varphi}$ in $\Prob_{\theta}$-probability, and on $A^{c}$ they
converge to the limit we just fixed. Hence one
and the same measurable function on $\mathcal C$ is $D^{\varphi}_{s,t}$ under $\Prob$ and under
$\pi$-almost every $\Prob_{\theta}$, and without that identification the convexity computation would
compare incomparable objects.

\ref{MP2} in the form \eqref{eq:MP2form} is the bound
$\E_{\Prob}[|D^{\varphi}_{s,t}|^{m}]\le\omega'_{\Lambda}(|t-s|)^{m}$ on a functional that is
\emph{affine} in $\Prob$. Applying Fubini, we obtain
\[
\E_{\Prob}\big[|D^{\varphi}_{s,t}|^{m}\big]
=\int\E_{\Prob_{\theta}}\big[|D^{\varphi}_{s,t}|^{m}\big]\,\pi(\d\theta)
\ \le\ \omega'_{\Lambda}(|t-s|)^{m},
\]
so that the bound passes to convex combinations and to mixtures. 
For \ref{MP1} the functional
$\Prob\mapsto\E_{\Prob}\big[\big|\E_{\Prob}[D^{\varphi}_{s,t}\mid\mathcal G_{s}]\big|^{m}\big]$ is not
affine in $\Prob$, which enters both as the measure the outer expectation integrates against and
through the conditioning operator $\E_{\Prob}[\,\cdot\mid\mathcal G_{s}]$. In particular
$\E_{\Prob}[D^{\varphi}_{s,t}\mid\mathcal G_{s}]$ averages the mixands with the conditional law of
$\theta$ given $\mathcal G_{s}$ and not with $\pi$, so Fubini does not carry the bound. We pass to a
product space on which $\theta$ is a random variable. Fix $s\le t$ and $\varphi$, put$\bar\Omega:=\mathcal C\times\mathrm{S}$,
with $\mathrm{S}$ the measurable space $\theta$ runs over, and
\[
\bar\Prob(\d\varpi,\d\theta):=\pi(\d\theta)\,\Prob_{\theta}(\d\varpi) ,
\]
a probability measure because $\theta\mapsto\Prob_{\theta}$ is measurable. We write $\sigma(\theta)$
for the $\sigma$-field of the second coordinate, and
$D:=D^{\varphi}_{s,t}$ for the version fixed above, read together with
$\mathcal G_{s}$ on $\bar\Omega$ through the first coordinate. Then $D\in L_{1}(\bar\Prob)$ by
\ref{MP2} for the mixands, and for bounded $\mathcal G_{s}$-measurable $G$,
\[
\E_{\bar\Prob}\big[DG\big]=\int\E_{\Prob_{\theta}}\big[DG\big]\,\pi(\d\theta)=\E_{\Prob}\big[DG\big],
\qquad\text{hence}\qquad
\E_{\bar\Prob}\big[D\mid\mathcal G_{s}\big]=\E_{\Prob}\big[D\mid\mathcal G_{s}\big] .
\]
Put $U:=\E_{\bar\Prob}[D\mid\mathcal G_{s}\vee\sigma(\theta)]$, which needs no measurable selection,
being a conditional expectation. Its $\theta$-section is $\mathcal G_{s}$-measurable, and testing its
defining identity against $G\,h(\theta)$ with $h$ bounded gives
$\E_{\Prob_{\theta}}[(D-U(\cdot,\theta))G]=0$ for $\pi$-almost every $\theta$. Running $G$ through a
countable algebra generating $\mathcal G_{s}$ leaves one null set, and a monotone class argument
extends the identity to every bounded $\mathcal G_{s}$-measurable $G$, dominated by
$|D-U(\cdot,\theta)|\in L_{1}(\Prob_{\theta})$. So $U(\cdot,\theta)$ is a version of
$\E_{\Prob_{\theta}}[D\mid\mathcal G_{s}]$ for $\pi$-almost every $\theta$. The tower property and
conditional Jensen for $|\cdot|^{m}$ now give the bound,
\[
\E_{\Prob}\Big[\big|\E_{\Prob}[D\mid\mathcal G_{s}]\big|^{m}\Big]
=\E_{\bar\Prob}\Big[\big|\E_{\bar\Prob}[U\mid\mathcal G_{s}]\big|^{m}\Big]
\ \le\ \E_{\bar\Prob}\big[|U|^{m}\big]
\ =\ \int\E_{\Prob_{\theta}}\Big[\big|\E_{\Prob_{\theta}}[D\mid\mathcal G_{s}]\big|^{m}\Big]\pi(\d\theta)
\ \le\ \omega_{\Lambda}(|t-s|)^{m},
\]
the last step by \ref{MP1} for each mixand.
So the bound $\omega_{\Lambda}(|t-s|)^{m}$ passes to convex combinations and to
mixtures over $\pi$; uniformity in $\theta$ is what allows the moduli to be taken common in the last
step.
\end{proof}

\begin{lemma}[The solution set is Borel]\label{lem:MPborel}
Assume the standing continuity \eqref{eq:standingcts} of \S\ref{subsec:equation},
and, for part \ref{mpb:borel} only, the linear
growth \eqref{eq:Egrowth} and \eqref{eq:MProughgrowth}. For each $\nu$ and each $m$, let $\mathcal S_{\nu}\subseteq\mathcal P(\mathcal C)$
denote the set of Borel probability measures on $\mathcal C$ solving the rough martingale problem of
Definition~\ref{def:roughMP} at that $m$ and with $\Law(X_{0})=\nu$.
\begin{enumerate}[label=\normalfont(\roman*),leftmargin=2.4em]
\item\label{mpb:version} \emph{(The version.)} Let
$\varphi\in C^{3}_{b}$ and let $\Prob$ satisfy \ref{MP0}. Then $N^{\varphi}$ has a version with
continuous paths, unique up to $\Prob$-indistinguishability, and it is the limit in
$\Prob$-probability, uniformly on $[0,T]$, of the following Riemann sums,
\begin{equation}\label{eq:MPriemann}
\sum_{i}\nabla\varphi(X_{t_{i}})\cdot\sigma_{t_{i}}(X_{t_{i}})\,\delta B_{t_{i},\,t_{i+1}\wedge\,\cdot}
\end{equation}
along any sequence of partitions of mesh tending to $0$. The joint law of $(X,B,N^{\varphi})$, and with
it the law of the field $D^{\varphi}$ of \eqref{eq:roughMP}, is determined by the law of $(X,B)$ alone.
\item\label{mpb:borel} \emph{(Borel and convex.)} $\mathcal S_{\nu}$, whose
suppressed second parameter is the exponent $m$, is a Borel subset of
$\mathcal P(\mathcal C)$ in the weak topology, and it is convex.
\end{enumerate}
\end{lemma}
\eqref{eq:standingcts} makes each Riemann sum \eqref{eq:MPriemann}, and each
term of \eqref{eq:roughMP} other than $N^{\varphi}$, a continuous functional of the path, the rough
coefficients $f$ and $Df\,f+f'$ entering that count alongside $b$ and $\sigma$, while
\eqref{eq:Egrowth} supplies the envelope that step~\ref{stp:MPbii} of the proof introduces and
step~\ref{stp:MPbiii} re-uses, and \eqref{eq:MProughgrowth} the state-degree count inside
step~\ref{stp:MPbiii}, the two level-two terms being quadratic in $f$. The convergence in \ref{mpb:version} is in
probability and not almost surely, so this route does not present $N^{\varphi}$ as a single Borel
functional of the path, and step~\ref{stp:MPbiii} works with a $\limsup$ of Riemann-sum functionals instead.
\begin{proof}
Part \ref{mpb:version} is Lemma~\ref{lem:itoraw} at $h=\nabla\varphi\cdot\sigma$, which
\eqref{eq:standingcts} makes continuous; step~\ref{stp:MPbi} below is where this
proof uses it. For part \ref{mpb:borel}, convexity is Lemma~\ref{lem:convexsel}. For measurability
there are the following \emph{four} quantifiers to reduce to
countable ones, namely the existential over moduli, the universal over $\varphi$ in a
$\tnorm\cdot$-ball, the universal over pairs $s\le t$, and the norm over $\Omega$.

\emph{(0) The moduli, by an $\varepsilon$--$\delta$ reformulation.}\stepnum{(0)}{stp:MPb0} This is the one that must not be
done by fixing $(\omega,\omega')$ and taking a union over a cofinal family. The set
$\mathfrak W=\{\omega\ \text{non-decreasing}:\omega(h)/h\to0\}$ has no countable cofinal subfamily for
pointwise domination. Let $(\omega_{j})_{j\ge1}\subseteq\mathfrak W$ be given, choose
$t_{j}\downarrow0$ such that $\omega_{j}(t_{j})\le t_{j}/(2j)$, let $\varepsilon$ denote a non-decreasing
interpolant through $(t_{j},1/j)$ and put $\omega(h):=h\varepsilon(h)$. This yields the following,
\[
\frac{\omega(h)}h=\varepsilon(h)\ \longrightarrow\ 0\quad(h\downarrow0),
\qquad
\omega(t_{j})=\frac{t_{j}}j\ >\ \frac{t_{j}}{2j}\ \ge\ \omega_{j}(t_{j})
\qquad(j\ge1),
\]
so $\omega\in\mathfrak W$ and $\omega\not\le\omega_{j}$ for every $j$.

For a family $(\xi_{s,t})$ of non-negative reals, the statement
``there is a non-decreasing $\omega$ such that $\omega(h)/h\to0$ and $\xi_{s,t}\le\omega(|t-s|)$ for all $s\le t$'' is
\emph{equivalent} to
\begin{equation}\label{eq:epsdelta}
\sup_{s\le t}\xi_{s,t}<\infty
\qquad\text{and}\qquad
\forall\,\varepsilon\in\Q_{+}\ \exists\,\delta\in\Q_{+}:\ \ \xi_{s,t}\le\varepsilon|t-s|
\ \text{ whenever } |t-s|\le\delta ,
\end{equation}
and likewise with $|t-s|^{1/2}$ for \ref{MP2}. ($\Rightarrow$ is the definition of
$\omega(h)/h\to0$ for the second half, and for the first half it is the monotonicity that
\ref{MP1} and \ref{MP2} require of their moduli, which gives
$\sup_{s\le t}\xi_{s,t}\le\omega(T)<\infty$. $\Leftarrow$ builds $\omega$ as the non-decreasing envelope of the resulting step
function.) Both quantifiers in \eqref{eq:epsdelta} are countable, and the uniformity over
$\tnorm\varphi\le\Lambda$ is retained by asking \eqref{eq:epsdelta} for
$\xi_{s,t}:=\sup\{\Lambda^{-1}\|\cdots\|:\tnorm\varphi\le\Lambda\}$, reduced to a countable supremum in
(ii) below. Here we can impose \eqref{eq:epsdelta} at each $\Lambda\in\N$ separately, which is harmless and in
fact redundant: by Remark~\ref{rem:MPstrong}(i) the choice $\omega_{\Lambda}=\Lambda\omega_{1}$ is
available on the whole solution set of Definition~\ref{def:roughMP}, so the countable intersection over
$\Lambda\in\N$ cuts out the set already cut out by the single condition at $\Lambda=1$.

\emph{(i) The pairs $(s,t)$.}\stepnum{(i)}{stp:MPbi} Fix $\varphi\in C^{3}_{b}$ and let $\Prob$ satisfy \ref{MP0}. Then
$(s,t)\mapsto D^{\varphi}_{s,t}$ has a $\Prob$-a.s.\ continuous version on $\{s\le t\}$, since every
term of \eqref{eq:roughMP} except $N^{\varphi}$ is a continuous functional of the path, and
$\delta Z,\ZZ,\CZ$ are continuous in $(s,t)$ by Assumption~\ref{ass:driver}. Here $N^{\varphi}$ is
not a continuous functional of the path, the It\^o map being discontinuous on $\mathcal C$, but it has
a version with continuous paths by \ref{mpb:version}. Continuity of the paths is not continuity of $\|D^{\varphi}_{s,t}\|_{L_{m}}$, which
would need uniform integrability, so we reduce \ref{MP2} to rational pairs exactly as \ref{MP1} is
reduced below.

For \ref{MP1} pathwise continuity is not enough.
The left-hand side of \ref{MP1} is $\|\E[D^{\varphi}_{s,t}\mid\mathcal G_{s}]\|_{L_{m}}$, in which the
conditioning $\sigma$-field is $\mathcal G_{s}$, so it changes as $s$ changes, and $(\mathcal G_{t})$ is the raw canonical filtration, not
assumed right-continuous. Approaching $s$ from above would produce $\mathcal G_{s+}$ and not
$\mathcal G_{s}$. We can approach it from below instead. The canonical paths being continuous,
$\mathcal G_{s-}=\bigvee_{r<s}\mathcal G_{r}=\mathcal G_{s}$. Take rationals $s_{j}\uparrow s$ and
$t_{j}\to t$ such that $s_{j}\le t_{j}$ \emph{and} $t_{j}-s_{j}\le t-s$, which is possible once
$s_{j}>s-1/j$, the interval $(t-1/j,\,s_{j}+(t-s)]$ being then non-empty. At $s=0$
there is nothing to do, $0$ being rational: take $s_{j}\equiv0$ and $t_{j}\uparrow t$ rational. The second constraint is what
makes the conclusion below available: we use the modulus in the
$\varepsilon$--$\delta$ form \eqref{eq:epsdelta} of step~\ref{stp:MPb0}, whose hypothesis $|t_{j}-s_{j}|\le\delta$ is inherited from
$|t-s|\le\delta$. On the set already cut out by \ref{MP2},
$D^{\varphi}_{s_{j},t_{j}}\to D^{\varphi}_{s,t}$ almost surely and, by the uniform \ref{MP2} bound at
$m\ge2>1$, in $L_{1}$ by Vitali. We also have
$\E[\xi\mid\mathcal G_{s_{j}}]\to\E[\xi\mid\mathcal G_{s}]$ in $L_{1}$ by martingale convergence along
$\mathcal G_{s_{j}}\uparrow\mathcal G_{s}$. Conditional expectation being an $L_{1}$-contraction, we
obtain
\[
\big\|\E[D^{\varphi}_{s_{j},t_{j}}\mid\mathcal G_{s_{j}}]-\E[D^{\varphi}_{s,t}\mid\mathcal G_{s}]\big\|_{L_{1}}
\ \le\ \big\|D^{\varphi}_{s_{j},t_{j}}-D^{\varphi}_{s,t}\big\|_{L_{1}}
+\big\|\E[D^{\varphi}_{s,t}\mid\mathcal G_{s_{j}}]-\E[D^{\varphi}_{s,t}\mid\mathcal G_{s}]\big\|_{L_{1}}
\ \longrightarrow\ 0 ,
\]
that is, $\E[D^{\varphi}_{s_{j},t_{j}}\mid\mathcal G_{s_{j}}]\to\E[D^{\varphi}_{s,t}\mid\mathcal G_{s}]$ in
$L_{1}$, and Fatou, applied along a subsequence on which the convergence is almost
sure, transfers the $L_{m}$-bound,
\[
\big\|\E[D^{\varphi}_{s,t}\mid\mathcal G_{s}]\big\|_{L_{m}}
\ \le\ \liminf_{j}\big\|\E[D^{\varphi}_{s_{j},t_{j}}\mid\mathcal G_{s_{j}}]\big\|_{L_{m}}
\ \le\ \varepsilon\liminf_{j}|t_{j}-s_{j}|\ =\ \varepsilon|t-s| ,
\]
the middle step by \eqref{eq:epsdelta} at the rational pair, because
$|t_{j}-s_{j}|\le|t-s|\le\delta$. The same Fatou, with the bound $N$ in place of
$\varepsilon|t_{j}-s_{j}|$, transfers the first half of \eqref{eq:epsdelta} from the rational pairs to
all of them. So we can impose \ref{MP1} too at rational pairs only.

\emph{(ii) The test functions.}\stepnum{(ii)}{stp:MPbii} We can give $C^{3}_{b}$ the topology of uniform convergence of
derivatives of order $\le3$ on compacts. The ball $B_{\Lambda}:=\{\varphi\in C^{3}_{b}:\tnorm\varphi\le\Lambda\}$, quotiented by constants (on which
$\tnorm\cdot$ vanishes and on which $D^{\varphi}\equiv0$), is separable in that topology. Not by
Arzel\`a--Ascoli: $\tnorm\varphi\le\Lambda$ bounds $D^{3}\varphi$, but supplies no equicontinuity for
it, so that theorem is unavailable at the top order. The map
$\varphi\mapsto(D^{k}\varphi)_{0\le k\le3}$ embeds $C^{3}(\R^{n})$, carrying the topology of uniform convergence on compacts,
into a finite product of spaces $C(\R^{n};\R^{N})$, each with local uniform convergence; that product
is metrizable, by the countably many seminorms $\sup_{|x|\le l}|\cdot|$, and separable, e.g.~the
polynomials with rational coefficients being dense on each ball, and every subspace of a
separable metric space is separable, the subspace $B_{\Lambda}$ and its quotient by
constants included. Let $\Phi_{\Lambda}\subseteq B_{\Lambda}$ denote a countable dense subset, and let
$\varphi_{j}\in\Phi_{\Lambda}$ converge in that topology to a $\varphi$ such that
$\tnorm\varphi\le\Lambda$. We can pass the limit with no domination by \eqref{eq:Dphibound} anywhere:
$\Xi_{s,t}$ is of total degree $3$ in the state, so $\Xi_{s,t}\in L_{m}$ would require \ref{MP0} at
exponent $3m$. Definition~\ref{def:roughMP} does not provide that exponent, and
the argument below avoids it. Every term of \eqref{eq:roughMP} other than $\delta N^{\varphi}_{s,t}$ is a
continuous functional of the path and of the derivatives of $\varphi$ on compacts, so each converges
pointwise on $\mathcal C$. For the one remaining term we apply Burkholder--Davis--Gundy and obtain,
with $\Delta_{j}(r):=(\nabla\varphi_{j}-\nabla\varphi)(X_{r})$,
\[
\big\|\delta N^{\varphi_{j}}_{s,t}-\delta N^{\varphi}_{s,t}\big\|_{L_{m}}
\ \le\ C_{m}\,\Big\|\Big(\int_{s}^{t}|\Delta_{j}(r)|^{2}\,\Frob{\sigma_{r}(X_{r})}^{2}\dr\Big)^{1/2}\Big\|_{L_{m}} .
\]
Here $\Delta_{j}$ tends to $0$ pointwise, the paths being continuous and hence of compact range,
and that $|\Delta_{j}|\le2\Lambda$. Applying \eqref{eq:Egrowth}, we obtain the pointwise domination
\[
\Big(\int_{s}^{t}|\Delta_{j}(r)|^{2}\,\Frob{\sigma_{r}(X_{r})}^{2}\dr\Big)^{1/2}
\ \le\ 2\Lambda C_{\mathrm{gr}}\big(1+\sup_{r\le T}|X_{r}|\big)\,|t-s|^{1/2} .
\]
That envelope is of degree one in the state, so \ref{MP0} at exponent $m$ integrates it and we do not require any higher exponent. The bound used being $\|\nabla\varphi_{j}\|_{\infty}\le\Lambda$
and not the Lipschitz estimate of \eqref{eq:Ntildebound}, still less the envelope of
\eqref{eq:Dphibound}. Dominated convergence then gives
$\delta N^{\varphi_{j}}_{s,t}\to\delta N^{\varphi}_{s,t}$ in $L_{m}$, hence
$D^{\varphi_{j}}_{s,t}\to D^{\varphi}_{s,t}$ in probability and, along a subsequence, almost surely, and
we deduce \ref{MP2} for $\varphi$ from \ref{MP2} for the $\varphi_{j}$ by Fatou along that subsequence.
For \ref{MP1} uniform integrability takes the place of domination: \ref{MP2} for the $\varphi_{j}$
yields
\[
\sup_{j}\big\|D^{\varphi_{j}}_{s,t}\big\|_{L_{m}}\ \le\ \omega'_{\Lambda}(|t-s|)\ <\ \infty,
\qquad m\ge2>1 ,
\]
so $(D^{\varphi_{j}}_{s,t})_{j}$ is uniformly integrable, and with the almost sure convergence
Vitali's theorem gives $D^{\varphi_{j}}_{s,t}\to D^{\varphi}_{s,t}$ in $L_{1}$. The duality bounds of (iii)
below are tested against bounded $\psi$ and therefore pass to the limit. The whole of step~\ref{stp:MPbii}
thus runs at exponent $m$, the family set aside in \eqref{eq:Dphibound} never being invoked. Hence
\ref{MP1}--\ref{MP2} for all $\varphi\in B_{\Lambda}$ are equivalent to the same for
$\varphi\in\Phi_{\Lambda}$, and we can impose them for $\Lambda\in\N$ only.

\emph{(iii) The conditional bounds.}\stepnum{(iii)}{stp:MPbiii} \ref{MP2} is now the unconditional bound
$\E_{\Prob}[|D^{\varphi}_{s,t}|^{m}]\le\omega'^{m}$ of \eqref{eq:MP2form}, and what has to be produced
is a Borel map on $\mathcal P(\mathcal C)$ computing its left-hand side wherever that side has a
meaning. Since \ref{mpb:version} does not present $D^{\varphi}$ as a single Borel
functional of the path, we use its Riemann-sum approximation. We denote by $D^{\varphi,l}$ the field obtained by replacing $N^{\varphi}$ by its $l$-th Riemann sum
\eqref{eq:MPriemann}. Each
$D^{\varphi,l}_{s,t}$ \emph{is} a fixed continuous functional of the path, so each
$\Prob\mapsto\E_{\Prob}[|D^{\varphi,l}_{s,t}|^{m}]$ is Borel, being the increasing limit
\[
\E_{\Prob}\big[|D^{\varphi,l}_{s,t}|^{m}\big]
=\lim_{N\uparrow\infty}\E_{\Prob}\big[|D^{\varphi,l}_{s,t}|^{m}\wedge N\big]
\]
of the integrals of the bounded continuous functionals $|D^{\varphi,l}_{s,t}|^{m}\wedge N$. Moreover,
$\E_{\Prob}[|D^{\varphi,l}_{s,t}|^{m}]\to\E_{\Prob}[|D^{\varphi}_{s,t}|^{m}]$ in $[0,\infty]$ for every
$\Prob$ satisfying \ref{MP0}, since Burkholder--Davis--Gundy against the degree-one envelope
$2\Lambda C_{\mathrm{gr}}(1+\sup_{r\le T}|X_{r}|)|t-s|^{1/2}$ used in step~\ref{stp:MPbii} yields
\[
\big\|D^{\varphi,l}_{s,t}-D^{\varphi}_{s,t}\big\|_{L_{m}}\ \longrightarrow\ 0
\qquad(l\to\infty),
\]
a quantity \ref{MP0} makes finite at the exponent
$m$. That convergence is available on the \ref{MP0} set only, so we use instead the map
$\Prob\mapsto\limsup_{l}\E_{\Prob}[|D^{\varphi,l}_{s,t}|^{m}]$: a $\limsup$ of Borel functions is Borel
on the whole of $\mathcal P(\mathcal C)$, and this one equals $\E_{\Prob}[|D^{\varphi}_{s,t}|^{m}]$ at
every $\Prob$ satisfying \ref{MP0}. Intersecting with the Borel set cut out by \ref{MP0}, as the
last paragraph of this proof does in any case, we can make the two conditions the same.
For \ref{MP1}, fix $\varphi$ and rational $s\le t$. By duality in
$L_{m}(\mathcal G_{s},\Prob)$ the condition
$\|\E[D^{\varphi}_{s,t}\mid\mathcal G_{s}]\|_{L_{m}(\Prob)}\le c$ is equivalent to the following,
\[
\big|\E_{\Prob}\big[D^{\varphi}_{s,t}\,\psi\big]\big|\ \le\ c\,\E_{\Prob}\big[|\psi|^{m^{*}}\big]^{1/m^{*}}
\qquad\text{for all }\psi\in\Psi_{s},\qquad \tfrac1m+\tfrac1{m^{*}}=1,
\]
where $\Psi_{s}$ denotes any countable family of bounded $\mathcal G_{s}$-measurable functionals of the
canonical path such that it is dense in $L_{m^{*}}(\mathcal G_{s})$ for every $\Prob$. We can take e.g.~the
$\Q$-algebra generated by $\{\varpi\mapsto g(\varpi_{r_{1}},\dots,\varpi_{r_{l}})\}$ with rational
$r_{i}\le s$ and $g$ in a countable subset of $C_{0}$ dense for uniform convergence, $C_{b}$ under
$\|\cdot\|_{\infty}$ not being separable. Such a family is dense because these functionals
generate $\mathcal G_{s}$ and $m^{*}<\infty$.

One implication is immediate. If
$\|\E[D^{\varphi}_{s,t}\mid\mathcal G_{s}]\|_{L_{m}(\Prob)}\le c$, then
$\E_{\Prob}[D^{\varphi}_{s,t}\psi]=\E_{\Prob}\big[\E[D^{\varphi}_{s,t}\mid\mathcal G_{s}]\,\psi\big]$
for every bounded $\mathcal G_{s}$-measurable $\psi$, and H\"older bounds that by
$c\,\E_{\Prob}[|\psi|^{m^{*}}]^{1/m^{*}}$.

The converse is the one that needs an argument. We prove it on the set where \ref{MP2} holds, so that
$D^{\varphi}_{s,t}\in L_{m}(\Prob)$ and $\E_{\Prob}[D^{\varphi}_{s,t}\psi]$ is finite at bounded
$\psi$. The inequality above bounds the linear functional
$\psi\mapsto\E_{\Prob}[D^{\varphi}_{s,t}\psi]$ by $c\,\|\psi\|_{L_{m^{*}}}$ on $\Psi_{s}$, which is
dense in $L_{m^{*}}(\mathcal G_{s})$. That functional therefore extends to $L_{m^{*}}(\mathcal G_{s})$
with norm at most $c$, and since $m^{*}<\infty$ some $\zeta\in L_{m}(\mathcal G_{s})$ with
$\|\zeta\|_{L_{m}}\le c$ represents the extension. It remains to identify $\zeta$. For
$\psi\in\Psi_{s}\cup\{1\}$ both $\E_{\Prob}[\zeta\psi]$ and
$\E_{\Prob}[\E[D^{\varphi}_{s,t}\mid\mathcal G_{s}]\psi]$ equal $\E_{\Prob}[D^{\varphi}_{s,t}\psi]$, so
$\E_{\Prob}[(\E[D^{\varphi}_{s,t}\mid\mathcal G_{s}]-\zeta)\psi]=0$ there. That family is
multiplicative and generates $\mathcal G_{s}$, and the difference lies in $L_{m}\subseteq L_{1}$, so
the monotone class theorem carries the identity to every bounded $\mathcal G_{s}$-measurable $\psi$.
Hence $\zeta=\E[D^{\varphi}_{s,t}\mid\mathcal G_{s}]$, and its norm bound is the
converse.

Each map
$\Prob\mapsto\E_{\Prob}[D^{\varphi}_{s,t}\psi]$ is Borel on the Borel set already cut out by
\ref{MP2}, where it is defined. We work on that set for convenience and not forced by integrability. Indeed, $D^{\varphi}$ is of total state degree $2$, not
$3$, its worst terms being
\[
\int_{s}^{t}\mathcal L_{r}\varphi(X_{r})\dr ,
\qquad
\textstyle\sum_{\kappa,\lambda}(\mathcal Q^{\kappa\lambda}_{s}\varphi)(X_{s})\ZZ^{\kappa\lambda}_{s,t}
\qquad\text{and}\qquad
\tfrac12\textstyle\sum_{\kappa,\lambda}(\mathcal Q^{\kappa\lambda}_{s}\varphi)(X_{s})\CZ^{\kappa\lambda}_{s,t},
\]
the middle one being the $\mathcal Q$-half of the $\ZZ$-coefficient of
\eqref{eq:roughMP}, by \eqref{eq:MPcancelG2}, and, unlike the third, present for a geometric
lift as well,
so \ref{MP0} at any admissible $m\ge2$ already makes it integrable, the degree $3$ belonging to the
envelope $\Xi$ and not to $D^{\varphi}$. The exponent $3m$ that \ref{mpe1} asks
for is therefore not needed for the integrability we use. It is needed for the degree count:
$m'\ge3m$ is the count of \eqref{eq:Dphibound}, which Remark~\ref{rem:MPfactor3} traces to the order
$\varkappa_{\varphi}$ of the test class. That count requires $m'\ge\varkappa_{\varphi}m$ at
$\varkappa_{\varphi}=3$, and no $\varkappa_{\varphi}$ attains the floor $1/\gamma$, so $m/\gamma$ is
an infimum and not a condition. The argument there uses the same approximation:
$\Prob\mapsto\E_{\Prob}[D^{\varphi,l}_{s,t}\psi]$ is Borel, $\psi$ being bounded and
$D^{\varphi,l}_{s,t}\psi$ a continuous functional of the, unbounded but reached
from the bounded continuous $(D^{\varphi,l}_{s,t}\psi)^{\pm}\wedge N$ as above, and
\[
\big|\E_{\Prob}\big[(D^{\varphi,l}_{s,t}-D^{\varphi}_{s,t})\psi\big]\big|
\ \le\ \|\psi\|_{\infty}\,\big\|D^{\varphi,l}_{s,t}-D^{\varphi}_{s,t}\big\|_{L_{1}}
\ \longrightarrow\ 0 ,
\]
the limit in $l$ again taken as a $\limsup_{l}$, Borel everywhere and equal to
$\E_{\Prob}[D^{\varphi}_{s,t}\psi]$ on the \ref{MP0} set. \ref{MP0} is Borel as well. The first part is the continuity of
$\Prob\mapsto\Prob\circ X^{-1}_{0}$. For the second, we do not use the martingale property of $B$
and of $B^{i}B^{j}-\delta_{ij}t$. Those conditions presuppose
$B_{t}\in L_{1}$, which a Borel measure on $\mathcal C$ need not supply, and recovering the Brownian
property from them needs L\'evy's characterisation. We test the independence directly, as
Lemma~\ref{lem:convexsel} does, by
\begin{equation}\label{eq:MPbmtest}
\E_{\Prob}\big[h(\delta B_{s,t})\,\psi\big]\ =\ \mathsf N_{t-s}(h)\ \E_{\Prob}[\psi]
\end{equation}
at rational $s\le t$, at $\psi\in\Psi_{s}\cup\{1\}$ and at $h$ in a countable
$\|\cdot\|_{\infty}$-dense subset of $C_{0}(\R^{d})$, where $\mathsf N_{u}$ denotes the centered
Gaussian law of covariance $u\,\mathrm{Id}$. Both sides integrate a bounded continuous functional of
the path, so this is again a countable family of Borel conditions, and no moment of $B$ enters it.
At fixed $h$, the two sides agree on the multiplicative family $\Psi_{s}\cup\{1\}$, which generates
$\mathcal G_{s}$; hence, at every bounded $\mathcal G_{s}$-measurable $\psi$, by the monotone class
theorem. The $h$ are dense in $C_{0}$, and both sides are $\|\cdot\|_{\infty}$-continuous in $h$.

A real pair $s\le t$ is reached from rational $s_{j}\uparrow s$, $t_{j}\to t$ as in
step~\ref{stp:MPbi}, but the test function has to be held back: \eqref{eq:MPbmtest} at $(s_{j},t_{j})$
is available only at $\mathcal G_{s_{j}}$-measurable $\psi$, and $\mathcal G_{s_{j}}$ is smaller than
$\mathcal G_{s}$. So fix $j_{0}$ and let $\psi$ be bounded and $\mathcal G_{s_{j_{0}}}$-measurable.
Every $j\ge j_{0}$ admits that $\psi$, and letting $j\to\infty$ gives \eqref{eq:MPbmtest} at
$(s,t)$ for it: $h(\delta B_{s_{j},t_{j}})\to h(\delta B_{s,t})$ bounded, by path continuity, and
$\mathsf N_{t_{j}-s_{j}}(h)\to\mathsf N_{t-s}(h)$. Both sides being linear in $\psi$ and are stable
under bounded monotone limits, the monotone class theorem carries the identity from the algebra
$\bigcup_{j}\mathcal G_{s_{j}}$ to $\bigvee_{j}\mathcal G_{s_{j}}=\mathcal G_{s-}=\mathcal G_{s}$,
the last equality by the path continuity recorded in Remark~\ref{rem:MPraw}. At $s=0$ there is
nothing to do, $0$ being rational. So
$\delta B_{s,t}$ is independent of $\mathcal G_{s}$ with law $\mathsf N_{t-s}$ at every pair, which
together with
\[
\Prob[B_{0}=0]=\lim_{k}\E_{\Prob}\big[(1-k|B_{0}|)^{+}\big]=1 ,
\]
a decreasing limit of integrals of bounded continuous functionals, is the Brownian
requirement of \ref{MP0}. The third is
\[
\sup_{R}\E_{\Prob}\Big[\big(\sup_{t}|X_{t}|\wedge R\big)^{m}\Big]\ <\ \infty ,
\]
a countable supremum of Borel maps.

Intersecting and unioning the countably many Borel conditions produced by (0)--(iii) (countable
intersections over $\varphi\in\Phi_{\Lambda}$, over rational $s\le t$, over $\psi\in\Psi_{s}$ and $h$, over
$\Lambda\in\N$ and over $\varepsilon\in\Q_{+}$, and countable unions over $\delta\in\Q_{+}$
and, for the first half of \eqref{eq:epsdelta}, over the bound $N\in\N$ in
$\bigcap\{\xi_{s,t}\le N\}$),
we can present $\mathcal S_{\nu}$ as a Borel set.
\end{proof}

\color{blue}
\begin{remark} We state Lemma~\ref{lem:MPborel} for two reasons. Lemma~\ref{lem:convexsel} manipulates $\mathcal S_{\nu}$ as a set, through
mixtures, and Borelness is the minimal regularity under which that is meaningful; and the
$\varepsilon$--$\delta$ reformulation \eqref{eq:epsdelta} of its proof, together with the diagonal
construction showing that no countable cofinal family of moduli exists, is what removes the
existential quantifier over moduli in a form uniform in $\Lambda$.
\end{remark}
\color{black}

\subsection{Graded test functions}\label{subsec:graded}

Definition~\ref{def:roughMP} requires its two moduli uniformly over $\tnorm\cdot$-bounded families of
test functions, where $\tnorm\varphi=\sum_{j\le3}\|D^{j}\varphi\|_{\infty}$. That class is calibrated
to test functions of degree one: the truncations $\varphi_{R}(x)=x_{i}\chi(x/R)$ used in the
proof of Proposition~\ref{prop:MPequiv} satisfy $\sup_{R}\tnorm{\varphi_{R}}<\infty$, which is what
makes the localization work there. For a monomial of degree $k\ge2$ it fails:
\[
\big\|D^{j}\big(x^{\alpha}\chi(x/R)\big)\big\|_{\infty}\ \asymp\ R^{k-j}\ \to\ \infty
\qquad(R\to\infty),\qquad j<k .
\]
So the localization of Proposition~\ref{prop:MPequiv} does not survive at degree $\ge2$, and an
argument that needs such test functions, e.g.~a moment hierarchy, cannot be run from the problem as stated. The grading below admits them once the moments of order $\kappa m$ are known, with
$\kappa=\max\{k,3\}$. It records a moment hierarchy and does not produce one. What the grading below changes is
 that the test class stays $C^{3}_{b}$ and a monomial of degree $k\ge2$ doesn't belong to it,
but the truncations $x^{\alpha}\chi(\cdot/R)$ now have $\sup_{R\ge1}\tnorm{\cdot}_{k}<\infty$. The natural strengthening is to weight each derivative by the
growth it is allowed.

\begin{definition}[Graded seminorms and the graded martingale problem]\label{def:graded}
For $k\in\N$ and $\varphi\in C^{3}(\R^{n})$ we put
\begin{equation}\label{eq:gradedseminorm}
\tnorm\varphi_{k}:=\sum_{j=1}^{3}\ \sup_{x\in\R^{n}}\ \frac{|D^{j}\varphi(x)|}{(1+|x|)^{k-j}} .
\end{equation}
We say that $\Prob$ solves the \emph{graded rough martingale problem of order $k$ at exponent $m$},
$m\ge2$, if it solves Definition~\ref{def:roughMP} at that exponent with \ref{MP1}--\ref{MP2} required
uniformly over $\big\{\varphi\in C^{3}_{b}(\R^{n}):\tnorm\varphi_{k}\le\Lambda\big\}$ rather than over
$\{\tnorm\varphi\le\Lambda\}$.
\end{definition}

\begin{proposition}[The graded problem asks only for more moments]\label{prop:gradedequiv}
Let $k\in\N$ and $m\ge2$, and put $\kappa:=\max\{k,3\}$. Let $X$ denote an
$L_{m_{0}}$-conditional solution of \eqref{eq:rsde} such that $m_{0}\ge\kappa m$ and
$\|\sup_{t\le T}|X_{t}|\|_{L_{m'}}<\infty$ for some $m'\ge\kappa m$, and assume \eqref{eq:Egrowth},
and the linear growth \eqref{eq:MProughgrowth} of the rough field. Then the law of $(X,B)$ solves the
graded rough martingale problem of order $k$ at exponent $m$.
Conversely, a solution of the graded problem of order $k\ge3$ solves Definition~\ref{def:roughMP}.
\end{proposition}

\begin{proof}
The forward direction is the proof of Proposition~\ref{prop:MPequiv}\ref{mpe1} with one change of
bookkeeping. The expansion is Taylor to order $3$ about
$X_{s}$ with \eqref{eq:davie} substituted, the germ cancelling by \eqref{eq:gendef}, \eqref{eq:MPcancelG2} and \eqref{eq:brackettaylor}. So \eqref{eq:paramsIto} is not used and the moment condition is the bare degree count $m'\ge\kappa m$. There, that expansion produced
\eqref{eq:Dphibound}, the $L_{m}$-bounds \eqref{eq:Ntildebound} and \eqref{eq:PQtildebound} for the
three members of $\mathcal K^{\varphi}$, and \eqref{eq:Dphisplit}; we repeat the inspection while
keeping the graded weights. A term carrying $D^{j}\varphi$, $1\le j\le3$, is accompanied by exactly
$j$ factors, each of them a piece of $\delta X$ or a coefficient and hence each of size
$O(1+\sup_{r}|X_{r}|)$, while \eqref{eq:gradedseminorm} yields
$|D^{j}\varphi(x)|\le\tnorm\varphi_{k}(1+|x|)^{k-j}$. This yields the following,
\[
(1+|x|)^{k-j}\,(1+|x|)^{j}=(1+|x|)^{k}\quad\text{for }j\le k,
\qquad
(1+|x|)^{k-j}\le1\quad\text{for }j>k .
\]
For $j\le k$ the two meet exactly and the term is $O((1+\sup_{r}|X_{r}|)^{k})$; for $j>k$, which can
occur only when $k\le2$, the graded weight is $\le1$ and the term is $O((1+\sup_{r}|X_{r}|)^{j})$ with
$j\le3$. We therefore obtain
\[
D^{\varphi}_{s,t}=-\,\mathcal K^{\varphi}_{s,t}+R^{\varphi}_{s,t},
\qquad
|R^{\varphi}_{s,t}|\ \le\ \tnorm\varphi_{k}\cdot\Xi^{(k)}_{s,t},
\qquad
\Xi^{(k)}_{s,t}\ \text{of total degree}\ \kappa=\max\{k,3\}\ \text{in the state},
\]
with the same $\varphi$-free rate factors and the same conditionally centred part as in
\eqref{eq:Dphisplit}. The family set aside is unchanged and does not raise the count. With the
graded weight we obtain
\[
\big|\nabla\varphi(X_{r})-\nabla\varphi(X_{s})\big|\ \le\ \tnorm\varphi_{k}\,\big|\delta X_{s,r}\big|\,
\big(1+\sup_{r}|X_{r}|\big)^{(k-2)\vee0},
\]
since $|D^{2}\varphi(x)|\le\tnorm\varphi_{k}(1+|x|)^{k-2}$; so \eqref{eq:Ntildebound} and the
$\tilde P^{\varphi}$-half of \eqref{eq:PQtildebound} hold with $\tnorm\varphi$ replaced by
$\tnorm\varphi_{k}$ and $(k-2)\vee0$ further state factors, at total degree
\[
2+\big((k-2)\vee0\big)=\max\{k,2\},
\qquad\text{i.e.~}2\ \text{at}\ k=0,1,2,\quad 3\ \text{at}\ k=3,\quad k\ \text{at every}\ k\ge2 .
\]
Similarly, we obtain
\[
\Op{D^{2}\varphi(X_{r})-D^{2}\varphi(X_{s})}\ \le\ \tnorm\varphi_{k}\,\big|\delta X_{s,r}\big|\,
\big(1+\sup_{r}|X_{r}|\big)^{(k-3)\vee0}
\]
from $\Op{D^{3}\varphi(x)}\le\tnorm\varphi_{k}(1+|x|)^{k-3}$, the $\tilde Q^{\varphi}$-half carrying in
addition the two state factors of $a=\sigma\sigma^{\top}$, at total degree
\[
3+\big((k-3)\vee0\big)=\max\{k,3\},
\qquad\text{i.e.~}3\ \text{at}\ k=0,1,2,3,\quad k\ \text{at every}\ k\ge3 .
\]
So the family set aside is of total degree $\max\{k,3\}=\kappa$ for every $k\in\N$, exactly within
the exponents already in use, and we can run the estimate of
Proposition~\ref{prop:MPequiv}\ref{mpe1} with $3$ replaced by $\kappa$, using
$\|\sup_{t\le T}|X_{t}|\|_{L_{\kappa m}}<\infty$ together with the Davie moduli at order $\kappa m$.
The graded weight has state degree $\max\{k-1,1\}\le\kappa-1$,
so instead of \eqref{eq:Wsplit} the split of $W^{\varphi}J$ needs $(\kappa-1)r\le m'$ with
$\tfrac1r=\tfrac1m-\tfrac1{m_{0}}$. At $m_{0}=\kappa m$ this gives $r=\kappa m/(\kappa-1)$, hence
$(\kappa-1)r=\kappa m\le m'$, which is $m'\ge\kappa m$. At $k=1$ the degree is $1$ and the
condition is \eqref{eq:Wsplit}.
The converse is $\tnorm\varphi_{k}\le\tnorm\varphi$ for $k\ge3$, every exponent $j-k$ in
\eqref{eq:gradedseminorm} being then non-positive, so that the graded $\Lambda$-ball contains the
ungraded one and the moduli transfer without loss.
\end{proof}

\section{Weak existence for systems of rough SDEs}\label{sec:weakex}

 $\RZ=(Z,\ZZ)\in\CC^{\gamma}([0,T];\R^{e})$ is a \emph{deterministic} rough path,
$\gamma\in(\tfrac13,\tfrac12)$, and we study
\begin{equation}\label{eq:rsdeE}
\d X_{t}=b(t,X_{t})\dt+\sigma(t,X_{t})\,\d B_{t}+(f,f')(t,X_{t})\,\d\RZ_{t},
\qquad X_{0}\sim\nu,\qquad X\in\R^{n},\ B\in\R^{d}.
\end{equation}

\subsection{Hypotheses}\label{subsec:Ehyp}

\begin{assumption}[It\^o data: continuity, linear growth, coordinatewise diffusion]\label{ass:Eito}
$b:[0,T]\times\R^{n}\to\R^{n}$ and $\sigma:[0,T]\times\R^{n}\to\R^{n\times d}$ are continuous, and there
is $C_{\mathrm{gr}}<\infty$ such that
\begin{equation}\label{eq:Egrowth}
|b(t,x)|+\Frob{\sigma(t,x)}\ \le\ C_{\mathrm{gr}}\big(1+|x|\big),\qquad (t,x)\in[0,T]\times\R^{n};
\end{equation}
and, with $a=\sigma\sigma^{\top}$, the \emph{coordinatewise diffusion bound}
\begin{equation}\label{eq:Eacoord}
a_{ii}(t,x)\ \le\ C_{\mathrm{gr}}\big(1+|x_{i}|\big)\big(1+|x|\big),
\qquad i=1,\dots,n,\quad (t,x)\in[0,T]\times\R^{n}.
\end{equation}
\eqref{eq:Eacoord} is only required in combination with branch \ref{Ecurvdiag} of
Assumption~\ref{ass:Erough}. Under \ref{Ecurvaffine}, assumption \eqref{eq:Egrowth} needs no addition: every
appeal to \eqref{eq:Eacoord} below sits in the \ref{Ecurvdiag} branch, as noted after
\eqref{eq:Taylorminaff}, so under \ref{Ecurvaffine} the coordinatewise bound may be removed from this
assumption.
\end{assumption}

\eqref{eq:Eacoord} is strictly stronger than the $\sigma$-half of \eqref{eq:Egrowth}, which yields
only $a_{ii}\le\Frob{\sigma}^{2}\le C^{2}_{\mathrm{gr}}(1+|x|)^{2}$. For $\sigma(x)=|x|\,e_{2}\otimes e_{1}$ and $x=(M,0,\dots,0)$ we have $\Frob{\sigma(x)}=|x|$,and \eqref{eq:Egrowth} holds with $C_{\mathrm{gr}}=1$, while $a_{22}(x)=M^{2}$ exceeds
$C_{\mathrm{gr}}(1+|x_{2}|)(1+|x|)=C_{\mathrm{gr}}(1+M)$ by a factor $M$. Conversely \eqref{eq:Eacoord}
implies that half up to a factor $\sqrt{C_{\mathrm{gr}}n}$. 

\begin{assumption}[The rough vector field]\label{ass:Erough}
$(f,f')$ is a deterministic controlled vector field: $f_{t}\in C^{3}(\R^{n};\R^{n\times e})$ and
$f'_{t}\in C^{2}(\R^{n};\R^{n\times e\times e})$ for each $t$, such that $(f,f')\in\mathscr D^{\beta,\beta'}_{\RZ}$
for some $\tfrac13<\beta\le\gamma$ and $\tfrac13<\beta'$. We denote by $C_{\RZ}$ the controlled-path
norm of the pair, i.e.\ the smallest constant such that
$|f_{t}(x)-f_{s}(x)-f'_{s}(x)\delta Z_{s,t}|\le C_{\RZ}(1+|x|)|t-s|^{\beta+\beta'}$ and
$|f'_{t}(x)-f'_{s}(x)|\le C_{\RZ}(1+|x|)|t-s|^{\beta'}$. The same bounds hold for one additional derivative, together with a time-modulus for the second derivative of $f$, at a constant
$C^{(1)}_{\RZ}<\infty$,
\begin{align}\label{eq:Eroughderiv}
   &\big|Df_{t}(x)-Df_{s}(x)-Df'_{s}(x)\delta Z_{s,t}\big|\le C^{(1)}_{\RZ}|t-s|^{\beta+\beta'},
\quad
\big|Df'_{t}(x)-Df'_{s}(x)\big|\le C^{(1)}_{\RZ}(1+|x|)|t-s|^{\beta'},\\
&\big|D^{2}f_{t}(x)-D^{2}f_{s}(x)\big|\le C^{(1)}_{\RZ}(1+|x|)|t-s|^{\beta'}, \nonumber
\end{align}
for all $s\le t$ and all $x$. The first of the three is unweighted, being used against
$|f_{w}(x)|\le C_{\mathrm{gr}}(1+|x|)$ at \eqref{eq:Gtemporal}, where a weight would make that product
quadratic and no Gr\"onwall would absorb it. The second and third keep their weights. $Df'$ and $D^{2}f$ are used only in the
truncation of Step~\ref{step:Ereg} and in \S\ref{app:FHL46}, both at fixed $(k,R)$, where the field
is truncated.
Further,
\begin{equation}\label{eq:Eroughgrowth}
\begin{gathered}
\sup_{t\le T}\ \|D^{j}f_{t}\|_{\infty}<\infty\ \ (j=1,2,3),
\qquad
\sup_{t\le T}\ \|D^{j}f'_{t}\|_{\infty}<\infty\ \ (j=1,2),
\\
\sup_{t\le T}\big(|f_{t}(x)|+|f'_{t}(x)|\big)\le C_{\mathrm{gr}}(1+|x|);
\end{gathered}
\end{equation}
and, finally, one of the following two curvature hypotheses is assumed.
\begin{enumerate}[label=\normalfont(\alph*),leftmargin=2.6em]
\item\label{Ecurvaffine} \emph{Norm-curvature case.} The curvature decays in the norm of the
state:
\begin{equation}\label{eq:Ecurvnormhyp}
C_{\mathrm{nrm}}\ :=\ \sup_{t\le T}\ \sup_{x}\ \big(1+|x|\big)\,\Op{D^{2}f_{t}(x)}\ <\ \infty .
\end{equation}
\item\label{Ecurvdiag} \emph{Diagonal case.} $f$ and $f'$ are diagonal: there are
$q^{\kappa}_{t,i}\in C^{3}(\R)$ and $p^{\kappa\lambda}_{t,i}\in C^{2}(\R)$ such that
\[
f^{\kappa}_{t,i}(x)=q^{\kappa}_{t,i}(x_{i}),
\qquad
f'^{\kappa\lambda}_{t,i}(x)=p^{\kappa\lambda}_{t,i}(x_{i}),
\qquad i\le n,\ \kappa,\lambda\le e,
\]
and the coordinatewise curvature decay holds:
\begin{equation}\label{eq:Ecurvdiag}
C_{\mathrm{cur}}\ :=\ \sup_{t\le T}\ \max_{i\le n,\ \kappa\le e}\ \sup_{\xi\in\R}\
\big(1+|\xi|\big)\,\big|\ddot q^{\kappa}_{t,i}(\xi)\big|\ <\ \infty .
\end{equation}
\end{enumerate}
\end{assumption}
\begin{remark}[Why the curvature hypothesis is coordinatewise]\label{rem:Ecurvature}
The germ defect of \eqref{eq:rsdeE} carries the Taylor remainder
$\mathcal T_{u,w}:=\delta[f(X)]_{u,w}-Df_{u}(X_{u})\delta X_{u,w}$, of size
$\tfrac12D^{2}f(\xi)(\delta X_{u,w},\delta X_{u,w})$, and on a greedy interval
$|\delta X_{u,w}|\asymp\Lambda_{\RZ}\Phi_{u}|w-u|^{\gamma}$ with $\Phi=V(X)$. Under
$\|D^{2}f\|_{\infty}<\infty$ alone that remainder is quadratic in $\Phi_{u}$, the greedy mesh
becomes state-dependent and the constant in \eqref{eq:Emoment} would depend on $\nu$, so one power of $\Phi_{u}$ must cancel.
The norm condition \eqref{eq:Ecurvnormhyp} is not suited for the diagonal class: for $f_{i}(x)=q(x_{i})$,
$\Op{D^{2}f(x)}=\max_{i}|\ddot q(x_{i})|$, so at $x^{(M)}=(\xi,M,0,\dots,0)$ with $\ddot q(\xi)\ne0$,
$(1+|x^{(M)}|)\Op{D^{2}f(x^{(M)})}\ge M|\ddot q(\xi)|$ diverges, and for $n\ge2$ it holds only for
diagonal fields with affine $q$. The condition itself is not affineness: $f_{1}=V$, $f_{i}\equiv0$ ($i\ge2$) obeys
$(1+|x|)\Op{D^{2}f(x)}\le\sqrt2$ by \eqref{eq:Vjet} while $D^{2}f\not\equiv0$. For $n\ge2$ neither branch contains the
other. At $n=1$ they agree.

 Substituting the Davie expansion
\eqref{eq:davie} into $D^{2}f(\xi)(\delta X_{u,w},\delta X_{u,w})$, pairs the rough term
$f\,\delta Z$, the martingale increment and the drift increment with one another. Bounding $D^{2}f$
only in the directions those pairings supply, rather than in every direction as
\eqref{eq:Ecurvnormhyp} does, would require
\begin{equation}\label{eq:Ecurv}
\sup_{t\le T}\ \sup_{x}\ \frac{
\big|D^{2}f_{t}(x)\big(f_{t}(x),f_{t}(x)\big)\big|
+\big|\!\operatorname{tr}\big[D^{2}f_{t}(x)\,a(t,x)\big]\big|
+\big|D^{2}f_{t}(x)\big(f_{t}(x),\,b(t,x)\big)\big|}{V(x)}\ <\ \infty ,
\end{equation}
one power of $V$ rather than two being exactly the cancellation the greedy mesh needs. This is a weaker
assumption than either branch of Assumption~\ref{ass:Erough}: \ref{Ecurvaffine} gives it through
\eqref{eq:Egrowth}, by $|S\!:\!a|\le\Op{S}\operatorname{tr}a$ for the middle term, and
\ref{Ecurvdiag} through \eqref{eq:Ecurvdiag} together with the coordinatewise diffusion bound
\eqref{eq:Eacoord}, which is what makes that term $O(V)$. It is nevertheless not what the proof uses,
for two reasons. First, \eqref{eq:Ecurv} accounts for only three of the six pairings: the expansion
also produces $D^{2}f(\delta M,\delta M)$, $D^{2}f(b,b)$, $D^{2}f(f,\delta M)$ and every term in $J$,
and of these the first is covered by $\operatorname{tr}[D^{2}f\,a]$ only in conditional mean and only
with $D^{2}f$ frozen at $X_{u}$, whereas Lemma~\ref{lem:SSL} at $\epsilon_{2}$ bounds the defect
itself, and the others not at all. Second, \eqref{eq:Ecurv} contracts $D^{2}f(x)$ against
coefficients at the same $x$, while the Taylor remainder evaluates $D^{2}f$ at an intermediate
$\xi\in[X_{u},X_{w}]$.
The proof uses instead
the pathwise two-branch bound \eqref{eq:Taylormin} below, one branch linear in $|\delta X|$ and one
quadratic with a weight in the denominator, which a diagonal field with \eqref{eq:Ecurvdiag} supplies
coordinate by coordinate. Its weight is the coordinate $1+|X_{u,i}|$ and cannot be improved to
$\Phi_{u}$: at $X_{u}=(M,0,\dots,0)$, $\delta X=(0,\eta,0,\dots,0)$, $|\mathcal T_{2}|\asymp|\ddot q_{2}(0)|\eta^{2}$
independently of $M$ while $|\delta X|^{2}/\Phi_{u}\asymp\eta^{2}/M\to0$. The germ of $V$ does admit
the global weight, $\partial_{i}V$ vanishing at $x_{i}=0$.
\end{remark}

\begin{example}[An example \ref{Ecurvdiag}.]\label{ex:Adiag}
\emph{An instance.} Fix $n\ge1$, take $e=1$, and let
\[
\d X_{i}=\drifz(X_{i})\,\drifo(\bar X)\dt+\sqrt{\alpha}\,\bsigz(X_{i})\,\bsigo(\bar X)\,\d B^{i}
+\varpi\,\wsigz(X_{i})\,\d\RZ,
\qquad \bar x:=\tfrac1n\textstyle\sum_{j}x_{j},
\]
on $\R^{n}$, with $\alpha\ge0$, where $\varpi\in\R$ denotes a constant, and where
$\drifz,\drifo,\bsigz,\bsigo,\wsigz$ denote continuous, nonnegative functions subject to the growth hypothesis:
there is $C_{\mathrm{gr}}<\infty$ such that, for all $x,m\in\R$,
\begin{equation}\label{eq:Agrowth}
\big|\drifz(x)\drifo(m)\big|\le C_{\mathrm{gr}}\big(1+|x|+|m|\big),\quad
\bsigz(x)^{2}\bsigo(m)^{2}\le C_{\mathrm{gr}}\big(1+|x|\big)\big(1+|m|\big),\quad
\wsigz(x)\le C_{\mathrm{gr}}\big(1+|x|\big).
\end{equation}
Further, $\wsigz\in C^{3}(\R)$, $\wsigz\not\equiv0$, such that the graded decay
$\wsigz''(\xi)=O\big((1+|\xi|)^{-1}\big)$ and
$\|\wsigz'''\|_{\infty}<\infty$ both hold. The non-vanishing is only required for the failure part below,
which divides by a value of $\wsigz$. The growth hypothesis
is stated on the products because each coefficient is a product of a coordinate and 
the empirical mean, and the product of two functions of linear growth is quadratic. 

\emph{Verification of Assumption~\ref{ass:Eito}.} Continuity is immediate, and $|\bar x|\le|x|$,
so that the first two bounds in \eqref{eq:Agrowth} yield
\[
|b(x)|\ \le\ C_{\mathrm{gr}}\sqrt n\,\big(1+2|x|\big),
\qquad \Frob{\sigma(x)}^{2}\ \le\ \alpha C_{\mathrm{gr}}n\big(1+|x|\big)^{2} ,
\]
i.e.~\eqref{eq:Egrowth}. The diffusion
matrix is diagonal with
\[
a_{ii}(x)=\alpha\,\bsigz(x_{i})^{2}\bsigo(\bar x)^{2}
\ \le\ \alpha C_{\mathrm{gr}}\big(1+|x_{i}|\big)\big(1+|\bar x|\big)
\ \le\ \alpha C_{\mathrm{gr}}\big(1+|x_{i}|\big)\big(1+|x|\big),
\]
and this is exactly \eqref{eq:Eacoord}, at a constant independent of $n$. We do not require $\bsigo$ to be constant or bounded, only $\varpi$ must be.

\emph{Verification of \ref{Ecurvdiag}.} With $\varpi$ constant, the rough field is
$f_{i}(x)=\varpi\,\wsigz(x_{i})$, a function of $x_{i}$ alone, and $f'\equiv0$ is diagonal. The structural half of \ref{Ecurvdiag} therefore holds with $q_{i}(\xi)=\varpi\,\wsigz(\xi)$. The
graded decay gives $|\wsigz''(\xi)|\le C(1+|\xi|)^{-1}$, hence
\[
\big(1+|\xi|\big)\big|\ddot q_{i}(\xi)\big|\ \le\ |\varpi|\,C
\qquad\text{for all }\xi\in\R,\ i\le n,
\]
and this is \eqref{eq:Ecurvdiag} with $C_{\mathrm{cur}}=|\varpi|C$; no decay beyond rate $1$ is
required. The three bounds $\|D^{j}f\|_{\infty}<\infty$, $j=1,2,3$, are also in place, but not all for
the same reason. Here $\|D^{2}f\|_{\infty}=|\varpi|\,\|\wsigz''\|_{\infty}\le|\varpi|C$ is the
graded decay with its weight dropped. The bound $\|Df\|_{\infty}<\infty$ follows from that decay
together with $\wsigz\ge0$ and its linear
growth. Indeed, let $|\wsigz'(\xi_{0})|=L$. The decay keeps $\wsigz'$ of one sign and of modulus at
least $L-C\log3$ throughout the interval between $\xi_{0}$ and $2\xi_{0}\pm2$, whose length is
$|\xi_{0}|+2$, so that
\[
\big|\wsigz\big(2\xi_{0}\pm2\big)-\wsigz(\xi_{0})\big|
\ \ge\ \big(L-C\log3\big)\big(|\xi_{0}|+2\big)
\ >\ 3C_{\mathrm{gr}}\big(1+|\xi_{0}|\big)
\qquad\text{once }L>C\log3+3C_{\mathrm{gr}} ,
\]
while the third bound of \eqref{eq:Agrowth} caps both endpoint values by
$3C_{\mathrm{gr}}(1+|\xi_{0}|)$. At the far endpoint $\wsigz$ therefore either exceeds that bound or is
negative. The third derivative
is a separate hypothesis. The next example shows why: it has the graded decay and an
unbounded third derivative, so it is a counterexample to the implication and is itself not
admissible for Assumption~\ref{ass:Erough}. An admissible example is given after.
Let us take
\[
\wsigz''(\xi)=\sin(\xi^{3})(1+\xi^{2})^{-1/2},
\qquad
\big(1+|\xi|\big)|\wsigz''(\xi)|\le\sqrt2,
\qquad
\wsigz'''(\xi)=3\xi^{2}\cos(\xi^{3})(1+\xi^{2})^{-1/2}+O(1) .
\]
The first two entries are the graded decay, and they leave $\wsigz'$ bounded, the integral
$I:=\int_{0}^{\infty}\sin(y^{3})(1+y^{2})^{-1/2}\dy$ being convergent
with $I=0.3343\ldots$, while the third
is unbounded. Here $\wsigz''$ is odd, so $\wsigz'$ is even. For every
choice but one $\wsigz'$ tends to the same non-zero limit $\wsigz'(0)+I$ at $\pm\infty$, making
$\wsigz-\wsigz(0)$ an odd function of linear growth and $\wsigz$ negative on a half-line, that is, outside
the class. We can normalize by $\wsigz'(0)=-I$. This makes $\wsigz'$ vanish at $\pm\infty$ with
$\|\wsigz'\|_{\infty}=I$ and, $\wsigz'$ being $O(\xi^{-3})$ there, makes $\wsigz$ bounded,
$\wsigz-\wsigz(0)$ ranging in $[-0.312,0.312]$. Then $\wsigz(0)=0.4$ makes $\wsigz$ nonnegative and of
(at most) linear growth, as \eqref{eq:Agrowth} requires. Every hypothesis of
\eqref{eq:Agrowth} and the graded decay are thus met, and $\|\wsigz'''\|_{\infty}=\infty$
nonetheless. This
is why we add $\|\wsigz'''\|_{\infty}<\infty$: Assumption~\ref{ass:Erough} requires $j=3$, and
$\varkappa=3$ of Appendix~\ref{app:FHL46} uses it.

\emph{An admissible example.} The function $\wsigz(\xi)=\sqrt{1+\xi^{2}}$
satisfies every hypothesis listed at the start of this example: it is smooth, nonnegative and
of linear growth, its first three derivatives are bounded by $1$, and
$(1+|\xi|)\,|\wsigz''(\xi)|\le2$.

With the three bounds and the third of
\eqref{eq:Agrowth}, Assumption~\ref{ass:Erough}\ref{Ecurvdiag} holds for every $\wsigz$
satisfying the hypotheses listed at the start of this example, the example just displayed
included, and Theorem~\ref{thm:weakex}
applies with $C$ depending on $n$.

Now replace
$\varpi$ by a non-constant $h\in C^{3}_{b}(\R)$, so that the rough field becomes a product of a
coordinate factor and a functional of the empirical mean,
\[
f_{i}(x)=g(x_{i})\,h(\bar x),\qquad g:=\wsigz .
\]
For $n\ge2$ it is diagonal exactly when $h$ is constant, since $\partial_{j}f_{i}=g(x_{i})h'(\bar x)/n$
for $j\ne i$, and only then it satisfies \ref{Ecurvdiag}. The following statement is about $n\ge3$. 
The two dimensions $n=1$ and $n=2$ are
genuine exceptions to it. Let $h$ be
non-constant. Then $D^{2}f_{i}$ carries, besides the same-coordinate block
$g''(x_{i})h(\bar x)\,\delta_{ij}\delta_{ik}$, a mixed block
$g'(x_{i})h'(\bar x)\,\delta_{ij}/n$ (and its transpose in $j\leftrightarrow k$) and a product
block $g(x_{i})h''(\bar x)/n^{2}$. Neither is diagonal, so \ref{Ecurvdiag} fails outright, and neither
is small. Contracted against the diagonal $a$ above, they contribute
\[
\frac{2\,g'(x_{i})h'(\bar x)\,a_{ii}}{n}
\ +\ \frac{g(x_{i})h''(\bar x)}{n^{2}}\sum_{j}a_{jj},
\]
and for functions of the model size (with $g$ of linear growth and $\bsigz=\bsigo=\sqrt{\,\cdot\,}$, so that
$a_{ii}(x)=\alpha x_{i}\bar x$) neither stays bounded relative to the Lyapunov weight $V$. Indeed, on
the diagonal $x_{i}\equiv s>0$ we have $\bar x=s$, $a_{ii}=\alpha s^{2}$ and $V\asymp s\sqrt n$, so
that, $g$ being of linear growth and $g',h',h''$ bounded, the mixed block is of order
$s^{2}/n=V^{2}/n^{2}$ and the product block of order $s^{3}/n=V^{3}/n^{5/2}$, hence of order $V$ and
$V^{2}$ relative to $V$. The configurations $(M,1,\dots,1)$ give the same two orders, and
both are unbounded relative to $V$.

 Replacing $\Phi=V(X)$ by a per-coordinate Lyapunov weight
$V(X_{i})$ leaves both blocks untouched, they being second derivatives in directions $j,k\ne i$ that a
weight which is a function of $x_{i}$ alone does not reach. The difficulty is the coupling, not the growth: it survives $h$ bounded, $h''$ bounded, and every choice of per-coordinate weight. The same
applies to any off-diagonal second-derivative block in directions $j,k\ne i$, however it appears, and in
particular to the blocks a global cut-off introduces into a diagonal field. That is why the cut-off of
Step~\ref{step:Ereg} must act coordinatewise.

Branch \ref{Ecurvaffine}  requires $\sup_{x}(1+|x|)\Op{D^{2}f(x)}<\infty$ and not affineness, as Remark~\ref{rem:Ecurvature} says,
so that non-affineness alone does not rule it out. For $n\ge3$ it does fail. Indeed, fix a pair $(\xi,m)$, put $x_{i}=\xi$, hold the
$n-3$ coordinates other than $i$ and two chosen ones at $0$, and send those two to $M+a$ and to
$-M+b$, where $a+b:=nm-\xi$. Then
\[
\bar x\ =\ \frac{\xi+(M+a)+(-M+b)}{n}\ =\ \frac{\xi+a+b}{n}\ =\ m
\]
for every $M$, while $|x|\ge|M+a|\to\infty$. Sending the two coordinates to $+M$ and $-M$ instead
would force $\bar x=\xi/n$ and leave $m$ unavailable, which at $n=3$ is the whole of the remaining
freedom. The limit does not isolate the product block on its own, the same-coordinate and
the mixed block being functions of $(\xi,m)$ as well, so that we test $D^{2}f(x)$ against
$v=e_{j}$ with $j\ne i$, possible since $n\ge3$: this annihilates those two blocks, which carry
$v_{i}^{2}$ and $v_{i}\sum_{j}v_{j}$, and leaves
\[
\big|D^{2}f_{i}(x)(v,v)\big|\ =\ \big|g(\xi)h''(m)\big|/n^{2} .
\]
Since $|v|=1$, this yields $(1+|x|)\Op{D^{2}f(x)}\gtrsim M\,|g(\xi)h''(m)|/n^{2}$ along that family, and
the product block therefore forces $g(\xi)h''(m)=0$ at every pair $(\xi,m)$. Since $g$ does not vanish
identically, we obtain $h''\equiv0$, hence $h$ affine, hence $h$ constant, $h$ being bounded.

Hence, for $n\ge3$, \ref{Ecurvdiag} holds for the system above when the mean modulation of the
rough field is constant and fails otherwise, \ref{Ecurvaffine} fails at every non-constant modulation
as well, and Theorem~\ref{thm:weakex} does not cover a non-constant one.
\end{example}

\begin{remark}[The two exceptional dimensions of Example~\ref{ex:Adiag}]\label{rem:Adiagsmalln}
The limitation concerns $n\ge3$. At $n=1$, $\bar x=x_{1}$ and $f_{1}(x)=(gh)(x_{1})$ is diagonal for
every $h$, the curvature half holding, once $(1+|\xi|)|(gh)''(\xi)|$ is bounded, as for $h=\arctan$. At
$n=2$, fixing $x_{i}$ and $\bar x$ fixes the configuration: at $e=1$,
$g(\xi)=h(\xi)=(1+\xi^{2})^{-1}$, $f'\equiv0$, the field is not diagonal and \ref{Ecurvdiag} fails,
while $C_{\mathrm{nrm}}$ is finite. Every other requirement of Assumption~\ref{ass:Erough}
holds in both cases, and under \ref{Ecurvaffine} \eqref{eq:Agrowth} supplies \eqref{eq:Egrowth}.
\end{remark}\color{blue}
\begin{remark}\label{rem:Efprime}
$f'$ denotes the Gubinelli derivative of the coefficient in time, so the coefficient in
$\delta f_{s,t}(\cdot)\approx f'_{s}(\cdot)\delta Z_{s,t}$.It only governs the time variation of the rough field. That class $f'\equiv0$ contains every autonomous field, the autonomous case being the one that
occurs at finite $N$, but it does not only consist of them, and the counterexample lies inside
Assumption~\ref{ass:Erough} itself, which permits $\beta+\beta'<1$, as at the triple
$\gamma=\beta=0.45$, $\beta'=0.35$. Fix such a triple, let
$g:\R^{n}\to\R^{n\times e}$ denote an affine map and denote $f_{t}(x):=t^{\,\beta+\beta'}g(x)$ with
$f'\equiv0$. Since $|t^{a}-s^{a}|\le|t-s|^{a}$ for $a\in(0,1]$, the controlled-path defect of $(f,0)$
is at most $|t-s|^{\beta+\beta'}|g(x)|$, so that $(f,0)\in\mathscr D^{\beta,\beta'}_{\RZ}$ with
$C_{\RZ}<\infty$ by the linear growth of $g$, while $|f'_{t}-f'_{s}|=0$ meets the second requirement trivially; the three bounds \eqref{eq:Eroughderiv} follow by the same computation, $g$ affine making
$Dg$ constant and $D^{2}g\equiv0$; $\sup_{t}\|D^{j}f_{t}\|_{\infty}<\infty$ for $j\le3$ and
$\sup_{t}|f_{t}(x)|\le T^{\beta+\beta'}|g(x)|\le C_{\mathrm{gr}}(1+|x|)$; and \ref{Ecurvaffine} holds
with $C_{\mathrm{nrm}}=0$, as does \ref{Ecurvdiag} with $C_{\mathrm{cur}}=0$ if $g$ is taken diagonal.
And $f$ is plainly not autonomous. The whole H\"older increment of $f$ in $t$ can be absorbed into the
remainder as soon as $\beta+\beta'<1$; the direction we use is the other one, autonomy
$\Rightarrow f'\equiv0$ and hence $\mathcal G'\equiv0$.
\end{remark}
\color{black}
\begin{assumption}[Initial moment]\label{ass:Einit}
$\nu$ is a Borel probability measure on $\R^{n}$ such that
$\langle\nu,|x|^{p}\rangle<\infty$ for some
\begin{equation}\label{eq:Ep}
p\ >\ \max\Big\{6,\ \tfrac1\gamma\Big\}\ =\ 6 .
\end{equation}
\end{assumption}

The first entry binds at every admissible $\gamma$, $1/\gamma\in(2,3)$ making the maximum $6$: the
range $m\in[2,p/3]$ of Theorem~\ref{thm:weakex} is non-empty at every $p$ satisfying \eqref{eq:Ep}.

The pathwise-uniqueness input of \cite{HuberII} asks instead
\begin{equation}\label{eq:Epcond}
p\ >\ 4\Big(\frac1\gamma-1\Big)^{2} ,
\end{equation}
and neither bound implies the other: \eqref{eq:Epcond} is the weaker one exactly on
$\gamma>1/(1+\sqrt{3/2})\approx0.44949$, where $4(\tfrac1\gamma-1)^{2}<6$. It is proved there.

We do not assume geometricity of $\RZ$ anywhere in this section, and we use no level beyond the
second. Theorem~\ref{thm:weakex} therefore holds for an arbitrary level-two lift, subject
to Assumption~\ref{ass:driver}. The bracket
changes nothing in Step~\ref{step:Emoment}, its germ coefficient obeying
\[
\tfrac12\mathcal Q^{\kappa\lambda}V=\tfrac12\textstyle\sum_{i,j}f^{\kappa}_{i}f^{\lambda}_{j}\partial^{2}_{ij}V ,
\qquad
\big|\tfrac12\mathcal Q^{\kappa\lambda}V\big|\ \le\ \tfrac12|f|^{2}\Op{D^{2}V}\ \le\ C_{1}V ,
\]
with the scale-free constant of the third bound of \eqref{eq:Vlyap}, entering
\eqref{eq:Egronwallineq} against $\CZ_{s,t}$, $|\CZ_{s,t}|=O(\|\CZ\|_{1}|t-s|)$ by
\eqref{eq:bracketLip}, absorbed by the greedy mesh, $\|\CZ\|_{1}$ entering $\Lambda_{\RZ}:=1+\|\RZ\|_{\gamma}+\|\CZ\|_{1}$.

\color{blue}
\begin{remark}[On \eqref{eq:Ep} and \eqref{eq:Epcond}]\label{rem:Ethreshold}
\begin{enumerate}
    \item $p\ge2$: Burkholder--Davis--Gundy in Step~\ref{step:Emoment}, and every passage through
$L_{m/2}$ in the proof of Lemma~\ref{lem:Escalefree}, which needs $m/2\ge1$.
\item $p>1/\gamma$: Kolmogorov--Chentsov in Step~\ref{step:Etight}, the H\"older exponent
of $\delta X$ being $\gamma$, together with Step~3 of the proof of
Lemma~\ref{lem:Escalefree}, whose series converges precisely when $\gamma m>1$. As a constraint on $p$
it is never active, $6>3>1/\gamma$ at every admissible $\gamma$. Interpreted at $m$ it is the floor of the
range $m\in(1/\gamma,p]$ on which we prove that lemma, the excluded interval $[2,1/\gamma]$ being
non-empty at every admissible $\gamma$; at e.g.~$\gamma=0.4$ it is $[2,2.5]$.
\item $m>4(\tfrac1\gamma-1)^{2}$: this is the floor \cite{HuberII} derives from
Lemma~\ref{lem:roughItoLn}, by capping the index of its controlled pair three times over. It is a
requirement on the exponent at which the uniqueness argument is run, and it plays no part here.
\item $p>6$: the martingale-problem conclusion of Theorem~\ref{thm:weakex} is stated on
$m\in[2,p/3]$, Proposition~\ref{prop:MPequiv}\ref{mpe1} asking $m_{0}=3m\le p$ and $m'=p\ge3m$
together, and the largest $m'$ available is $m'=p$, since with coefficients of linear growth no
moment can exceed what $\nu$ provided. That range is non-empty precisely when $p\ge6$. 
\end{enumerate}
\end{remark}
\color{black}

\subsection{The theorem}\label{subsec:Ethm}

\begin{theorem}[Weak existence for systems of rough SDEs, norm-curvature or diagonal rough field]\label{thm:weakex}
Let Assumptions~\ref{ass:driver}, \ref{ass:Eito}, \ref{ass:Erough} and \ref{ass:Einit} hold,
the coordinatewise bound \eqref{eq:Eacoord} of Assumption~\ref{ass:Eito} only in branch
\ref{Ecurvdiag}, and let
Definition~\ref{def:roughMP} be taken with $D=\R^{n}$. Then the
rough martingale
problem of Definition~\ref{def:roughMP} for $(b,\sigma,f,f',\RZ)$ with initial law $\nu$ admits at least
one solution at every exponent $m\in[2,p/3]$. Further, there is a stochastic basis carrying a Brownian
motion $B$ and an $L_{m}$-conditional solution $X$ of \eqref{eq:rsdeE} such that $\Law(X_{0})=\nu$, for every
$m\in(1/\gamma,p]$ and hence for every $m\in[2,p]$. Moreover
\begin{equation}\label{eq:Emoment}
\E\Big[\sup_{t\le T}\big(1+|X_{t}|\big)^{p}\Big]
\ \le\ \exp\!\Big\{C\Big(1+\big(1+\|\RZ\|_{\gamma}^{1/\gamma}+\|\CZ\|^{1/\gamma}_{1}\big)T\Big)\Big\}\,
\big(1+\langle\nu,|x|^{p}\rangle\big),
\end{equation}
with $C=C\big(C_{\mathrm{gr}},C_{\mathrm{cur}}\text{ or }C_{\mathrm{nrm}},\|Df\|_{\infty},\|Df'\|_{\infty},C_{\RZ},C^{(1)}_{\RZ},\beta,\beta',\gamma,p,n,e\big)$
independent of $\nu$. For a geometric lift $\CZ\equiv0$.
\end{theorem}
Six of the twelve entries of $C$ are analytic constants, and each
is a supremum over $[0,T]$: $C_{\mathrm{gr}}$ by \eqref{eq:Egrowth}, \eqref{eq:Eacoord} and \eqref{eq:Eroughgrowth}; $C_{\mathrm{cur}}$ by \eqref{eq:Ecurvdiag} under \ref{Ecurvdiag} and $C_{\mathrm{nrm}}$
by \eqref{eq:Ecurvnormhyp} under \ref{Ecurvaffine}, whichever branch is used, the
other constant not being assumed finite; $\|Df\|_{\infty}$ and $\|Df'\|_{\infty}$ by
\eqref{eq:Eroughgrowth}; $C_{\RZ}$ by \eqref{eq:cvfdefect} and $C^{(1)}_{\RZ}$ by
\eqref{eq:Eroughderiv}. The last two are moreover taken against $\RZ$, as are the indices
$\beta,\beta',\gamma$; the remaining entries $p,n,e$ are the moment exponent of the initial law and
the two dimensions. What the statement asserts is that beyond these, the explicit factor multiplying
$T$ in the exponent is the whole of the dependence on the driver and on the horizon, and that $\nu$
enters $C$ nowhere.

The proof applies \cite[Thm.~4.6]{FHL} at fixed $(k,R)$, and
\S\ref{app:FHL46} checks its hypotheses.

Proposition~\ref{prop:MPequiv}\ref{mpe1} converts $X$ into a solution of the martingale problem at
every $m\in[2,p/3]$, which is non-empty by \eqref{eq:Ep}.

The additive $1$ in the exponent is the ``$+1$'' of \eqref{eq:Egreedycount}, from the first greedy
interval, and cannot be dropped: at $n=e=d=1$, $b=\sigma=0$, $f(x)=x$, $f'=0$, a geometric lift,
$Z_{t}=at^{\gamma}$, $\nu=\delta_{x_{0}}$, $X_{t}=x_{0}e^{Z_{t}}$ is the only solution (affine field,
the affine proposition of \cite{HuberII}) and the left-hand side is $(1+x_{0}e^{aT^{\gamma}})^{p}$, so as
$x_{0}\to\infty$ the exponent must dominate $paT^{\gamma}$, of order $\Lambda_{\RZ}T^{\gamma}\gg T$
as $T\downarrow0$: a bound proportional to $T$ would be false. The $1$ covers that, since
\[
\Lambda_{\RZ}T^{\gamma}\ =\ \big(\Lambda^{1/\gamma}_{\RZ}T\big)^{\gamma}
\ \le\ 1+\Lambda^{1/\gamma}_{\RZ}T ,
\]
by $u^{\gamma}\le1+u$ for $u\ge0$ and $\gamma\in(0,1)$.
\color{blue}
\begin{remark}[What Theorem~\ref{thm:weakex} does not claim]\label{rem:Escope}
\begin{enumerate}
    \item No uniqueness. Weak existence is proved by compactness along a subsequence and different
subsequences not being shown to agree; uniqueness is the subject of \cite{HuberII}, under
hypotheses strictly stronger than this theorem's. 
\item No uniformity in $n$. The Lyapunov
constant $C_{1}$ depending on $n$ through $|x|^{2}=\sum_{i}x^{2}_{i}$.
\item No invariance. The
solution is $\R^{n}$-valued, and invariance of a closed $D\subsetneq\R^{n}$ is not addressed in this
paper and left to one of the companion papers, which are currently work in progress.
\end{enumerate}
\end{remark}
\color{black}
\subsection{Proof}\label{subsec:Eproof}

Both left-hand sides of \eqref{eq:daviemoduli} are non-decreasing in $m$, and its third clause descends with $m$, so that an
$L_{m}$-conditional solution for every $m\in(1/\gamma,p]$ is one for every $m\in[2,p]$, the two ranges
meeting because $1/\gamma<3$. The interval $[2,p/3]$ is non-empty, since \eqref{eq:Ep} requires $p>\max\{6,\tfrac1\gamma\}=6$ at every admissible $\gamma$, so that $p/3>2$.

\subsubsection*{Step 1: regularization and truncation}
\makeatletter\def\@currentlabel{1}\makeatother\label{step:Ereg}

Let $b^{k},\sigma^{k}$ denote mollifications of $b,\sigma$ in the space variable. Then
$b^{k},\sigma^{k}$ are locally Lipschitz in $x$ uniformly in $t$, they converge to $b,\sigma$
locally uniformly, and they satisfy
\eqref{eq:Egrowth} with the same constant $C_{\mathrm{gr}}$ up to a factor $2$ (mollify at scale
$\le1$ and use that the mollifier has unit mass). They are not globally Lipschitz. Indeed, for continuous $b$ of linear growth we obtain
\[
|Db^{k}(x)|\ \le\ \|D\rho_{k}\|_{L^{1}}\sup_{|y-x|\le1/k}|b(y)|\ \lesssim\ k(1+|x|) ,
\]
a Lipschitz constant growing linearly in the state. The cut-off below produces global Lipschitz
continuity.
The mollification also passes on condition \eqref{eq:Eacoord}. Jensen's inequality yields
\[
a^{k}_{ii}(x)=|\sigma^{k}_{i\cdot}(x)|^{2}\le\int a_{ii}(x-y)\rho_{k}(y)\dy
\le C_{\mathrm{gr}}(2+|x_{i}|)(2+|x|)\le4C_{\mathrm{gr}}(1+|x_{i}|)(1+|x|) .
\]
We use the $\Theta_{R}$ of Definition~\ref{def:approxfamily}, and we also fix a
scalar $\theta_{R}\in C^{\infty}_{c}(\R)$, such that the same properties hold in one variable. We denote
\[
b_{R}^{k}:=\Theta_{R}b^{k},\qquad \sigma^{k}_{R}:=\Theta_{R}\sigma^{k},\qquad
(f_{R},f'_{R}):=
\begin{cases}
\big(\Theta_{R}f,\ \Theta_{R}f'\big) & \text{under \ref{Ecurvaffine}},\\[2pt]
\big(q^{\kappa}_{[R],t,i}(x_{i})\big)_{i,\kappa},\ \big(p^{\kappa\lambda}_{[R],t,i}(x_{i})\big)_{i,\kappa\lambda}
& \text{under \ref{Ecurvdiag}},
\end{cases}
\]
where in the second case $q_{[R]}:=\theta_{R}q$ and $p_{[R]}:=\theta_{R}p$, coordinate by coordinate.
Here $q_{[R]}$ denotes the coordinatewise truncation, that is, a different operator
from the global $\Theta_{R}\,\cdot\,$ of Definition~\ref{def:approxfamily}.

Under \ref{Ecurvdiag} the cut-off must act coordinatewise. A
cut-off $\Theta_{R}$ in the full vector destroys diagonality, since $\Theta_{R}(x)q_{i}(x_{i})$ is not
a function of $x_{i}$ alone, and its second derivative acquires exactly the off-diagonal blocks
$q_{i}\partial^{2}_{jk}\Theta_{R}$ and $\partial_{j}\Theta_{R}\dot q_{i}\delta_{ik}$ with $j,k\ne i$ that
no coordinatewise weight controls, and that the closing paragraph of Example~\ref{ex:Adiag} covers.
The coordinatewise cut-off keeps the
field diagonal, and it preserves \eqref{eq:Ecurvdiag} uniformly in $R\ge1$. With
$q_{[R],i}=\theta_{R}q_{i}$, Leibniz' rule yields
\[
\big(1+|\xi|\big)\big|\ddot q_{[R],i}(\xi)\big|
\le\big(1+|\xi|\big)\Big(\big|\ddot\theta_{R}q_{i}\big|+2\big|\dot\theta_{R}\dot q_{i}\big|+\big|\theta_{R}\ddot q_{i}\big|\Big)(\xi)
\ \le\ C\Big(\tfrac{(1+2R)^{2}}{R^{2}}C_{\mathrm{gr}}+\tfrac{1+2R}{R}\|Df\|_{\infty}\Big)+C_{\mathrm{cur}},
\]
using that $\ddot\theta_{R}$ and $\dot\theta_{R}$ are supported in $\{|\xi|\le2R\}$. That is the constant the truncated
field satisfies. It is at most
$\bar C_{\mathrm{cur}}:=C_{\mathrm{cur}}+C(9C_{\mathrm{gr}}+3\|Df\|_{\infty})$ at every $R\ge1$,
since $(1+2R)^{2}/R^{2}\le9$ and $(1+2R)/R\le3$ there. Every later use of \eqref{eq:Ecurvdiag} at
$f_{R}$ carries $\bar C_{\mathrm{cur}}$ in place of $C_{\mathrm{cur}}$, and
$\bar C_{\mathrm{cur}}$ is independent of $k$ and $R$.

Under \ref{Ecurvaffine} we use the full cut-off, and what it preserves is the norm
condition \eqref{eq:Ecurvnormhyp}. Leibniz' rule yields
\[
D^{2}(\Theta_{R}f)=\Theta_{R}D^{2}f+2D\Theta_{R}\otimes Df+fD^{2}\Theta_{R} .
\]
The first summand satisfies
$(1+|x|)\Op{\Theta_{R}D^{2}f(x)}\le C_{\mathrm{nrm}}$ by \eqref{eq:Ecurvnormhyp}, since
$0\le\Theta_{R}\le1$. The other two are supported in $\{|x|\le2R\}$, where
\[
\Op{2D\Theta_{R}\otimes Df+fD^{2}\Theta_{R}}
\ \le\ CR^{-1}\|Df\|_{\infty}+CR^{-2}C_{\mathrm{gr}}(1+2R)\ \le\ CR^{-1} ,
\qquad
(1+2R)\cdot CR^{-1}\ \le\ 3C .
\]
Hence,
\begin{equation}\label{eq:Ecurvnorm}
\sup_{x}\big(1+|x|\big)\Op{D^{2}f_{R}(x)}\ \le\ C_{\mathrm{nrm}}+3C
\qquad\text{uniformly in }R\ge1 .
\end{equation}
This, and not $D^{2}f\equiv0$, is what branch \ref{Ecurvaffine} uses. For an
affine field $C_{\mathrm{nrm}}=0$ and the bound is $3C$, $D^{2}f_{R}$ then living only where the cut-off
varies, where it is $O(1/R)$. \color{blue}Among the untruncated diagonal fields only the affine ones satisfy
the norm condition. This is what makes branch \ref{Ecurvaffine} run with
the global weight $V(X_{u})$ in Step~1 of the proof of Lemma~\ref{lem:Escalefree}.\color{black}

In both cases we truncate both members of the pair by the same cut-off, and
$(f_{R},f'_{R})$ is again a controlled vector field. The cut-off is deterministic and
time-independent, so that for $\vartheta=\Theta_{R}$ or $\vartheta=\theta_{R}(x_{i})$,
\[
\delta(f_{R})_{s,t}(x)=\vartheta(x)\delta f_{s,t}(x)=\vartheta(x)f'_{s}(x)\delta Z_{s,t}+O(|t-s|^{\beta+\beta'})
=f'_{R,s}(x)\delta Z_{s,t}+O(|t-s|^{\beta+\beta'}) ,
\]
uniformly in $x$, the remainder inheriting the bound of $(f,f')$ since $0\le\vartheta\le1$, so that
$(f_{R},f'_{R})\in\mathscr D^{\beta,\beta'}_{\RZ}C^{3}_{b}$ with the same indices and the same
controlled-path norm. The same computation for higher derivatives gives \eqref{eq:Eroughderiv} for
the truncated pair, at the same indices apart from the third term and at a constant uniform in
$R$ \color{blue}(Lemma~\ref{lem:Escalefree} requiring
constants free of $(k,R)$)\color{black}. For the first, unweighted term of \eqref{eq:Eroughderiv} this holds by
Leibniz,
\[
R^{Df_{R}}_{s,t}=\vartheta R^{Df}_{s,t}+R^{f}_{s,t}\otimes D\vartheta .
\]
While
$\|D\vartheta\|_{\infty}=O(R^{-1})$, the companion factor $R^{f}$ carries a weight bounded by
$1+2R$ where $D\vartheta$ lives, so
\[
\sup_{x}\big|R^{f}_{s,t}(x)\big|\,\|D\vartheta\|_{\infty}
\ \le\ C_{\RZ}(1+2R)|t-s|^{\beta+\beta'}\cdot CR^{-1}
\ \le\ 3CC_{\RZ}|t-s|^{\beta+\beta'}\qquad(R\ge1),
\]
Hence the unweighted constant of the truncated pair is at most $C^{(1)}_{\RZ}+3CC_{\RZ}$, uniformly in
$R$. Under
\ref{Ecurvaffine}, $\vartheta=\Theta_{R}$ is supported in $\{|x|\le2R\}$, where $1+|x|\le1+2R$. Under
\ref{Ecurvdiag} the support of $D\theta_{R}(x_{i})$ is the annulus $\{R\le|x_{i}|\le2R\}$, on which
$1+|x|$ is not bounded. There, however, $f^{\kappa}_{i}(x)=q^{\kappa}_{i}(x_{i})$ is a function
of $x_{i}$ alone, so that $R^{f}_{i}$ is too, and evaluating \eqref{eq:cvfdefect} at $x=\xi e_{i}$
yields
\[
|R^{f}_{i}|\le C_{\RZ}(1+|x_{i}|)|t-s|^{\beta+\beta'}\le C_{\RZ}(1+2R)|t-s|^{\beta+\beta'} ,
\]
after the $\ell^{2}$-sum over the surviving components. It applies identically to
$R^{Df}$, $\delta Df'$ and $\delta D^{2}f$, each diagonal under \ref{Ecurvdiag}.
The sum produces the square root of the number of surviving components, $ne$ for the unprimed
$R^{f},R^{Df},\delta D^{2}f$ and $ne^{2}$ for the primed $\delta f',\delta Df'$, both absorbed into
constants, which carry $n$ and $e$ throughout. The weighted second and third terms are inherited the
same way and with the same uniformity, the third at order $\gamma\wedge\beta'$ rather than $\beta'$
because $\delta D^{2}f_{R}$ carries $\delta Df$ against $D\vartheta$,
Appendix~\ref{app:FHL46} checking it at $\alpha=\beta\wedge\beta'\le\gamma\wedge\beta'$. Finally, the
truncation does not increase $C_{\RZ}$ itself. Leibniz again yields
\[
\delta Df'_{R}=\vartheta\,\delta Df'+\delta f'\otimes D\vartheta ,
\qquad
\delta D^{2}f_{R}=\vartheta\,\delta D^{2}f+2D\vartheta\otimes\delta Df+\delta f\,D^{2}\vartheta ,
\]
each summand carrying one of the moduli of \eqref{eq:cvfdefect} or \eqref{eq:Eroughderiv} at an order
$\ge\beta\wedge\beta'$. Each of these bounds is in $\mathcal C_{b}$ at fixed $R$, for the two reasons just given: under \ref{Ecurvaffine} because
$\{\Theta_{R}\ne0\}\subseteq\{|x|\le2R\}$, where $1+|x|\le1+2R$; under \ref{Ecurvdiag}, because every component of $\delta f$, $\delta f'$,
$\delta Df$ and $\delta D^{2}f$ is a function of $x_{i}$ alone, so that its weighted bound
evaluated at $x=\xi e_{i}$ reads $1+|x_{i}|\le1+2R$, the $\ell^{2}$-sum over $i$ producing a
factor $\sqrt{ne}$. At $\alpha=\beta\wedge\beta'$ (as \eqref{eq:FHL46alpha} demands) this is
\[
(f_{R},f'_{R})\in\mathscr D^{2\alpha}_{\RZ}\mathcal C^{3}_{b},
\qquad
(Df_{R},Df'_{R})\in\mathscr D^{\alpha,\alpha}_{\RZ}\mathcal C^{2}_{b},
\]
and with $\varkappa=3$ and $\alpha''=\alpha$ that is \eqref{eq:FHL46hyp} in full. These are bounded and globally Lipschitz and
satisfy \eqref{eq:Egrowth} with the same $C_{\mathrm{gr}}$, since $0\le\vartheta\le1$, and their first
three derivatives inherit
\[
\|D^{j}f_{R}\|_{\infty}\le C\textstyle\sum_{1\le i\le j}\|D^{i}f\|_{\infty}+C
\qquad\text{for }1\le j\le3,\ \text{uniformly in }R\ge1 ,
\]
by Leibniz and $\|D^{j}\Theta_{R}\|_{\infty}+\|\theta_{R}^{(j)}\|_{\infty}=O(R^{-j})$ together with the
linear growth of $f$. The It\^o data is bounded and globally Lipschitz at fixed $(k,R)$ \color{blue}(which is what
\cite[Thm.~4.6]{FHL} requires)\color{black}. Leibniz yields
\[
D\big(\Theta_{R}b^{k}\big)=\Theta_{R}Db^{k}+b^{k}\otimes D\Theta_{R} ,
\]
and for fixed $(k,R)$ each summand is bounded: on $\{|x|\le2R\}$, where only the product is non-zero,
the mollification bound and \eqref{eq:Egrowth} yield
\[
|Db^{k}|\lesssim k(1+2R),
\qquad
|b^{k}|\le2C_{\mathrm{gr}}(1+2R),
\qquad
\|D\Theta_{R}\|_{\infty}=O(R^{-1}) .
\]
The function $b^{k}_{R}$ is therefore bounded with bounded derivative, hence globally Lipschitz on
$\R^{n}$, and the same computation yields it for $\sigma^{k}_{R}$.

By \cite[Thm.~4.6]{FHL}, applied at the index $\alpha=\beta\wedge\beta'$ of \eqref{eq:FHL46alpha} and
not at $\gamma$, the truncated equation has a unique $L_{m,\infty}$-solution $X^{k,R}$ on $[0,T]$,
simultaneously for every $m\ge2$, on any basis carrying $B$ and an $\mathcal F_{0}$-measurable
$X_{0}\sim\nu$.  Truncation having made the coefficients bounded, that solution satisfies the strong form
\eqref{eq:daviemoduliinfty}. At $\alpha=\beta\wedge\beta'$ the driver is still a rough path on $[0,T]$ and the theorem's whole
requirement reduces to $\beta\wedge\beta'>\tfrac13$, which
Assumption~\ref{ass:Erough} states. Conditions (a) and (c) of \cite[Def.~4.2]{FHL} do not
mention $\alpha$ and condition (b)'s inequalities are increasing in it, so that $X^{k,R}$ is an
$L_{m,\infty}$-integrable solution at every index in $[\alpha,\gamma]$, and in particular at $\gamma$.

For each $m$ in the range Lemma~\ref{lem:Escalefree} fixes, the Davie remainder satisfies, at fixed $(k,R)$,
\[
\big\|J_{u,v}\big\|_{L_{m}}\ \lesssim\ |v-u|^{2\gamma} ,
\qquad
\big\|\E_{u}J_{u,v}\big\|_{L_{m}}\ \lesssim\ |v-u|^{\epsilon_{1}} ,
\]
with $\epsilon_{1}$ as in \eqref{eq:eps1value}.

\begin{enumerate}[label=(\roman*), leftmargin=*, labelindent=0pt, align=left]
    \item What the source states: \cite[Prop.~4.3]{FHL} states, for any $L_{m,\mfn }$-integrable
solution and every $u\le v$,
\begin{equation}\label{eq:FHLdavierates}
\Big\|\big\|J_{u,v}\mid\mathcal F_{u}\big\|_{L_{m}}\Big\|_{L_{\mfn }}
\ \lesssim\ |v-u|^{\gamma+(\gamma\wedge\bar\alpha)} ,
\qquad
\big\|\E_{u}J_{u,v}\big\|_{L_{\mfn }}
\ \lesssim\ |v-u|^{\gamma+(\gamma\wedge\bar\alpha)+\bar\alpha'} ,
\end{equation}
where $(\bar\alpha,\bar\alpha')$ denote indices such that the composed pair
$\big(f_{R}(X^{k,R}),(Df_{R}f_{R}+f'_{R})(X^{k,R})\big)$ is a stochastic controlled rough path. The norms in \eqref{eq:FHLdavierates} are stronger than the ones
\eqref{eq:Jfinite} requires, \cite[Prop.~2.3(ii)]{FHL} yielding
\[
\|\xi\|_{L_{m}}\le\big\|\|\xi\mid\mathcal F_{u}\|_{L_{m}}\big\|_{L_{\mfn }} ,
\qquad
\|\xi\|_{L_{m}}\le\|\xi\|_{L_{\mfn }}\quad\text{for }m\le \mfn  .
\]
Here $\mfn =\infty$, \cite[Thm.~4.6]{FHL} having produced
an $L_{m,\infty}$-solution. \cite[Prop.~4.3]{FHL} has the membership
\cite[Def.~4.2(b)]{FHL} as a hypothesis. At fixed $(k,R)$ that membership holds by construction,
since $X^{k,R}$ comes from \cite[Thm.~4.6]{FHL}. Its continuity clause carries no index, so the
membership still holds when \cite[Prop.~4.3]{FHL} is applied again at the pair of indices
$(\gamma,\bar\beta)$ of the next item.
\item The truncated field lies not only in
$\mathscr D^{\beta,\beta'}_{\RZ}$ but in $\mathscr D^{\gamma,\bar\beta}_{\RZ}$, where
\begin{equation}\label{eq:Jbarbeta}
\bar\beta\ :=\ \min\big\{\beta+\beta'-\gamma,\ \gamma\big\}\ \in\ (0,\gamma] ,
\end{equation}
positive because $2\beta+\beta'>1$ (a consequence of the bounds $\beta,\beta'>\tfrac13$
that Assumption~\ref{ass:Erough} states) and
$\beta\le\gamma<\tfrac12$ require
\[
\beta+\beta'\ >\ 1-\beta\ \ge\ 1-\gamma\ >\ \gamma ,
\]
i.e.~\eqref{eq:rhoexp}. Indeed, the controlled-path remainder of $(f_{R},f'_{R})$ is of order
$\beta+\beta'\ge\gamma+\bar\beta$ and $\delta f'_{R}$ of order $\beta'\ge\bar\beta$, the latter by the
two cases of \eqref{eq:Jbarbeta} separately: if $\bar\beta=\beta+\beta'-\gamma$, then
$\gamma\ge\beta$ yields $\beta'\ge\bar\beta$. If $\bar\beta=\gamma$, then $\beta+\beta'\ge2\gamma$, and
$\beta\le\gamma$ requires $\beta'\ge\gamma=\bar\beta$. \cite[Rem.~4.4]{FHL} requires two things of the solution. Its moment hypothesis
$\sup_{t}\|g_{t}(X^{k,R}_{t})\|_{L_{\infty}}<\infty$, $g\in\{b,\sigma\sigma^{\top},f,Df\,f,f'\}$,
holds because the truncated coefficients are bounded at fixed $(k,R)$. Its H\"older hypothesis
$\delta[f_{R}(X^{k,R})]\in C^{\bar\beta}_{2}L_{m,\infty}$ holds as well. First,
$\delta X^{k,R}\in C^{\gamma}_{2}L_{m,\infty}$: in \eqref{eq:davie} the drift is $O(|t-s|)$ and the two
driver terms are $O(|t-s|^{\gamma})$ and $O(|t-s|^{2\gamma})$ in that norm, all three by the same
boundedness, the martingale term is $O(|t-s|^{1/2})$ by conditional Burkholder--Davis--Gundy, and
$J$ is $o(|t-s|^{1/2})$ by \eqref{eq:daviemoduliinfty}. Since $\gamma<\tfrac12$ every summand is
$O(|t-s|^{\gamma})$ once $|t-s|$ is below a threshold $h_{\sharp}\le1$ depending on $(k,R,m)$, below
which the second modulus of \eqref{eq:daviemoduliinfty} holds at constant $1$. A longer pair is cut
into at most $\lceil T/h_{\sharp}\rceil$ short ones and the pieces are added by conditional Minkowski. Each piece
may be taken at its own left endpoint, since
$\|\E[|\xi|^{m}\mid\mathcal F_{s}]^{1/m}\|_{L_{\infty}}
\le\|\E[|\xi|^{m}\mid\mathcal F_{u}]^{1/m}\|_{L_{\infty}}$ for $s\le u$ by the tower property, and
$|t-s|\ge h_{\sharp}$ turns the count into a constant. This puts $\delta X^{k,R}$ in
$C^{\gamma}_{2}L_{m,\infty}$.
For $f_{R}$ the state part is its Lipschitz bound and the time part is the increment of the
truncation above, $f'_{R,s}(x)\delta Z_{s,t}+O(|t-s|^{\beta+\beta'})$ uniformly in $x$ at a constant
depending on $R$, of order $\gamma\wedge(\beta+\beta')\ge\bar\beta$. Hence
\cite[Rem.~4.4]{FHL} places $\big(X^{k,R},f_{R}(X^{k,R})\big)$ in
$\mathscr D^{\gamma,\bar\beta}_{\RZ}L_{m,\infty}$; and \cite[Lem.~3.11]{FHL}, the composition lemma,
applied at $\varkappa=3$ (which $f_{R}\in C^{3}_{b}$ allows, Appendix~\ref{app:FHL46}) and at $\mfn =\infty$, carries this to
$\big(f_{R}(X^{k,R}),(Df_{R}f_{R}+f'_{R})(X^{k,R})\big)\in\mathscr D^{\gamma,\bar\beta'}_{\RZ}L_{m,\infty}$
with $\bar\beta'=\min\{(\varkappa-2)\gamma,\bar\beta\}=\bar\beta$. One hypothesis of that lemma is not
in \eqref{eq:cvfdefect}: it asks $(f_{R},f'_{R})$ in
$\mathscr D^{\gamma,\bar\beta}L_{m,\infty}\mathcal C^{\varkappa-1}_{b}$, and by
\cite[Def.~3.7(b)]{FHL} that class asks $[\![\delta Df_{R}]\!]_{\bar\beta}<\infty$ as well as the two
brackets \eqref{eq:cvfdefect} delivers. The first two bounds of \eqref{eq:Eroughderiv} deliver it:
\[
\delta Df_{R}=Df'_{R,s}\delta Z+O(|t-s|^{\beta+\beta'})
\]
with $Df'_{R}$ bounded, hence of order $\gamma\wedge(\beta+\beta')\ge\bar\beta$. Hence
$\bar\alpha=\gamma$ and $\bar\alpha'=\bar\beta$ in \eqref{eq:FHLdavierates}, and the two rates it
delivers are $2\gamma$ and $2\gamma+\bar\beta$.
\item The first rate is exactly the first bound of \eqref{eq:Jfinite}. For the
second, we must check
\[
2\gamma+\bar\beta\ \ge\ \epsilon_{1} ,
\qquad
\epsilon_{1}=\min\{\tfrac{1+3\gamma}2,\ \beta+\beta'+\gamma\} ,
\]
with $\epsilon_{1}$ of \eqref{eq:eps1value}, and we can check the two cases of \eqref{eq:Jbarbeta}
separately. \emph{If $\bar\beta=\beta+\beta'-\gamma$}, then
\[
2\gamma+\bar\beta=\gamma+\beta+\beta' ,
\]
one of the two entries of that minimum, and therefore at least $\epsilon_{1}$. \emph{If
$\bar\beta=\gamma$}, then
\[
2\gamma+\bar\beta=3\gamma\ >\ \tfrac{1+3\gamma}2\ \ge\ \epsilon_{1} ,
\]
because $6\gamma>1+3\gamma$ is $\gamma>\tfrac13$.

\end{enumerate}
\color{blue}
\begin{remark}\label{rem:Emollifyfactor}
We do not claim invariance, but an argument deducing
invariance of a closed set $D$ from invariance for the approximating equations would require the
approximating coefficients to preserve $D$. On the orthant
the diffusion coefficient $\mathcal B_{0}(x)=\sqrt{x^{+}\mathfrak b(x^{+})}$ and $\mathcal B_{0}(0)=0$ is what keeps a
solution nonnegative, while a mollifier of scale $1/k$ averages $\mathcal B_{0}$ over a neighborhood of
the origin on which it is strictly positive. Mollifying the factor $\mathfrak b$ keeps the zero,
but $x\mapsto\sqrt{x^{+}\mathfrak b^{k}(x^{+})}$ is not Lipschitz at the origin, and it is global
Lipschitz continuity at fixed $k$ that Step~\ref{step:Ereg} hands to \cite[Thm.~4.6]{FHL}. A
regularization such that both hold is
\[
\mathcal B^{k}_{0}(x)\ :=\ \sqrt{x^{+}\mathfrak b^{k}(x^{+})+\tfrac1k}\ -\ \sqrt{\tfrac1k}\ ,
\]
with $\mathfrak b^{k}$ the mollification of $\mathfrak b$. It vanishes at $x=0$ and on $\{x\le0\}$. It
is Lipschitz on bounded sets with constant $O(\sqrt k)$, hence globally Lipschitz and bounded after the
cut-off $\Theta_{R}$; it converges to $\mathcal B_{0}$ locally uniformly, since
$|\sqrt{u_{k}+\tfrac1k}-\sqrt u|\le\sqrt{|u_{k}-u|+\tfrac1k}$; and it inherits \eqref{eq:Egrowth} and
\eqref{eq:Eacoord} from $x\mathfrak b$, because
$0\le\mathcal B^{k}_{0}(x)^{2}\le x^{+}\mathfrak b^{k}(x^{+})$.
\end{remark}
\color{black}
\subsubsection*{Step 2: the moment bound, uniform in $(k,R)$}
\makeatletter\def\@currentlabel{2}\makeatother\label{step:Emoment}

Let $V(x):=(1+|x|^{2})^{1/2}$ denote the Lyapunov function, so that $V\in C^{\infty}$,
$\tfrac1{\sqrt2}(1+|x|)\le V(x)\le1+|x|$, and
\begin{equation}\label{eq:Vjet}
\|DV\|_{\infty}\le1,\qquad \Op{D^{2}V(x)}\le\frac1{V(x)},\qquad \Op{D^{3}V(x)}\le\frac{C}{V(x)^{2}} .
\end{equation}
Recall the norm convention of \S\ref{subsec:rp}:  \eqref{eq:Vjet}'s second bound reads
\[
\sup_{|v|=1}|D^{2}V(v,v)|\le V^{-1} .
\]
That
bound is sharp, and the Frobenius norm does not satisfy it for $n\ge2$. For $\sigma$, it is the other way around, and $\Frob{\sigma}^{2}=\operatorname{tr}(\sigma \sigma^{\top})$ is the
stronger form of \eqref{eq:Egrowth}, which is the one used below.
Let $\mathcal L_{t}$, $\mathcal G$ and $\mathcal G^{(2)}$, the latter standing for
$\mathcal G^{\kappa}\mathcal G^{\lambda}+\mathcal G'^{\kappa\lambda}$, denote the rough generator of the truncated system. From \eqref{eq:Egrowth}, \eqref{eq:Vjet} and $\|Df\|_{\infty}<\infty$ we obtain the following bounds,
\begin{equation}\label{eq:Vlyap}
|\mathcal L_{t}V|\le\tfrac12\Frob{\sigma_{R}^{k}}^{2}\Op{D^{2}V}+|b^{k}_{R}||DV|\le C_{1}V,
\qquad
|\mathcal GV|\le C_{1}V,
\qquad
|\mathcal G^{(2)}V|\le C_{1}V,
\end{equation}
with $C_{1}=C_{1}(C_{\mathrm{gr}},\|Df\|_{\infty},\|Df'\|_{\infty},C_{\mathrm{nrm}}\text{ or }\bar C_{\mathrm{cur}})$ independent of $k$ and $R$. For the third bound we
use
\[
\mathcal G^{\kappa}\mathcal G^{\lambda}V+\mathcal G'^{\kappa\lambda}V
=\sum_{i,j}f^{\kappa}_{i}(\partial_{i}f^{\lambda}_{j})\partial_{j}V
+\sum_{i,j}f^{\kappa}_{i}f^{\lambda}_{j}\partial^{2}_{ij}V+\sum_{i}f'^{\kappa\lambda}_{i}\partial_{i}V
\]
together with
\[
|f|^{2}\Op{D^{2}V}\le C(1+|x|)^{2}/V\le CV ,
\qquad
|f'|\,|DV|\le C(1+|x|)\le CV .
\]
Similarly, the martingale denoted
$N^{V}:=\int DV(X^{k,R})\cdot\sigma^{k}_{R}(X^{k,R})\,\d B$ has
\begin{equation}\label{eq:Vbracket}
\d\langle N^{V}\rangle_{t}=\big|DV(X^{k,R}_{t})\,\sigma^{k}_{R}(t,X^{k,R}_{t})\big|^{2}\dt
\ \le\ C_{1}\,V(X^{k,R}_{t})^{2}\dt .
\end{equation}
\begin{enumerate}[label=(a\arabic*), start=0, wide=0pt]
    \item\label{proof:Eproof:step2:a0} \emph{The germ coefficients are Lipschitz, with scale-free constants.}  We denote by
\begin{equation}\label{eq:Vgerm}
\Gamma^{1,\kappa}(x):=\big(\mathcal G^{\kappa}_{t}V\big)(x),
\qquad
\Gamma^{2,\kappa\lambda}(x):=\big(\mathcal G^{\kappa}_{t}\mathcal G^{\lambda}_{t}
+\mathcal G'^{\kappa\lambda}_{t}\big)V(x),
\qquad
\Gamma^{3,\kappa\lambda}(x):=\tfrac12\big(\mathcal Q^{\kappa\lambda}_{t}V\big)(x)
\end{equation}
the three germ coefficients of \eqref{eq:roughMP} at $\varphi=V$. Then, uniformly in $t$ and in
$(k,R)$,
\begin{equation}\label{eq:VgermLip}
|\Gamma^{1}(x)|+|\Gamma^{2}(x)|+|\Gamma^{3}(x)|\ \le\ C_{1}V(x),
\qquad
|D\Gamma^{1}(x)|\le C_{1},
\qquad
|D\Gamma^{2}(x)|+|D\Gamma^{3}(x)|\ \le\ C_{1}V(x) .
\end{equation}
The first bound is \eqref{eq:Vlyap} together with
$|\mathcal Q^{\kappa\lambda}V|\le|f|^{2}\Op{D^{2}V}\le CV$. For the second, Leibniz' rule and
\eqref{eq:Vjet} yield
\[
D(\mathcal G^{\kappa}V)=(Df^{\kappa})^{\top}DV+(D^{2}V)f^{\kappa} ,
\qquad
|D\Gamma^{1}|\le\|Df\|_{\infty}+C_{\mathrm{gr}}(1+|x|)/V(x)\le C_{1} .
\]
For the third, we
differentiate the three summands displayed after \eqref{eq:Vlyap}, which produces eight terms. Seven are
bounded outright: each carries either $\|Df\|_{\infty}$, or
$\|Df'\|_{\infty}$, or a factor \[ |f|^{j}\Op{D^{j+1}V}\le C_{\mathrm{gr}}^{j}(1+|x|)^{j}\cdot CV^{-j}\le C
\qquad (j=1,2) .
\]
The
eighth is $\textstyle\sum_{i,j}f^{\kappa}_{i}(\partial_{k}\partial_{i}f^{\lambda}_{j})\partial_{j}V$. Under \ref{Ecurvaffine},
\eqref{eq:Ecurvnorm} and \eqref{eq:Vjet} yield the following bound,
\[
\big|\textstyle\sum_{i,j}f^{\kappa}_{i}(\partial_{k}\partial_{i}f^{\lambda}_{j})\partial_{j}V\big|
\ \le\ C_{\mathrm{gr}}(1+|x|)\Op{D^{2}f}\,\|DV\|_{\infty}
\ \le\ C_{\mathrm{gr}}(C_{\mathrm{nrm}}+3C) .
\]
Under \ref{Ecurvdiag} the only non-zero second derivatives are
$\partial^{2}_{jj}f^{\lambda}_{j}=\ddot q^{\lambda}_{j}(x_{j})$, so that the eighth term has $k$-th
component $q^{\kappa}_{k}(x_{k})\,\ddot q^{\lambda}_{k}(x_{k})\,\partial_{k}V(x)$, and we obtain the following bound,
\[
\big|\sum_{i,j}f^{\kappa}_{i}\,\partial_{k}\partial_{i}f^{\lambda}_{j}\,\partial_{j}V\big|
\ \le\ C_{\mathrm{gr}}\big(1+|x_{k}|\big)\,\big|\ddot q^{\lambda}_{k}(x_{k})\big|\,\big|\partial_{k}V\big|
\ \le\ C_{\mathrm{gr}}\bar C_{\mathrm{cur}} ,
\]
a bounded quantity. Hence,
$|D\Gamma^{2}|+|D\Gamma^{3}|\le C_{1}$ is bounded in either case. This is stronger than
\eqref{eq:VgermLip} claims and is what the proof of Lemma~\ref{lem:Escalefree} uses.

 For all
$\xi,v\in\R^{n}$ and every $\kappa$,
\begin{equation}\label{eq:D2Gamma1}
\big|D^{2}\Gamma^{1,\kappa}(\xi)(v,v)\big|\ \le\ \frac{C_{1}\,|v|^{2}}{V(\xi)} .
\end{equation}
Leibniz' rule yields
\[
\partial^{2}_{jk}\Gamma^{1,\kappa}
=\textstyle\sum_{i}(\partial^{2}_{jk}f^{\kappa}_{i})\partial_{i}V
+\textstyle\sum_{i}(\partial_{j}f^{\kappa}_{i})\partial^{2}_{ik}V
+\textstyle\sum_{i}(\partial_{k}f^{\kappa}_{i})\partial^{2}_{ij}V
+\textstyle\sum_{i}f^{\kappa}_{i}\partial^{3}_{ijk}V .
\]
For the last three, \eqref{eq:Vjet} yields the following bound,
\[ 2\|Df\|_{\infty}\Op{D^{2}V}+|f|\,\Op{D^{3}V} \ \le\ \big(2\|Df\|_{\infty}+\sqrt2\,CC_{\mathrm{gr}}\big)/V(\xi) .
\]
For the first, \eqref{eq:Ecurvnorm} and $1+|\xi|\ge V(\xi)$ yield under \ref{Ecurvaffine}
\[
\Opb{D^{2}f(\xi)}\,|v|^{2}\|DV\|_{\infty}\ \le\ (C_{\mathrm{nrm}}+3C)|v|^{2}/V(\xi) ,
\]
Under \ref{Ecurvdiag} it contracts to $\sum_{i}\ddot q^{\kappa}_{i}(\xi_{i})\,v_{i}^{2}\,\partial_{i}V(\xi)$.
Here $|\partial_{i}V(\xi)|=|\xi_{i}|/V(\xi)\le(1+|\xi_{i}|)/V(\xi)$, so that \eqref{eq:Ecurvdiag}
yields
\[
\big|\textstyle\sum_{i}\ddot q^{\kappa}_{i}(\xi_{i})v^{2}_{i}\partial_{i}V(\xi)\big|
\ \le\ \bar C_{\mathrm{cur}}|v|^{2}/V(\xi) .
\]

Every constant here is free of $k$ and $R$.
\item \label{proof:Eproof:step2:a1} \emph{A scale-free bound on the sewing remainder.} 
\begin{lemma}[Scale-free a priori estimates for the truncated system]\label{lem:Escalefree}
Let Assumptions~\ref{ass:driver}, \ref{ass:Eito}, \ref{ass:Erough} and
\ref{ass:Einit} hold, the coordinatewise bound \eqref{eq:Eacoord} only in branch
\ref{Ecurvdiag}, let $X=X^{k,R}$ denote the solution of the
truncated, mollified system of Step~\ref{step:Ereg} and denote $\Phi_{t}:=V(X_{t})$. Fix $m\in(1/\gamma,p]$
and denote, for $0\le s\le t\le T$,
\[
\mathcal N_{s,t}:=\Big\|\sup_{r\in[s,t]}\Phi_{r}\Big\|_{L_{m}},
\qquad
\Lambda_{\RZ}:=1+\|\RZ\|_{\gamma}+\|\CZ\|_{1} .
\]
For that $m$,
\begin{equation}\label{eq:Jfinite}
\mathcal J^{k,R}\ :=\ \sup_{0\le u<v\le T}\frac{\big\|J_{u,v}\big\|_{L_{m}}}{|v-u|^{2\gamma}}\ <\ \infty ,
\qquad
\mathcal J^{k,R}_{\mathrm{c}}\ :=\ \sup_{0\le u<v\le T}\frac{\big\|\E_{u}J_{u,v}\big\|_{L_{m}}}{|v-u|^{\epsilon_{1}}}\ <\ \infty ,
\end{equation}
where $J$ denotes the remainder of \eqref{eq:davie} for $X^{k,R}$ and $\epsilon_{1}>1$ denotes the exponent
\eqref{eq:eps1value}. Both suprema are finite by \cite[Prop.~4.3, Rem.~4.4]{FHL}, at
values allowed to depend on $(k,R)$ and on $m$ (Remark~\ref{rem:Jfinite}).

There are $c_{0}\ge1$ and $C_{3}$, depending only on
$\big(C_{\mathrm{gr}},C_{\mathrm{cur}}\text{ or }C_{\mathrm{nrm}},\|Df\|_{\infty},\|Df'\|_{\infty},C_{\RZ},C^{(1)}_{\RZ},\beta,\beta',\gamma,m,n,e\big)$
and in particular \emph{free of $k$, of $R$ and of $T$}, such that with
$h_{\star}:=\big(c_{0}\Lambda_{\RZ}\big)^{-1/\gamma}\wedge T$ we have, for all $s\le t$ such that
$t-s\le h_{\star}$,
\begin{align}
\big\|\delta X_{s,t}\big\|_{L_{m}}&\ \le\ C_{3}\,\Lambda_{\RZ}\,\mathcal N_{s,t}\,|t-s|^{\gamma},
\label{eq:EscaleX}\\[2pt]
\Big\|\sup_{r\in[s,t]}\big|\Phi^{\natural}_{s,r}\big|\Big\|_{L_{m}}
&\ \le\ C_{3}\,\Lambda^{2}_{\RZ}\,\mathcal N_{s,t}\,|t-s|^{2\gamma},
\label{eq:EscaleNat}\\[2pt]
\Big\|\E\big[J_{s,t}\mid\mathcal F_{s}\big]\Big\|_{L_{m}}
&\ \le\ C_{3}\,\Lambda^{3}_{\RZ}\,\mathcal N_{s,t}\,|t-s|^{\epsilon_{1}},
\qquad
\epsilon_{1}\ :=\ \min\Big\{\tfrac{1+3\gamma}{2},\ \beta+\beta'+\gamma\Big\}\ >\ 1,
\label{eq:EscaleCond}
\end{align}
where
\[
\Phi^{\natural}_{s,t}:=\int_{s}^{t}\big(\Gamma^{1},\Gamma^{2}\big)(X)\,\d\RZ
-\Gamma^{1}(X_{s})\,\delta Z_{s,t}-\Gamma^{2}(X_{s})\,\ZZ_{s,t}
+\int_{s}^{t}\Gamma^{3}(X_{r})\,\d\CZ_{r}-\Gamma^{3}(X_{s})\,\CZ_{s,t}
\]
denotes the remainder of the rough It\^o formula applied to $V(X)$.
\end{lemma}

\begin{proof}
The proof of the lemma can be found in Appendix~\ref{app:Escalefree}.
\end{proof}
\item\label{proof:Eproof:step2:a2} \emph{The rough Gr\"onwall, in $L_{m}$.} We apply the rough It\^o formula (
Lemma~\ref{lem:roughItoLn}) to $V(X^{k,R})$ and denote $\Phi_{t}:=V(X^{k,R}_{t})$. That lemma requires
that $V$ have bounded derivatives of order $1\le j\le2$ with $D^{2}V$ globally Lipschitz, all three
supplied by \eqref{eq:Vjet}. This is its hypothesis at $N=2$, forced by 
$\varkappa_{\varphi}=3$. The coefficients are bounded
at fixed $(k,R)$, so that \eqref{eq:rItoLn} holds, and the parameter conditions
\eqref{eq:paramsIto} are satisfied at $\mfn =8$, the controlled pair being in $L_{4,\mfn }$ at every
$\mfn <\infty$. The arithmetic is given by Proposition~\ref{prop:Athreshuses}\ref{app:Athresh1}. On $[s,t]\subseteq[0,T]$ we
obtain the following estimate,
\begin{equation}\label{eq:Egronwallineq}
\begin{aligned}
\sup_{r\in[s,t]}\Phi_{r}\ \le\ &\Phi_{s}+C_{1}\!\int_{s}^{t}\!\Phi_{r}\dr
+C_{1}\big(\|\RZ\|_{\gamma}|t-s|^{\gamma}+\|\RZ\|_{\gamma}|t-s|^{2\gamma}
+\|\CZ\|_{1}|t-s|\big)\sup_{r\in[s,t]}\Phi_{r}\\
&+\sup_{r\in[s,t]}\big|\delta N^{V}_{s,r}\big|+\sup_{r\in[s,t]}\big|\Phi^{\natural}_{s,r}\big|,
\end{aligned}
\end{equation}
where the third term collects the three germ coefficients \eqref{eq:Vgerm} against
$\delta Z_{s,t}$, $\ZZ_{s,t}$ and $\CZ_{s,t}$, bounded by \eqref{eq:VgermLip}, and where
$\Phi^{\natural}$ denotes the sewing remainder. We take $\|\cdot\|_{L_{m}}$ of
\eqref{eq:Egronwallineq} for $m\in(1/\gamma,p]$, the range of Lemma~\ref{lem:Escalefree}, and use
that lemma's \eqref{eq:EscaleNat}. Denoting
$\mathcal N(t):=\|\sup_{r\le t}\Phi_{r}\|_{L_{m}}$, and $\mathcal N_{s,t}$ as in
Lemma~\ref{lem:Escalefree}, this yields
\begin{equation}\label{eq:EgronwallLm}
\mathcal N_{s,t}\ \le\ \mathcal N(s)+C_{1}\!\int_{s}^{t}\!\mathcal N(r)\dr
+\underbrace{\Big[C_{1}\Lambda_{\RZ}|t-s|^{\gamma}+C_{3}\Lambda^{2}_{\RZ}|t-s|^{2\gamma}\Big]}_{=:\ \theta(|t-s|)}
\mathcal N_{s,t}
+\Big\|\sup_{r\in[s,t]}\big|\delta N^{V}_{s,r}\big|\Big\|_{L_{m}} .
\end{equation}
Since $2\gamma>\gamma$, we can choose $c_{1}=c_{1}(C_{1},C_{3},\gamma)$
large enough such that the greedy mesh denoted
$h:=\big(c_{1}\Lambda_{\RZ}\big)^{-1/\gamma}\wedge h_{\star}$
yields the following bound,
\[
\theta(h)\ =\ C_{1}\Lambda_{\RZ}h^{\gamma}+C_{3}\Lambda^{2}_{\RZ}h^{2\gamma}
\ \le\ C_{1}c_{1}^{-1}+C_{3}c_{1}^{-2}\ \le\ \tfrac12 ,
\]
and we can absorb both contributions, picking up a factor $2$.

The scalar Gr\"onwall on each subinterval contributes $e^{2C_{1}h}$ and the absorption a further factor
$2$, and there are
\begin{equation}\label{eq:Egreedycount}
\lceil T/h\rceil\ \le\ 1+T\big((c_{0}\vee c_{1})\Lambda_{\RZ}\big)^{1/\gamma}
\end{equation}
of them. Indeed, $h=\big((c_{0}\vee c_{1})\Lambda_{\RZ}\big)^{-1/\gamma}\wedge T$ by the definition of
$h_{\star}$, so that either $h=T$ and the count is $1$, or $h$ is the first term and
$\lceil T/h\rceil\le1+T/h$. Multiplying and using $h\le h_{\star}\le c_{0}^{-1/\gamma}\le1$, we obtain the factor
\[
\big(2e^{2C_{1}h}\big)^{\lceil T/h\rceil}
\ =\ \exp\!\Big\{\big(\log2+2C_{1}h\big)\lceil T/h\rceil\Big\}
\ \le\ \exp\!\Big\{C_{2}\big(1+\Lambda^{1/\gamma}_{\RZ}T\big)\Big\} .
\]
Hence
\begin{equation}\label{eq:Epathwise}
\begin{gathered}
\Big\|\sup_{t\le T}\Phi_{t}\Big\|_{L_{m}}\ \le\ \mathcal E_{T}\Big(\big\|\Phi_{0}\big\|_{L_{m}}
+\max_{j}\Big\|\sup_{t\le T}\big|\delta N^{V}_{t_{j},t}\big|\Big\|_{L_{m}}\Big),\\[2pt]
\mathcal E_{T}:=\exp\!\Big\{C_{2}\big(1+\Lambda^{1/\gamma}_{\RZ}T\big)\Big\}
=\exp\!\Big\{C_{2}'\Big(1+\big(1+\|\RZ\|^{1/\gamma}_{\gamma}+\|\CZ\|^{1/\gamma}_{1}\big)T\Big)\Big\},
\end{gathered}
\end{equation}
where $C_{2}=C_{2}(C_{1},C_{3},c_{0},\gamma)$ is free of $k$, $R$, $\nu$ and $T$, and where the $t_{j}$ denote the
greedy mesh points. Iterating the one-step bound puts the count
\eqref{eq:Egreedycount} on the second summand as well. It is absorbed into $\mathcal E_{T}$ by
enlarging $C_{2}$: the count is at most $1+cu$ with $u=\Lambda^{1/\gamma}_{\RZ}T$ and
$c=c(c_{0},c_{1},\gamma)$, and $\log(1+cu)\le(1+\log(1+c))(1+u)$. For a geometric lift $\CZ\equiv0$ and the factor multiplying $T$ reads
$1+\|\RZ\|^{1/\gamma}_{\gamma}$. 
\item[(b)]\stepnum{(b)}{proof:Eproof:step2:b} \emph{A priori finiteness.} We denote
$\mathcal M:=\E\big[\sup_{t\le T}\Phi^{p}_{t}\big]$ and show $\mathcal M<\infty$ at fixed $(k,R)$.
We write $X$ for $X^{k,R}$ here. Every constant may depend on $(k,R)$. The Davie expansion
gives no pathwise bound on $X$: \eqref{eq:daviemoduliinfty}, available at fixed $(k,R)$, bounds a
conditional moment of $J$ and not $|J|$, and summing the germ over a partition of mesh $T/N$ leaves
$C(R)\Lambda_{\RZ}T^{\gamma}N^{1-\gamma}$, which diverges.

\emph{Short intervals.} The truncated coefficients are bounded by a deterministic $C(R)$, and
Definition~\ref{def:davie} at $m=p$ gives an $h_{1}>0$ with $\|J_{u,v}\|_{L_{p}}\le\tau^{1/2}$ for
$\tau:=|v-u|\le h_{1}$. So \eqref{eq:davie}, the triangle inequality and the
Burkholder--Davis--Gundy inequality yield, for those pairs,
\[
\big\|\delta X_{u,v}\big\|_{L_{p}}
\ \le\ C(R)\big(\tau+\tau^{1/2}+\Lambda_{\RZ}\tau^{\gamma}+\Lambda_{\RZ}\tau^{2\gamma}\big)+\tau^{1/2}
\ \le\ K\tau^{\gamma},
\qquad \tau\le h_{1}\wedge1 ,
\]
since $\gamma<\tfrac12$.

\emph{All intervals.} Chaining over $\lceil T/(h_{1}\wedge1)\rceil$ such intervals and using
Assumption~\ref{ass:Einit} gives $\sup_{r\le T}\|X_{r}\|_{L_{p}}<\infty$, and hence
$\|\delta X_{u,v}\|_{L_{p}}\le K'|v-u|^{\gamma}$ at every pair, the longer ones by that supremum and
$|v-u|^{\gamma}\ge(h_{1}\wedge1)^{\gamma}$.

\emph{The supremum.} Raising that bound to the power $p$ gives
$\E|\delta X_{u,v}|^{p}\lesssim|v-u|^{1+\varepsilon}$ with $\varepsilon=p\gamma-1>0$ by
\eqref{eq:Ep}. Writing $[X]_{\alpha}:=\sup_{u\ne v}|\delta X_{u,v}|/|v-u|^{\alpha}$, Kolmogorov--Chentsov in its
quantitative form then bounds $\E\big[[X]^{p}_{\alpha}\big]$ for $\alpha<\varepsilon/p$, and with
$\sup_{t\le T}|X_{t}|\le|X_{0}|+T^{\alpha}[X]_{\alpha}$ and $\|X_{0}\|_{L_{p}}<\infty$ we obtain
$\E\big[\sup_{t\le T}|X_{t}|^{p}\big]<\infty$. Hence $\mathcal M<\infty$, by $\Phi\le1+|X|$. This is
what makes $\mathcal H$ finite in the proof of Lemma~\ref{lem:Escalefree}. $\mathcal J<\infty$ is
instead \eqref{eq:Jfinite}.

\item[(c)]\stepnum{(c)}{proof:Eproof:step2:c} \emph{Closing the martingale term.} We raise \eqref{eq:Epathwise} to the power $m$ and take
expectations. Applying the Burkholder--Davis--Gundy inequality together with \eqref{eq:Vbracket} and
Doob's inequality, we obtain
\[
\E\Big[\sup_{t\le T}|N^{V}_{t}|^{m}\Big]\le c_{m}\,\E\big[\langle N^{V}\rangle_{T}^{m/2}\big]
\le c_{m}\big(C_{1}T\big)^{m/2}\,\E\Big[\sup_{t\le T}\Phi^{m}_{t}\Big] ,
\]
so that, denoting $\Phi_{0}=V(X_{0})$ and $\mathcal M_{m}:=\E[\sup_{t\le T}\Phi^{m}_{t}]$, this yields
\begin{equation}\label{eq:Eabsorb}
\mathcal M_{m}\ \le\ 2^{m-1}\mathcal E^{m}_{T}\,\E\big[V(X_{0})^{m}\big]
\ +\ 2^{m-1}\mathcal E^{m}_{T}\,c_{m}(C_{1}T)^{m/2}\,\mathcal M_{m} .
\end{equation}
We choose $T_{0}\in(0,T]$ such that
\[
2^{m-1}\mathcal E^{m}_{T_{0}}c_{m}\big(C_{1}T_{0}\big)^{m/2}\ \le\ \tfrac12 ,
\]
which is possible because $T_{0}\mapsto\mathcal E^{m}_{T_{0}}T_{0}^{m/2}$ is continuous and vanishes at
$T_{0}=0$. More explicitly, we can take $T_{0}:=\tau_{0}\wedge\Lambda^{-1/\gamma}_{\RZ}\wedge T$ with
$\tau_{0}=\tau_{0}(m,C_{1},C_{2})$, since $\Lambda^{1/\gamma}_{\RZ}T_{0}\le1$ bounds
$\mathcal E_{T_{0}}$ by $e^{2C_{2}}$. We then run \ref{proof:Eproof:step2:a0}--\ref{proof:Eproof:step2:c} on $[0,T_{0}]$ and iterate over the
\[
\lceil T/T_{0}\rceil\ \le\ 1+\big(\tau^{-1}_{0}+\Lambda^{1/\gamma}_{\RZ}\big)T
\]
successive intervals, each time with the previous endpoint as initial datum. Each contributes the same
constant factor $2^{m}\mathcal E^{m}_{T_{0}}$, so that the accumulated constant is again of the form
$\exp\{C(1+(1+\|\RZ\|^{1/\gamma}_{\gamma})T)\}$ after enlarging $C$. Using
$\tfrac1{\sqrt2}(1+|x|)\le V(x)\le1+|x|$, this is \eqref{eq:Emoment} at exponent $m$, for $X^{k,R}$, with
a constant free of $k$ and $R$.

Taking $m=p$, we obtain \eqref{eq:Emoment} at the exponent $p$, and hence for every $m\in[2,p]$.
What this section's argument requires is $\gamma m>1$, hence $m>1/\gamma>2$ by $\gamma<\tfrac12$.
 
\end{enumerate}

\color{blue}
\begin{remark}\label{rem:Etemporal}
The germ coefficients of \eqref{eq:rsdeE} depend on time as well as on the state, so that their
two-index increments split as in \eqref{eq:Gsplit}, and \eqref{eq:DGGLip}, a bound on a spatial
derivative, controls only the spatial half. For an autonomous field the temporal half vanishes.

 The temporal half of $\delta[Df\,f+f']$ contains
$(Df_{w}-Df_{u})(x)f_{w}(x)$, whose second factor is of linear growth and cannot be improved,
Assumption~\ref{ass:Erough} requires only $|f_{t}(x)|\le C_{\mathrm{gr}}(1+|x|)$. Had the first term
of \eqref{eq:Eroughderiv} carried the weight $1+|x|$ that the other two carry, the product would be of
order $(1+|x|)^{2}$, whose $L_{m}$ norm is $\|\Phi_{u}\|^{2}_{L_{2m}}$, which the $L_{m}$ quantity
$\mathcal N_{s,t}$ does not dominate. Interpolating against $|Df_{w}-Df_{u}|\le2\|Df\|_{\infty}$
leaves $(1+|x|)^{1+\theta}\tau^{\theta(\beta+\beta')}$ with $\theta\in(0,1)$, where $\epsilon_{1}$
needs $\theta$ bounded away from $0$, forcing a moment exponent strictly between $m$ and $2m$. Since
\eqref{eq:Emoment} is proved by running the whole scheme at the top exponent $m=p$ supplied by $\nu$,
there is no higher exponent available and the Gr\"onwall does not close.

The unweighted form is a mild strengthening: it implies the weighted form, since
$1+|x|\ge1$, and is not implied by it. Take e.g.~$n=e=1$, $f'\equiv0$ and
$f_{t}(x)=\int_{0}^{x}\sin(\omega(t)y)\dy$ with $\omega$ H\"older of index $\beta+\beta'$ and
bounded away from $0$, since
$f_{t}(x)=(1-\cos(\omega(t)x))/\omega(t)$ has $\partial_{\omega}f\to x^{2}/2$ as $\omega\downarrow0$,
so that without it the temporal defect of $f$ is quadratic in the state and \eqref{eq:cvfdefect} fails,
whereas $\inf_{t}\omega(t)=:\omega_{-}>0$ yields $|\partial_{\omega}f|\le|x|/\omega_{-}+2/\omega^{2}_{-}$.
Then $Df_{t}(x)=\sin(\omega(t)x)$ is bounded and the weighted statements holds by
$|\sin(ax)-\sin(bx)|\le\min\{|a-b||x|,2\}\le|a-b|(1+|x|)$, while $\sup_{x}|Df_{t}(x)-Df_{s}(x)|$ does
not tend to $0$ as $t-s\to0$, so that the unweighted one fails. The issue arises from $(1+|x|)|D^{2}f|$ being unbounded, so that the field satisfies neither
curvature branch. What the strengthening excludes is a bounded derivative whose oscillation in time
does not decay as the state grows; every autonomous field satisfies it trivially, as does the example
$f_{t}=t^{\beta+\beta'}g$ of Remark~\ref{rem:Efprime} with $g$ affine.
\end{remark}
\color{black}

\subsubsection*{Step 3: tightness, and identification of the limit (jointly in $(k,R)$)}
\makeatletter\def\@currentlabel{3}\makeatother\label{step:Etight}

Both indices are removed here, and in the same limit.  We can treat
$\{X^{k,R}\}$ as one doubly indexed family and send the two indices to infinity together.

\emph{(a) Tightness.}\stepnum{(a)}{stp:Et-a} Applying the Burkholder--Davis--Gundy inequality to the Davie expansion for
$X^{k,R}$, and using \eqref{eq:Egrowth} together with \eqref{eq:Emoment}, we obtain for
$m\in(1/\gamma,p]$ and $s\le t$ with $|t-s|\le h_{\star}$,
\begin{equation}\label{eq:Eholder}
\big\|\delta X^{k,R}_{s,t}\big\|_{L_{m}}
\ \lesssim\ \underbrace{|t-s|}_{\text{drift}}
+\underbrace{|t-s|^{1/2}}_{\text{martingale}}
+\underbrace{\|\RZ\|_{\gamma}|t-s|^{\gamma}}_{\text{first level}}
+\underbrace{\|\RZ\|_{\gamma}|t-s|^{2\gamma}}_{\text{second level}}
+\underbrace{C_{3}\Lambda^{2}_{\RZ}|t-s|^{2\gamma}}_{\text{Davie remainder}}
\ \lesssim\ \Lambda_{\RZ}\,|t-s|^{\gamma},
\end{equation}
uniformly in $k$ and in $R$, the exponent $\gamma$ is rate determining, since $\gamma<\tfrac12<1$ and
$2\gamma>\gamma$. The Davie-remainder entry is the second bound of
\eqref{eq:HJfinal} in the proof of Lemma~\ref{lem:Escalefree}, and the final bound is
\eqref{eq:EscaleX}, with $\mathcal N_{s,t}$ bounded by \eqref{eq:Emoment}.

We can take $m=p$ and raise \eqref{eq:Eholder} to the power $p$, which yields
\[
\E\Big[\big|\delta X^{k,R}_{s,t}\big|^{p}\Big]\ \lesssim\ \Lambda^{p}_{\RZ}\,|t-s|^{1+\varepsilon},
\qquad
\varepsilon\ :=\ p\gamma-1\ >\ 0 ,
\]
uniformly in $k$ and in $R$, where $\varepsilon>0$ holds by \eqref{eq:Ep}, since
$p>6>3>1/\gamma$. Kolmogorov--Chentsov then yields
the tightness of the doubly indexed family
$\{\Law(X^{k,R},B)\}_{k\ge1,\,R\ge1}$ on
$C([0,T];\R^{n}\times\R^{d})$. The initial laws are all $\nu$. Intervals longer than $h_{\star}$ are
covered by concatenating at most $\lceil T/h_{\star}\rceil$ of them.

\emph{(b) Skorokhod.}\stepnum{(b)}{stp:Et-b} We can fix any sequence $(k_{j},R_{j})_{j\ge1}$ with $k_{j}\to\infty$ and
$R_{j}\to\infty$, extract a subsequence along which $\Law(X^{k_{j},R_{j}},B)$ converges, and choose a
probability space carrying $(\widetilde X^{j},\widetilde B^{j})$ such that
$\Law(\widetilde X^{j},\widetilde B^{j})=\Law(X^{k_{j},R_{j}},B)$ and
$(\widetilde X^{j},\widetilde B^{j})\to(\widetilde X,\widetilde B)$ a.s.\ in
$C([0,T];\R^{n}\times\R^{d})$. Each
$\widetilde B^{j}$ is a Brownian motion for the filtration generated by
$(\widetilde X^{j},\widetilde B^{j})$, and each $\widetilde X^{j}$ is an $L_{m}$-conditional solution of the
$(k_{j},R_{j})$-th truncated, mollified equation driven by $\widetilde B^{j}$, by
Lemma~\ref{lem:lawdet}.

Let $s\le t$,
let $g\in C_{b}(\R^{d})$, let $\psi$ denote a bounded continuous cylinder functional of
$(\widetilde X,\widetilde B)|_{[0,s]}$, and let $\psi^{j}$ denote the same functional of
$(\widetilde X^{j},\widetilde B^{j})$. Then
\[
\E\big[g\big(\widetilde B^{j}_{t}-\widetilde B^{j}_{s}\big)\,\psi^{j}\big]
\ =\ \E\big[g\big(\widetilde B^{j}_{t}-\widetilde B^{j}_{s}\big)\big]\,\E\big[\psi^{j}\big],
\qquad j\ge1 ,
\]
and by bounded convergence we obtain
\[
\E\big[g\big(\widetilde B_{t}-\widetilde B_{s}\big)\,\psi\big]
\ =\ \E\big[g\big(\widetilde B_{t}-\widetilde B_{s}\big)\big]\,\E\big[\psi\big] .
\]
With $\mathcal G_{t}:=\sigma(\widetilde X_{r},\widetilde B_{r}:r\le t)$, the chosen $\psi$ generate $\mathcal G_{s}$. Hence, $\widetilde B_{t}-\widetilde B_{s}$ is independent of
$\mathcal G_{s}$ and is $N(0,(t-s)I)$, its law being the limit of those of
$\widetilde B^{j}_{t}-\widetilde B^{j}_{s}$. With $\widetilde B$ continuous and $\widetilde B_{0}=0$,
this is the definition of a $(\mathcal G_{t})$-Brownian motion.

 Here $\delta Z$
and $\ZZ$ are deterministic, and the It\^o integral is, under each
measure, almost surely equal to a random variable measurable for the generated filtration and
continuous in $(s,t)$ (Lemma~\ref{lem:itoraw}), so that the Davie remainder
\eqref{eq:davie} is given by the same formula on either space and its joint law with the pair is
determined by the $\Law(X,B)$ alone. Both moduli of
\eqref{eq:daviemoduli}, taken with respect to the generated filtration, are therefore determined
by $\Law(X,B)$, and Lemma~\ref{lem:filtdescend} hands the moduli that Step~\ref{step:Ereg}
produces on the ambient filtration to the generated one.

\emph{(c) Identification.}\stepnum{(c)}{stp:Et-c} The drift term passes to the limit,
\[
\int_{s}^{t}b^{k_{j}}_{R_{j}}\big(r,\widetilde X^{j}_{r}\big)\dr
\ \longrightarrow\
\int_{s}^{t}b\big(r,\widetilde X_{r}\big)\dr
\qquad\text{in }L_{m}\ \text{ for every }m<p,
\]
by the locally uniform
convergence $b^{k_{j}}_{R_{j}}\to b$ of Step~\ref{step:Ereg}, the a.s.\ convergence of
$\widetilde X^{j}$, dominated convergence and uniform integrability from \eqref{eq:Emoment} with
$m<p$.

The It\^o term needs slightly more. We want
\begin{equation}\label{eq:Itostability}
\int_{s}^{t}\sigma^{k_{j}}_{R_{j}}\big(r,\widetilde X^{j}_{r}\big)\,\d\widetilde B^{j}_{r}
\ \longrightarrow\
\int_{s}^{t}\sigma\big(r,\widetilde X_{r}\big)\,\d\widetilde B_{r}
\qquad\text{in }L_{m}\ \text{ for every }m<p.
\end{equation}
What supplies \eqref{eq:Itostability} is a martingale argument. The Skorokhod space carries every $j$ at once. What stays
$j$-dependent is the filtration, each identity below being read for the one $\mathcal G^{j}$ its index
names, and only expectations cross the limit. $\mathcal G^{j}_{t}:=\sigma\big(\widetilde X^{j}_{r},\widetilde B^{j}_{r}:r\le t\big)$.

\emph{(i)}\stepnum{(i)}{stp:Et-ci} We denote
$M^{k,R}_{t}:=\int_{0}^{t}\sigma^{k}_{R}(r,X^{k,R}_{r})\,\d B_{r}$. Applying the
Burkholder--Davis--Gundy inequality together with \eqref{eq:Egrowth} and \eqref{eq:Emoment}, we obtain the following bounds,
\[
\big\|\delta M^{k,R}_{s,t}\big\|_{L_{p}}
\ \lesssim\ \Big\|\Big(\int_{s}^{t}\big\|\sigma^{k}_{R}\big(r,X^{k,R}_{r}\big)\big\|^{2}\dr\Big)^{1/2}\Big\|_{L_{p}}
\ \lesssim\ |t-s|^{1/2},
\qquad
\E\Big[\sup_{t\le T}\big|M^{k,R}_{t}\big|^{p}\Big]\ \lesssim\ 1,
\]
both uniformly in $k$ and in $R$. Raising the first bound to the power $p$ yields
\[
\E\Big[\big|\delta M^{k,R}_{s,t}\big|^{p}\Big]\ \lesssim\ |t-s|^{p/2},
\qquad \tfrac{p}{2}\ >\ 1 ,
\]
since $p>2$, so that Kolmogorov--Chentsov applies to $M^{k,R}$ as it did to $X^{k,R}$ in \ref{stp:Et-a}. Hence
$\{\Law(X^{k,R},B,M^{k,R})\}_{k,R}$ is tight on $C([0,T];\R^{n}\times\R^{d}\times\R^{n})$ and we can
run steps \ref{stp:Et-a} and \ref{stp:Et-b} on the triple. This is not a second extraction on top of
step~\ref{stp:Et-b}: that step is to be read from the start with the triple
$(X^{k,R},B,M^{k,R})$ in place of the pair, so that one subsequence and one Skorokhod
representation serve throughout, and the pair statements of \ref{stp:Et-b} are the first two components of the
triple statements. We denote by $\widetilde M^{j}$ the
third component of the $j$-th Skorokhod copy and by $\widetilde M$ its limit. At fixed $(k,R)$ the
integrand $\sigma^{k}_{R}(\cdot,X^{k,R})$ is continuous and adapted, so that $M^{k,R}$ is the u.c.p.\
limit of the left-point Riemann sums $S_{l}(X^{k,R},B)$ along any partitions of vanishing mesh, each
$S_{l}$ a fixed continuous functional of the path pair, by the construction the proof of
Lemma~\ref{lem:MPborel}\ref{mpb:version} opens with, run on $\sigma^{k}_{R}$ in place of
$\nabla\varphi\cdot\sigma$. From equality in law of the triples we obtain
\[
\sup_{t\le T}\big|S_{l}\big(\widetilde X^{j},\widetilde B^{j}\big)_{t}-\widetilde M^{j}_{t}\big|
\ \longrightarrow\ 0
\qquad\text{in probability as }l\to\infty ,
\]
On the Skorokhod space
, the same sums converge u.c.p.\ to
$\int_{0}^{\cdot}\sigma^{k_{j}}_{R_{j}}\big(r,\widetilde X^{j}_{r}\big)\d\widetilde B^{j}_{r}$,
$\widetilde B^{j}$ being a $\mathcal G^{j}$-Brownian motion and $\widetilde X^{j}$
$\mathcal G^{j}$-adapted by (b). Two u.c.p.\ limits of one sequence agree almost
surely, and we obtain
\begin{equation}\label{eq:Mjtransfer}
\widetilde M^{j}_{t}=\int_{0}^{t}\sigma^{k_{j}}_{R_{j}}\big(r,\widetilde X^{j}_{r}\big)\,\d\widetilde B^{j}_{r}
\qquad\text{almost surely, for all }t\le T .
\end{equation}
First, $\widetilde M^{j}_{t}$ is by \eqref{eq:Mjtransfer} a limit in probability of
$\mathcal G^{j}_{t}$-measurable random variables, so the filtration generated by the $j$-th triple
coincides with $\mathcal G^{j}$ up to $\Prob$-null sets, $\widetilde B^{j}$ is a Brownian motion for it,
and the closedness argument of (b) applies identically to the triple: $\widetilde B$ is a Brownian motion
for $\mathcal G^{\sharp}_{t}:=\sigma\big(\widetilde X_{r},\widetilde B_{r},\widetilde M_{r}:r\le t\big)$.
Second, $\widetilde M^{j}\to\widetilde M$ almost surely uniformly on $[0,T]$, hence in $L_{m}$ for every
$m<p$ by the $p$-th moment bound. Those null sets, and the stochastic integration,
Kunita--Watanabe associativity and Doob--Meyer uniqueness of brackets used in (ii) and (iii), are read
on the usual augmentation of the filtration, which changes neither integrals, brackets, nor the
Brownian property.

\emph{(ii) The brackets of the limit.}\stepnum{(ii)}{stp:Et-cii} We fix $s\le t$, let $\psi$ denote a bounded continuous cylinder
functional of the limit triple on $[0,s]$ and $\psi^{j}$ the same functional of the $j$-th triple. Such
$\psi$ generate $\mathcal G^{\sharp}_{s}$, and each $\psi^{j}$ is bounded and
$\mathcal G^{j}_{s}$-measurable. At fixed $(k_{j},R_{j})$ the coefficients are bounded, so that by
\eqref{eq:Mjtransfer} $\widetilde M^{j}$ is a square-integrable continuous $\mathcal G^{j}$-martingale
such that
\[
\big\langle\widetilde M^{j,i},\widetilde B^{j,\kappa}\big\rangle_{t}
=\int_{0}^{t}\big(\sigma^{k_{j}}_{R_{j}}\big)^{i\kappa}\big(r,\widetilde X^{j}_{r}\big)\dr,
\qquad
\big\langle\widetilde M^{j,i},\widetilde M^{j,l}\big\rangle_{t}
=\int_{0}^{t}\big(a^{k_{j}}_{R_{j}}\big)^{il}\big(r,\widetilde X^{j}_{r}\big)\dr,
\]
where $a^{k}_{R}:=\sigma^{k}_{R}(\sigma^{k}_{R})^{\top}$ denotes the diffusion matrix. Testing the three martingales against
$\psi^{j}$, we obtain, for every $j$,
\[
\E\big[\delta\widetilde M^{j,i}_{s,t}\,\psi^{j}\big]=0,\qquad
\E\Big[\Big(\delta\big(\widetilde M^{j,i}\widetilde B^{j,\kappa}\big)_{s,t}
-\int_{s}^{t}\big(\sigma^{k_{j}}_{R_{j}}\big)^{i\kappa}\big(r,\widetilde X^{j}_{r}\big)\dr\Big)\psi^{j}\Big]=0,
\]
and the same with $\big(\widetilde M^{j,i}\widetilde M^{j,l},\,a^{k_{j}}_{R_{j}}\big)$ in place of
$\big(\widetilde M^{j,i}\widetilde B^{j,\kappa},\,\sigma^{k_{j}}_{R_{j}}\big)$. Each passes to the limit
in $L_{1}$. The linear terms do so by (i). For the products, $\widetilde M^{j}_{t}\to\widetilde M_{t}$
in $L_{m}$ and $\widetilde B^{j}_{t}\to\widetilde B_{t}$ in every $L_{q}$, the $\widetilde B^{j}$ being
Brownian motions and the convergence almost sure, and H\"older's inequality yields the following bound,
\[
\big\|\widetilde M^{j,i}_{t}\widetilde B^{j,\kappa}_{t}
-\widetilde M^{i}_{t}\widetilde B^{\kappa}_{t}\big\|_{L_{m/2}}
\ \le\ \big\|\widetilde M^{j,i}_{t}-\widetilde M^{i}_{t}\big\|_{L_{m}}
\big\|\widetilde B^{j,\kappa}_{t}\big\|_{L_{m}}
+\big\|\widetilde M^{i}_{t}\big\|_{L_{m}}
\big\|\widetilde B^{j,\kappa}_{t}-\widetilde B^{\kappa}_{t}\big\|_{L_{m}}
\ \longrightarrow\ 0 ,
\]
so that they converge in $L_{m/2}\subseteq L_{1}$, the exponent $m$ exceeding $1/\gamma>2$. For the
$\dr$ integrals, the agreement on compacts below and \eqref{eq:standingcts} give almost sure
convergence of the integrands, and \eqref{eq:Egrowth} with \eqref{eq:Emoment} at $p>2$ the uniform
integrability, for $a^{k_{j}}_{R_{j}}$ as well since
$\Frob{a^{k}_{R}}\le\Frob{\sigma^{k}_{R}}^{2}\lesssim(1+|x|)^{2}$, and $\psi^{j}\to\psi$ almost surely and
boundedly. Hence $\widetilde M$ is a continuous $\mathcal G^{\sharp}$-martingale such that
\begin{equation}\label{eq:Mbrackets}
\big\langle\widetilde M^{i},\widetilde B^{\kappa}\big\rangle_{t}=\int_{0}^{t}\sigma^{i\kappa}\big(r,\widetilde X_{r}\big)\dr,
\qquad
\big\langle\widetilde M^{i},\widetilde M^{l}\big\rangle_{t}=\int_{0}^{t}a^{il}\big(r,\widetilde X_{r}\big)\dr .
\end{equation}

\emph{(iii) The representation.}\stepnum{(iii)}{stp:Et-ciii} $\widetilde B$ is a $\mathcal G^{\sharp}$-Brownian
motion by (i), and $r\mapsto\sigma(r,\widetilde X_{r})$ is continuous and
$\mathcal G^{\sharp}$-adapted, so that
$\mathcal R:=\widetilde M-\int_{0}^{\cdot}\sigma(r,\widetilde X_{r})\,\d\widetilde B_{r}$ is a
continuous local martingale such that $\mathcal R_{0}=0$. Its bracket expands into the following four terms,
\[
\big\langle\mathcal R^{i},\mathcal R^{l}\big\rangle_{t}
\ =\ \big\langle\widetilde M^{i},\widetilde M^{l}\big\rangle_{t}
-\big\langle\widetilde M^{i},\int_{0}^{\cdot}\sigma^{l\cdot}\d\widetilde B\big\rangle_{t}
-\big\langle\int_{0}^{\cdot}\sigma^{i\cdot}\d\widetilde B,\widetilde M^{l}\big\rangle_{t}
+\big\langle\int_{0}^{\cdot}\sigma^{i\cdot}\d\widetilde B,
\int_{0}^{\cdot}\sigma^{l\cdot}\d\widetilde B\big\rangle_{t} ,
\]
each equal to $\int_{0}^{t}a^{il}(r,\widetilde X_{r})\dr$. The first
is the second identity of \eqref{eq:Mbrackets}. The two mixed ones are the associativity of the
bracket with the stochastic integral, read against the first identity of \eqref{eq:Mbrackets},
\[
\begin{gathered}
\big\langle\widetilde M^{i},\int_{0}^{\cdot}\sigma^{l\cdot}\d\widetilde B\big\rangle_{t}
=\sum_{\kappa}\int_{0}^{t}\sigma^{l\kappa}\d\big\langle\widetilde M^{i},\widetilde B^{\kappa}\big\rangle
=\int_{0}^{t}a^{il}(r,\widetilde X_{r})\dr,
\\
\big\langle\int_{0}^{\cdot}\sigma^{i\cdot}\d\widetilde B,\widetilde M^{l}\big\rangle_{t}
=\sum_{\kappa}\int_{0}^{t}\sigma^{i\kappa}\d\big\langle\widetilde B^{\kappa},\widetilde M^{l}\big\rangle
=\int_{0}^{t}a^{li}(r,\widetilde X_{r})\dr ,
\end{gathered}
\]
and the fourth is the bracket formula for stochastic integrals against a Brownian motion,
\[
\big\langle\int_{0}^{\cdot}\sigma^{i\cdot}\d\widetilde B,\int_{0}^{\cdot}\sigma^{l\cdot}\d\widetilde B\big\rangle_{t}
=\sum_{\kappa}\int_{0}^{t}\sigma^{i\kappa}\big(r,\widetilde X_{r}\big)\sigma^{l\kappa}\big(r,\widetilde X_{r}\big)\dr
=\int_{0}^{t}a^{il}\big(r,\widetilde X_{r}\big)\dr .
\]
Carrying the signs $+,-,-,+$ of the expansion, the four cancel:
\[
\big\langle\mathcal R^{i},\mathcal R^{l}\big\rangle_{t}
\ =\ \int_{0}^{t}a^{il}\big(r,\widetilde X_{r}\big)\dr
-\int_{0}^{t}a^{il}\big(r,\widetilde X_{r}\big)\dr
-\int_{0}^{t}a^{il}\big(r,\widetilde X_{r}\big)\dr
+\int_{0}^{t}a^{il}\big(r,\widetilde X_{r}\big)\dr\ =\ 0 ,
\]
so that $\langle\mathcal R^{i},\mathcal R^{l}\rangle\equiv0$, whence $\mathcal R\equiv0$ and
\[
\widetilde M_{t}=\int_{0}^{t}\sigma\big(r,\widetilde X_{r}\big)\,\d\widetilde B_{r}.
\]
The right-hand side is a u.c.p.\ limit of left-point Riemann sums of a continuous adapted integrand, so
that it is the same object on $\mathcal G^{\sharp}$ or on the smaller $\mathcal G$. $\widetilde B$
is a Brownian motion for both. In particular, $\widetilde M$ is $\mathcal G$-adapted. Combining with the
$L_{m}$ convergence $\widetilde M^{j}\to\widetilde M$ of (i) and taking increments yields
\eqref{eq:Itostability} for every $s\le t$ and every $m<p$.

The identification does not care which of the two indices is moving: on a
fixed compact $\{|x|\le M\}$ we have $\Theta_{R}=1$ and $\theta_{R}(x_{i})=1$ for every $i$ as soon as
$R\ge M$, so there $b^{k}_{R}=b^{k}$, $\sigma^{k}_{R}=\sigma^{k}$ and $(f_{R},f'_{R})=(f,f')$
exactly, and what is left is the locally uniform convergence $b^{k}\to b$,
$\sigma^{k}\to\sigma$ of Step~\ref{step:Ereg}. For the rough term the integrand is the controlled pair
carried by the equation:
\[
\big(Y^{j}_{r},Y^{j\prime}_{r}\big):=\Big(f_{R_{j},r}\big(\widetilde X^{j}_{r}\big),\
\big(Df_{R_{j},r}f_{R_{j},r}+f'_{R_{j},r}\big)\big(\widetilde X^{j}_{r}\big)\Big)
\ \longrightarrow\
\big(Y_{r},Y'_{r}\big):=\Big(f_{r}\big(\widetilde X_{r}\big),\ \big(Df_{r}f_{r}+f'_{r}\big)\big(\widetilde X_{r}\big)\Big)
\]
in $L_{m}$ at each fixed $r$ and each $m<p$,
by the agreement on compacts, continuity of $f_{r}$ and $Df_{r}f_{r}+f'_{r}$, a.s.\ convergence, the linear growth
$|f_{R}|+|f'_{R}|\le C_{\mathrm{gr}}(1+|x|)$ that truncation does not interfere with, and uniform integrability
from \eqref{eq:Emoment}. Hence the pair is the one of \eqref{eq:davie}.

Stability of the rough integral is what yields the integral form of
\eqref{eq:rsdeE} in the limit. By \cite[Prop.~4.3]{FHL} an integrable solution there is
equivalently a process such that
\[
\widetilde X_{t}=\widetilde X_{0}+\int_{0}^{t}b\big(r,\widetilde X_{r}\big)\dr
+\int_{0}^{t}\sigma\big(r,\widetilde X_{r}\big)\,\d\widetilde B_{r}+\int_{0}^{t}(Y,Y')\d\RZ
\qquad\text{for all }t\in[0,T],\ \text{a.s.},
\]
\cite[Prop.~4.3]{FHL} carries as a hypothesis the membership condition \cite[Def.~4.2(b)]{FHL},
membership in $\mathscr D^{\bar\alpha,\bar\alpha'}_{\RZ}L_{m,\mfn }$ of \cite[Def.~3.1]{FHL}. The limit pair $(\widetilde X,Y)$ satisfies that condition: the four bounds it requires
for are established at the end of this step. One clause of \cite[Def.~3.1]{FHL} is left open there,
the right-continuity in $s$ of $(s,t)\mapsto\E_{s}R^{Y}_{s,t}$, and nothing in this paper uses
it.

The stability lemma \cite[Lem.~4.20]{FHL} asks that pair to lie in the extended
controlled class of \cite[Def.~4.12]{FHL}. We apply it in the weaker form of
Lemma~\ref{lem:roughstab}, which requires only what the proof in \cite{FHL} uses: control of the pair through the four quantities
\[
\big\|\delta Y\big\|_{\gamma;m,\mfn },
\qquad
\sup_{r\le T}\big\|Y'_{r}\big\|_{L_{\mfn }},
\qquad
\sup_{s<t}\frac{\big\|\E_{s}R^{Y}_{s,t}\big\|_{L_{\mfn }}}{|t-s|^{\gamma+\beta_{\star}}},
\qquad
\sup_{s<t}\frac{\big\|\E_{s}\delta Y'_{s,t}\big\|_{L_{\mfn }}}{|t-s|^{\beta_{\star}}},
\qquad
R^{Y}_{s,t}:=\delta Y_{s,t}-Y'_{s}\delta Z_{s,t},
\]
Their sum is
$\Gamma^{\gamma,\beta_{\star};m,\mfn }(Y,Y';[0,T])$, and the lemma requires it to be bounded uniformly along the
sequence and finite in the limit.

 At $\mfn =\infty$ the second quantity is an essential
supremum over $\Omega$, while $Y^{j\prime}=(Df_{R_{j}}f_{R_{j}}+f'_{R_{j}})(\widetilde X^{j})$ has
linear growth. More precisely,
\[
\sup_{r\le T}\big\|Y^{j\prime}_{r}\big\|_{L_{\infty}}
\ \le\ \big(\sup_{R}\|Df_{R}\|_{\infty}+1\big)C_{\mathrm{gr}}\big(1+2R_{j}\big)
\ \xrightarrow[\ j\to\infty\ ]{}\ \infty ,
\]
and at the limit $(Df\,f+f')(\widetilde X)$ is unbounded on $\Omega$ as soon as $\nu$ is, so that the
term is $+\infty$ there. Any $\mfn \in[m,p]$ is available, and we take $\mfn =m$, at which
\[
\big\|\delta Y\big\|_{\gamma;m,m}\ =\ \sup_{s<t}\frac{\|\delta Y_{s,t}\|_{L_{m}}}{|t-s|^{\gamma}}
\]
by \cite[Prop.~2.3(i)]{FHL}.

$R^{Y}$ is exactly the remainder
$R$ of \eqref{eq:Rsplitfive}, and Step~\ref{proof:Eproof:step2:step1}\ref{proof:Eproof:step2:step1:1b} of the proof of Lemma~\ref{lem:Escalefree} bounds it
uniformly in $(k,R)$ at order $\min\{\tfrac{1+\gamma}{2},\beta+\beta'\}$.
Collecting the orders, with $\varrho^{(2)}$ at $\beta+\beta'$ by \eqref{eq:rhoorders} and the Taylor
term $\mathcal T$ at $\tfrac{1+\gamma}{2}$ by \eqref{eq:Emismatch}, yields the following bound,
\[
\big\|\E_{s}R^{Y}_{s,t}\big\|_{L_{m}}
\ \lesssim\ \underbrace{|t-s|}_{\text{drift}}
+\underbrace{|t-s|^{2\gamma}}_{\text{second level},\ J,\ \varrho^{(1)}}
+\underbrace{|t-s|^{\beta+\beta'}}_{\varrho^{(2)}}
+\underbrace{|t-s|^{\frac{1+\gamma}{2}}}_{\mathcal T}
\ \lesssim\ |t-s|^{\min\{\frac{1+\gamma}{2},\ \beta+\beta'\}} ,
\]
the last term rate determining, because $\gamma>\tfrac13$. Since $\tfrac{1+\gamma}{2}<\beta+\beta'$ does occur inside
Assumption~\ref{ass:Erough}, at e.g.~$\gamma=\beta=\beta'=0.45$, we lower the pair to
$(\gamma,\beta_{\star})$, where we denote
\[
\beta_{\star}\ :=\ \epsilon_{1}-2\gamma\ =\ \min\Big\{\frac{1-\gamma}{2},\ \beta+\beta'-\gamma\Big\}
\ =\ \min\Big\{\frac{1-\gamma}{2},\ \bar\beta\Big\}\ \in\ (0,\bar\beta] ,
\]
with $\epsilon_{1}$ as in \eqref{eq:eps1value} and $\bar\beta$ as in \eqref{eq:Jbarbeta}. The second
equality holds because $\tfrac{1-\gamma}{2}\le\gamma$ at $\gamma\ge\tfrac13$, which is also why
$\beta_{\star}\le\bar\beta$, and $\beta_{\star}>0$ because $\bar\beta>0$. Here $\gamma+\beta_{\star}=\min\{\tfrac{1+\gamma}{2},\beta+\beta'\}$ is the
order just named, and $\beta_{\star}\le\bar\beta\le\beta'$ is the order of the fourth quantity. The
parameter conditions of \cite[Thm.~3.4]{FHL}, on which \cite[Lem.~4.20]{FHL} rests, at the
driver's index $\alpha=\gamma$
\[
\gamma+\gamma\ =\ 2\gamma\ >\ \tfrac12,
\qquad
\gamma+(\gamma\wedge\gamma)+\beta_{\star}\ =\ 2\gamma+\beta_{\star}\ =\ \epsilon_{1}\ >\ 1 ,
\]
the first by $\gamma>\tfrac13$ and the second by \eqref{eq:eps1value}. The driver's index is available, since \cite[Thm.~3.4]{FHL}
constrains only its own triple $(\alpha,\beta,\beta')$, read here at $(\gamma,\gamma,\beta_{\star})$,
and does not require $\mathscr D^{2\alpha}_{\RZ}$ membership.

\begin{lemma}[Stability of the rough integral, after {\cite[Lem.~4.20]{FHL}}]\label{lem:roughstab}
Let $m\ge2$, let $\mfn \in[m,\infty]$, and let $(\gamma,\beta_{\star})$ meet the parameter conditions of
\cite[Thm.~3.4]{FHL}. Let $(\Omega,\mathcal A,\Prob)$ carry, for each $j\in\N\cup\{\infty\}$, a
filtration $(\mathcal G^{j}_{t})$ and a pair $(Y^{j},Y^{j\prime})$ progressively measurable for it.
Assume
\[
\sup_{j}\ \Gamma^{\gamma,\beta_{\star};m,\mfn }\big(Y^{j},Y^{j\prime};[0,T]\big)\ <\ \infty
\]
and $(Y^{j}_{s},Y^{j\prime}_{s})\to(Y^{\infty}_{s},Y^{\infty\prime}_{s})$ in $L_{m}$ at every
$s\in[0,T]$. Then
\[
\sup_{t\le T}\Big|\int_{0}^{t}(Y^{j},Y^{j\prime})\d\RZ
-\int_{0}^{t}(Y^{\infty},Y^{\infty\prime})\d\RZ\Big|\ \longrightarrow\ 0
\qquad\text{in }L_{m} .
\]
\end{lemma}
This is \cite[Lem.~4.20]{FHL}, and its two-line proof, with one hypothesis
relaxed: the controlled class is dropped. The proof only needs the requirements of \cite[Thm.~3.4]{FHL},
progressive measurability and the finiteness of that theorem's $\Gamma_{1}\vee\Gamma_{2}$, which the
statement above dominates: under it the germ
$A^{j}_{s,t}=Y^{j}_{s}\delta Z_{s,t}+Y^{j\prime}_{s}\ZZ_{s,t}$ satisfies Lemma~\ref{lem:SSL} and the
integral is its sewing. Of condition~(c) of \cite[Def.~3.1]{FHL} only the seminorm bound is
used. The $C_{2}L$-continuity it also carries, and with it the clause of
\cite[Def.~2.5]{FHL}, is not. One filtration per index is allowed in
\cite[Lem.~4.20]{FHL}. The measure space here is common to all of them.

That proof runs through \cite[Cor.~2.10]{FHL}, which requires the hypotheses of
\cite[Thm.~2.9(ii)]{FHL} uniformly in $j$ and $\lim_{j}\big\|\sup_{t\in[s,T]}|A^{j}_{s,t}-A^{\infty}_{s,t}|\big\|_{L_{m}}=0$
for each $s$. The
latter is the $L_{m}$-convergence hypothesis of the lemma, $A^{j}_{s,t}$ being a linear function of
$(Y^{j}_{s},Y^{j\prime}_{s})$ with deterministic coefficients: the supremum is bounded by
$\|\RZ\|_{\gamma}T^{\gamma}|Y^{j}_{s}-Y^{\infty}_{s}|+\|\RZ\|_{\gamma}T^{2\gamma}|Y^{j\prime}_{s}-Y^{\infty\prime}_{s}|$. Of the former, two items are needed
beyond $\Gamma_{1}\vee\Gamma_{2}$, and the proof of \cite[Thm.~3.4]{FHL} provides both: the third defect bound
with $\Gamma_{3}\lesssim\Gamma_{2}$ at the rate
$\varepsilon_{3}=\gamma+(\gamma\wedge\gamma)-\tfrac1m=2\gamma-\tfrac1m$, positive at every
$m\ge2$ since $2\gamma>\tfrac12$, and the almost sure continuity of $t\mapsto A^{j}_{s,t}$, which
holds because $Z$ and $\ZZ$ are continuous.

 At $(\gamma,\beta_{\star};m,m)$ and for
$|t-s|\le h_{\star}$ we obtain the following four bounds,
\[
\begin{aligned}
\big\|\delta Y^{j}_{s,t}\big\|_{L_{m}}&\ \lesssim\ |t-s|^{\gamma},
&\qquad
\sup_{r\le T}\big\|Y^{j\prime}_{r}\big\|_{L_{m}}
&\ \le\ \big(\sup_{R}\|Df_{R}\|_{\infty}+1\big)C_{\mathrm{gr}}
\Big(1+\sup_{r\le T}\big\|\widetilde X^{j}_{r}\big\|_{L_{m}}\Big),\\[2pt]
\big\|\E_{s}R^{Y^{j}}_{s,t}\big\|_{L_{m}}&\ \lesssim\ |t-s|^{\gamma+\beta_{\star}},
&\qquad
\big\|\E_{s}\delta Y^{j\prime}_{s,t}\big\|_{L_{m}}&\ \lesssim\ |t-s|^{\beta_{\star}} .
\end{aligned}
\]
The first bound is \eqref{eq:Eholder}, together with
$\|D^{i}f_{R}\|_{\infty}\le C\sum_{1\le l\le i}\|D^{l}f\|_{\infty}+C$ for $1\le i\le3$ uniformly in $R\ge1$ and the temporal
defect of $(f_{R},f'_{R})$, of order $\gamma\wedge(\beta+\beta')=\gamma$ by \eqref{eq:rhoexp}. The
second is finite by \eqref{eq:Emoment} at $m\le p$, the third is Step~\ref{proof:Eproof:step2:step1}\ref{proof:Eproof:step2:step1:1b}, and the fourth is
\eqref{eq:Gincrement}, of order $\min\{\gamma,\beta'\}\ge\bar\beta\ge\beta_{\star}$, its spatial half
carrying $\|D(Df\,f+f')\|_{\infty}<\infty$ from \eqref{eq:DGGLip} and its temporal half the defects of
Assumption~\ref{ass:Erough}. Every constant is independent of $(k,R)$, $\mathcal H$ and $\mathcal J$ being
bounded by $C_{3}\Lambda_{\RZ}\mathcal N_{s,t}$ and $C_{3}\Lambda^{2}_{\RZ}\mathcal N_{s,t}$ at
\eqref{eq:HJfinal} and $\mathcal N_{s,t}$ by \eqref{eq:Emoment}.

Pairs with $|t-s|>h_{\star}$ are covered in a crude manner. The conditional expectation is a contraction on
$L_{m}$, and both suprema are finite uniformly in $j$ by linear growth and \eqref{eq:Emoment}, so that
\[
\frac{\big\|\delta Y^{j}_{s,t}\big\|_{L_{m}}}{|t-s|^{\gamma}}
\ \le\ \frac{2\sup_{r}\|Y^{j}_{r}\|_{L_{m}}}{h_{\star}^{\gamma}},
\qquad
\frac{\big\|\E_{s}\delta Y^{j\prime}_{s,t}\big\|_{L_{m}}}{|t-s|^{\beta_{\star}}}
\ \le\ \frac{2\sup_{r}\|Y^{j\prime}_{r}\|_{L_{m}}}{h_{\star}^{\beta_{\star}}},
\]
and, $R^{Y}$ being $\delta Y-Y'\delta Z$ and $|\delta Z_{s,t}|\le\|\RZ\|_{\gamma}T^{\gamma}$,
\[
\frac{\big\|\E_{s}R^{Y^{j}}_{s,t}\big\|_{L_{m}}}{|t-s|^{\gamma+\beta_{\star}}}
\ \le\ \frac{2\sup_{r}\|Y^{j}_{r}\|_{L_{m}}
+\|\RZ\|_{\gamma}T^{\gamma}\,2\sup_{r}\|Y^{j\prime}_{r}\|_{L_{m}}}
{h_{\star}^{\gamma+\beta_{\star}}} .
\]
Each denominator is at least $h_{\star}$ raised to the entry's exponent, $\gamma$, $0$,
$\gamma+\beta_{\star}$ and $\beta_{\star}$ respectively, so that what we pick up is the single factor
$h_{\star}^{-(\gamma+\beta_{\star})}$ and no $\vee1$ is needed: $h_{\star}\le c_{0}^{-1/\gamma}\le1$ by
\eqref{eq:Ebookkeep}, since $c_{0}\ge1$ and $\Lambda_{\RZ}\ge1$, so that the four factors are ordered
by their exponents.

\cite[Thm.~3.4]{FHL} produces $\int(Y,Y')\d\RZ$
as the limit in probability, uniform in time, of the Riemann sums
$\sum\big(Y_{u}\delta Z_{u,v}+Y'_{u}\ZZ_{u,v}\big)$. That characterisation makes the integral independent of both the indices and the
filtration at which it was built. Step~\ref{step:Ereg} builds it at $(\gamma,\bar\beta;m,\infty)$ on
the ambient filtration of the $j$-th truncated system, whose germ is subtracted in \eqref{eq:davie},
while the construction here runs at $(\gamma,\beta_{\star};m,m)$ on $\mathcal G^{j}$. The sewing
hypotheses hold in both cases, and the Riemann sums are the same random variables either way, so
the two limits in probability agree almost surely at each $t$ and, both being continuous, up to
indistinguishability.

The first two quantities pass by continuity of the $L_{m}$ norm.
At each fixed $(s,t)$ we have $\delta Y^{j}_{s,t}\to\delta Y_{s,t}$
and $Y^{j\prime}_{s}\to Y'_{s}$ in $L_{m}$, so that
\[
\big\|\delta Y_{s,t}\big\|_{L_{m}}=\lim_{j}\big\|\delta Y^{j}_{s,t}\big\|_{L_{m}},
\qquad\text{hence}\qquad
\big\|\delta Y\big\|_{\gamma;m,m}
=\sup_{s<t}\frac{\|\delta Y_{s,t}\|_{L_{m}}}{|t-s|^{\gamma}}
\ \le\ \limsup_{j}\big\|\delta Y^{j}\big\|_{\gamma;m,m} ,
\]
and similarly
$\sup_{r\le T}\|Y'_{r}\|_{L_{m}}\le\limsup_{j}\sup_{r\le T}\|Y^{j\prime}_{r}\|_{L_{m}}$. The last two
quantities are conditional, and each $j$ has its own conditioning $\sigma$-field:
Step~\ref{proof:Eproof:step2:step1}\ref{proof:Eproof:step2:step1:1b} conditions on $\mathcal G^{j}_{s}$ and the limit entry on $\mathcal G_{s}$, and
$L_{m}$-convergence transfers nothing between them, so that they pass instead by duality.
Let $\Psi_{s}$ denote the family of bounded cylinder functionals,
dense in $L_{m^{*}}(\mathcal G_{s})$ with $\tfrac1m+\tfrac1{m^{*}}=1$, and let $\psi^{j}$ denote the
same functional of $(\widetilde X^{j},\widetilde B^{j})$. Then
\[
\big|\E\big[R^{Y}_{s,t}\psi\big]\big|
=\lim_{j}\big|\E\big[\E\big[R^{Y^{j}}_{s,t}\mid\mathcal G^{j}_{s}\big]\,\psi^{j}\big]\big|
\ \le\ \limsup_{j}\big\|\E\big[R^{Y^{j}}_{s,t}\mid\mathcal G^{j}_{s}\big]\big\|_{L_{m}}\,
\big\|\psi\big\|_{L_{m^{*}}},
\]
Here $R^{Y^{j}}_{s,t}\to R^{Y}_{s,t}$ in $L_{m}$ while $\psi^{j}\to\psi$ boundedly and almost
surely, so that $\|\psi^{j}\|_{L_{m^{*}}}\to\|\psi\|_{L_{m^{*}}}$. Taking the supremum over $\Psi_{s}$
returns
\[
\big\|\E\big[R^{Y}_{s,t}\mid\mathcal G_{s}\big]\big\|_{L_{m}}
\ \le\ \limsup_{j}\big\|\E\big[R^{Y^{j}}_{s,t}\mid\mathcal G^{j}_{s}\big]\big\|_{L_{m}} ,
\]
and the same computation with $\delta Y^{j\prime}_{s,t}$ in place of
$R^{Y^{j}}_{s,t}$ does the fourth. Collecting the four estimates yields the following bound,
\[
\Gamma^{\gamma,\beta_{\star};m,m}\big(Y,Y';[0,T]\big)
\ \le\ \limsup_{j}\Gamma^{\gamma,\beta_{\star};m,m}\big(Y^{j},Y^{j\prime};[0,T]\big)\ <\ \infty ,
\]
and with it, the quantitative half of the membership \cite[Def.~4.12]{FHL}, the third and fourth components
just bounded being exactly the seminorms of that definition's remainder condition,
$(s,t)\mapsto\E_{s}R^{Y}_{s,t}$ in $C^{\gamma+\beta_{\star}}_{2}L_{m}$, and of its condition~(d$'$),
$(s,t)\mapsto\E_{s}\delta Y'_{s,t}$ in $C^{\beta_{\star}}_{2}L_{m}$. Of the remaining conditions, progressive measurability for $(\mathcal G_{t})$ and continuity of $r\mapsto Y_{r}$ into $L_{m}$ are
immediate from step~(b) and \eqref{eq:Emoment}. One condition is not verified here in full: membership in $C_{2}L_{m}$ in the sense of \cite[Def.~2.5]{FHL} also requires $(s,t)\mapsto\E_{s}R^{Y}_{s,t}$ to be
continuous into $L_{m}$. Continuity in $t$ at fixed $s$ follows from the bound on the first
entry and the contraction property of $\E_{s}$.

\emph{From the left it holds.} Let $s_{j}\uparrow s$ and $t_{j}\to t$ with $s_{j}\le t_{j}$. Since
$\E_{s_{j}}$ is an $L_{m}$-contraction,
\[
\big\|\E_{s_{j}}R^{Y}_{s_{j},t_{j}}-\E_{s}R^{Y}_{s,t}\big\|_{L_{m}}
\ \le\ \big\|R^{Y}_{s_{j},t_{j}}-R^{Y}_{s,t}\big\|_{L_{m}}
\ +\ \big\|\E_{s_{j}}R^{Y}_{s,t}-\E_{s}R^{Y}_{s,t}\big\|_{L_{m}} .
\]
For the first summand, $R^{Y}_{u,v}=\delta Y_{u,v}-Y'_{u}\delta Z_{u,v}$ and
\[
\big\|R^{Y}_{s_{j},t_{j}}-R^{Y}_{s,t}\big\|_{L_{m}}
\le\big\|Y_{t_{j}}-Y_{t}\big\|_{L_{m}}+\big\|Y_{s_{j}}-Y_{s}\big\|_{L_{m}}
+\big\|Y'_{s_{j}}-Y'_{s}\big\|_{L_{m}}\big|\delta Z_{s_{j},t_{j}}\big|
+\big\|Y'_{s}\big\|_{L_{m}}\big|\delta Z_{s_{j},t_{j}}-\delta Z_{s,t}\big| ,
\]
and all four vanish: the $L_{m}$-continuity of $r\mapsto Y_{r}$ is the one just recorded, that of
$r\mapsto Y'_{r}$ is the unconditional bound $\|\delta Y'_{u,w}\|_{L_{m}}\lesssim
|w-u|^{\min\{\gamma,\beta'\}}$ established under condition~(d) below, which does not use the present
step, $\sup_{r}\|Y'_{r}\|_{L_{m}}<\infty$ is the second entry of $\Gamma$, and $\delta Z$ is
deterministic and continuous. For the second summand, $R^{Y}_{s,t}\in L_{m}$ by those same two
bounds, and $\mathcal G_{s_{j}}\uparrow\bigvee_{r<s}\mathcal G_{r}=\mathcal G_{s-}=\mathcal G_{s}$,
the last equality because $(\mathcal G_{t})$ is generated by the continuous paths
$(\widetilde X,\widetilde B)$, as in the proof of Lemma~\ref{lem:MPborel}; Doob's martingale convergence theorem in
$L_{m}$, available at $m\ge2>1$, then gives $\E_{s_{j}}R^{Y}_{s,t}\to\E_{s}R^{Y}_{s,t}$ in $L_{m}$.
So $(s,t)\mapsto\E_{s}R^{Y}_{s,t}$ is continuous in $t$ and left-continuous in $s$, into $L_{m}$.

\emph{From the right it is open, but nothing below needs it.} Approaching $s$ from above replaces
$\mathcal G_{s}$ by $\mathcal G_{s+}$, which would need $(\mathcal G_{t})$ right-continuous.
The clause is never used. Step~\ref{step:Etight} applies \cite[Lem.~4.20]{FHL}, proved there from
\cite[Thm.~3.4]{FHL} and \cite[Cor.~2.10]{FHL}, and between them those ask only for progressive
measurability, the finiteness of the four quantities bounded above, which are \cite{FHL}'s
$\Gamma^{\gamma,\beta_{\star};m,m}(Y,Y';[0,T])$, almost sure continuity of $t\mapsto A_{s,t}$ at
fixed $s$, and convergence at each fixed $s$ in $L_{m}$. The continuity of \cite[Def.~2.5]{FHL}
is not among them. The clause enters only through the membership
\cite[Def.~4.2(b)]{FHL} recorded at the end of Step~\ref{step:Etight}, which compares our notion of solution with
\cite{FHL}'s and is used nowhere in this paper.

 At
$(\bar\alpha,\bar\alpha')=(\gamma,\beta_{\star})$ and $\mfn =m$, the conditions (a)--(c) of \cite[Def.~3.1]{FHL} are the progressive
measurability and the first and third entries of $\Gamma$, and only its condition~(d),
$Y'\in C^{\beta_{\star}}L_{m,m}$ with $\sup_{r}\|Y'_{r}\|_{L_{m}}<\infty$, is left. The second half of
(d) is the second entry. The first half is the unconditional half of \eqref{eq:Gincrement}. It
yields the following bound, at order $\min\{\gamma,\beta'\}\ge\beta_{\star}$, with
$\mathsf G_{t}:=Df_{t}f_{t}+f'_{t}$,
\[
\big\|\delta Y'_{u,w}\big\|_{L_{m}}=\big\|\delta[\mathsf G(X)]_{u,w}\big\|_{L_{m}}
\ \le\ \big(C_{\mathrm{lip}}\mathcal H+C\,\Lambda_{\RZ}\,\mathcal N_{s,t}\big)\,|w-u|^{\min\{\gamma,\beta'\}}.
\]
The factor $\Lambda_{\RZ}$ being the one \eqref{eq:Gincrement} attaches to its first temporal piece and
which survives the collapse of the three orders to $\min\{\gamma,\beta'\}$, since $\Lambda_{\RZ}\ge1$
and $|w-u|\le h_{\star}\le1$. The bound is independent of $(k,R)$ by \eqref{eq:HJfinal} and
\eqref{eq:Emoment}, and it holds for the untruncated field as well, its two halves being
\eqref{eq:DGGLip} and \eqref{eq:Gtemporal} at constants Step~\ref{step:Ereg} makes uniform in $R$. It
passes to the limit at each fixed pair by the $L_{m}$-continuity used for the first entry, no
conditioning being involved, and pairs with $|t-s|>h_{\star}$ are covered as before, at the exponent
$\beta_{\star}$. Since $\|\cdot\|_{m,m}=\|\cdot\|_{L_{m}}$ by \cite[Prop.~2.3(i)]{FHL}, this is
$\|\delta Y'\|_{\beta_{\star};m,m}<\infty$, whence also the $L_{m}$-continuity of $r\mapsto Y'_{r}$
that $C^{\beta_{\star}}L_{m,m}$ requires. Condition~(d) holds. With it \cite[Def.~4.2(b)]{FHL} holds at the limit, except for the
right-continuity in $s$ of $(s,t)\mapsto\E_{s}R^{Y}_{s,t}$, which is open, as noted above, and which
nothing in this paper uses.

All the $\widetilde X^{j}$ live on the one Skorokhod space, which is the common
measure space Lemma~\ref{lem:roughstab} needs. Only the filtration depends on
$j$. $\widetilde B^{j}$ is a
Brownian motion for $\mathcal G^{j}$, with no reason to be one for the $\sigma$-field generated by the
whole family, while enlarging the filtration can only enlarge the third quantity, by conditional
Jensen.
The descent of
Step~\ref{proof:Eproof:step2:step1}\ref{proof:Eproof:step2:step1:1b} to the raw filtration is simply
conditional Jensen. This is performed on the original basis of $X^{k_{j},R_{j}}$, where an ambient
filtration is present and $\mathcal G^{j}$ is the filtration the pair generates. On
the Skorokhod space $\mathcal G^{j}$ is itself the ambient one. Lemma~\ref{lem:lawdet} then carries
the descended bound to $\widetilde X^{j}$: the third and fourth quantities are conditional $L_{m}$
norms of continuous functionals of the pair, the first two plain $L_{m}$ norms of
such. Denoting by
$\mathcal F_{s}$ the ambient $\sigma$-field, so that
$\mathcal G^{j}_{s}\subseteq\mathcal F_{s}$, the tower property and the contraction of conditional
expectation on $L_{m}$ yield
\[
\big\|\E\big[\xi\mid\mathcal G^{j}_{s}\big]\big\|_{L_{m}}
\ =\ \big\|\E\big[\E[\xi\mid\mathcal F_{s}]\mid\mathcal G^{j}_{s}\big]\big\|_{L_{m}}
\ \le\ \big\|\E\big[\xi\mid\mathcal F_{s}\big]\big\|_{L_{m}},
\qquad \xi\in L_{m} .
\]
Hence
Lemma~\ref{lem:roughstab} applies and yields
\[
\sup_{t\le T}\Big|\int_{0}^{t}(Y^{j},Y^{j\prime})\d\RZ-\int_{0}^{t}(Y,Y')\d\RZ\Big|
\ \longrightarrow\ 0
\qquad\text{in }L_{m},\ \text{ at every }m\in(1/\gamma,p) .
\]
Passing to the limit in the integral form of the $j$-th equation gives the same form
for $(X,B)$. The endpoint $m=p$ is not needed. What remains of the solution property is the two
moduli, we establish next.

The convergences established yield $J^{j}_{s,t}\to J_{s,t}$ in $L_{m}$ at each fixed pair $s\le t$,
$J$ being the same Borel functional \eqref{eq:davie} of $(X,B)$ on either space. That
is at every $m<p$. At $m=p$ we have convergence in probability together with
$\sup_{j}\|J^{j}_{s,t}\|_{L_{p}}<\infty$, which \eqref{eq:davie} gives at every pair, uniformly in
$j$, from \eqref{eq:Egrowth} and \eqref{eq:MProughgrowth} at constants independent of $(k,R)$,
Burkholder--Davis--Gundy and \eqref{eq:Emoment} at exponent $p$.  Fatou along an almost surely convergent subsequence, together with that bound, gives
$J_{s,t}\in L_{m}(\Prob)$ at every pair, the last condition of
Definition~\ref{def:davie}. The
second modulus then passes the same way from $\mathcal J\le C_{3}\Lambda^{2}\mathcal N$ at \eqref{eq:HJfinal}. The first does not follow from it. Indeed, an unconditional $L_{m}$ bound yields
only
\[
\big\|\E_{s}J_{s,t}\big\|_{L_{m}}\ \le\ \big\|J_{s,t}\big\|_{L_{m}}\ =\ o\big(|t-s|^{1/2}\big).
\]
It comes instead from \eqref{eq:EscaleCond}, and it passes to the limit by duality.
We fix $s\le t$ and let $\Psi_{s}$ denote the test family of bounded cylinder
functionals $\psi=g\big(\varpi_{r_{1}},\dots,\varpi_{r_{l}}\big)$ with $r_{i}\le s$ and $g\in C_{b}$,
evaluated at the canonical path of the pair, dense in $L_{m^{*}}(\mathcal G_{s})$ for every measure on
$\mathcal C$. Again $\psi^{j}$
denotes the same functional evaluated at $(\widetilde X^{j},\widetilde B^{j})$, which
\eqref{eq:EscaleCond} for the $j$-th approximant may be tested against, that estimate conditioning on
the $j$-th generated filtration. It is $\mathcal G^{j}_{s}$-measurable and bounded, and
$\psi^{j}\to\psi$ almost surely and boundedly, so that
$\E[J^{j}_{s,t}\psi^{j}]\to\E[J_{s,t}\psi]$ and
$\|\psi^{j}\|_{L_{m^{*}}}\to\|\psi\|_{L_{m^{*}}}$. Applying the same duality to \eqref{eq:EscaleCond}, we obtain the following bound,
\[
\big|\E\big[J_{s,t}\psi\big]\big|
\ =\ \lim_{j}\big|\E\big[J^{j}_{s,t}\psi^{j}\big]\big|
\ =\ \lim_{j}\big|\E\big[\E\big[J^{j}_{s,t}\mid\mathcal G^{j}_{s}\big]\,\psi^{j}\big]\big|
\ \le\ C_{3}\Lambda^{3}\mathcal N\,|t-s|^{\epsilon_{1}}\,\big\|\psi\big\|_{L_{m^{*}}} ,
\]
with $\mathcal N$ bounded uniformly in $j$ by \eqref{eq:Emoment}. Taking the
supremum over $\Psi_{s}$ yields
\[
\big\|\E\big[J_{s,t}\mid\mathcal G_{s}\big]\big\|_{L_{m}}
\ \le\ C\,|t-s|^{\epsilon_{1}}\ =\ o\big(|t-s|\big) ,
\]
the first modulus on the filtration $(\mathcal G_{s})$, generated by the pair
$(\widetilde X,\widetilde B)$. Every ingredient is $\mathcal G_{s}$-measurable, $J$ being a Borel
functional of the pair by step~\ref{stp:Et-b}, and $\Psi_{s}$ is dense in $L_{m^{*}}(\mathcal G_{s})$, so that taking the supremum over it returns nothing about a larger $\sigma$-field.

Finally, \eqref{eq:Emoment} survives by Fatou, and Proposition~\ref{prop:MPequiv}\ref{mpe1}, at the
exponent shift it carries, turns its law into a solution of the rough martingale problem.
\hfill$\square$
\subsection{An abstract Lyapunov form}

\begin{remark}\label{rem:Elyapunov}
\eqref{eq:Vlyap}--\eqref{eq:Vbracket} say that $V$ is a Lyapunov function for the whole rough
generator and not only for the It\^o part, and that is what the rough Gr\"onwall of (a$_{2}$) uses: given \eqref{eq:EscaleNat}, the passage from \eqref{eq:Egronwallineq} to \eqref{eq:Epathwise} uses
nothing about $b,\sigma,f$ beyond
\[
|\mathcal L_{t}V|+|\mathcal GV|+|\mathcal G^{(2)}V|\ \le\ C_{1}V,
\qquad
|DV\cdot\sigma|^{2}\ \le\ C_{1}V^{2} .
\]
It does not follow that Theorem~\ref{thm:weakex} holds with \eqref{eq:Egrowth} replaced by an abstract
Lyapunov condition.

\emph{(i) }The drift is used directly, in $L_{m}$ along the path. Piece $(\beta)$
of the $\epsilon_{1}$-estimate produces
$\tau^{1/2}\big(\int_{u}^{w}\|b(r,X_{r})\|^{2}_{L_{m}}\dr\big)^{1/2}$, the same quantity reappearing in
the Davie expansion that closes $\mathcal H$, and \eqref{eq:Egrowth}, and nothing else, bounds it by
$\sqrt2C_{\mathrm{gr}}\mathcal N_{s,t}\tau$. A Lyapunov bound $|\mathcal L_{t}V|\le C_{V}V$ controls
$\sup_{t}V(X_{t})$ and says nothing about $\|b(r,X_{r})\|_{L_{m}}$, the quantity an increment
estimate needs. For the dissipative $b(x)=-x|x|^{2}$ we obtain
$\|b(r,X_{r})\|_{L_{m}}=\|X_{r}\|^{3}_{L_{3m}}$, neither linear in $\mathcal N_{s,t}$ nor at exponent
$m$, so that the inequality for $\mathcal H$ is cubic and at three
times the exponent and the absorption of \ref{proof:Eproof:step2:step1:1c} does not close. 

\emph{(ii) }Under branch \ref{Ecurvdiag}
there is nothing left for a Lyapunov condition to add on $\sigma$ either, \eqref{eq:Eacoord} already
implying the $\sigma$-half of \eqref{eq:Egrowth}: summing it over $i$ and using
$\sum_{i}(1+|x_{i}|)\le n(1+|x|)$ yields $\operatorname{tr}a\le C_{\mathrm{gr}}n(1+|x|)^{2}$, that is,
$\|\sigma\|\le\sqrt{C_{\mathrm{gr}}n}\,(1+|x|)$, at a factor $\sqrt n$ this section's constants carry
anyway. Under \ref{Ecurvaffine} the $\sigma$-half is used directly and \eqref{eq:Eacoord} is not required
at all. What a genuine generalization would have to replace is therefore not the growth conditions but
the bounds \eqref{eq:Vjet}, special to $V=(1+|x|^{2})^{1/2}$ and false for
e.g.~$V=(1+|x|^{2})^{3/2}$. That is Open Problem~\ref{op:lyapunov}.
\end{remark}

\begin{openproblem}[An abstract Lyapunov form of Theorem~\ref{thm:weakex}]\label{op:lyapunov}
Replace $1+|x|$ by a coercive $V$ in the growth bounds,
\[
|b|+\Frob{\sigma}\ \lesssim\ V,
\qquad
a_{ii}\ \lesssim\ \big(1+|x_{i}|\big)V,
\qquad
|f|+|f'|\ \lesssim\ V ,
\]
and find what replaces \eqref{eq:Vjet}. The difficulty is entirely in \eqref{eq:Vjet}:
$\|DV\|_{\infty}<\infty$ alone forces $V$ to grow at most linearly, so no Lyapunov
function of superlinear growth satisfies \eqref{eq:Vjet}. For instance $V(x)=(1+|x|^{2})^{3/2}$ with $b(x)=-x|x|^{2}$ satisfies
$\mathcal L_{t}V\le C_{V}V$ and $|b|\le V$, and Theorem~\ref{thm:weakex} does not apply.
\end{openproblem}

\section{From pathwise uniqueness to uniqueness in law}\label{sec:YWtransfer}

The Yamada--Watanabe implication asserts that weak existence together with pathwise uniqueness yields
a strong solution and uniqueness in law. For rough stochastic differential equations we are not aware
of it having been established. We establish it in this section, and we begin by naming the obstacle.

For a classical It\^o equation the phrase ``$X$ solves the equation'' denotes a pathwise identity
between $X$, the driving Brownian motion $B$ and a stochastic integral. Let $B$ remain a Brownian
motion for a larger filtration, and let the integrand be adapted to the smaller one. Then the identity
persists, i.e.~the notion is insensitive to the filtration. The Yamada--Watanabe construction
glues two solutions over a common Brownian motion and applies pathwise uniqueness on the glued space,
using exactly that insensitivity, in the direction of enlargement.

The notion of Definition~\ref{def:weaksol} is not a pathwise identity insensitive to
enlargement of the filtration: it is a pair of bounds on conditional moments of the Davie
remainder,
\begin{equation*}
\big\|\E[J_{s,t}\mid\mathcal F_{s}]\big\|_{L_{m}(\Prob)}=o(|t-s|),
\qquad
\big\|\E\big[|J_{s,t}|^{m}\mid\mathcal F_{s}\big]^{1/m}\big\|_{L_{m}(\Prob)}=o(|t-s|^{1/2}),
\end{equation*}
that is, the two moduli of \eqref{eq:daviemoduli}. The rough-specific ingredients are
Lemmas~\ref{lem:filtdescend}, \ref{lem:lawdet} and \ref{lem:immersion}. The argument is otherwise set in the framework of Kurtz's abstract Yamada--Watanabe--Engelbert theorem \cite{Kurtz14}.

\subsection{Immersion}\label{subsec:immersion}

Filtrations are $\Prob$-complete, except where a filtration is said to be raw, as at
Definition~\ref{def:roughMP}, on which $\Gamma_{\nu}$ below is posed and which
Remark~\ref{rem:MPraw} requires to stay raw, and at
Definition~\ref{def:compatible}. Right-continuity is not assumed.

\begin{definition}[Immersion; hypothesis (H)]\label{def:immersion}
Let $(\mathcal F_{t})_{t\in[0,T]}\subseteq(\mathcal G_{t})_{t\in[0,T]}$ denote filtrations. We say
$(\mathcal F_{t})$ is immersed in $(\mathcal G_{t})$, and write
$(\mathcal F_{t})\hookrightarrow(\mathcal G_{t})$, if every $(\mathcal F_{t})$-martingale is a
$(\mathcal G_{t})$-martingale.
\end{definition}

\begin{lemma}[Equivalent formulations]\label{lem:immersionchar}
For $(\mathcal F_{t})\subseteq(\mathcal G_{t})$ the following are equivalent.
\begin{enumerate}[label=\normalfont(\roman*),leftmargin=2.2em]
\item\label{imm1} $(\mathcal F_{t})\hookrightarrow(\mathcal G_{t})$.
\item\label{imm2} $\E[\xi\mid\mathcal G_{t}]=\E[\xi\mid\mathcal F_{t}]$ a.s., for every $t$ and every
$\xi\in L^{1}(\mathcal F_{T})$.
\item\label{imm3} $\mathcal F_{T}$ and $\mathcal G_{t}$ are conditionally independent given
$\mathcal F_{t}$, for every $t$.
\end{enumerate}
\end{lemma}

\begin{proof}
\ref{imm1}$\Rightarrow$\ref{imm2}: let $\xi\in L^{1}(\mathcal F_{T})$ and denote by $M$ the martingale
$M_{t}=\E[\xi\mid\mathcal F_{t}]$. It is uniformly integrable with $M_{T}=\xi$, and by \ref{imm1} it is
a $(\mathcal G_{t})$-martingale. We thus obtain the following identity,
\begin{equation*}
\E\big[\xi\mid\mathcal F_{t}\big]\ =\ M_{t}\ =\ \E\big[M_{T}\mid\mathcal G_{t}\big]
\ =\ \E\big[\xi\mid\mathcal G_{t}\big] .
\end{equation*}
\ref{imm2}$\Rightarrow$\ref{imm1} is immediate, and
\ref{imm2}$\Leftrightarrow$\ref{imm3} is the standard characterisation of conditional independence.
Indeed, letting $\xi$ range over bounded $\mathcal F_{T}$-measurable variables and testing against
bounded $\mathcal G_{t}$-measurable ones, \ref{imm2} reads
\begin{equation*}
\E\big[\xi\,\eta\big]\ =\ \E\big[\E[\xi\mid\mathcal F_{t}]\,\eta\big],
\qquad \xi\in L^{\infty}(\mathcal F_{T}),\quad \eta\in L^{\infty}(\mathcal G_{t}) .
\end{equation*}
See e.g.~\cite[Ch.~VI]{Protter} or \cite{BremaudYor}.
\end{proof}

\begin{lemma}[The solution notion is immersion-invariant]\label{lem:immersion}
Let $X$ denote an $L_{m}$-conditional solution of \eqref{eq:rsde} on the stochastic basis
$(\Omega,\mathcal A,(\mathcal F_{t}),\Prob)$ with Brownian motion $B$, and let
$(\mathcal G_{t})\supseteq(\mathcal F_{t})$ denote a filtration such that
$(\mathcal F_{t})\hookrightarrow(\mathcal G_{t})$. Then $B$ is a $(\mathcal G_{t})$-Brownian motion,
$X$ is an $L_{m}$-conditional solution of \eqref{eq:rsde} on
$(\Omega,\mathcal A,(\mathcal G_{t}),\Prob)$
with the same Brownian motion and the same rough path, and the two moduli in \eqref{eq:daviemoduli}
may be taken to be \emph{the same functions}. The same holds for the $L_{m,\infty}$ form
\eqref{eq:daviemoduliinfty}.
\end{lemma}

\begin{proof}
\emph{(a) $B$ is a $(\mathcal G_{t})$-Brownian motion.}\stepnum{(a)}{stp:imm-a} No appeal to
L\'evy's characterisation is needed. Let $s\le t$ and
let $h$ denote a bounded measurable function. Then $h(\delta B_{s,t})\in L^{1}(\mathcal F_{T})$, so
Lemma~\ref{lem:immersionchar}\ref{imm2} yields the following identity,
\begin{equation}\label{eq:immersionBM}
\E\big[h(\delta B_{s,t})\mid\mathcal G_{s}\big]\ =\ \E\big[h(\delta B_{s,t})\mid\mathcal F_{s}\big]
\ =\ \E\big[h(\delta B_{s,t})\big],
\end{equation}
the second equality because $B$ is an $(\mathcal F_{t})$-Brownian motion. Thus $\delta B_{s,t}$ is
independent of $\mathcal G_{s}$ with the right law. Here $B$ is continuous with $B_{0}=0$, so
\eqref{eq:immersionBM} is the definition.

\emph{(b) The It\^o integrals are unchanged.}\stepnum{(b)}{stp:imm-b} The integrand $r\mapsto\sigma_{r}(X_{r})$ is continuous
and $(\mathcal F_{r})$-adapted, hence also $(\mathcal G_{r})$-adapted. Let an integrand be
adapted to both filtrations and let the integrator be a continuous local martingale for both. Then the
two stochastic integrals coincide. Indeed, both are the u.c.p.\ limits, along any sequence of
partitions $\pi$ of vanishing mesh, of the same left-point Riemann sums
\begin{equation}\label{eq:immersionriemann}
\sum_{[u,v]\in\pi}\sigma_{u}(X_{u})\,\delta B_{u,v},
\end{equation}
and limits in probability are a.s.\ unique. Therefore the Davie remainder
\[
J_{s,t}=\delta X_{s,t}-\int_{s}^{t}\!b_{r}(X_{r})\dr-\int_{s}^{t}\!\sigma_{r}(X_{r})\,\mathrm dB_{r}
-f_{s}(X_{s})\,\delta Z_{s,t}-\big(Df_{s}f_{s}+f'_{s}\big)(X_{s})\,\ZZ_{s,t}
\]
is the same random field whether computed on $(\mathcal F_{t})$ or on $(\mathcal G_{t})$.
No property of the rough path is used here: $\delta Z$ and $\ZZ$ are deterministic.

\emph{(c) The conditional bounds transfer with equality.}\stepnum{(c)}{stp:imm-c} The remainder $J_{s,t}$ is
$\mathcal F_{t}$-measurable, and so is $|J_{s,t}|^{m}$. Both are integrable, the first by the
$L_{m}$-bound of the solution notion and the second by definition. Applying
Lemma~\ref{lem:immersionchar}\ref{imm2} at time $s$ with
$\xi=J_{s,t}$ and with $\xi=|J_{s,t}|^{m}$, we obtain the following identities,
\begin{equation}\label{eq:immersionidentity}
\E\big[J_{s,t}\mid\mathcal G_{s}\big]=\E\big[J_{s,t}\mid\mathcal F_{s}\big],
\qquad
\E\big[|J_{s,t}|^{m}\mid\mathcal G_{s}\big]=\E\big[|J_{s,t}|^{m}\mid\mathcal F_{s}\big]
\qquad\text{a.s.}
\end{equation}
Taking $L_{m}(\Prob)$-norms in \eqref{eq:immersionidentity} yields
\begin{align}
\big\|\E[J_{s,t}\mid\mathcal G_{s}]\big\|_{L_{m}(\Prob)}
&=\big\|\E[J_{s,t}\mid\mathcal F_{s}]\big\|_{L_{m}(\Prob)}=o(|t-s|),\nonumber\\
\big\|\E\big[|J_{s,t}|^{m}\mid\mathcal G_{s}\big]^{1/m}\big\|_{L_{m}(\Prob)}
&=\big\|\E\big[|J_{s,t}|^{m}\mid\mathcal F_{s}\big]^{1/m}\big\|_{L_{m}(\Prob)}
=o(|t-s|^{1/2}),\label{eq:immersionnorms}
\end{align}
that is, \eqref{eq:daviemoduli} on $(\mathcal G_{t})$ with the same moduli. Taking
$L_{\infty}(\Omega)$-norms instead reproduces \eqref{eq:daviemoduliinfty} on $(\mathcal G_{t})$, again
with the same moduli. The norm is applied after the identity, which is why the lemma is
indifferent to which of the two forms is used. Together with (a) and (b), $X$ is an
$L_{m}$-conditional solution on $(\mathcal G_{t})$.
\end{proof}

\subsection{Compatibility, and the glued measure}\label{subsec:gluing}

\begin{definition}[Compatibility, after {\cite[Def.~2.1]{Kurtz14}}]\label{def:compatible}
Let $X$ and $B$ denote continuous processes on a common probability space, and let
$\mathcal F^{X,B}_{t}=\sigma(X_{r},B_{r}:r\le t)$ and $\mathcal F^{B}_{t}=\sigma(B_{r}:r\le t)$ denote the generated filtrations. We say
$X$ is compatible with $B$ if
\begin{equation}\label{eq:compat}
\E\big[h(B)\mid\mathcal F^{X,B}_{t}\big]=\E\big[h(B)\mid\mathcal F^{B}_{t}\big]
\qquad\text{a.s., for every }t\in[0,T]\text{ and every bounded measurable }h .
\end{equation}
Equivalently, $\sigma(B)$ and $\mathcal F^{X,B}_{t}$ are conditionally independent given
$\mathcal F^{B}_{t}$.
\end{definition}

\cite[Def.~2.1]{Kurtz14} says temporally compatible and completes the generated filtrations.
We take them raw, \eqref{eq:compat} being an a.s.\ identity.

\begin{lemma}[Every weak solution is compatible]\label{lem:solcompat}
On a stochastic basis $(\Omega,\mathcal A,(\mathcal F_{t}),\Prob)$ let $X$ and $B$ be
continuous and $(\mathcal F_{t})$-adapted, with $B$ an $(\mathcal F_{t})$-Brownian motion in
$\R^{d}$. This holds at every weak solution of \eqref{eq:rsde} in the sense of
Definition~\ref{def:weaksol}. Then $X$ is compatible with $B$. Moreover $X$ is compatible with the
input $Y:=(X_{0},B)$ of \S\ref{subsec:YWthm}, the pair of \cite[Def.~2.1]{Kurtz14}.
Denoting
$\mathcal F^{X,Y}_{t}=\sigma(X_{r},X_{0},B_{r}:r\le t)$ and
$\mathcal F^{Y}_{t}=\sigma(X_{0})\vee\mathcal F^{B}_{t}$, we have
\begin{equation}\label{eq:compatinput}
\E\big[h(Y)\mid\mathcal F^{X,Y}_{t}\big]=\E\big[h(Y)\mid\mathcal F^{Y}_{t}\big]
\qquad\text{a.s., for every }t\in[0,T]\text{ and every bounded measurable }h .
\end{equation}
\end{lemma}

\begin{proof}
By hypothesis $B$ is an $(\mathcal F_{t})$-Brownian motion. Hence the increment process
$\beta^{t}:=(B_{t+u}-B_{t})_{u\in[0,T-t]}$ is independent of $\mathcal F_{t}$, and therefore of
$\mathcal F^{X,B}_{t}\subseteq\mathcal F_{t}$. We write a bounded measurable $h(B)$ as
$\tilde h(B|_{[0,t]},\beta^{t})$. Here $B|_{[0,t]}$ is $\mathcal F^{B}_{t}$-measurable while
$\beta^{t}\perp\mathcal F^{X,B}_{t}$, so freezing the first argument yields
\begin{equation}\label{eq:compatfreeze}
\E\big[h(B)\mid\mathcal F^{X,B}_{t}\big]=H_{t}\big(B|_{[0,t]}\big),
\qquad H_{t}(w):=\E\big[\tilde h(w,\beta^{t})\big] .
\end{equation}
The right-hand side of \eqref{eq:compatfreeze} is $\mathcal F^{B}_{t}$-measurable. Conditioning it
further on $\mathcal F^{B}_{t}\subseteq\mathcal F^{X,B}_{t}$ leaves it unchanged, and by the tower
property we obtain
\begin{equation*}
\E\big[h(B)\mid\mathcal F^{B}_{t}\big]
=\E\Big[\E\big[h(B)\mid\mathcal F^{X,B}_{t}\big]\ \Big|\ \mathcal F^{B}_{t}\Big]
=H_{t}\big(B|_{[0,t]}\big)
=\E\big[h(B)\mid\mathcal F^{X,B}_{t}\big],
\end{equation*}
i.e.~\eqref{eq:compat}.

\eqref{eq:compatfreeze} yields more than \eqref{eq:compat}, and the surplus is needed in
\eqref{eq:compatinput}. Here, $\E[h_{1}(B)\mid\mathcal F^{X,B}_{t}]$ is
$\mathcal F^{B}_{t}$-measurable for every bounded measurable $h_{1}$. Applying the tower property to
any $\sigma$-field in between, we obtain the following identity,
\begin{equation}\label{eq:compatinter}
\E\big[h_{1}(B)\mid\mathcal H_{t}\big]=\E\big[h_{1}(B)\mid\mathcal F^{B}_{t}\big]
\qquad\text{for every }\mathcal H_{t}\text{ with }
\mathcal F^{B}_{t}\subseteq\mathcal H_{t}\subseteq\mathcal F^{X,B}_{t}.
\end{equation}
Note further that $X_{0}$ is $\mathcal F^{X,B}_{0}$-measurable. This yields the following filtration identities,
\begin{equation}\label{eq:compatfiltid}
\mathcal F^{X,Y}_{t}=\mathcal F^{X,B}_{t},
\qquad
\mathcal F^{B}_{t}\ \subseteq\ \mathcal F^{Y}_{t}=\sigma(X_{0})\vee\mathcal F^{B}_{t}
\ \subseteq\ \mathcal F^{X,B}_{t},
\end{equation}
so that $\mathcal F^{Y}_{t}$ is one such $\mathcal H_{t}$. Folding $X_{0}$
into the conditioning is a step and not a restatement, and \eqref{eq:compatinter}
allows it. Let now $h(x_{0},w)=g(x_{0})h_{1}(w)$, with $g,h_{1}$ denoting bounded measurable functions. The factor
$g(X_{0})$ is measurable for both $\mathcal F^{X,Y}_{t}$ and $\mathcal F^{Y}_{t}$, hence comes out of
both sides of \eqref{eq:compatinput}, and we obtain the following chain,
\begin{equation}\label{eq:compatproduct}
\E\big[g(X_{0})h_{1}(B)\mid\mathcal F^{X,Y}_{t}\big]
=g(X_{0})\,\E\big[h_{1}(B)\mid\mathcal F^{X,B}_{t}\big]
=g(X_{0})\,\E\big[h_{1}(B)\mid\mathcal F^{Y}_{t}\big]
=\E\big[g(X_{0})h_{1}(B)\mid\mathcal F^{Y}_{t}\big],
\end{equation}
the middle equality being \eqref{eq:compatinter} applied at $\mathcal H_{t}=\mathcal F^{X,B}_{t}$ and at
$\mathcal H_{t}=\mathcal F^{Y}_{t}$. Such products form a multiplicative class generating the product
$\sigma$-field on $\R^{n}\times C([0,T];\R^{d})$, so the functional monotone class theorem extends
\eqref{eq:compatproduct} to \eqref{eq:compatinput} for every bounded measurable $h$.
\end{proof}

\begin{remark}\label{rem:compatfree}
Lemma~\ref{lem:solcompat} uses only that $B$ is an $(\mathcal F_{t})$-Brownian motion, supplied by
\ref{MP0} on the filtration generated by $(X,B)$, while \eqref{eq:compat}, \eqref{eq:compatinput}
constrain only the law of $(X,B)$. Every $\Prob$ of Definition~\ref{def:roughMP} is compatible with
the input.
\end{remark}

\begin{proposition}[The glued measure, and its immersion property]\label{prop:gluing}
For $k=1,2$ let
\[
\big(\Omega^{k},\mathcal A^{k},\Prob^{k},X^{k},B^{k}\big)
\]
denote a probability space carrying continuous processes $X^{k},B^{k}$ such that $X^{k}$
is an $L_{m}$-conditional solution on the raw filtration of $(X^{k},B^{k})$ and \ref{MP0} holds at
exponent $m$, for the same deterministic
rough path $\RZ$, the same coefficients and the same initial law $\nu$. By \ref{MP0}
the process $B^{k}$ has increments independent of the raw field at time $0$, which contains
$\sigma(X^{k}_{0})$, so $\Law_{\Prob^{k}}(X^{k}_{0},B^{k})=\nu\otimes\mathcal W$. This is
\eqref{eq:gluejointlaw} below, and it is what makes the kernels $\mu^{k}$ of \eqref{eq:glue} defined
$\nu\otimes\mathcal W$-almost everywhere. Let
\[
\widehat\Omega:=\R^{n}\times C([0,T];\R^{n})^{2}\times C([0,T];\R^{d}),
\]
with coordinate maps $(x_{0},X^{1},X^{2},B)$; here $\widehat\Omega$ is written with $\R^{n}$ in place
of the state space $D$ for readability, the marginals constructed below being carried by
$C([0,T];D)$ because each $\Prob^{k}$ is, and $D$ being closed. We denote by
$\mu^{k}(\mathrm dx\mid x_{0},w)$ a regular
conditional distribution of $X^{k}$ given $(X^{k}_{0},B^{k})=(x_{0},w)$ under $\Prob^{k}$, and define
\begin{equation}\label{eq:glue}
\Q(\mathrm dx_{0},\mathrm dx^{1},\mathrm dx^{2},\mathrm dw)
:=\nu(\mathrm dx_{0})\,\mu^{1}(\mathrm dx^{1}\mid x_{0},w)\,\mu^{2}(\mathrm dx^{2}\mid x_{0},w)\,
\mathcal W(\mathrm dw),
\end{equation}
$\mathcal W$ denoting Wiener measure on $C([0,T];\R^{d})$. We denote
\[
\mathcal G_{t}:=\sigma\big(x_{0},\,X^{1}_{r},\,X^{2}_{r},\,B_{r}:r\le t\big),
\qquad
\mathcal F^{k}_{t}:=\sigma\big(x_{0},\,X^{k}_{r},\,B_{r}:r\le t\big)
\]
(all augmented, and not right-continuous). Then the following hold under $\Q$:
\begin{enumerate}[label=\normalfont(\alph*),leftmargin=2.4em]
\item\label{gl1} $\Law_{\Q}(X^{k},B)=\Law_{\Prob^{k}}(X^{k},B^{k})$ for $k=1,2$;
\item\label{gl2} $(\mathcal F^{k}_{t})\hookrightarrow(\mathcal G_{t})$ for $k=1,2$;
\item\label{gl3} $B$ is a $(\mathcal G_{t})$-Brownian motion, and $X^{1},X^{2}$ are both
\emph{$L_{m}$-conditional} solutions of \eqref{eq:rsde} on the single stochastic basis
$(\widehat\Omega,\mathcal G_{T},(\mathcal G_{t}),\Q)$, driven by the same $B$, the same
$\RZ$, and with the same initial value $X^{1}_{0}=X^{2}_{0}=x_{0}$, and both satisfy \ref{MP0}.
\end{enumerate}
\end{proposition}

\begin{proof}
\ref{gl1} Integrating out the other copy in \eqref{eq:glue} yields
\begin{equation}\label{eq:gluemarginal}
\Q\big(\mathrm dx_{0},\mathrm dx^{k},\mathrm dw\big)
=\nu(\mathrm dx_{0})\,\mu^{k}(\mathrm dx^{k}\mid x_{0},w)\,\mathcal W(\mathrm dw),
\qquad k=1,2,
\end{equation}
each $\mu^{j}(\cdot\mid x_{0},w)$, $j=1,2$, being a probability measure. The right-hand side of
\eqref{eq:gluemarginal} is
$\Law_{\Prob^{k}}(X^{k}_{0},X^{k},B^{k})$, since the outer factor
$\nu\otimes\mathcal W$ is the joint law of $(X^{k}_{0},B^{k})$ and not merely the product of its
marginals. That identification uses the independence of the initial datum from the
driver, and $\Law_{\Prob^{k}}(B^{k})=\mathcal W$ does not provide this. It is provided by
\ref{MP0}, which asks that $B$ is a Brownian motion for the canonical filtration
$(\mathcal G_{t})$ of the pair $(X,B)$. $B_{0}=0$, so the whole path $B$ then consists of
increments independent of $\mathcal G_{0}$, and $\sigma(X_{0})\subseteq\mathcal G_{0}$. Hence $B^{k}$ is
independent of $X^{k}_{0}$, and with $\Law_{\Prob^{k}}(X^{k}_{0})=\nu$ we obtain the following joint law,
\begin{equation}\label{eq:gluejointlaw}
\Law_{\Prob^{k}}\big(X^{k}_{0},B^{k}\big)=\nu\otimes\mathcal W .
\end{equation}
Finally $\mu^{k}$ was constructed as a conditional law under $\Prob^{k}$.

\ref{gl2} We verify Lemma~\ref{lem:immersionchar}\ref{imm2} for $k=1$.
The case $k=2$ is symmetric.
Fix $s\le t$ and bounded measurable
\[
\xi=F\big(x_{0},X^{1}|_{[0,t]},B|_{[0,t]}\big)\in\mathcal F^{1}_{t},
\qquad
\eta=G\big(X^{2}|_{[0,s]}\big),
\qquad
\zeta=H\big(x_{0},X^{1}|_{[0,s]},B|_{[0,s]}\big)\in\mathcal F^{1}_{s} .
\]
The products $\eta\zeta$ form a multiplicative class generating
$\mathcal G_{s}=\mathcal F^{1}_{s}\vee\sigma(X^{2}|_{[0,s]})$, and the filtrations here are
augmented, so cylinder variables represent the general bounded measurable ones up to $\Q$-null sets.
By the functional monotone class theorem it therefore suffices to
show that
\begin{equation}\label{eq:towrite}
\E_{\Q}\big[\xi\,\eta\,\zeta\big]=\E_{\Q}\big[\E[\xi\mid\mathcal F^{1}_{s}]\,\eta\,\zeta\big].
\end{equation}
By construction, under $\Q$ the two copies $X^{1},X^{2}$ are conditionally independent given
$(x_{0},B)$, that is, given the whole path $B$. The variables $\xi,\zeta$ are
$\sigma(x_{0},X^{1},B)$- and $\eta$ is $\sigma(X^{2})$-measurable, so the tower property together with
that conditional independence yields
\begin{equation}\label{eq:condfactor}
\E_{\Q}\big[\xi\eta\zeta\big]
=\E_{\Q}\Big[\ \underbrace{\E_{\Q}\big[\xi\zeta\mid x_{0},B\big]}_{=:A(x_{0},B)}\cdot
\underbrace{\E_{\Q}\big[\eta\mid x_{0},B\big]}_{=:\beta(x_{0},B)}\ \Big].
\end{equation}
Next we apply compatibility of the other solution. By \ref{gl1} and
Lemma~\ref{lem:solcompat}, the process
$X^{2}$ is compatible with $B$ under $\Q$. Definition~\ref{def:compatible} therefore yields
\begin{equation}\label{eq:gluecondind}
\sigma\big(X^{2}|_{[0,s]}\big)\ \perp\!\!\!\perp\ \sigma(B)\ \big|\ \sigma\big(B|_{[0,s]}\big) .
\end{equation}
What \eqref{eq:betameas} needs is conditional independence given
$\sigma(B|_{[0,s]})\vee\sigma(x_{0})$. Folding $x_{0}$ into the conditioning is permitted by the following elementary rule:
\begin{equation}\label{eq:condindrule}
\mathcal A\perp\!\!\!\perp\mathcal B\mid\mathcal C
\quad\text{and}\quad
\mathcal D\subseteq\mathcal A
\qquad\Longrightarrow\qquad
\mathcal A\perp\!\!\!\perp\mathcal B\mid\mathcal C\vee\mathcal D .
\end{equation}
We apply \eqref{eq:condindrule} to \eqref{eq:gluecondind} with
\[
\mathcal A=\sigma\big(X^{2}|_{[0,s]}\big),\qquad \mathcal B=\sigma(B),\qquad
\mathcal C=\sigma\big(B|_{[0,s]}\big),\qquad \mathcal D=\sigma(x_{0}),
\]
the inclusion $\mathcal D\subseteq\mathcal A$ holding modulo $\Q$-null sets, by
$x_{0}=X^{2}_{0}$ $\Q$-almost surely, which is stated at the end of this proof. That is enough, since
conditional independence is insensitive to null sets. From the resulting conditional independence
$\mathcal A\perp\!\!\!\perp\mathcal B\mid\mathcal C\vee\mathcal D$ we obtain the following identity,
\begin{equation}\label{eq:gluetower}
\E\big[\eta\mid\mathcal B\vee\mathcal C\vee\mathcal D\big]
=\E\big[\eta\mid\mathcal C\vee\mathcal D\big]
\end{equation}
for every integrable $\mathcal A$-measurable $\eta$, in particular for our $\eta$, a
functional of $X^{2}|_{[0,s]}$. Now $\mathcal B\vee\mathcal C\vee\mathcal D=\sigma(x_{0},B)$ and
$\mathcal C\vee\mathcal D=\sigma(x_{0},B|_{[0,s]})$, so that \eqref{eq:gluetower} reads
\begin{equation}\label{eq:betameas}
\beta(x_{0},B)=\E_{\Q}\big[\eta\mid x_{0},B|_{[0,s]}\big]=:\beta_{s}\big(x_{0},B|_{[0,s]}\big).
\end{equation}
In particular $\beta(x_{0},B)$ is
$\sigma(x_{0},B|_{[0,s]})\subseteq\mathcal F^{1}_{s}$-measurable. Substituting
\eqref{eq:betameas} into \eqref{eq:condfactor} and undoing the conditioning, we obtain
\[
\E_{\Q}\big[\xi\eta\zeta\big]=\E_{\Q}\big[\xi\,\zeta\,\beta_{s}\big]
=\E_{\Q}\big[\E[\xi\mid\mathcal F^{1}_{s}]\,\zeta\,\beta_{s}\big],
\]
the last step by the tower property, $\zeta\beta_{s}$ being $\mathcal F^{1}_{s}$-measurable. Running
the identical computation \eqref{eq:condfactor}--\eqref{eq:betameas} with $\xi$ replaced by the
$\mathcal F^{1}_{s}$-measurable variable $\E[\xi\mid\mathcal F^{1}_{s}]$ yields
\begin{equation}\label{eq:gluesecondrun}
\E_{\Q}\big[\E[\xi\mid\mathcal F^{1}_{s}]\eta\zeta\big]
=\E_{\Q}\big[\E[\xi\mid\mathcal F^{1}_{s}]\zeta\beta_{s}\big].
\end{equation}
Collecting the two identities, the right-hand sides agree, and that is \eqref{eq:towrite}.

\ref{gl3} By \ref{gl1} the law of $(X^{k},B)$ under $\Q$ is the law of $(X^{k},B^{k})$ under
$\Prob^{k}$, and the latter is an $L_{m}$-conditional solution by hypothesis. That the
former is one as well is given by Lemma~\ref{lem:lawdet}, the property, and the hypotheses under which the
lemma reads it, being determined by the law of $(X,B)$: each $X^{k}$ is
an $L_{m}$-conditional solution on its own generated filtration
$\sigma(X^{k}_{r},B_{r}:r\le t)$ under $\Q$. The filtration
$(\mathcal F^{k}_{t})$ above is the augmentation of that field once $\sigma(x_{0})$ is
adjoined, and since $X^{k}_{0}=x_{0}$ $\Q$-a.s., as recorded at the end of this proof, $\sigma(x_{0})$
already lies in the augmentation of the raw field. So only the augmentation has to be justified, and
that is the null-set argument of Remark~\ref{rem:MPraw}, which uses only that every set of the
augmentation differs from a set of the raw field by a null set: applied with $J_{s,t}$ and
$|J_{s,t}|^{m}$ in place of $D^{\varphi}_{s,t}$ and $|D^{\varphi}_{s,t}|^{m}$, it leaves both moduli of
\eqref{eq:daviemoduli} unchanged, and it leaves the Brownian requirement of \ref{MP0} unchanged as
well, an increment independent of the raw field being independent of its augmentation. Hence, the property
holds on $(\mathcal F^{k}_{t})$.
This is the step that fixes the solution class for the rest of the section: \ref{gl1}
controls only the law of $(X^{k},B)$, and Lemma~\ref{lem:lawdet} is what carries that back to the
equation, at Davie order $m$. By
\ref{gl2} and Lemma~\ref{lem:immersion}, each is then an
$L_{m}$-conditional solution on $(\mathcal G_{t})$, with $B$
a $(\mathcal G_{t})$-Brownian motion and the same moduli. Of \ref{MP0}, the initial
law and the moment bound are functions of $\Law(X^{k},B)$ and transfer by \ref{gl1}. The Brownian
condition, on $(\mathcal G_{t})$, is not, but it is provided by  step~\ref{stp:imm-a} of
Lemma~\ref{lem:immersion}. Finally, $X^{k}_{0}=x_{0}$ $\Q$-a.s.\ for
both $k$, since $\mu^{k}(\cdot\mid x_{0},w)$ is by construction supported on paths starting at $x_{0}$.
\end{proof}

We do not use right-continuity here. Passing to $(\mathcal G_{t+})$ loses the solution
property, and what survives is immersion between the right-continuous regularisations,
$(\mathcal F_{t+})$ into $(\mathcal G_{t+})$.

\subsection{The Yamada--Watanabe theorem}\label{subsec:YWthm}

\begin{definition}[Pathwise uniqueness]\label{def:pathuniq}
Fix $m\ge2$.
\emph{Pathwise uniqueness} holds for \eqref{eq:rsde} if the following is satisfied. Let $X^{1}$ and
$X^{2}$ denote $L_{m}$-conditional solutions in the sense of Definition~\ref{def:davie} on one and
the same stochastic basis $(\Omega,\mathcal A,(\mathcal F_{t}),\Prob)$, both satisfying
the moment bound of \ref{MP0} at
exponent $m$, driven by the
same Brownian motion $B$ and the same rough path $\RZ$, such that $\Prob[X^{1}_{0}=X^{2}_{0}]=1$. Then
\begin{equation}\label{eq:pathuniqconc}
\Prob\big[X^{1}_{t}=X^{2}_{t}\ \forall t\in[0,T]\big]=1 .
\end{equation}
\end{definition}

\begin{theorem}[Yamada--Watanabe for rough stochastic differential equations]\label{thm:YWtransfer}
Fix a deterministic $\RZ\in\CC^{\gamma}([0,T];\R^{e})$, $\gamma\in(\tfrac13,\tfrac12)$, and
let $(b,\sigma,f,f')$ denote coefficients as in \S\ref{sec:setting}, with $b$ and $\sigma$ of linear growth
\eqref{eq:Egrowth} and the rough field of linear growth \eqref{eq:MProughgrowth}. Parts \ref{ywa} and \ref{ywb} use neither growth bound, only the standing
continuity \eqref{eq:standingcts}. Both enter at part \ref{ywc}, through
Proposition~\ref{prop:MPequiv}, and at the Borelness of $\Gamma$, which the proof records without
using. Let
$m\ge2$, the floor of Definitions~\ref{def:davie} and \ref{def:pathuniq}. For an initial law $\nu$ denote
\begin{align*}
    \Gamma_{\nu}\ :=\ \big\{&\Prob\ \text{Borel on }\mathcal C=C([0,T];D\times\R^{d}):\
\text{the canonical }(X,B)\text{ is an }L_{m}\text{-conditional solution}\\
&\text{of \eqref{eq:rsde} for }(b,\sigma,f,f',\RZ)\text{ on the raw filtration,}\\
&\text{with \ref{MP0} at exponent }m\text{, which contains }\Law(X_{0})=\nu\big\} .
\end{align*}
It is a set of laws by construction, the condition being posed on the canonical space. What
Lemma~\ref{lem:lawdet} provides is that the property of a realisation can be read off its law, which
is what Proposition~\ref{prop:gluing}\ref{gl3} needs.

This is \cite{Kurtz14}'s constraint set with the initial law carried inside it rather than imposed
afterwards, which is the same object taken one $\nu$ at a time. Let $\mathcal N$ denote a set of
initial laws on $\R^{n}$ and suppose the following.
\begin{enumerate}[label=\normalfont(\roman*),leftmargin=2.4em]
\item\label{yw1} \emph{(weak existence)} for every $\nu\in\mathcal N$ there exists a weak solution of
\eqref{eq:rsde} with initial law $\nu$, and it is an $L_{m}$-conditional solution
satisfying \ref{MP0} at exponent $m$.
\item\label{yw2} \emph{(pathwise uniqueness)} Definition~\ref{def:pathuniq} holds.
\end{enumerate}
Then the following hold for every $\nu\in\mathcal N$:
\begin{enumerate}[label=\normalfont(\alph*),leftmargin=2.4em]
\item\label{ywa} \emph{uniqueness in law} holds \emph{on all of $\Gamma_{\nu}$}: that set has at most
one element, and by \ref{yw1} exactly one. In particular any two weak solutions with initial law $\nu$
whose laws lie in $\Gamma_{\nu}$ induce the same law on
$C([0,T];\R^{n})$; and the uniqueness is \emph{joint}, that is, what is determined is the law of the
pair $(X,B)$ on $C([0,T];\R^{n}\times\R^{d})$;
\item\label{ywb} \emph{strong existence} holds: there is a Borel map
$\Psi:\R^{n}\times C([0,T];\R^{d})\to C([0,T];D)$ such that, for every $t$, the map
$y\mapsto\Psi(y)|_{[0,t]}$ is measurable for the $\nu\otimes\mathcal W$-completion of
$\sigma(x_{0},w_{r}:r\le t)$, and such that
$X=\Psi(X_{0},B)$ a.s.\ under the unique element of $\Gamma_{\nu}$ given by
\ref{ywa};
\item\label{ywc} the rough martingale problem of Definition~\ref{def:roughMP} at exponent
$m$ has, for each $\nu\in\mathcal N$, at most one solution with initial law $\nu$, its
solution set at $m$ being contained in $\Gamma_{\nu}$ by
Proposition~\ref{prop:MPequiv}\ref{mpe2}; and exactly one if the solution of \ref{yw1} is in addition
an $L_{3m}$-conditional solution whose supremum lies in $L_{m'}$ for some $m'\ge3m$, by
\ref{mpe1}.
\end{enumerate}
\end{theorem}

\begin{proof}
The solution of \ref{yw1} has its law in $\Gamma_{\nu}$: by
Lemma~\ref{lem:filtdescend} it is an $L_{m}$-conditional solution on the filtration generated by the
pair, and \ref{MP0} descends with it, an increment independent of the larger field being independent
of the smaller one. Lemma~\ref{lem:lawdet} then carries the property to the canonical pair under
that law. So $\Gamma_{\nu}$ is non-empty.

Let $\nu\in\mathcal N$ and let $\mu^{1},\mu^{2}\in\Gamma_{\nu}$, realised by processes
$(X^{1},B^{1})$ and $(X^{2},B^{2})$ on possibly
different bases. Proposition~\ref{prop:gluing} produces a single basis
$(\widehat\Omega,\mathcal G_{T},(\mathcal G_{t}),\Q)$ carrying a $(\mathcal G_{t})$-Brownian motion
$B$ and two $L_{m}$-conditional solutions of \eqref{eq:rsde} driven by that same $B$ and the same $\RZ$,
with $X^{1}_{0}=X^{2}_{0}=x_{0}$ $\Q$-a.s.\ and with the correct marginal
laws. Hypothesis \ref{yw2} applies on that basis and yields $X^{1}=X^{2}$ $\Q$-a.s.

Hence $\Law_{\Q}(X^{1},B)=\Law_{\Q}(X^{2},B)$, and by Proposition~\ref{prop:gluing}\ref{gl1} these
are $\mu^{1}$ and $\mu^{2}$. Collecting the identities yields
\begin{equation}\label{eq:YWlawequal}
\mu^{1}=\Law_{\Q}\big(X^{1},B\big)=\Law_{\Q}\big(X^{2},B\big)=\mu^{2} .
\end{equation}
This
is \ref{ywa}, in the joint form, proved for every pair of elements of $\Gamma_{\nu}$.
The argument used only membership in that set.

For \ref{ywb} we invoke the second half of the classical Yamada--Watanabe argument, in Kurtz's form.
The constraint set is $\Gamma:=\Gamma_{\nu}$ of the statement, carried over to laws $\mu_{x_{0},X,B}$ on
$\R^{n}\times C([0,T];\R^{n})\times C([0,T];\R^{d})$ by adjoining $x_{0}=X_{0}$, a functional of $X$
that adds nothing. Its elements are the laws of $L_{m}$-conditional solutions, by
definition, so Definition~\ref{def:pathuniq} applies to them, and $\Gamma$ is non-empty by the
first paragraph. It is Borel, by the argument of
Lemma~\ref{lem:MPborel}\ref{mpb:borel}, with $J_{s,t}$ in place of $D^{\varphi}_{s,t}$, where the
reduction over test functions is not needed. The first half of \eqref{eq:epsdelta} does not come
from a monotone modulus, as it does in Lemma~\ref{lem:MPborel}, but from the bound
$\sup_{s\le t}\|J_{s,t}\|_{L_{m}}<\infty$, which \eqref{eq:davie} gives from \eqref{eq:Egrowth}, \eqref{eq:MProughgrowth}, Burkholder--Davis--Gundy and
$\|\sup_{t\le T}|X_{t}|\|_{L_{m}}<\infty$; its image in
$\mathcal P(\R^{n}\times C([0,T];\R^{n})\times C([0,T];\R^{d}))$ under $\mu\mapsto\Law(X_{0},X,B)$ is
Borel as well, that map being a continuous injection between Polish spaces and
\cite[Thm.~15.1]{Kechris95} making the image of a Borel set under such a map Borel. Borelness is not
used in the proof. It is recorded because in the terminology of \cite[\S1]{Kurtz14}
it exhibits \eqref{eq:rsde} as
a stochastic model $(\Gamma,\nu\otimes\mathcal W)$ with input $Y=(x_{0},B)$ and output $X$.
Every element of $\Gamma$ is compatible with that input, by \eqref{eq:compatinput} of
Lemma~\ref{lem:solcompat}. That is not what makes
\cite[Thm.~1.5]{Kurtz14} applicable, since that theorem imposes no condition on $\Gamma$ at all. It is
what the adaptedness part of \ref{ywb} cites \cite[Prop.~2.13]{Kurtz14} for. That
proposition is the one result of \cite{Kurtz14} the proof needs. The other appeals to that paper
are to a definition restated in full at \eqref{eq:compat}, to \cite[Lem.~1.3(c)]{Kurtz14} for a
fact proved below, and to \cite[Thm.~1.5]{Kurtz14}, which the proof does not invoke. What
Proposition~\ref{prop:gluing}\ref{gl2} uses is \eqref{eq:compat} together with the folding rule
\eqref{eq:condindrule}, and not \eqref{eq:compatinput}.

Hypothesis \ref{yw1} says the solution set is nonempty. What remains is pointwise uniqueness in the
sense of \cite[Def.~1.4]{Kurtz14}, quantifying over every triple $X_{1},X_{2},Y$ on a
common probability space with $\mu_{X_{1},Y},\mu_{X_{2},Y}\in\Gamma$, and not only over the glued $\Q$
of \eqref{eq:glue}. The gluing paragraph above yields it, at $\mu^{1}=\mu^{2}=\mu$ for the element $\mu$
of $\Gamma$ that \ref{yw1} provides. The two
copies are then conditionally independent given $Y$ with the same conditional law
$\mu(\mathrm dx\mid y)$, so $X^{1}=X^{2}$ $\Q$-a.s.\ reads
\[
\int\big(\mu(\cdot\mid y)\otimes\mu(\cdot\mid y)\big)\big(\{x^{1}\ne x^{2}\}\big)\,
(\nu\otimes\mathcal W)(\mathrm dy)\ =\ 0 ,
\]
and a Borel probability measure $\varrho$ on a Polish space with
$\varrho\otimes\varrho(\{x^{1}\ne x^{2}\})=0$ is a Dirac mass. Indeed, otherwise $\varrho(A)\in(0,1)$
for some Borel $A$, and from $A\times A^{c}\subseteq\{x^{1}\ne x^{2}\}$ we obtain
\begin{equation*}
\varrho\otimes\varrho\big(\{x^{1}\ne x^{2}\}\big)\ \ge\
\varrho\otimes\varrho\big(A\times A^{c}\big)\ =\ \varrho(A)\big(1-\varrho(A)\big)\ >\ 0 .
\end{equation*}
Hence the following holds,
\begin{equation}\label{eq:diracdisint}
\mu(\cdot\mid y)=\delta_{\Psi(y)}
\qquad\text{for }\nu\otimes\mathcal W\text{-a.e.\ }y ,
\end{equation}
and \eqref{eq:diracdisint} is to say that $\mu$ is the law
of a strong solution, that is, $X=\Psi(Y)$ a.s.\ whenever $\mu_{x_{0},X,B}=\mu$.
This is the direction of \cite[Lem.~1.3(c)]{Kurtz14} that we use:
$\Prob[X\neq\Psi(Y)]=\int\mu\big(\{x\neq\Psi(y)\}\mid y\big)\,(\nu\otimes\mathcal W)(\d y)=0$ by
\eqref{eq:diracdisint}. Here $\Psi$ is Borel after
modification on a $\nu\otimes\mathcal W$-null set, the set of $y$ at
which the kernel is Dirac being Borel of full measure and $\delta_{x}\mapsto x$ Borel on it. It is the adaptedness part that names the completion. \cite[Prop.~2.13]{Kurtz14}
asks of $X$ that it be a strong solution and compatible with the input, and returns
$\mathcal F^{X}_{t}\subseteq\mathcal F^{Y}_{t}$ at the completed filtrations. Both hypotheses are in
hand: $X=\Psi(Y)$ by \eqref{eq:diracdisint}, and the compatibility is \eqref{eq:compatinput}.
Neither uses the growth bounds. By \ref{ywa} the measure $\mu$ is the only element of $\Gamma$,
so $X_{1}=\Psi(Y)=X_{2}$ a.s.\ for every triple as above. That is \cite[Def.~1.4]{Kurtz14} outright, so
the hypotheses of \cite[Thm.~1.5]{Kurtz14} hold. However, we do not need to invoke this theorem. What it would provide is what the two previous paragraphs have produced directly: joint uniqueness in law is \eqref{eq:YWlawequal}, and the strong solution is the map $\Psi$
built from \eqref{eq:diracdisint} and \cite[Lem.~1.3(c)]{Kurtz14}. The agreement with
\cite[Thm.~1.5]{Kurtz14} is noted because it locates the result in the standard framework. Part \ref{ywc} follows from \ref{ywa}: the solution set of the martingale
problem at $m$ lies in $\Gamma_{\nu}$ by \ref{mpe2}, and $\Gamma_{\nu}$ is a singleton, while
\ref{mpe1} places the solution of \ref{yw1} in that set under the stated extra
integrability.
\end{proof}

\begin{remark}[Composing results (A) with (B)]\label{rem:YWcompose}
Fix a $p$ meeting \eqref{eq:Ep} and let $\mathcal N:=\{\nu:\langle\nu,|x|^{p}\rangle<\infty\}$.
Under Assumptions~\ref{ass:driver}, \ref{ass:Eito}, \ref{ass:Erough} and \ref{ass:Einit}, with
$D=\R^{n}$, Theorem~\ref{thm:weakex} provides
\ref{yw1} at every $m\in[2,p]$: it produces a weak solution with initial law $\nu$ that is an
$L_{m}$-conditional solution, and \eqref{eq:Emoment} gives the moment bound of \ref{MP0}. The two
growth bounds Theorem~\ref{thm:YWtransfer}\ref{ywc} needs are \eqref{eq:Egrowth} of
Assumption~\ref{ass:Eito} and, through $\|Df\|_{\infty}<\infty$, \eqref{eq:Eroughgrowth} of
Assumption~\ref{ass:Erough}. Only \ref{yw2} is left, and it is the subject of
\cite{HuberII}. At $m\le p/3$ the composition gives the exactly one half of \ref{ywc} as well,
Theorem~\ref{thm:weakex} then supplying an $L_{3m}$-conditional solution with supremum in $L_{p}$
and $p\ge3m$. Above that only at most one is available.
\end{remark}

\color{blue}
\begin{remark}\label{rem:YWours}
\emph{We imported:} \cite[Prop.~2.13]{Kurtz14}, which returns adaptedness of the strong
solution with respect to the completed canonical filtration of the input. That is everything the
proof of Theorem~\ref{thm:YWtransfer} uses from \cite{Kurtz14}. Part \ref{ywa} is
proved by Proposition~\ref{prop:gluing} with hypothesis \ref{yw2}, and \ref{ywc} by
Proposition~\ref{prop:MPequiv}. Part \ref{ywb} builds the strong solution directly, from the Dirac
disintegration \eqref{eq:diracdisint}, and \cite[Lem.~1.3(c)]{Kurtz14} is quoted there for a one-line
fact whose needed direction the proof reproves. \emph{Different, and
necessary:} First, Lemma~\ref{lem:lawdet}, exhibiting the notion of
Definition~\ref{def:davie}, a family of conditional bounds and originally a property of a basis, as a
constraint $\Gamma$ on the law. Second, Lemma~\ref{lem:immersion}, allowing the two glued solutions to
be read as solutions on the common filtration $(\mathcal G_{t})$, as Definition~\ref{def:pathuniq}
requires. For a classical It\^o equation both are automatic, and that is why neither has an analogue in
\cite{Kurtz14}.
\end{remark}
\color{black}

\appendix

\section{Deferred proofs, and what the imported results need}

\subsection{Proof of Lemma~\ref{lem:Escalefree}}\label{app:Escalefree}
The lemma is stated at Step~\ref{step:Emoment}\ref{proof:Eproof:step2:a1} of
\S\ref{subsec:Eproof} and used in Steps~\ref{step:Emoment} and \ref{step:Etight}. The
paragraphs before Step~\ref{proof:Eproof:step2:step1} fix what its proof may and may not assume.

\begin{proof}
\emph{Preliminaries.} The floor on $m$ is $1/\gamma$ and not $2$, Step~\ref{proof:Eproof:step2:step3} below turning the sewing bound
at a fixed pair into a bound on $\sup_{r\in[s,t]}|\Phi^{\natural}_{s,r}|$ by dyadic chaining, a series
that converges only when $\gamma m>1$. The range $(1/\gamma,p]$ is non-empty, since \eqref{eq:Ep} yields $p>6>3>1/\gamma$, and, $1/\gamma$ being $<3$, it contains every $m\in[3,p]$. Moreover $\mathcal N_{s,t}<\infty$ by
Step~\ref{step:Emoment}\ref{proof:Eproof:step2:b} of \S\ref{subsec:Eproof}, and branch \ref{Ecurvaffine} enters the estimates only through
\eqref{eq:Ecurvnorm}, by $C_{\mathrm{nrm}}$. Neither bound of \eqref{eq:Jfinite} is a hypothesis of the lemma: the citation delivers that, and all we use is, finiteness at the two
exponents, which boundedness of the truncated coefficients does not give (Remark~\ref{rem:Jfinite}).

\emph{The constants are free of the horizon.} The horizon enters the statement only through
$h_{\star}=(c_{0}\Lambda_{\RZ})^{-1/\gamma}\wedge T$, and every smallness used below is
$\Lambda\tau^{\gamma}\le c_{0}^{-1}$ or $\tau\le c_{0}^{-1/\gamma}\le1$ of \eqref{eq:Ebookkeep}, a
comparison against $1$ and never against $T$, so that no power of the horizon is absorbed into $c_{0}$
or $C_{3}$. The suprema $C_{\RZ}$ and $C^{(1)}_{\RZ}$ over $[0,T]$ are data of
Assumption~\ref{ass:Erough}.

\emph{The $\Gamma^{3}$-part of $\Phi^{\natural}$ is a classical integral.} By \eqref{eq:bracketLip} the bracket
is Lipschitz, so $\int\Gamma^{3}\d\CZ$ is an ordinary Lebesgue--Stieltjes integral. Its remainder is
bounded pathwise, and by \ref{proof:Eproof:step2:a0} we obtain the following bound,
\begin{align*}
\Big|\int_{s}^{t}\Gamma^{3}(X_{r})\,\d\CZ_{r}-\Gamma^{3}(X_{s})\,\CZ_{s,t}\Big|
&\ \le\ \|\CZ\|_{1}\sup_{r\in[s,t]}|\delta\Gamma^{3}(X)_{s,r}|\,|t-s|\\
&\ \le\ C_{1}\|\CZ\|_{1}\big(\sup_{r\in[s,t]}\Phi_{r}\big)|t-s| .
\end{align*}
This is of order $1$ and linear in $\sup\Phi$, hence of order $2\gamma$ on $|t-s|\le h_{\star}\le1$, by \eqref{eq:Ebookkeep}. The order is $1$ and not $1+\gamma$ because the supremum
sits inside the norm, and $\|\sup_{r}|\delta X_{s,r}|\|_{L_{m}}$ is not among the quantities the lemma
controls. Only $(\Gamma^{1},\Gamma^{2})$ is sewn against $\RZ$, the $\CZ$-part being
classical. 
Nevertheless, $\CZ$-terms belong to the germ and to $\Phi^{\natural}$, so
the $\CZ$-paired terms of the defect in Step~\ref{proof:Eproof:step2:step2} are part of what
Lemma~\ref{lem:SSL} is provided with. What is classical here is the integral $\int\Gamma^{3}\d\CZ$ itself, not
the estimate of the defect.

The third conclusion \eqref{eq:EscaleCond} carries the conditional half of the solution notion,
\eqref{eq:daviemoduli} being two-sided. Its second modulus is an unconditional $L_{m}$ bound and
follows from $\mathcal J\le C_{3}\Lambda^{2}\mathcal N$ at \eqref{eq:HJfinal}. The first is a bound on
$\|\E_{s}J_{s,t}\|_{L_{m}}$ at rate $o(|t-s|)$, and no unconditional bound delivers it, since
\[
\|\E_{s}J_{s,t}\|_{L_{m}}\le\|J_{s,t}\|_{L_{m}}=O(|t-s|^{2\gamma}),
\qquad 2\gamma\ <\ 1 ,
\]
is short of rate $1$. We insist on exactly this distinction for the conditional-mean condition \ref{MP1} of Definition~\ref{def:roughMP}. \eqref{eq:EscaleCond} is the second conclusion of
Lemma~\ref{lem:SSL}, applied to the germ of Step~\ref{proof:Eproof:step2:step1}, $\epsilon_{1}>1$ is what makes it $o(|t-s|)$, and it
is uniform in $(k,R)$ while the approximants' own conditional moduli are not.

All three bounds are stochastic sewing, that is, Lemma~\ref{lem:SSL}, the ordinary $L_{m}$ lemma with
no weight and no conditional $L_{\infty}$ quantity, applied to germs whose defects we estimate through
the pathwise Taylor bound \eqref{eq:Taylormin} below and not through $\sup$-norms of the truncated
coefficients.
\eqref{eq:EscaleX} and \eqref{eq:EscaleNat} come from the first bound of \eqref{eq:SSLconc}, at
the germ of Step~\ref{proof:Eproof:step2:step1} and at the germ of $V$ in Step~\ref{proof:Eproof:step2:step2} respectively, while \eqref{eq:EscaleCond} comes
from the second, at the germ of Step~\ref{proof:Eproof:step2:step1}, and is read off the closed system rather than fed back
into it. The quantities estimated appear on both sides, are finite by the
truncation, and we can absorb them because the
mesh is short.

For $s\le u\le v\le t$ we denote by
\[
\mathcal H:=\!\!\sup_{s\le u\le v\le t}\!\frac{\|\delta X_{u,v}\|_{L_{m}}}{|v-u|^{\gamma}},
\qquad
\mathcal J:=\!\!\sup_{s\le u\le v\le t}\!\frac{\|J_{u,v}\|_{L_{m}}}{|v-u|^{2\gamma}},
\]
where $J$ denotes the remainder of the Davie expansion
\eqref{eq:davie}. Both are finite at fixed $(k,R)$, but not for the same
reason.

\emph{$\mathcal H<\infty$.} This is Step~\ref{step:Emoment}\ref{proof:Eproof:step2:b} of
\S\ref{subsec:Eproof}, whose "All intervals" part bounds
$\|\delta X_{u,v}\|_{L_{p}}\le K'|v-u|^{\gamma}$ at every pair, with $K'=K'(k,R)$. Since $m\le p$, the
same bound holds in $L_{m}$. The Davie modulus does not provide it by itself: it controls
$\|J_{u,v}\|_{L_{m}}$ only below a threshold depending on $(k,R)$, whereas the supremum defining
$\mathcal H$ runs up to $h_{\star}$, which does not.

\emph{$\mathcal J<\infty$ is \eqref{eq:Jfinite}.} For $J$ rather
than for $\delta X$, the same crude estimate yields only the following,
\[
\|J_{u,v}\|_{L_{m}}\le\big(\mathcal H+C(R)\Lambda^{2}\big)\tau^{\gamma},
\qquad\text{hence}\qquad
\|J_{u,v}\|_{L_{m}}/\tau^{2\gamma}\le\big(\mathcal H+C(R)\Lambda^{2}\big)\tau^{-\gamma}\to\infty
\quad\text{as }v\downarrow u ,
\]
the $\mathcal H$ coming from the $\delta X_{u,v}$ that \eqref{eq:davie} leaves on the other side.
Definition~\ref{def:davie} gives us $o(\tau^{1/2})$, weaker than $O(\tau^{2\gamma})$ since
$2\gamma>\tfrac12$ at every admissible $\gamma$: boundedness gives the exponent $\gamma$, while
$\mathcal J$ is posed at $2\gamma$. What yields $\mathcal J<\infty$ is \eqref{eq:Jfinite}, namely \cite[Prop.~4.3, Rem.~4.4]{FHL}, verified at Step~\ref{step:Ereg}. Both
$\mathcal H$ and $\mathcal J$ are unconditional $L_{m}$ quantities at the exponent of the conclusion;
no conditional $L_{\infty}$ quantity, no weight and no norm above $L_{m}$ occurs anywhere in this
proof.

\begin{remark}\label{rem:Jfinite}
Both of the bounds of \eqref{eq:Jfinite} are \cite[Prop.~4.3, Rem.~4.4]{FHL} applied to the truncated, mollified system, Step~\ref{step:Ereg} giving
the arithmetic that carries the source's two rates $2\gamma$ and $2\gamma+\bar\beta$ onto the two
exponents, at values allowed to depend on $(k,R)$ and on $m$ in any way.
Boundedness of the truncated coefficients supplies neither, obtaining from \eqref{eq:davie} only
\[
\|J_{u,v}\|_{L_{m}}\lesssim_{R}\tau^{\gamma}
\]
while the first bound divides by $\tau^{2\gamma}$, and Definition~\ref{def:davie}'s moduli
$o(\tau^{1/2})$ and $o(\tau)$ are short of $2\gamma$ and of $\epsilon_{1}>1$ respectively.

The first bound gives the finiteness of the quantity $\mathcal J$, which makes the
absorption of \ref{proof:Eproof:step2:step1:1c} a relation between finite numbers rather than $\infty\le\infty$, and it is used at
the order $2\gamma$: replacing $\|J_{u,w}\|_{L_{m}}\le\mathcal J\tau^{2\gamma}$ by the bound
$\big(\mathcal H+C(R)\Lambda^{2}\big)\tau^{\gamma}$ would drop the $J$-contribution to $\epsilon_{1}$
from
\[
\gamma+2\gamma\ =\ 3\gamma\ >\ 1
\qquad\text{to}\qquad
\gamma+\gamma\ =\ 2\gamma<1 ,
\]
and Lemma~\ref{lem:SSL}'s hypothesis $\epsilon_{1}>1$ would fail. The second is used at the
identification of \ref{proof:Eproof:step2:step1:1c}, where the uniqueness part of Lemma~\ref{lem:SSL} identifies the sewing
$\mathcal A$ with $\delta X-\int b-\int\sigma\,\d B$, both bounds of \eqref{eq:SSLconc} cutting
out the class in which that sewing is unique, at the exponents $\epsilon_{1},\epsilon_{2}$ that
\ref{proof:Eproof:step2:step1:1a}--\ref{proof:Eproof:step2:step1:1b} assemble and at some finite constants. What we assume of $\mathcal H$ and $\mathcal J$ is
finiteness at an exponent, while what \eqref{eq:HJfinal} produces is a constant free of $(k,R)$.
\end{remark}

Throughout, $\tau:=|w-u|$, $c_{i}:=1+|X_{u,i}|$, $\Lambda:=\Lambda_{\RZ}\ge1$, and every interval is
contained in $[s,t]$ with $t-s\le h_{\star}$, so that
\begin{equation}\label{eq:Ebookkeep}
\Lambda\,\tau^{\gamma}\ \le\ \Lambda h_{\star}^{\gamma}\ \le\ c_{0}^{-1},
\qquad\text{hence also}\qquad
\tau\ \le\ \big(c_{0}\Lambda\big)^{-1/\gamma}\ \le\ c_{0}^{-1/\gamma}\ \le\ 1 ,
\end{equation}
the last inequality by the normalisation $c_{0}\ge1$ of the statement. That normalisation is harmless,
$c_{0}$ being only ever enlarged below, which shortens the greedy mesh $h_{\star}$ and weakens no
conclusion. We use it, silently, at three kinds of places. First, wherever a power of
$\tau$ is traded for the smaller $\tau^{\gamma}$, as in
\[
\Lambda\tau\le1,
\qquad
\Lambda\tau^{\beta+\beta'}\le1
\]
in the $\Gamma_{2}$ bookkeeping of (1c) and in the three temporal pieces at \eqref{eq:eps2value}.
Second, wherever a power of $c_{0}$ is traded for a smaller one, as in
\[
c_{0}^{-(1+\gamma)/\gamma}\le c_{0}^{-1}
\qquad\text{at \eqref{eq:EmismatchAMGM}},
\qquad
c_{0}^{-1/(2\gamma)}\le1
\]
in that same bookkeeping. Third, in the $\Gamma^{3}$-part above, where an estimate of order $1$ is read
as one of order $2\gamma<1$. Every absorption below is made by choosing $c_{0}$ large, and we indicate
in each case which inequality of \eqref{eq:Ebookkeep} supplies the smallness.

\medskip
\begin{enumerate}[label=Step~\arabic*, ref=\arabic*, start=0, wide=0pt]
    \item \label{proof:Eproof:step2:step0} \emph{\textbf{Two elementary observations.}}
    \begin{enumerate}[label=(0\alph*), wide=0pt]
        \item \label{proof:Eproof:step2:step0:0a} For $c\ge1$ we denote by
$g_{c}(\theta):=\theta/(c+\sqrt\theta)$, $\theta\ge0$. $g_{c}$ is increasing with $g_{c}(0)=0$ and
\[
g_{c}''(\theta)\ =\ -\,\frac{3c\theta+\theta^{3/2}}{4\theta^{3/2}\big(c+\sqrt\theta\big)^{3}}\ <\ 0 ,
\]
so that $g_{c}$ is concave on $[0,\infty)$. Being concave and vanishing at $0$, it is subadditive and
satisfies $g_{c}(\lambda\theta)\le\lambda g_{c}(\theta)$ for $\lambda\ge1$. Further, it admits the
bound.
\begin{equation}\label{eq:concave}
g_{c}(\theta)\ \le\ \min\Big\{\frac{\theta}{c},\ \sqrt\theta\Big\} .
\end{equation}
Let $c$ be $\mathcal F_{u}$-measurable and let $Y\in L_{2}$. Writing $g_{c}$ as the infimum of its
tangents and evaluating at the $\mathcal F_{u}$-measurable point $\E_{u}[Y^{2}]$, we obtain
\begin{equation}\label{eq:concaveJensen}
\E_{u}\Big[\frac{Y^{2}}{c+|Y|}\Big]\ =\ \E_{u}\big[g_{c}(Y^{2})\big]\ \le\ g_{c}\big(\E_{u}[Y^{2}]\big) . 
\end{equation}
The two branches of \eqref{eq:concave} replace a dichotomy on $\Omega$. Such a dichotomy would need
moments at $3m$ for its rare branch. Nothing is split on $\Omega$, since we can split the
estimate into additive pieces and choose a branch of \eqref{eq:concave} for each, subadditivity
of $g_{c}$ making that legitimate piece by piece. What we give up instead is an exponent rather than a
moment, and we give it up once, in Step~\ref{proof:Eproof:step2:step1}\ref{proof:Eproof:step2:step1:1b}\ref{proof:Eproof:step2:step1:1b:gamma}.
\item \label{proof:Eproof:step2:step0:0b} Let
$\mathcal T_{u,w}:=\delta\big[f_{R}(X)\big]^{\mathrm{sp}}_{u,w}-Df_{R,u}(X_{u})\,\delta X_{u,w}$ denote the
second-order remainder of the expansion of $f_{R,u}$ in the state variable, where
$\delta[f_{R}(X)]^{\mathrm{sp}}_{u,w}:=f_{R,u}(X_{w})-f_{R,u}(X_{u})$. Here and below $f$ denotes the
truncated field $f_{R}$ of Step~\ref{step:Ereg}, and that is where the branch of
Assumption~\ref{ass:Erough} is visible. Two forms of the bound are available.

Under \ref{Ecurvaffine} what the truncated field has, is the norm condition
\eqref{eq:Ecurvnorm}, uniformly in $R$. In the affine
sub-case the untruncated field has $\mathcal T\equiv0$, while the truncated one does not. Even there the
argument runs through \eqref{eq:Ecurvnorm}, and that is why
Assumption~\ref{ass:Erough}\ref{Ecurvaffine} is stated at \eqref{eq:Ecurvnormhyp} rather than at
$D^{2}f\equiv0$. On $\{|\delta X_{u,w}|\le\tfrac12\Phi_{u}\}$ every point $\xi$ of the segment
$[X_{u},X_{w}]$ is such that, since $\Phi_{u}\le1+|X_{u}|$,
\[
1+|\xi|\ge(1+|X_{u}|)-|\delta X_{u,w}|\ge\tfrac12\Phi_{u} ,
\]
so that \eqref{eq:Ecurvnorm} yields $|\mathcal T_{u,w}|\le C|\delta X_{u,w}|^{2}/\Phi_{u}$. On the
complement the mean value theorem yields the following bound,
\[
|\mathcal T_{u,w}|\le2\|Df\|_{\infty}|\delta X_{u,w}|
\le4\|Df\|_{\infty}|\delta X_{u,w}|^{2}/\Phi_{u} .
\]
Hence, pathwise on all of $\Omega$, and with $C_{\mathcal T}:=\max\{4\|Df\|_{\infty},C_{\mathrm{nrm}}+3C\}$ in this branch, we obtain
\begin{equation}\label{eq:Taylorminaff}
\big|\mathcal T_{u,w}\big|\ \le\ C_{\mathcal T}\min\Big\{\frac{|\delta X_{u,w}|^{2}}{\Phi_{u}},\ |\delta X_{u,w}|\Big\},
\qquad\text{\ref{Ecurvaffine}} ,
\end{equation}
with the global weight $\Phi_{u}$. This is the same form as \eqref{eq:TaylorminV} below, and we can run
the affine branch of the proof identically to Step~\ref{proof:Eproof:step2:step2}. In particular \eqref{eq:Eacoord} is never used
under \ref{Ecurvaffine}, and the governing order there permits the larger first entry
$3\gamma$ in \eqref{eq:eps1value}, an improvement we do not use. No
coordinatewise growth bound on $f$ is available under \ref{Ecurvaffine}, e.g.~for $f(x)=(x_{2},0)$ it
is false.

Under \ref{Ecurvdiag} the field decouples across coordinates,
\[
\mathcal T^{\kappa}_{u,w,i}=q^{\kappa}_{u,i}(X_{w,i})-q^{\kappa}_{u,i}(X_{u,i})
-\dot q^{\kappa}_{u,i}(X_{u,i})\,\delta X_{u,w,i} ,
\]
and two bounds hold, each pathwise on all of $\Omega$. Applying the mean value theorem to
$\dot q$ yields the first,
\[
|\mathcal T^{\kappa}_{u,w,i}|\le2\|Df\|_{\infty}|\delta X_{u,w,i}| .
\]
For the second, let $|\delta X_{u,w,i}|\le\tfrac12c_{i}$. Then every point $\xi$ of the segment
$[X_{u,i},X_{w,i}]$ is such that $1+|\xi|\ge c_{i}-|\delta X_{u,w,i}|\ge\tfrac12c_{i}$, so that
\eqref{eq:Ecurvdiag} yields
\[
|\ddot q^{\kappa}_{u,i}(\xi)|\le2\bar C_{\mathrm{cur}}/c_{i},
\qquad\text{hence}\qquad
|\mathcal T^{\kappa}_{u,w,i}|\le \bar C_{\mathrm{cur}}(\delta X_{u,w,i})^{2}/c_{i} .
\]
If instead $|\delta X_{u,w,i}|>\tfrac12c_{i}$, the first bound yields
\[
|\mathcal T^{\kappa}_{u,w,i}|\le4\|Df\|_{\infty}(\delta X_{u,w,i})^{2}/c_{i} .
\]
With $C_{\mathcal T}:=\max\{4\|Df\|_{\infty},\bar C_{\mathrm{cur}}\}$ in this branch, we therefore obtain
\begin{equation}\label{eq:Taylormin}
\big|\mathcal T^{\kappa}_{u,w,i}\big|
\ \le\ C_{\mathcal T}\,\min\Big\{\frac{\big(\delta X_{u,w,i}\big)^{2}}{c_{i}},\ \big|\delta X_{u,w,i}\big|\Big\}
\ \le\ \frac{2\,C_{\mathcal T}\,\big(\delta X_{u,w,i}\big)^{2}}{c_{i}+\big|\delta X_{u,w,i}\big|},
\end{equation}
the last step by considering the regimes $|\delta X_{u,w,i}|\le c_{i}$ and $|\delta X_{u,w,i}|>c_{i}$.

Three features of \eqref{eq:Taylormin} matter, and the first two are shared with
\eqref{eq:Taylorminaff}. First, it is an inequality between random variables on all of $\Omega$: no
stopping time and no localization appears in this proof. Second, its weight is the coordinate $c_{i}$
and not $\Phi_{u}$, and this cannot be improved. At $X_{u}=(M,0,\dots,0)$ and
$\delta X_{u,w}=(0,\eta,0,\dots,0)$ we obtain the following,
\[
|\mathcal T_{u,w}|\asymp|\ddot q_{2}(0)|\eta^{2},
\qquad
|\delta X_{u,w}|^{2}/\Phi_{u}\asymp\eta^{2}/M ,
\]
the former independent of $M$. Third, its second branch is linear, and that is what keeps the
$\epsilon_{2}$-estimate at exponent $m$. In what follows we run the branch \ref{Ecurvdiag}, that is, the
harder one. Under \ref{Ecurvaffine} every occurrence of $c_{i}$ is replaced by
$\Phi_{u}$ and every $\ell^{2}$-sum over $i$ disappears. The four pieces are written out after
\eqref{eq:eps1value}.
\end{enumerate}%
    \item \label{proof:Eproof:step2:step1}\emph{\textbf{The germ of the equation.}} Let $A_{u,v}:=f_{u}(X_{u})\delta Z_{u,v}+(Df\,f+f')_{u}(X_{u})\ZZ_{u,v}$ denote the germ, whose defect is
\[
\delta A_{u,w,v}\ =\ -\,R_{u,w}\,\delta Z_{w,v}\ -\ \delta\big[(Df\,f+f')(X)\big]_{u,w}\,\ZZ_{w,v},
\qquad
R_{u,w}:=\delta\big[f(X)\big]_{u,w}-(Df\,f+f')_{u}(X_{u})\delta Z_{u,w} .
\]
We expand $f_{w}(X_{w})-f_{u}(X_{u})$ to second order in $x$ and to first order in $t$, and we insert
the Davie expansion of $\delta X_{u,w}$. The first-order terms cancel against
$(Df\,f+f')_{u}(X_{u})\delta Z_{u,w}$. What remains is
\begin{equation}\label{eq:Rsplitfive}
R_{u,w}\ =\ Df_{u}(X_{u})\Big[\int_{u}^{w}\!b\,\dr+\delta M_{u,w}
+(Df\,f+f')_{u}(X_{u})\ZZ_{u,w}+J_{u,w}\Big]
\ +\ \mathcal T_{u,w}\ +\ \varrho_{u,w},
\end{equation}
where $\varrho=\varrho^{(1)}+\varrho^{(2)}$ denotes the two time-increment terms, i.e.~
$\varrho^{(1)}_{u,w}:=\big(f'_{u}(X_{w})-f'_{u}(X_{u})\big)\delta Z_{u,w}$ and
$\varrho^{(2)}_{u,w}:=f_{w}(X_{w})-f_{u}(X_{w})-f'_{u}(X_{w})\delta Z_{u,w}$. These have different
orders:
\begin{equation}\label{eq:rhoorders}
\big\|\varrho^{(1)}_{u,w}\big\|_{L_{m}}\le\|Df'\|_{\infty}\Lambda\,\mathcal H\,\tau^{2\gamma},
\qquad
\big\|\varrho^{(2)}_{u,w}\big\|_{L_{m}}\le \sqrt2\,C_{\RZ}\,\mathcal N_{s,t}\,\tau^{\beta+\beta'},
\end{equation}
where $C_{\RZ}$ denotes the controlled-path norm of $(f,f')$ in $\mathscr D^{\beta,\beta'}_{\RZ}$.
Step~\ref{step:Ereg} shows that the truncation does not increase it. The exponent $\beta+\beta'$
is not $2\gamma$ and need not exceed it. Here Assumption~\ref{ass:Erough} demands $\beta'>\tfrac13$ and
$\beta\le\gamma$, and that permits e.g.~
\[
\gamma=\beta=0.45,
\qquad
\beta'=0.35,
\qquad
\beta+\beta'=0.80<0.90=2\gamma .
\]
What we use, is
\begin{equation}\label{eq:rhoexp}
\beta+\beta'+\gamma\ \ge\ 2\beta+\beta'\ >\ 1 ,
\end{equation}
since $\beta\le\gamma$. The term $\varrho^{(2)}$ is therefore admissible for $\epsilon_{1}$ on its own account and not
by comparison with the second-level term. It carries $\mathcal N_{s,t}$ linearly and is not mentioned
again, while $\varrho^{(1)}$ is treated exactly as the second-level term.
We use two Lipschitz facts repeatedly. First, evaluating the global growth bound of Assumption~\ref{ass:Erough} at
$x=\xi e_{i}$ yields, under \ref{Ecurvdiag}, the following bounds
\[
|f_{i}(x)|\le C_{\mathrm{gr}}(1+|x_{i}|),
\qquad
|(Df\,f+f')_{i}(x)|\le\big(\|Df\|_{\infty}+1\big)C_{\mathrm{gr}}(1+|x_{i}|) .
\]
Second,
\begin{equation}\label{eq:DGGLip}
\big\|D\big(Df\,f+f'\big)\big\|_{\infty}\ \le\ \|Df\|^{2}_{\infty}+C_{\mathrm{gr}}\bar C_{\mathrm{cur}}+\|Df'\|_{\infty}\ =:\ C_{\mathrm{lip}}
\qquad\text{under \ref{Ecurvdiag}},
\end{equation}
Here $D(Df\,f)=(Df)^{2}+(D^{2}f)f$, and under \ref{Ecurvdiag}, the second summand has $i$-th diagonal
entry $\ddot q_{i}(x_{i})q_{i}(x_{i})$. By \eqref{eq:Ecurvdiag} and the coordinatewise growth, that entry is bounded by $\bar C_{\mathrm{cur}}C_{\mathrm{gr}}$, the same cancellation as in the
eighth term of \ref{proof:Eproof:step2:a0}. Under \ref{Ecurvaffine} it is bounded instead by
$C_{\mathrm{gr}}(C_{\mathrm{nrm}}+3C)$ in place of $C_{\mathrm{gr}}\bar C_{\mathrm{cur}}$, by
\eqref{eq:Ecurvnorm} and the linear growth of $f$.
Without a curvature hypothesis $Df\,f$ is not globally Lipschitz for a field of linear growth,
and \eqref{eq:DGGLip} would fail.

\begin{enumerate}[label=(1\alph*), wide=0pt]
    \item \label{proof:Eproof:step2:step1:1a} \emph{The hypothesis at $\epsilon_{2}$.}  We can take $L_{m}$-norms in \eqref{eq:Rsplitfive}. Applying
the Burkholder--Davis--Gundy inequality to the martingale term, we obtain the following four bounds,
\begin{align*}
\Big\|Df_{u}(X_{u})\int_{u}^{w}\!b\,\dr\Big\|_{L_{m}}
&\ \le\ \|Df\|_{\infty}\sqrt2C_{\mathrm{gr}}\mathcal N_{s,t}\tau ,\\
\big\|Df_{u}(X_{u})\,\delta M_{u,w}\big\|_{L_{m}}
&\ \le\ \|Df\|_{\infty}c_{m}\sqrt2 C_{\mathrm{gr}}\mathcal N_{s,t}\tau^{1/2} ,\\
\big\|Df_{u}(X_{u})(Df\,f+f')_{u}(X_{u})\,\ZZ_{u,w}\big\|_{L_{m}}
&\ \le\ C\Lambda^{2}\mathcal N_{s,t}\tau^{2\gamma} ,\\
\big\|Df_{u}(X_{u})\,J_{u,w}\big\|_{L_{m}}
&\ \le\ \|Df\|_{\infty}\mathcal J\tau^{2\gamma} .
\end{align*}
The two time-increment terms sit at the orders $2\gamma$ and $\beta+\beta'$ of \eqref{eq:rhoorders},
both $\ge\gamma$ since
\[
\beta+\beta'>1-\beta\ge1-\gamma>\gamma .
\]
For the Taylor term we use the second branch of \eqref{eq:Taylormin}, yielding
\begin{equation}\label{eq:Teps2}
\big\|\mathcal T_{u,w}\big\|_{L_{m}}\ \le\ C_{\mathcal T}\sqrt e\,\big\|\,|\delta X_{u,w}|\,\big\|_{L_{m}}
\ \le\ C_{\mathcal T}\sqrt e\ \mathcal H\ \tau^{\gamma}
\qquad\text{(linear in $\mathcal H$, at exponent $m$).}
\end{equation}
Since $\gamma<\tfrac12$ the determining order is $\gamma$. Collecting the estimates yields the following bounds,
\[
\|R_{u,w}\|_{L_{m}}\le C\Lambda^{2}(\mathcal N_{s,t}+\mathcal H+\mathcal J)\tau^{\gamma},
\qquad
\big\|R_{u,w}\,\delta Z_{w,v}\big\|_{L_{m}}\ \le\ C\Lambda^{3}(\mathcal N_{s,t}+\mathcal H+\mathcal J)\tau^{\gamma}|v-w|^{\gamma} ,
\]
so that against $\delta Z_{w,v}$ this contributes at order $2\gamma$.
$\delta[(Df\,f+f')(X)]_{u,w}$ still remains. It is paired with $\ZZ_{w,v}$ and carries the
full two-index increment of a time-dependent coefficient along the path. We denote by
$\mathsf G_{t}:=Df_{t}f_{t}+f'_{t}$ and split
\begin{equation}\label{eq:Gsplit}
\delta\big[\mathsf G(X)\big]_{u,w}
=\underbrace{\mathsf G_{w}(X_{w})-\mathsf G_{w}(X_{u})}_{\text{spatial}}
+\underbrace{\mathsf G_{w}(X_{u})-\mathsf G_{u}(X_{u})}_{\text{temporal}} .
\end{equation}
By \eqref{eq:DGGLip} the spatial part obeys the following bound,
\[
\big\|\mathsf G_{w}(X_{w})-\mathsf G_{w}(X_{u})\big\|_{L_{m}}\ \le\ C_{\mathrm{lip}}\mathcal H\tau^{\gamma} ,
\]
by a bound on $\|D\mathsf G\|_{\infty}$ controlling that part. The temporal part does not vanish for a
non-autonomous field, so we have to estimate it. Expanding yields
\[
\mathsf G_{w}(x)-\mathsf G_{u}(x)=\big(Df_{w}-Df_{u}\big)(x)\,f_{w}(x)+Df_{u}(x)\big(f_{w}-f_{u}\big)(x)
+\big(f'_{w}-f'_{u}\big)(x) ,
\]
and by \eqref{eq:Eroughderiv}, \eqref{eq:cvfdefect} and $\|Df\|_{\infty},\|Df'\|_{\infty}<\infty$, the
three summands yield
\begin{equation}\label{eq:Gtemporal}
\begin{gathered}
\big|\mathsf G_{w}(x)-\mathsf G_{u}(x)\big|
\ \le\ C_{\mathrm{tmp}}\big(1+|x|\big)\Big(\Lambda\tau^{\gamma}+\tau^{\beta+\beta'}\Big)+C_{\RZ}\big(1+|x|\big)\tau^{\beta'},
\\
C_{\mathrm{tmp}}:=C_{\mathrm{tmp}}\big(C_{\mathrm{gr}},C^{(1)}_{\RZ},C_{\RZ},\|Df\|_{\infty},\|Df'\|_{\infty}\big),
\end{gathered}
\end{equation}
a bound, which is linear in the state. Indeed, \eqref{eq:Eroughderiv} in its unweighted form yields
\[
|Df_{w}-Df_{u}|(x)\le\|Df'\|_{\infty}\Lambda\tau^{\gamma}+C^{(1)}_{\RZ}\tau^{\beta+\beta'} ,
\]
and $|f_{w}(x)|\le C_{\mathrm{gr}}(1+|x|)$ provides the single power. Taking $L_{m}$ norms yields,
\begin{equation}\label{eq:Gincrement}
\big\|\delta[\mathsf G(X)]_{u,w}\big\|_{L_{m}}
\ \le\ C_{\mathrm{lip}}\mathcal H\,\tau^{\gamma}
\ +\ C\,\mathcal N_{s,t}\Big(\Lambda\tau^{\gamma}+\tau^{\beta+\beta'}+\tau^{\beta'}\Big).
\end{equation}
Instead of collapsing the temporal orders into $\gamma\wedge\beta'$, we keep them separate. Only
the first one carries a factor $\Lambda$, coming from the two $\delta Z$'s inside the defects of $Df$ and
$f$, while the $\tau^{\beta+\beta'}$ and $\tau^{\beta'}$ pieces are the controlled-path defects. Attaching $\Lambda$ to all three would put $\Lambda^{2-\beta'/\gamma}$ on
the coefficient of $\mathcal N_{s,t}$ after the greedy bookkeeping
$h\le(c_{0}\Lambda)^{-1/\gamma}$, and for $\beta'<\gamma$ that exceeds the $O(\Lambda)$ allowed by
\eqref{eq:Esystem} and \eqref{eq:HJfinal}. Against $\ZZ_{w,v}=O(\Lambda\tau^{2\gamma})$, the coefficients below being written
with $\Lambda\ge1$ so that one power covers all three, the three
contribute at the orders
\[
3\gamma,
\qquad
2\gamma+\beta+\beta',
\qquad
2\gamma+\beta' ,
\]
with coefficients $\Lambda^{3}\mathcal N_{s,t}$, $\Lambda^{2}\mathcal N_{s,t}$ and
$\Lambda^{2}\mathcal N_{s,t}$, the first carrying the surplus factor $\Lambda\tau^{\gamma}\le c_{0}^{-1}$
of \eqref{eq:Ebookkeep}. Each order is $\ge\epsilon_{1}$. The first, because
$3\gamma\ge\epsilon_{1}$ as above, and the second and third because $\beta\le\gamma$ yields
\[
2\gamma+\beta'\ge\gamma+\beta+\beta'\ge\epsilon_{1} ,
\]
and $2\gamma+\beta+\beta'$ is still larger. Hence, the hypothesis of Lemma~\ref{lem:SSL} holds at
\begin{equation}\label{eq:eps2value}
\epsilon_{2}\ =\ 2\gamma\ >\ \tfrac12
\qquad\text{with}\qquad
\Gamma_{2}\ \le\ C\Big(\Lambda\,\mathcal H+\Lambda^{2}\mathcal N_{s,t}
+\Lambda h_{\star}^{\gamma}\,\mathcal J\Big),
\end{equation}
which is the list of \ref{proof:Eproof:step2:step1:1a} divided by $\tau^{2\gamma}$ and
evaluated at $\tau\le h_{\star}$: the two $\Lambda^{2}\mathcal H\tau^{\gamma}$ entries are
$\Lambda\mathcal H$ times $\Lambda\tau^{\gamma}\le c_{0}^{-1}$, and the $\Lambda^{3}\mathcal
N_{s,t}\tau^{\gamma}$ entry is $\Lambda^{2}\mathcal N_{s,t}$ times the same factor. The
$\mathcal J$-term carries the factor $\Lambda h_{\star}^{\gamma}\le c_{0}^{-1}$ of
\eqref{eq:Ebookkeep}, independent of $\Lambda$, which is what the absorption in
\eqref{eq:Esystem} needs. We write $h_{\star}$ and
not $\tau=|w-u|$ here because Lemma~\ref{lem:SSL} requires a $\Gamma_{2}$ free of the triple, and
$\tau\le h_{\star}$ on the range the lemma is applied on. The temporal half of
\eqref{eq:Gincrement} sits inside the $\mathcal N_{s,t}$-term here and in \eqref{eq:Esystem} below, its
three pieces being
\[
\Lambda^{2}\mathcal N_{s,t}\cdot\Lambda\tau^{\gamma},
\qquad
\Lambda^{2}\mathcal N_{s,t}\cdot\tau^{\beta+\beta'},
\qquad
\Lambda^{2}\mathcal N_{s,t}\cdot\tau^{\beta'} ,
\]
all bounded on $|t-s|\le h_{\star}$. The Taylor term fixes $\epsilon_{2}=2\gamma$, the
remaining contributions being of order $\tfrac12+\gamma$ or better. $2\gamma>\tfrac12$ needs
only $\gamma>\tfrac14$.
\item \label{proof:Eproof:step2:step1:1b} \emph{The hypothesis at $\epsilon_{1}$.} We apply $\E_{u}$ to \eqref{eq:Rsplitfive}. The
martingale term vanishes exactly, $Df_{u}(X_{u})$ being $\mathcal F_{u}$-measurable and $\delta M_{u,w}$ of
zero conditional mean. The remaining terms are of the orders
\[
1\ \ \text{(drift)},
\qquad
2\gamma\ \ \text{(second level and }\varrho^{(1)}\text{)},
\qquad
\beta+\beta'\ \ (\varrho^{(2)}) ,
\]
the last being admissible by \eqref{eq:rhoexp} (not because $\beta+\beta'\ge2\gamma$). Conditional Jensen yields
\[
\|\E_{u}J_{u,w}\|_{L_{m}}\le\|J_{u,w}\|_{L_{m}}\le\mathcal J\tau^{2\gamma} .
\]
The crude bound suffices, since $2\gamma+\gamma=3\gamma>1$, so that the refined conclusion of
Lemma~\ref{lem:SSL} on $\E_{u}J$ does not enter. It enters afterwards, in Step~\ref{proof:Eproof:step2:step1}\ref{proof:Eproof:step2:step1:1d},
where it produces \eqref{eq:EscaleCond}. The only term requiring work is $\mathcal T$. We can treat it by
\eqref{eq:Taylormin} and \eqref{eq:concaveJensen}. Denoting by
$\mathcal A_{i}:=\E_{u}\big[(\delta X_{u,w,i})^{2}\big]$, we obtain,
\[
\E_{u}\big|\mathcal T^{\kappa}_{u,w,i}\big|\ \le\ 2C_{\mathcal T}\,g_{c_{i}}\big(\mathcal A_{i}\big),
\qquad
\big|\E_{u}\mathcal T_{u,w}\big|\ \le\ \Big(\sum_{i,\kappa}\big(\E_{u}|\mathcal T^{\kappa}_{u,w,i}|\big)^{2}\Big)^{1/2} .
\]
Along \eqref{eq:davie} we can split $\mathcal A_{i}$ into four pieces,
\[
\mathcal A_{i}\le4\big(\mathcal A^{b}_{i}+\mathcal A^{M}_{i}+\mathcal A^{Z}_{i}+\mathcal A^{J}_{i}\big),
\]
where, the second by the conditional It\^o isometry,
\begin{align*}
\mathcal A^{b}_{i}&=\E_{u}\big(\int_{u}^{w}b_{i}\dr\big)^{2}\le\tau\,\E_{u}\!\int_{u}^{w}\!b_{i}^{2}\dr ,
&\mathcal A^{M}_{i}&=\E_{u}\!\int_{u}^{w}\!a_{ii}(r,X_{r})\dr ,\\
\mathcal A^{Z}_{i}&=\big(|f_{i}(X_{u})||\delta Z_{u,w}|+|(Df\,f+f')_{i}(X_{u})||\ZZ_{u,w}|\big)^{2} ,
&\mathcal A^{J}_{i}&=\E_{u}J^{2}_{u,w,i} .
\end{align*}
By subadditivity of $g_{c_{i}}$ and $g_{c_{i}}(4\theta)\le4g_{c_{i}}(\theta)$ we obtain
\[
g_{c_{i}}(\mathcal A_{i})\le4\sum_{\alpha}g_{c_{i}}(\mathcal A^{\alpha}_{i}) .
\]
Applying Minkowski's inequality in $\ell^{2}(i)$, the outer index being $\alpha$ so that we can choose a different
branch of \eqref{eq:concave} for each piece, it suffices to bound the four
$\ell^{2}(i)$-sums separately. All $L_{m/2}$-manipulations below are allowed because $m\ge2$.
\begin{enumerate}[label=(\greek*), wide=0pt]
    \item \label{proof:Eproof:step2:step1:1b:alpha} \emph{Rough.} By the coordinatewise growth above and \eqref{eq:Ebookkeep}, we obtain
$\mathcal A^{Z}_{i}\le C\,c^{2}_{i}\,\Lambda^{2}\tau^{2\gamma}$. On the branch $\theta/c_{i}$ this
yields
\[
g_{c_{i}}(\mathcal A^{Z}_{i})\le Cc_{i}\Lambda^{2}\tau^{2\gamma},
\qquad
\big(\sum_{i}c^{2}_{i}\big)^{1/2}\le\sqrt{2n}\,\Phi_{u},
\]
so that the $L_{m}$-norm of the $\ell^{2}$-sum obeys
\[
\Big\|\Big(\sum_{i}g_{c_{i}}\big(\mathcal A^{Z}_{i}\big)^{2}\Big)^{1/2}\Big\|_{L_{m}}
\ \le\ C\sqrt{2ne}\,\Lambda^{2}\mathcal N_{s,t}\tau^{2\gamma} .
\]
\item \label{proof:Eproof:step2:step1:1b:beta} \emph{Drift.} On the branch $\sqrt\theta$ we obtain
$\big(\sum_{i}\mathcal A^{b}_{i}\big)^{1/2}=\big(\tau\E_{u}\!\int_{u}^{w}|b(r,X_{r})|^{2}\dr\big)^{1/2}$.
Applying conditional Jensen, the triangle inequality in $L_{m/2}$ and \eqref{eq:Egrowth}, we obtain,
\[
\Big\|\Big(\sum_{i}\mathcal A^{b}_{i}\Big)^{1/2}\Big\|_{L_{m}}
=\Big\|\tau\,\E_{u}\!\int_{u}^{w}\!|b|^{2}\Big\|^{1/2}_{L_{m/2}}
\le\tau^{1/2}\Big(\int_{u}^{w}\!\big\|b(r,X_{r})\big\|^{2}_{L_{m}}\dr\Big)^{1/2}
\le\sqrt2\,C_{\mathrm{gr}}\,\mathcal N_{s,t}\,\tau .
\] 
\item \label{proof:Eproof:step2:step1:1b:gamma}  \emph{Martingale.} By \eqref{eq:Eacoord} and
$1+|X_{r,i}|\le c_{i}+|\delta X_{u,r,i}|$ we obtain
$\mathcal A^{M}_{i}\le C_{\mathrm{gr}}\big(\mathcal A^{M,1}_{i}+\mathcal A^{M,2}_{i}\big)$, where
\[
\mathcal A^{M,1}_{i}:=c_{i}\,\E_{u}\!\int_{u}^{w}\!\big(1+|X_{r}|\big)\dr,
\qquad
\mathcal A^{M,2}_{i}:=\E_{u}\!\int_{u}^{w}\!\big|\delta X_{u,r,i}\big|\big(1+|X_{r}|\big)\dr .
\]
For the first we can take the branch $\theta/c_{i}$. The bound $\E_{u}\int(1+|X_{r}|)\dr$ then no longer
depends on $i$, the $\ell^{2}$-sum produces a factor $\sqrt{n}$, and we obtain the following bound,
\[
\Big\|\Big(\sum_{i}g_{c_{i}}\big(\mathcal A^{M,1}_{i}\big)^{2}\Big)^{1/2}\Big\|_{L_{m}}
\ \le\ \sqrt{2ne}\;\mathcal N_{s,t}\tau ,
\]
On the second we can take the branch $\sqrt\theta$. Since $\sum_{i}|y_{i}|\le\sqrt n\,|y|$,
\[
\Big(\sum_{i}\mathcal A^{M,2}_{i}\Big)^{1/2}
=\Big(\E_{u}\!\int_{u}^{w}\!\Big(\sum_{i}|\delta X_{u,r,i}|\Big)\big(1+|X_{r}|\big)\dr\Big)^{1/2}
\le\Big(\sqrt n\,\E_{u}\!\int_{u}^{w}\!\big|\delta X_{u,r}\big|\big(1+|X_{r}|\big)\dr\Big)^{1/2},
\]
and conditional Jensen together with H\"older's inequality in $L_{m/2}$ yields
\begin{equation}\label{eq:Emismatch}
\Big\|\Big(\sum_{i}\mathcal A^{M,2}_{i}\Big)^{1/2}\Big\|_{L_{m}}
\ \le\ n^{1/4}\Big(\int_{u}^{w}\!\big\|\delta X_{u,r}\big\|_{L_{m}}\big\|1+|X_{r}|\big\|_{L_{m}}\dr\Big)^{1/2}
\ \le\ n^{1/4}\big(\sqrt2\,\mathcal H\,\mathcal N_{s,t}\big)^{1/2}\,\tau^{\frac{1+\gamma}{2}} .
\end{equation}
The order is $\tfrac{1+\gamma}{2}$, and this term fixes $\epsilon_{1}$. It comes from
the mismatch of evaluation points: the weight $c_{i}$ sits at $u$ while $a_{ii}$ is evaluated at
$r\in[u,w]$, and \eqref{eq:Eacoord} cancels the coordinate only at coinciding times. The square-root
branch keeps the estimate at exponent $m$. The branch $\theta/c_{i}$ would give the better order
$1+\gamma$, but it would force $\|\delta X\|_{L_{2m}}\|1+|X|\|_{L_{2m}}$, the higher moment
this proof exists to avoid.
\item \label{proof:Eproof:step2:step1:1b:delta} \emph{Remainder.} On the branch $\sqrt\theta$ we obtain
\[
\big(\sum_{i}\mathcal A^{J}_{i}\big)^{1/2}=\big(\E_{u}|J_{u,w}|^{2}\big)^{1/2},
\qquad
\Big\|\big(\E_{u}|J_{u,w}|^{2}\big)^{1/2}\Big\|_{L_{m}}=\big\|\E_{u}|J_{u,w}|^{2}\big\|^{1/2}_{L_{m/2}}\le\|J_{u,w}\|_{L_{m}}\le\mathcal J\tau^{2\gamma} .
\]
\end{enumerate}
We collect the four estimates and use $\tfrac{1+\gamma}{2}<2\gamma$ (exactly
$\gamma>\tfrac13$). The governing order among the pieces of $\mathcal T$ is $\tfrac{1+\gamma}{2}$. Against
$\delta Z_{w,v}$, and together with \eqref{eq:rhoexp}, we therefore obtain
\begin{equation}\label{eq:eps1value}
\epsilon_{1}\ =\ \min\Big\{\frac{1+3\gamma}{2},\ \beta+\beta'+\gamma\Big\}\ >\ 1 ,
\end{equation} each entry exceeding $1$ for its own reason. The first because $\gamma>\tfrac13$, the
second by \eqref{eq:rhoexp}.  In the autonomous case $f'\equiv0$, we have $\varrho\equiv0$ and $\epsilon_{1}=\tfrac12(1+3\gamma)$ outright.

The governing order under \ref{Ecurvaffine}  is higher, but $\epsilon_{1}$ always
denotes \eqref{eq:eps1value}. The four pieces are the
same, with $\mathcal A:=\E_{u}|\delta X_{u,w}|^{2}$ and with the global weight $\Phi_{u}$ of
\eqref{eq:Taylorminaff} in place of $c_{i}$, so that no $\ell^{2}$-sum over $i$ occurs:
\begin{itemize}[leftmargin=1.4em]
\item \emph{Rough.} $\mathcal A^{Z}\le C\Phi^{2}_{u}\Lambda^{2}\tau^{2\gamma}$ by
\eqref{eq:Eroughgrowth} and \eqref{eq:Ebookkeep}. On the branch $\theta/\Phi_{u}$ this is
$C\Phi_{u}\Lambda^{2}\tau^{2\gamma}$, of $L_{m}$-norm $C\Lambda^{2}\mathcal N_{s,t}\tau^{2\gamma}$,
at order $2\gamma$.
\item \emph{Drift.} On the branch $\sqrt\theta$, at order $1$, exactly as in
\ref{proof:Eproof:step2:step1:1b:beta}.
\item \emph{Martingale.} Here \eqref{eq:Eacoord} is not used. By \eqref{eq:Egrowth},
$\mathcal A^{M}=\E_{u}\!\int_{u}^{w}\!\operatorname{tr}a_{r}(X_{r})\dr$ with
$\operatorname{tr}a_{r}(X_{r})=\Frob{\sigma_{r}(X_{r})}^{2}\le C^{2}_{\mathrm{gr}}(1+|X_{r}|)^{2}$,
and $(1+|X_{r}|)^{2}\le4\Phi^{2}_{u}+4|\delta X_{u,r}|^{2}$ splits it in two. The first half goes on
the branch $\theta/\Phi_{u}$, at order $1$. The second goes on the branch $\sqrt\theta$, where
conditional Jensen and Minkowski's integral inequality in $L_{m/2}$ yield
\[
\Big\|\Big(\E_{u}\!\int_{u}^{w}\!\big|\delta X_{u,r}\big|^{2}\dr\Big)^{1/2}\Big\|_{L_{m}}
\ \le\ \Big(\int_{u}^{w}\!\big\|\delta X_{u,r}\big\|^{2}_{L_{m}}\dr\Big)^{1/2}
\ \le\ \mathcal H\,\tau^{\gamma+\frac12} ,
\]
at order $\gamma+\tfrac12$.
\item \emph{Remainder.} On the branch $\sqrt\theta$, at order $2\gamma$, as in
\ref{proof:Eproof:step2:step1:1b:delta}.
\end{itemize}
The governing order is $2\gamma$, since $2\gamma<\gamma+\tfrac12<1$ at $\gamma<\tfrac12$, and
against $\delta Z_{w,v}$ that is $3\gamma$. The mismatch of evaluation points behind
\eqref{eq:Emismatch} does not occur: the weight is global, so the split is by $\Phi_{u}$ and
$|\delta X_{u,r}|$ and leaves the second factor at degree two rather than one.

Three later places change under \ref{Ecurvaffine}, and none of them for the worse. In
\eqref{eq:Teps2} the linear branch of \eqref{eq:Taylorminaff} replaces the linear branch of
\eqref{eq:Taylormin}, so the factor $\sqrt e$ coming from the $\ell^{2}$-sum disappears and the
exponent stays $\gamma$. In the enumeration of \ref{proof:Eproof:step2:step1:1d}, kind~(4) is empty,
since \eqref{eq:Emismatch} does not arise, and the $\mathcal A^{M,2}$-half of the martingale piece,
of order $\gamma+\tfrac12$ with coefficient $\Lambda\mathcal H\le C_{3}\Lambda^{2}\mathcal N_{s,t}$,
is of kind~(5). Every $\Gamma_{1}$ contribution is then of order strictly above $2\gamma$, the half
just named at $2\gamma+\tfrac12$, so the arithmetic--geometric mean step \eqref{eq:EmismatchAMGM} is
not needed and \eqref{eq:Esystem}--\eqref{eq:HJfinal} close a fortiori.
The remaining contribution to $\|\E_{u}\delta A_{u,w,v}\|_{L_{m}}$ is
$\|\E_{u}\delta[(Df\,f+f')(X)]_{u,w}\|_{L_{m}}|\ZZ_{w,v}|$. Conditional Jensen and
\eqref{eq:Gincrement}, the latter carrying the temporal half of that increment alongside the spatial
one, yield
\[
\big\|\E_{u}\delta[(Df\,f+f')(X)]_{u,w}\big\|_{L_{m}}\big|\ZZ_{w,v}\big|
\ \le\ C_{\mathrm{lip}}\mathcal H\Lambda^{2}\tau^{3\gamma}
\ +\ C\Big(\Lambda^{3}\mathcal N_{s,t}\,\tau^{3\gamma}
+\Lambda^{2}\mathcal N_{s,t}\,\tau^{2\gamma+\beta+\beta'}
+\Lambda^{2}\mathcal N_{s,t}\,\tau^{2\gamma+\beta'}\Big),
\]
Here we keep the three temporal pieces apart exactly as in \eqref{eq:Gincrement}, since only the first
of them carries a factor $\Lambda$. All four orders are $\ge\epsilon_{1}$. For $3\gamma$ we obtain
$3\gamma>\tfrac12(1+3\gamma)\ge\epsilon_{1}$, that is, exactly $\gamma>\tfrac13$
, in both branches, $\epsilon_{1}$ always denoting \eqref{eq:eps1value}. For $2\gamma+\beta'$, and consequently for $2\gamma+\beta+\beta'$, we obtain
$2\gamma+\beta'\ge\gamma+\beta+\beta'\ge\epsilon_{1}$, since $\beta\le\gamma$. Lemma~\ref{lem:SSL} requires order $\ge\epsilon_{1}$, so that equality is admissible, and all four are absorbed with the rest.
\item\label{proof:Eproof:step2:step1:1c}   Lemma~\ref{lem:SSL}, applied at
\eqref{eq:eps1value} and \eqref{eq:eps2value}, produces a sewing $\mathcal A$ such that
\begin{equation}\label{eq:SSLtwo}
\big\|\mathcal A_{u,v}-A_{u,v}\big\|_{L_{m}}\le C\big(\Gamma_{1}|v-u|^{\epsilon_{1}}+\Gamma_{2}|v-u|^{\epsilon_{2}}\big),
\end{equation}
Further, $\mathcal A_{u,v}-A_{u,v}=J_{u,v}$, the two remainders solving the same sewing problem and its
solution being unique. What places $\delta X-\int b-\int\sigma\,\d B$ in the uniqueness class is \eqref{eq:Jfinite}, in both of its bounds. Indeed, the two estimates of \eqref{eq:SSLconc} separate that class, at the exponents \eqref{eq:eps1value} and \eqref{eq:eps2value} and at finite constants, which
Definition~\ref{def:davie} does not provide there, its $o(\tau^{1/2})$ being short of $2\gamma$ and its
$o(\tau)$ short of $\epsilon_{1}$. The uniqueness part this rests on is the class of Lemma~\ref{lem:SSL}, which is wider than the class
\cite[Thm.~2.9(i)]{FHL} prints and which the paragraph after that lemma settles. \eqref{eq:Jfinite}
is itself a citation,
\cite[Prop.~4.3, Rem.~4.4]{FHL}, verified at Step~\ref{step:Ereg}. 
Lemma~\ref{lem:SSL} delivers $|v-u|^{\min\{\epsilon_{1},\epsilon_{2}\}}$ and \eqref{eq:SSLtwo} yields
\[
\|J_{u,v}\|_{L_{m}}=O(|v-u|^{2\gamma})
\qquad\text{and not}\qquad
O(|v-u|^{3\gamma}) .
\]
Every use of \eqref{eq:EscaleX}--\eqref{eq:EscaleNat} below needs only $2\gamma>\gamma$, but the
constant in \eqref{eq:Emoment} depends on the difference.
We can divide \eqref{eq:SSLtwo} by $|v-u|^{2\gamma}$ and take the supremum. Feeding
\eqref{eq:Rsplitfive} back into the Davie expansion, whose other terms are bounded by
$C\Lambda\mathcal N_{s,t}\tau^{\gamma}$, yields,
\begin{equation}\label{eq:Hclose}
\mathcal J\ \le\ C\big(\Gamma_{1}\,h^{\epsilon_{1}-2\gamma}+\Gamma_{2}\big),
\qquad
\mathcal H\ \le\ C\Lambda\,\mathcal N_{s,t}\ +\ \mathcal J\,h^{\gamma},
\qquad
\epsilon_{1}-2\gamma\ \ge\ \tfrac{1-\gamma}{2}\wedge\big(\beta+\beta'+\gamma-2\gamma\big)\ >\ 0,
\end{equation}
where $h:=t-s\le h_{\star}$ denotes the length of the interval. It remains to see, term by term, that in the resulting inequalities the
coefficients of $\mathcal H$ and $\mathcal J$ are small and the coefficient of $\mathcal N_{s,t}$ is
$O(\Lambda)$. The powers of $\Lambda$ matter, since dropping them makes the absorption fail for large
$\|\RZ\|_{\gamma}$, that is, in the only regime the mesh is for.

\emph{From $\Gamma_{2}$.} Dividing the contributions to $\|\delta A_{u,w,v}\|_{L_{m}}$ listed in \ref{proof:Eproof:step2:step1:1a}
by $\tau^{2\gamma}$ yields the following bound,
\[
\begin{gathered}
\Gamma_{2}\ \le\ C\Big[\underbrace{\Lambda\mathcal H}_{\text{Taylor}}
+\underbrace{\Lambda^{2}\mathcal H\,\tau^{\gamma}
+\Lambda^{2}\mathcal N_{s,t}\big(\tau^{\beta+\beta'}+\tau^{\beta'}\big)}_{\delta(Df\,f+f'),\ \text{spatial and temporal}}
+\underbrace{\Lambda\mathcal N_{s,t}\big(\tau^{\frac12-\gamma}+\tau^{1-\gamma}+\tau^{\beta+\beta'-\gamma}\big)}_{\text{mart., drift, }\varrho^{(2)}}
\\
+\underbrace{\Lambda^{2}\mathcal H\,\tau^{\gamma}}_{\varrho^{(1)}}
+\underbrace{\Lambda^{3}\mathcal N_{s,t}\,\tau^{\gamma}}_{\text{second level, and the }\Lambda\text{-carrying temporal piece}}
+\underbrace{\Lambda\mathcal J\,\tau^{\gamma}}_{J}\Big],
\end{gathered}
\]
where each factor $\Lambda$ comes from one $\delta Z_{w,v}$ or one $\ZZ_{w,v}$. Where a higher
power is used, the bound is deliberately crude, $\Lambda\ge1$ making it valid. We multiply by
$h^{\gamma}$, and \eqref{eq:Ebookkeep} then yields the absorption
inequalities
\begin{align*}
\Lambda h^{\gamma}&\le c_{0}^{-1},
&\Lambda^{2}h^{2\gamma}&\le c_{0}^{-2},
&\Lambda h^{\frac12}&\le c_{0}^{-\frac1{2\gamma}}\Lambda^{1-\frac1{2\gamma}}\le1,\\
\Lambda h&\le1,
&\Lambda h^{\beta+\beta'}&\le1,
&\Lambda^{2}h^{\gamma+\beta'}&\le c_{0}^{-1}\Lambda,\\
\Lambda^{2}h^{\gamma+\beta+\beta'}&\le c_{0}^{-1}\Lambda,
&\Lambda^{3}h^{2\gamma}&\le c_{0}^{-2}\Lambda,
&\Lambda h^{2\gamma}&\le c_{0}^{-2}\Lambda^{-1} .
\end{align*}
The third holds since $\gamma<\tfrac12$ and $c_{0}\ge1$. The fourth and the fifth hold by $h\le1$ and
$\beta+\beta'>\gamma$, putting those two powers of $h$ below $h^{\gamma}$. The sixth and the seventh,
for the two temporal pieces, hold by $\Lambda h^{\gamma}\le c_{0}^{-1}$ and $h\le1$. Hence
$\Gamma_{2}h^{\gamma}$ contributes to $\mathcal H$ the coefficients
\[
Cc_{0}^{-1}\ \text{ on }\mathcal H,
\qquad
Cc_{0}^{-2}\ \text{ on }\mathcal J,
\qquad
C\Lambda\ \text{ on }\mathcal N_{s,t},
\]
as required.

\emph{From $\Gamma_{1}$.} Every contribution is of order strictly greater
than $2\gamma$ in $|v-u|$, and that is what the division by $\tau^{\epsilon_{1}}$ followed by
multiplication by $h^{\epsilon_{1}-2\gamma}$ needs. \eqref{eq:Emismatch} is treated separately for a different
reason: its coefficient $\Lambda(\mathcal H\mathcal N_{s,t})^{1/2}$ is not linear in $\mathcal H$, so
it cannot enter the system \eqref{eq:Esystem} as it stands. The list of the remaining contributions and of
their orders is given by
\begin{align*}
\text{drift of \eqref{eq:Rsplitfive}}&:\ 1+\gamma ,
&\text{second-level term and }\E_{u}J_{u,w}&:\ 3\gamma ,\\
(\alpha),\ (\delta),\ \varrho^{(1)}\ \text{against }\delta Z_{w,v}&:\ 2\gamma+\gamma=3\gamma ,
&(\beta),\ (\gamma_{1})&:\ 1+\gamma\ge3\gamma ,\\
\varrho^{(2)}&:\ \beta+\beta'+\gamma ,
&\text{the four }\ZZ\text{-paired pieces}&:\ 3\gamma,\ 3\gamma ,\\
&&&\phantom{:\ }2\gamma+\beta+\beta',\ 2\gamma+\beta' ,
\end{align*}
the entry $(\beta),(\gamma_{1})$ using $\gamma\le\tfrac12$. Every one exceeds $2\gamma$, the last two by
$\beta,\beta'>0$, and $\varrho^{(2)}$ by \eqref{eq:rhoexp} and $\gamma<\tfrac12$, which
give $\beta+\beta'+\gamma>1>2\gamma$. They do not all need to exceed $3\gamma$, and at $\gamma$ close to $\tfrac12$
they will not. That is why we state the criterion at $2\gamma$. At
e.g.~$(\gamma,\beta,\beta')=(0.49,0.49,0.34)$ both $\varrho^{(2)}$ and the $\tau^{2\gamma+\beta'}$ piece sit
at
\[
\beta+\beta'+\gamma\ =\ 2\gamma+\beta'\ =\ 1.32\ ,
\qquad
1.32<1.47=3\gamma .
\]
Each becomes a term with a positive power of $h$ after division, and a factor $\Lambda^{3}$
at worst. Using \eqref{eq:Ebookkeep} once more, we obtain,
\[
\mathcal N_{s,t}\text{-terms}:\ O(\Lambda^{2}),
\qquad
\mathcal H\text{-terms}:\ O(\Lambda),
\qquad
\mathcal J\text{-term}:\ O(c_{0}^{-1}) .
\]
The bound $O(\Lambda^{2})$ on the $\mathcal N_{s,t}$-terms is saturated by the
$\Lambda^{3}\mathcal N_{s,t}\tau^{3\gamma}$ part, and that is why \eqref{eq:HJfinal} carries
$\Lambda^{2}$ on $\mathcal J$.

The one term requiring care is \eqref{eq:Emismatch}. It is of order $\tfrac{1+3\gamma}{2}$, the first entry of \eqref{eq:eps1value} and hence at least
$\epsilon_{1}$, and therefore
contributes to $\mathcal H$, via \eqref{eq:Hclose}, the quantity
\[
C\Lambda\big(\mathcal H\mathcal N_{s,t}\big)^{1/2}h^{\frac{1+3\gamma}{2}-\gamma}
=C\Lambda\big(\mathcal H\mathcal N_{s,t}\big)^{1/2}h^{\frac{1+\gamma}{2}} .
\]
Applying the arithmetic--geometric mean inequality $xy\le\tfrac14x^{2}+y^{2}$ with
$x=\mathcal H^{1/2}$ and $y=C\Lambda h^{(1+\gamma)/2}\mathcal N^{1/2}_{s,t}$, we obtain,
\begin{equation}\label{eq:EmismatchAMGM}
C\Lambda h^{\frac{1+\gamma}{2}}\big(\mathcal H\mathcal N_{s,t}\big)^{1/2}
\ \le\ \tfrac14\,\mathcal H\ +\ C^{2}\Lambda^{2}h^{1+\gamma}\,\mathcal N_{s,t},
\qquad
\Lambda^{2}h^{1+\gamma}\ \le\ c_{0}^{-\frac{1+\gamma}{\gamma}}\,\Lambda^{2-\frac{1+\gamma}{\gamma}}\ \le\ c_{0}^{-1},
\end{equation}
the last inequality is due to $2-\tfrac{1+\gamma}{\gamma}=1-\tfrac1\gamma<0$ for $\gamma<1$ and
$\Lambda\ge1$, and because $\tfrac{1+\gamma}{\gamma}\ge1$ and $c_{0}\ge1$ yield
$c_{0}^{-(1+\gamma)/\gamma}\le c_{0}^{-1}$. The coefficient of $\mathcal H$ is therefore $\tfrac14$, free of $c_{0}$ and of $\Lambda$, and the coefficient of $\mathcal N_{s,t}$ is $O(1)$. Its
contribution to $\mathcal J$ carries instead $h^{\frac{1+3\gamma}{2}-2\gamma}=h^{(1-\gamma)/2}$, and the same
inequality yields
\[
C\Lambda h^{\frac{1-\gamma}{2}}\big(\mathcal H\mathcal N_{s,t}\big)^{1/2}
\ \le\ \tfrac14\mathcal H+C^{2}\Lambda^{2}h^{1-\gamma}\mathcal N_{s,t},
\qquad
\Lambda^{2}h^{1-\gamma}\le c_{0}^{-1}\Lambda ,
\]
the last by $3-\tfrac1\gamma\in(0,1)$ for $\gamma\in(\tfrac13,\tfrac12)$. The mismatch sits in $\Gamma_{1}$, so
\eqref{eq:Esystem} carries it in its second inequality, where $\Lambda\ge1$ lets
$C\Lambda\mathcal H$ absorb the $\tfrac14\mathcal H$. It reaches $\mathcal H$ through
$\mathcal J h^{\gamma}$.

\emph{Assembling.} Collecting the last three paragraphs, and denoting by $C$ a constant depending only
on the data listed in the statement, we obtain the following system,
\begin{equation}\label{eq:Esystem}
\mathcal H\ \le\ C\Lambda\,\mathcal N_{s,t}\ +\ c_{0}^{-1}\Lambda^{-1}\,\mathcal J,
\qquad
\mathcal J\ \le\ C\Lambda\,\mathcal H\ +\ C\Lambda^{2}\,\mathcal N_{s,t}\ +\ Cc_{0}^{-1}\,\mathcal J .
\end{equation}
The first is \eqref{eq:Hclose} together with $h^{\gamma}\le(c_{0}\Lambda)^{-1}$. In the second, the
$\mathcal H$-coefficient is $O(\Lambda)$ and not small. It comes from the
Taylor term of $\Gamma_{2}$, entering \eqref{eq:Hclose} with no power of $h$. That is the reason
the two inequalities must be combined in this order. We can take $c_{0}$ large enough such that
$Cc_{0}^{-1}\le\tfrac12$, and the second then yields
$\mathcal J\le2C\Lambda\mathcal H+2C\Lambda^{2}\mathcal N_{s,t}$. Substituting into the first yields
\[
\mathcal H\ \le\ C\Lambda\,\mathcal N_{s,t}+c_{0}^{-1}\Lambda^{-1}\big(2C\Lambda\mathcal H+2C\Lambda^{2}\mathcal N_{s,t}\big)
\ =\ \big(C+2Cc_{0}^{-1}\big)\Lambda\,\mathcal N_{s,t}\ +\ 2Cc_{0}^{-1}\,\mathcal H ,
\]
where the coefficient of $\mathcal H$ is now $2Cc_{0}^{-1}$, independent of $\Lambda$. We can enlarge $c_{0}$
once more, such that $2Cc_{0}^{-1}\le\tfrac12$. With that enlargement, and using that $\mathcal H$ and
$\mathcal J$ are finite at fixed $(k,R)$, so that these inequalities are relations between finite
numbers rather than $\infty\le\infty$, bringing the terms involving $\mathcal H$ from the right to the
left-hand-side yields the following bounds,
\begin{equation}\label{eq:HJfinal}
\mathcal H\ \le\ C_{3}\,\Lambda\,\mathcal N_{s,t},
\qquad
\mathcal J\ \le\ C_{3}\,\Lambda^{2}\,\mathcal N_{s,t} .
\end{equation}
The first is \eqref{eq:EscaleX}. Note the asymmetry: $\mathcal J$ carries $\Lambda^{2}$ and not
$\Lambda$, and that is why \eqref{eq:EscaleNat} and \eqref{eq:Eholder} are stated with $\Lambda^{2}$.
It is harmless, since
\[
\Lambda^{2}|t-s|^{2\gamma}\le c_{0}^{-1}\Lambda|t-s|^{\gamma}
\qquad\text{on }|t-s|\le h_{\star} .
\]
Here $C(R)$ enters only to make $\mathcal H$ finite, and the first bound of \eqref{eq:Jfinite} only to make $\mathcal J$ finite.
\item \label{proof:Eproof:step2:step1:1d} \emph{The conditional bound \eqref{eq:EscaleCond}.} The germ of Step~1 satisfies the
$\epsilon_{1}$-hypothesis of Lemma~\ref{lem:SSL} with the $\Gamma_{1}$ assembled in \ref{proof:Eproof:step2:step1:1b}, so that
lemma's second conclusion applies and yields,
\[
\big\|\E_{u}\big[\mathcal A_{u,v}-A_{u,v}\big]\big\|_{L_{m}}\ \le\ C\,\Gamma_{1}\,|v-u|^{\epsilon_{1}} ,
\]
and $\mathcal A_{u,v}-A_{u,v}=J_{u,v}$ by the identification made at \eqref{eq:SSLtwo}. It
remains to bound $\Gamma_{1}$ independent of $(k,R)$. By the defect formula of
Step~\ref{proof:Eproof:step2:step1}, $\E_{u}\delta A_{u,w,v}$ has two parts. The first is $\E_{u}R_{u,w}$ paired with
$\delta Z_{w,v}$, contributing one factor $\Lambda$ to every piece of $\E_{u}R_{u,w}$ assembled in
\ref{proof:Eproof:step2:step1:1b}. The second is $\E_{u}\delta[(Df\,f+f')(X)]_{u,w}$ paired with $\ZZ_{w,v}$, contributing two. That
second part has the two halves of \eqref{eq:Gincrement}: its spatial half is of kind~(5) below and its
temporal half, being linear in $\mathcal N_{s,t}$, is of kind~(2), so that the enumeration covers it.
The pieces are of five kinds, and the martingale term of \eqref{eq:Rsplitfive} contributes
nothing, vanishing under $\E_{u}$:
\begin{enumerate}[label=\normalfont(\arabic*),leftmargin=2.4em]
\item $\Lambda\,\mathcal N_{s,t}$: the drift of \eqref{eq:Rsplitfive}, the pieces \ref{proof:Eproof:step2:step1:1b:beta} and
\ref{proof:Eproof:step2:step1:1b:gamma} of $\E_{u}\mathcal T$ (the latter being its $\mathcal A^{M,1}$-half), and $\varrho^{(2)}$;
\item $\Lambda\cdot\Lambda^{2}\mathcal N_{s,t}=\Lambda^{3}\mathcal N_{s,t}$: the second-level term
$Df_{u}(X_{u})(Df\,f+f')_{u}(X_{u})\ZZ_{u,w}$ of \eqref{eq:Rsplitfive} and the piece $(\alpha)$ of
$\E_{u}\mathcal T$;
\item $\Lambda\,\mathcal J\le C_{3}\Lambda^{3}\mathcal N_{s,t}$: the $J$-term of
\eqref{eq:Rsplitfive} and the piece $(\delta)$ of $\E_{u}\mathcal T$;
\item $\Lambda\big(\mathcal H\mathcal N_{s,t}\big)^{1/2}\le C\Lambda^{3/2}\mathcal N_{s,t}$: the
martingale mismatch \eqref{eq:Emismatch}, namely the $\mathcal A^{M,2}$-half of $(\gamma)$ and the
one piece of order $\tfrac{1+3\gamma}{2}$;
\item $\Lambda^{2}\mathcal H\le C_{3}\Lambda^{3}\mathcal N_{s,t}$: two pieces, namely
$\varrho^{(1)}$, whose bound $\|Df'\|_{\infty}\Lambda\mathcal H$ in \eqref{eq:rhoorders} already carries
one $\Lambda$ before the one from $\delta Z_{w,v}$ is counted, and the $\ZZ$-contribution
$C_{\mathrm{lip}}\mathcal H\Lambda^{2}$.
\end{enumerate}
Here \eqref{eq:HJfinal} yields $\mathcal H\le C_{3}\Lambda\mathcal N_{s,t}$ and
$\mathcal J\le C_{3}\Lambda^{2}\mathcal N_{s,t}$, and $\Lambda\ge1$. Hence every one of the five is at
most a constant multiple of $\Lambda^{3}\mathcal N_{s,t}$, and collecting the five kinds yields
\[
\Gamma_{1}\ \le\ C\Lambda^{3}\mathcal N_{s,t} .
\]
The exponent $3$ is the largest of the five, attained at (2), (3) and (5), and it is not optimised;
what matters below is only that the bound is a constant multiple of $\mathcal N_{s,t}$ with a power of
$\Lambda$ and no dependence on $k$ or $R$. This is \eqref{eq:EscaleCond}, and $\epsilon_{1}>1$ by
\eqref{eq:eps1value}. Nothing in this step re-enters the calculations.

\end{enumerate}
\item \label{proof:Eproof:step2:step2} \noindent\emph{\textbf{The germ of $V$.}}
We repeat the computation for the germ
$\Gamma^{1}(X_{u})\delta Z_{u,v}+\Gamma^{2}(X_{u})\ZZ_{u,v}+\Gamma^{3}(X_{u})\CZ_{u,v}$.
Its defect is estimated exactly as in Step~\ref{proof:Eproof:step2:step1}, with $f$ replaced by $\Gamma^{1}$ and $Df\,f+f'$ by
$\Gamma^{2}$, and \ref{proof:Eproof:step2:a0} supplies the inputs used, since $\Gamma^{1}$ and $\Gamma^{2}$
are globally Lipschitz with constants independent of $(k,R)$, so that the spatial half satisfies
\[
\big\|\Gamma^{i}_{w}(X_{w})-\Gamma^{i}_{w}(X_{u})\big\|_{L_{m}}\ \le\ C_{1}\mathcal H\tau^{\gamma}
\]
at the same exponent $m$ and with no product of random factors. We treat the temporal half below, at
the three separate orders of \eqref{eq:Gincrement} and linearly in $\mathcal N_{s,t}$.

The one difference is that here the weight may be taken global. Only
$\Gamma^{1}$ needs a second-order expansion. Indeed, in the defect it is paired with
$\delta Z_{w,v}=O(|v-w|^{\gamma})$, so that its controlled-path remainder must be of order $2\gamma$,
and \eqref{eq:D2Gamma1} delivers it in contracted form. We can combine \eqref{eq:D2Gamma1} on the
event $\{|\delta X_{u,w}|\le\tfrac12\Phi_{u}\}$, where $V$ being $1$-Lipschitz yields
\[
V(X_{u}+\theta\delta X_{u,w})\ge\Phi_{u}-|\delta X_{u,w}|\ge\tfrac12\Phi_{u} ,
\]
with the global Lipschitz bound $|D\Gamma^{1}|\le C_{1}$ on its complement. Again pathwise on all of
$\Omega$, we obtain the following bound,
\begin{equation}\label{eq:TaylorminV}
\big|\mathcal T^{\Gamma}_{u,w}\big|
\ \le\ 4C_{1}\min\Big\{\frac{|\delta X_{u,w}|^{2}}{\Phi_{u}},\ |\delta X_{u,w}|\Big\}
\ \le\ \frac{8C_{1}|\delta X_{u,w}|^{2}}{\Phi_{u}+|\delta X_{u,w}|},
\end{equation}
where
$\mathcal T^{\Gamma}_{u,w}:=\Gamma^{1}_{u}(X_{w})-\Gamma^{1}_{u}(X_{u})-D\Gamma^{1}_{u}(X_{u})\delta X_{u,w}$.
As in Step~\ref{proof:Eproof:step2:step1}, $\mathcal T^{\Gamma}$ denotes the remainder of the expansion in the state variable. The time-increment of $\Gamma^{1}_{t}=\mathcal G^{\kappa}_{t}V$ is carried exactly as $\varrho$
was in \eqref{eq:rhoorders}. This increment is non-zero whenever $f_{t}$ depends on $t$, so
that eq.\ \eqref{eq:TaylorminV} would be false with the full two-index increment on the left. The pair
$(\Gamma^{1},\Gamma^{1\prime})$ with $\Gamma^{1\prime,\kappa\lambda}:=f'^{\kappa\lambda}\cdot DV$ is
controlled with the same indices $(\beta,\beta')$ and a norm controlled by $C_{\RZ}$ and
$\|DV\|_{\infty}\le1$. Hence its two time-increment terms satisfy \eqref{eq:rhoorders} with
$\mathcal N_{s,t},\mathcal H$ unchanged and contribute at order $\beta+\beta'+\gamma>1$ by
\eqref{eq:rhoexp}. In the autonomous case $f'\equiv0$ they vanish. The weight is now $\Phi_{u}$, for
the reason given after \eqref{eq:D2Gamma1}: $\partial_{i}V$ vanishes where $\ddot q_{i}$ is largest.
Hence \eqref{eq:Eacoord} is not used in this step. With $c=\Phi_{u}$ in
\eqref{eq:concaveJensen} and $\mathcal A:=\E_{u}|\delta X_{u,w}|^{2}$ decomposed as before, the
martingale piece satisfies
\[
\E_{u}\int\operatorname{tr}a(r,X_{r})\dr\le2C^{2}_{\mathrm{gr}}\E_{u}\int(1+|X_{r}|)^{2}\dr ,
\]
and the elementary splitting
\[
(1+|X_{r}|)^{2}\le4\Phi^{2}_{u}+4|\delta X_{u,r}|^{2} ,
\]
again by $V$ being $1$-Lipschitz, yields a first piece of order $1$ on the branch $\theta/\Phi_{u}$ and
a second of order $\tfrac12+\gamma$ on the branch $\sqrt\theta$, the latter because
\[
\|(\E_{u}\int|\delta X_{u,r}|^{2}\dr)^{1/2}\|_{L_{m}}\le(\int\|\delta X_{u,r}\|^{2}_{L_{m}}\dr)^{1/2}\le\mathcal H\tau^{\frac{1+2\gamma}{2}} .
\]
Since $2\gamma<\tfrac12+\gamma$ for $\gamma<\tfrac12$, the binding order among the pieces of
$\mathcal T^{\Gamma}$ is $2\gamma$, and we obtain the following exponents,
\[
\epsilon^{V}_{1}\ =\ \min\big\{3\gamma,\ \beta+\beta'+\gamma\big\}\ >\ 1,
\qquad
\epsilon^{V}_{2}\ =\ 2\gamma\ >\ \tfrac12 ,
\]
both again with constants linear in $\mathcal N_{s,t}$ and $\mathcal H$ at exponent $m$. Routing this
step through \eqref{eq:Eacoord} and the coordinatewise weight would be legitimate, but it would give
the worse $\epsilon^{V}_{1}=\tfrac12(1+3\gamma)$. We use the global weight here.

We do not perform a Taylor expansion of $\Gamma^{2}$, and none can be performed: $D^{2}\Gamma^{2}$ contains
$f\,(D^{3}f)\,DV=O(V)$, which does not decay, Assumption~\ref{ass:Erough} bounding $\|D^{3}f\|_{\infty}$ while the
curvature hypothesis constrains only $D^{2}f$, and it contains $(D^{2}f')\,DV=O(1)$, there being by
design no curvature condition on $f'$. We actually do not need one, since
$\delta\Gamma^{2}$ and $\delta\Gamma^{3}$ are paired with $\ZZ_{w,v}=O(|v-w|^{2\gamma})$ and
$\CZ_{w,v}=O(|v-w|)$, where the first-order Lipschitz bound of \ref{proof:Eproof:step2:a0} already gives order
$3\gamma>1$ for $\epsilon_{1}$, through
\[
\|\E_{u}|\delta X_{u,w}|\|_{L_{m}}\le\mathcal H\tau^{\gamma}
\]
unconditionally, and more than $\tfrac12$ for $\epsilon_{2}$.

$\Gamma^{2}$ and $\Gamma^{3}$ inherit the time dependence of $(f,f')$, so that $\delta\Gamma^{2}$ and
$\delta\Gamma^{3}$ carry a temporal increment alongside the spatial one, exactly as \eqref{eq:Gsplit}
does, and we can estimate it the same way. Differentiating the summands of $\Gamma^{2}$ in time rather
than in space produces $\delta f\cdot Df\cdot DV$, $f\cdot\delta(Df)\cdot DV$,
$\delta f\cdot f\cdot D^{2}V$ and $\delta f'\cdot DV$. All four are linear in the state at the
relevant order. The first and the fourth are, because $\|DV\|_{\infty}\le1$ and $\delta f,\delta f'$
are weighted once by \eqref{eq:cvfdefect}. The third is, because \eqref{eq:Vjet} yields the following bound,
\[ |f|^{2}\Op{D^{2}V}\le C_{\mathrm{gr}}^{2}(1+|x|)^{2}V^{-1}\le CV ,
\]
the Lyapunov weight in the denominator absorbing the second power. The second is, because
\eqref{eq:Eroughderiv} is unweighted in $Df$, so that
\[
|f|\,|\delta(Df)|\le C_{\mathrm{gr}}(1+|x|)\big(\|Df'\|_{\infty}\Lambda\tau^{\gamma}+C^{(1)}_{\RZ}\tau^{\beta+\beta'}\big) .
\]
For
$\Gamma^{3}=\tfrac12\mathcal Q^{\kappa\lambda}V=\tfrac12\textstyle\sum_{i,j}f^{\kappa}_{i}f^{\lambda}_{j}\partial^{2}_{ij}V$, being half of one summand of $\Gamma^{2}$ and not a copy of it, only the third of the four forms
occurs and the bound holds consequently. The resulting temporal orders are the three of
\eqref{eq:Gincrement}, namely $\gamma$ with a factor $\Lambda$ and $\beta+\beta'$ and $\beta'$
without one, and we keep them apart here for the same reason. Against $\ZZ_{w,v}$ they give
\[
\Lambda^{3}\mathcal N_{s,t}\tau^{3\gamma},
\qquad
\Lambda^{2}\mathcal N_{s,t}\tau^{2\gamma+\beta+\beta'},
\qquad
\Lambda^{2}\mathcal N_{s,t}\tau^{2\gamma+\beta'} ,
\]
and against $\CZ_{w,v}$, whose size is $\Lambda\tau$ by \eqref{eq:bracketLip}, they give
\[
\Lambda^{2}\mathcal N_{s,t}\tau^{1+\gamma},
\qquad
\Lambda\mathcal N_{s,t}(\tau^{1+\beta+\beta'}+\tau^{1+\beta'}) .
\]
All of these are absorbed as in \ref{proof:Eproof:step2:step1:1c}.
Lemma~\ref{lem:SSL} and the same
absorption as in \ref{proof:Eproof:step2:step1:1c} then yield, for each fixed $r\in[s,t]$,
\[
\|\Phi^{\natural}_{s,r}\|_{L_{m}}\le C\Lambda^{2}\mathcal N_{s,t}|r-s|^{2\gamma} .
\]
The bound just proved is a bound on the sewing of the germ. Lemma~\ref{lem:Escalefree}
instead defines $\Phi^{\natural}$ through the integrals of Lemma~\ref{lem:roughItoLn}. The two
objects agree, and one sequence of partitions of mesh tending to $0$ shows it. Fix a pair and take
partitions containing its left endpoint. Lemma~\ref{lem:SSL} gives the sewing as the $L_{m}$ limit
of the compensated Riemann sums of the germ. In those sums the $\delta Z$ and $\ZZ$ parts converge
to $\int(\Gamma^{1},\Gamma^{2})\d\RZ$ by \cite[Thm.~3.4]{FHL}, in probability and uniformly in time.
The integral $\int(\Gamma^{1},\Gamma^{2})\d\RZ$ at the composed pair is defined by
Lemma~\ref{lem:roughItoLn}\ref{rIto:class}, whose hypotheses hold here by
Proposition~\ref{prop:Athreshuses}\ref{app:Athresh2}. The $\CZ$ part converges pathwise to
$\int\Gamma^{3}\d\CZ$, the integrator being the increment of a Lipschitz path and the integrand
almost surely continuous. A limit in probability is almost surely unique, so the two objects agree
at each pair.
\item \label{proof:Eproof:step2:step3} \emph{\textbf{ From a fixed pair to the supremum.}}
\eqref{eq:EscaleNat} requires $\|\sup_{r\in[s,t]}|\Phi^{\natural}_{s,r}|\|_{L_{m}}$, and this needs one
further ingredient. That is where the second entry of \eqref{eq:Ep} is used. For $s\le r'\le r\le t$,
and with the three-index convention of \S\ref{subsec:rp},
additivity of the sewing yields the exact identity
\[
\Phi^{\natural}_{s,r}\ =\ \Phi^{\natural}_{s,r'}+\Phi^{\natural}_{r',r}-\delta A_{s,r',r},
\qquad
\Phi^{\natural}_{u,v}:=\mathcal A_{v}-\mathcal A_{u}-A_{u,v},
\]
Let us fix $r\in[s,t]$ and denote by $\mathfrak s_{k}$ the largest dyadic point $s+j2^{-k}(t-s)$,
$0\le j\le2^{k}$, such that $\mathfrak s_{k}\le r$, so that $\mathfrak s_{0}=s$,
$\mathfrak s_{k}\uparrow r$ and $|\mathfrak s_{k+1}-\mathfrak s_{k}|\in\{0,2^{-(k+1)}(t-s)\}$. The next
step uses continuity of $v\mapsto\Phi^{\natural}_{s,v}$, and it is not the continuity part of Lemma~\ref{lem:SSL} that delivers it. By the rough It\^o formula (Lemma~\ref{lem:roughItoLn}), $\Phi^{\natural}_{s,v}$ equals $V(X_{v})-V(X_{s})$ minus the drift integral,
minus the It\^o integral, and minus the three frozen germ terms $\Gamma^{1}(X_{s})\delta Z_{s,v}$,
$\Gamma^{2}(X_{s})\ZZ_{s,v}$ and $\Gamma^{3}(X_{s})\CZ_{s,v}$. Every one of those is almost surely
continuous in $v$, the last three because $Z$, $\ZZ$ and $\CZ$ are continuous deterministic paths.
Applying the identity with $(s,r',r)=(s,\mathfrak s_{k},\mathfrak s_{k+1})$, telescoping and using that
continuity, we obtain the following identity,
\[
\Phi^{\natural}_{s,r}\ =\ \sum_{k\ge0}\Big(\Phi^{\natural}_{\mathfrak s_{k},\mathfrak s_{k+1}}
-\delta A_{s,\,\mathfrak s_{k},\,\mathfrak s_{k+1}}\Big),
\]
whence, the pair $(\mathfrak s_{k},\mathfrak s_{k+1})$ ranging over at most $2^{k+1}$ possibilities for
each $k$, we obtain the following bound,
\[
\sup_{r\in[s,t]}\big|\Phi^{\natural}_{s,r}\big|
\ \le\ \sum_{k\ge0}\ \max_{0\le j<2^{k+1}}\Big(\big|\Phi^{\natural}_{r_{k+1,j},r_{k+1,j+1}}\big|
+\big|\delta A_{s,\,r_{k+1,j},\,r_{k+1,j+1}}\big|\Big),
\qquad r_{k,j}:=s+j2^{-k}(t-s),
\]
and $\|\max_{j<2^{k+1}}|Y_{j}|\|_{L_{m}}\le2^{(k+1)/m}\max_{j}\|Y_{j}\|_{L_{m}}$. Taking $L_{m}$-norms,
the first family contributes
\[
\sum_{k}2^{k/m}\big(2^{-k}|t-s|\big)^{2\gamma} ,
\]
convergent because $2\gamma>\tfrac12\ge\tfrac1m$. The second contributes
\[
\sum_{k}2^{k/m}\big(2^{-k}\big)^{\gamma}|t-s|^{\epsilon_{2}} ,
\]
the defect carrying an explicit factor $\delta Z$, $\ZZ$ or $\CZ$ on the short interval, the last
of order $|v-w|\le|v-w|^{\gamma}$, and this converges
precisely when $\gamma m>1$. That is the second entry of \eqref{eq:Ep}, used again at the Kolmogorov--Chentsov step of Step~\ref{step:Etight}. Both sums are bounded by
$C(\gamma,m)|t-s|^{2\gamma}$. The $\Phi^{\natural}$ chained here is the one Lemma~\ref{lem:Escalefree} states, $\CZ$-terms
included, so this is \eqref{eq:EscaleNat}. Its $\CZ$-part is what the Preliminaries bound pathwise, of
order $1$ and hence of order $2\gamma$ on $|t-s|\le h_{\star}\le1$ since
$2\gamma<1$.
\end{enumerate}%
\end{proof}

\subsection{The envelope in Proposition~\ref{prop:MPequiv}: Steps 7--13}\label{app:MPenvelope}
These steps continue the proof of Proposition~\ref{prop:MPequiv}\ref{mpe1} from Step~6.

\emph{Step~7: the three terms the envelope cannot carry, and the envelope
\eqref{eq:Dphibound}.}\stepnum{7}{stp:MP7} What remains to be subtracted from \eqref{eq:MPleftover} is
$\int_{s}^{t}\mathcal L_{r}\varphi(X_{r})\dr+\delta N^{\varphi}_{s,t}$, and this is where the three
mismatches appear. Every term of \eqref{eq:MPleftover} carries at least one derivative of $\varphi$,
so $D^{\varphi}$ is expressed in terms of $J$ and of Taylor remainders of $\varphi$ of order $\le3$;
three terms of that expression are not pointwise dominated by a $\varphi$-free quantity, and all three
have to be set aside before any envelope is written. The first-order Taylor term of
\eqref{eq:MPtaylor} contributes $\nabla\varphi(X_{s})\cdot\int_{s}^{t}\sigma_{r}(X_{r})\d B_{r}$, while
\eqref{eq:roughMP} subtracts
$\delta N^{\varphi}_{s,t}=\int_{s}^{t}\nabla\varphi(X_{r})\cdot\sigma_{r}(X_{r})\d B_{r}$, and their difference is
\[
\nabla\varphi(X_{s})\cdot\int_{s}^{t}\sigma_{r}(X_{r})\d B_{r}-\delta N^{\varphi}_{s,t}
\ =\ -\,\delta\tilde N^{\varphi}_{s,t},
\]
where
\begin{equation}\label{eq:Ntilde}
\delta\tilde N^{\varphi}_{s,t}:=\int_{s}^{t}\big(\nabla\varphi(X_{r})-\nabla\varphi(X_{s})\big)
\cdot\sigma_{r}(X_{r})\,\d B_{r}
\end{equation}
denotes a stochastic integral whose integrand depends on $\varphi$. A pathwise majorant of the form
$\tnorm\varphi\cdot(\varphi\text{-free})$ is not available for it. Every other occurrence of $\int\sigma\d B$ in the expansion carries a
$\mathcal G_{s}$-measurable prefactor $D^{j}\varphi(X_{s})$ and is pointwise dominated.

The other two escape for a reason about the shape of the
envelope rather than about $\varphi$: they are the companion mismatches in the Lebesgue
integrals. Recall from \eqref{eq:gendef} that
$\mathcal L_{r}\varphi=\tfrac12a_{r}\!:\!D^{2}\varphi+b_{r}\cdot\nabla\varphi$ with
$a=\sigma\sigma^{\top}$. That same first-order Taylor term contributes
$\nabla\varphi(X_{s})\cdot\int_{s}^{t}b_{r}(X_{r})\dr$ against the $b$-half of
$\int_{s}^{t}\mathcal L_{r}\varphi(X_{r})\dr$, and the second-order term contributes
$\tfrac12D^{2}\varphi(X_{s})\!:\!\int_{s}^{t}a_{r}(X_{r})\dr$ against its $a$-half, the two differences being
\begin{align*}
\nabla\varphi(X_{s})\cdot\int_{s}^{t}b_{r}(X_{r})\dr
-\int_{s}^{t}\nabla\varphi(X_{r})\cdot b_{r}(X_{r})\dr
\ &=\ -\,\tilde P^{\varphi}_{s,t},\\
\tfrac12D^{2}\varphi(X_{s})\!:\!\int_{s}^{t}a_{r}(X_{r})\dr
-\tfrac12\int_{s}^{t}D^{2}\varphi(X_{r})\!:\!a_{r}(X_{r})\dr
\ &=\ -\,\tilde Q^{\varphi}_{s,t},
\end{align*}
where
\begin{equation}\label{eq:PQtilde}
\tilde P^{\varphi}_{s,t}:=\int_{s}^{t}\big(\nabla\varphi(X_{r})-\nabla\varphi(X_{s})\big)
\cdot b_{r}(X_{r})\dr,
\qquad
\tilde Q^{\varphi}_{s,t}:=\tfrac12\int_{s}^{t}\big(D^{2}\varphi(X_{r})-D^{2}\varphi(X_{s})\big)
\!:\!a_{r}(X_{r})\dr,
\end{equation}
and nothing else of $\int_{s}^{t}\mathcal L_{r}\varphi(X_{r})\dr$ survives the cancellation: its
$b$-half is absorbed by the first-order Taylor term up to $\tilde P^{\varphi}$, and its $a$-half by the
It\^o compensator of $\big(\int_{s}^{t}\sigma_{r}(X_{r})\d B_{r}\big)^{\otimes2}$ up to
$\tilde Q^{\varphi}$. These two are not dominated under the integral sign by a $\varphi$-free
quantity of the envelope. Their majorants follow from
$\|D^{2}\varphi\|_{\infty}\vee\|D^{3}\varphi\|_{\infty}\le\tnorm\varphi$ and from \eqref{eq:Egrowth} in
the two forms $|b_{r}(X_{r})|\le C_{\mathrm{gr}}(1+|X_{r}|)$ and
$|a_{r}(X_{r})|\le C_{\mathrm{gr}}^{2}(1+|X_{r}|)^{2}$:
\begin{align*}
\big|\tilde P^{\varphi}_{s,t}\big|
\ &\le\
\tnorm\varphi\,C_{\mathrm{gr}}\big(1+\sup_{r\in[s,t]}|X_{r}|\big)\int_{s}^{t}\big|\delta X_{s,r}\big|\dr
\ \le\ \tnorm\varphi\,C_{\mathrm{gr}}\big(1+\sup_{r\in[s,t]}|X_{r}|\big)\,|t-s|
\sup_{r\in[s,t]}\big|\delta X_{s,r}\big| ,\\
\big|\tilde Q^{\varphi}_{s,t}\big|
\ &\le\
\tfrac12\tnorm\varphi\,C_{\mathrm{gr}}^{2}\big(1+\sup_{r\in[s,t]}|X_{r}|\big)^{2}
\int_{s}^{t}\big|\delta X_{s,r}\big|\dr\\
\ &\le\ \tfrac12\tnorm\varphi\,C_{\mathrm{gr}}^{2}\big(1+\sup_{r\in[s,t]}|X_{r}|\big)^{2}\,|t-s|
\sup_{r\in[s,t]}\big|\delta X_{s,r}\big| ,
\end{align*}
that is, the majorant of $\tilde Q^{\varphi}$ is that of $\tilde P^{\varphi}$ with
$\tfrac12C_{\mathrm{gr}}^{2}(1+\sup_{r\in[s,t]}|X_{r}|)^{2}$ in place of
$C_{\mathrm{gr}}(1+\sup_{r\in[s,t]}|X_{r}|)$. Both majorants carry the running supremum
$\sup_{r\in[s,t]}|\delta X_{s,r}|$, which the envelope $\Xi_{s,t}$ below cannot accommodate. So the terms set
aside form a family of three, and what the Taylor expansion produces is the decomposition
\begin{equation}\label{eq:Dphibound}
\begin{gathered}
D^{\varphi}_{s,t}=-\,\mathcal K^{\varphi}_{s,t}+R^{\varphi}_{s,t},
\qquad
\mathcal K^{\varphi}_{s,t}:=\delta\tilde N^{\varphi}_{s,t}+\tilde P^{\varphi}_{s,t}
+\tilde Q^{\varphi}_{s,t},
\qquad
\big|R^{\varphi}_{s,t}\big|\ \le\ \tnorm\varphi\cdot\Xi_{s,t}\quad\text{pointwise},\\
\Xi_{s,t}:=C\Big(|J_{s,t}|+\sum_{j=0}^{3}\big(1+\sup_{r\in[s,t]}|X_{r}|\big)^{3-j}
\big|\delta X_{s,t}\big|^{j}\,\Theta^{(j)}_{s,t}\Big),
\end{gathered}
\end{equation}
together with, for $\delta\tilde N^{\varphi}$, the following $L_{m}$-bound, proved in
Step~\ref{stp:MP9},
\begin{equation}\label{eq:Ntildebound}
\big\|\delta\tilde N^{\varphi}_{s,t}\big\|_{L_{m}}
\ \le\ C_{m}C_{\mathrm{gr}}\,\tnorm\varphi\,
\sup_{r\in[s,t]}\big\|\delta X_{s,r}\big\|_{L_{2m}}
\Big(1+\big\|\sup_{t\le T}|X_{t}|\big\|_{L_{2m}}\Big)\,|t-s|^{1/2}
\ =\ O\big(|t-s|^{\gamma+1/2}\big).
\end{equation}
The two Lebesgue mismatches set aside are bounded in $L_{m}$ likewise, by Minkowski's integral
inequality alone \color{blue}(again Step~9)\color{black}:
\begin{equation}\label{eq:PQtildebound}
\begin{aligned}
\big\|\tilde P^{\varphi}_{s,t}\big\|_{L_{m}}+\big\|\tilde Q^{\varphi}_{s,t}\big\|_{L_{m}}
\ &\le\ C_{\mathrm{gr}}\big(1+\tfrac12C_{\mathrm{gr}}\big)\,\tnorm\varphi\,
\Big(1+\big\|\sup_{t\le T}|X_{t}|\big\|_{L_{3m}}\Big)^{2}
\sup_{r\in[s,t]}\big\|\delta X_{s,r}\big\|_{L_{3m}}\,|t-s|\\
&\ =\ O\big(|t-s|^{1+\gamma}\big).
\end{aligned}
\end{equation}
Here the $\Theta^{(j)}$ denote $\varphi$-free quantities, with $\Theta^{(j)}\equiv1$ for $j\ge2$, while
$\Theta^{(0)}$ and $\Theta^{(1)}$ are the following explicit quantities. We set
\begin{equation}\label{eq:Thetaatoms}
\mathsf{S}(X,s,t):=1+\sup_{r\in[s,t]}|X_{r}|,
\qquad
\mathsf m_{s,t}:=\frac{\big|\int_{s}^{t}\sigma_{r}(X_{r})\d B_{r}\big|}{\mathsf{S}(X,s,t)} ,
\end{equation}
and set
\begin{equation}\label{eq:Thetadef}
\Theta^{(1)}_{s,t}:=|t-s|+\big|\delta Z_{s,t}\big|+\big|\ZZ_{s,t}\big|+\mathsf m_{s,t} ,
\qquad
\Theta^{(0)}_{s,t}:=|t-s|+\big(\Theta^{(1)}_{s,t}\big)^{2} .
\end{equation}
Both are built from the four quantities $|t-s|$, $|\delta Z_{s,t}|$, $|\ZZ_{s,t}|$ and
$\mathsf m_{s,t}$, and every term of $\Theta^{(0)}$ is a product of exactly two of them with a single
exception, the factor $|t-s|$, which Remark~\ref{rem:MPtsexception} shows to be
harmless.

 As in \eqref{eq:Dphisplit} and \eqref{eq:MUcirc} of Step~\ref{stp:MP10}, we  write $R^{\varphi}=W^{\varphi}J+\mathcal M^{\varphi,\circ}
+\Upsilon^{\varphi,\circ}$. Reading
\eqref{eq:Egrowth} and \eqref{eq:MProughgrowth} off \eqref{eq:MPpieces} gives
\[
\begin{gathered}
|A_{s,t}|\le C_{\mathrm{gr}}\mathsf{S}(X,s,t)|t-s|,
\qquad
|M_{s,t}|=\mathsf{S}(X,s,t)\,\mathsf m_{s,t},
\\
|G^{1}_{s,t}|\le C_{\mathrm{gr}}\mathsf{S}(X,s,t)\big|\delta Z_{s,t}\big|,
\qquad
|G^{2}_{s,t}|\le C_{\mathrm{gr}}\mathsf{S}(X,s,t)\big|\ZZ_{s,t}\big| ,
\end{gathered}
\]
so that each of the four contributes one power of $\mathsf{S}(X,s,t)$ and one of the four quantities
above. Bounding $\|D^{j}\varphi\|_{\infty}\le\tnorm\varphi$ for $j=1,2,3$ and dividing by
$\tnorm\varphi$, the summands contribute as follows, at the stated power $j$ of
$\big|\delta X_{s,t}\big|$ and up to the constant $C$ of \eqref{eq:Dphibound}:
\begin{center}
\begin{tabular}{@{}ll@{}}
\toprule
summand of $R^{\varphi}_{s,t}$ & contributes \\
\midrule
$\nabla\varphi(X_{s})\cdot J_{s,t}$ & the term $|J_{s,t}|$\\
$D^{2}\varphi(X_{s})\big(\,\cdot\,,G^{1}_{s,t}+G^{2}_{s,t}\big)J_{s,t}$
  & $|\delta Z_{s,t}|+|\ZZ_{s,t}|$ at $j=1$;
    $\big(|\delta Z_{s,t}|+|\ZZ_{s,t}|\big)\Theta^{(1)}_{s,t}$ at $j=0$\\
$\tfrac12D^{2}\varphi(X_{s})\!:\!\big(M^{\otimes2}_{s,t}-\int_{s}^{t}a_{r}(X_{r})\dr\big)$
  & $\mathsf m^{2}_{s,t}$ and the lone $|t-s|$, both at $j=0$\\
$D^{2}\varphi(X_{s})\big(G^{1}_{s,t}+G^{2}_{s,t},M_{s,t}\big)$
  & $\big(|\delta Z_{s,t}|+|\ZZ_{s,t}|\big)\mathsf m_{s,t}$ at $j=0$\\
$D^{2}\varphi(X_{s})\big(G^{1}_{s,t},G^{2}_{s,t}\big)
   +\tfrac12D^{2}\varphi(X_{s})\!:\!\big(G^{2}_{s,t}\big)^{\otimes2}$
  & $|\delta Z_{s,t}||\ZZ_{s,t}|+|\ZZ_{s,t}|^{2}$ at $j=0$\\
$D^{2}\varphi(X_{s})\big(M_{s,t},J_{s,t}\big)$
  & $\mathsf m_{s,t}$ at $j=1$; $\mathsf m_{s,t}\Theta^{(1)}_{s,t}$ at $j=0$\\
$\tfrac12D^{2}\varphi(X_{s})\!:\!J^{\otimes2}_{s,t}$
  & $1$ at $j=2$; $\Theta^{(1)}_{s,t}$ at $j=1$; $\big(\Theta^{(1)}_{s,t}\big)^{2}$ at $j=0$\\
$D^{2}\varphi(X_{s})\big(A_{s,t},\delta X_{s,t}-\tfrac12A_{s,t}\big)$
  & $|t-s|$ at $j=1$; $|t-s|^{2}$ at $j=0$\\
$\mathcal R^{\varphi,(3)}_{s,t}$ & $1$ at $j=3$\\
\bottomrule
\end{tabular}
\end{center}
The entries at $j=1$ add up to at most $2\Theta^{(1)}_{s,t}$, the four distinct quantities
appearing there being exactly $|t-s|$, $|\delta Z_{s,t}|$, $|\ZZ_{s,t}|$ and $\mathsf m_{s,t}$.
Every entry at $j=0$ is either $|t-s|$ or bounded by $\big(\Theta^{(1)}_{s,t}\big)^{2}$,
and the entries at $j=2$ and $j=3$ are $1$. This is \eqref{eq:Thetadef}, the multiplicities being
absorbed by the constant $C$. The three rows carrying $J_{s,t}$ beside a second factor are the ones Step~\ref{stp:MP8}
treats, and are the only ones where \eqref{eq:Jresub} is used.

The bracket enters the expansion only through
$\tfrac12D^{2}\varphi(X_{s})\!:\!(G^{1}_{s,t})^{\otimes2}$ by way of \eqref{eq:brackettaylor}, and
\eqref{eq:MPcancel2c} absorbs that summand into the germ, so no $\CZ$ survives into
$R^{\varphi}$. Carrying $|\CZ_{s,t}|$ in \eqref{eq:Thetadef} would change no rate below,
\eqref{eq:bracketLip} making it $O(|t-s|)$. They are
not state-free: the last of the listed factors is a functional of the path, and the only
property used of it is that the normalisation by $1+\sup_{r\in[s,t]}|X_{r}|$ leaves it of state degree
$0$. That normalisation is what keeps the total degree at $3$: $\int_{s}^{t}\sigma_{r}(X_{r})\d B_{r}$
grows linearly in the state by \eqref{eq:Egrowth}, and an unnormalised factor would add a fourth power
to a term already carrying $(1+\sup_{r}|X_{r}|)^{3-j}|\delta X_{s,t}|^{j}$. The total degree in the
state is therefore $3$ in every term of the sum in $\Xi$, and not $6$, the power of the supremum and the power of
$\delta X$ being complementary,
\[
\underbrace{(3-j)}_{\text{from }\ (1+\sup_{r\in[s,t]}|X_{r}|)^{3-j}}
\ +\ \underbrace{j}_{\text{from }\ |\delta X_{s,t}|^{j}}\ =\ 3,
\qquad j=0,1,2,3,
\]
because substituting \eqref{eq:davie} for $\delta X$ trades one power of $|\delta X|$ for one
coefficient factor and leaves the total degree unchanged. The term carrying $|\delta X_{s,t}|^{3}$ is
the third-order Taylor remainder of $\varphi$, and carries no coefficient factor.

\emph{Step~8: the three terms in which $J$ stands beside a second factor.}\stepnum{8}{stp:MP8} The summands
$D^{2}\varphi(X_{s})(\,\cdot\,,G^{1}_{s,t}+G^{2}_{s,t})J_{s,t}$,
$D^{2}\varphi(X_{s})(M_{s,t},J_{s,t})$ and
$\tfrac12D^{2}\varphi(X_{s})\!:\!J^{\otimes2}_{s,t}$ of \eqref{eq:MPleftover} are not of the form just
considered: $|J_{s,t}|$ is no factor of $\Theta^{(0)}$ or of $\Theta^{(1)}$, each summand carries
$|J_{s,t}|$ multiplied by a further factor that the bare $C|J_{s,t}|$ of \eqref{eq:Dphibound} does not
cover, and neither of those factors is $\varphi$-free of the listed kind, $M_{s,t}$ carrying no
deterministic bound at all and $G^{1}_{s,t}+G^{2}_{s,t}$ a coefficient evaluated along the path. We
put the first two summands in the stated form by substituting \eqref{eq:davie} once more, in the form
\begin{equation}\label{eq:Jresub}
\big|J_{s,t}\big|\ \le\ \big|\delta X_{s,t}\big|+\big|A_{s,t}\big|+\big|M_{s,t}\big|
+\big|G^{1}_{s,t}\big|+\big|G^{2}_{s,t}\big| ,
\end{equation}
and bounding each of the five summands with \eqref{eq:Egrowth} and
\eqref{eq:MProughgrowth}, that is,
$|A_{s,t}|\le C_{\mathrm{gr}}(1+\sup_{r}|X_{r}|)|t-s|$, $|M_{s,t}|=(1+\sup_{r}|X_{r}|)\cdot
\big|\int_{s}^{t}\sigma\d B\big|/(1+\sup_{r}|X_{r}|)$, $|G^{1}_{s,t}|\le C_{\mathrm{gr}}
(1+|X_{s}|)|\delta Z_{s,t}|$ and $|G^{2}_{s,t}|\le C_{\mathrm{gr}}(1+|X_{s}|)|\ZZ_{s,t}|$. Every
product that results carries either one power of $|\delta X_{s,t}|$ beside one listed factor, and so
sits at $j=1$, or two listed factors and no $|\delta X_{s,t}|$, and so sits at $j=0$; in both cases the
state degree is $3$, since each substituted coefficient factor contributes exactly the power of
$1+\sup_{r}|X_{r}|$ that the power of $|\delta X_{s,t}|$ gives up. In particular no term of
$\Theta^{(0)}$ produced this way stands alone, the smallest being $|t-s|$ multiplied by
$\big|\int_{s}^{t}\sigma\d B\big|/(1+\sup_{r}|X_{r}|)$.

The third summand carries $|J_{s,t}|$ twice, and we apply \eqref{eq:Jresub} to each of the
two factors. Each of the twenty-five products that result is bounded using
\eqref{eq:Egrowth} and \eqref{eq:MProughgrowth}, exactly as above. Denote by $j\in\{0,1,2\}$ the number of factors equal to
$|\delta X_{s,t}|$. The remaining $2-j$ factors are each one of $|A_{s,t}|,|M_{s,t}|,
|G^{1}_{s,t}|,|G^{2}_{s,t}|$, and each of those contributes one power of $1+\sup_{r\in[s,t]}|X_{r}|$
together with one of the quantities listed for $\Theta^{(0)}$ and $\Theta^{(1)}$. So the product sits
at $|\delta X_{s,t}|^{j}$ with state degree $2-j$ and with $2-j$ of the listed factors beside it, and
it is dominated by the term of $\Xi_{s,t}$ at that same $j$, which carries
$(1+\sup_{r\in[s,t]}|X_{r}|)^{3-j}$, one power to spare at every $j$. At $j=2$ the product is
$|\delta X_{s,t}|^{2}$ and $\Theta^{(2)}\equiv1$; at $j=0$ it is a product of two of the listed
quantities, so it is again not the lone factor $|t-s|$ of Remark~\ref{rem:MPtsexception}.

\emph{Step~9: size, conditional mean and degree of the three terms set aside.}\stepnum{9}{stp:MP9}
Nothing is lost by setting aside the family
$\mathcal K^{\varphi}=\delta\tilde N^{\varphi}+\tilde P^{\varphi}+\tilde Q^{\varphi}$ of
\eqref{eq:Dphibound}, and three things have to be settled about the three quantities appearing on the left.

\emph{(i) Their size.}\stepnum{(i)}{stp:MP9i} Applying Burkholder--Davis--Gundy at exponent $m\ge2$ and then Minkowski's
integral inequality in $L_{m/2}$, we obtain
\begin{align*}
\big\|\delta\tilde N^{\varphi}_{s,t}\big\|_{L_{m}}
\ &\le\ C_{m}\Big\|\int_{s}^{t}\big|\nabla\varphi(X_{r})-\nabla\varphi(X_{s})\big|^{2}
\,\Frob{\sigma_{r}(X_{r})}^{2}\dr\Big\|_{L_{m/2}}^{1/2}\\
\ &\le\ C_{m}\Big(\int_{s}^{t}\big\|\,\big|\nabla\varphi(X_{r})-\nabla\varphi(X_{s})\big|^{2}
\,\Frob{\sigma_{r}(X_{r})}^{2}\,\big\|_{L_{m/2}}\dr\Big)^{1/2},
\end{align*}
with $C_{m}$ denoting the Burkholder constant. 
$|\nabla\varphi(X_{r})-\nabla\varphi(X_{s})|\le\|D^{2}\varphi\|_{\infty}|\delta X_{s,r}|\le\tnorm\varphi\,|\delta X_{s,r}|$,
and that the linear growth \eqref{eq:Egrowth} of $\sigma$ gives
$\Frob{\sigma_{r}(X_{r})}\le C_{\mathrm{gr}}(1+|X_{r}|)$ with its constant $C_{\mathrm{gr}}$. Applying
Cauchy--Schwarz at the conjugate pair $(2,2)$ inside the norm, we obtain
\[
\big\|\,\big|\delta X_{s,r}\big|^{2}\big(1+|X_{r}|\big)^{2}\,\big\|_{L_{m/2}}
\ \le\ \big\|\,\big|\delta X_{s,r}\big|^{2}\,\big\|_{L_{m}}\,\big\|\big(1+|X_{r}|\big)^{2}\big\|_{L_{m}}
\ =\ \big\|\delta X_{s,r}\big\|^{2}_{L_{2m}}\,\big\|1+|X_{r}|\big\|^{2}_{L_{2m}} ,
\]
and hence we obtain
\begin{align*}
\big\|\delta\tilde N^{\varphi}_{s,t}\big\|_{L_{m}}
\ &\le\ C_{m}C_{\mathrm{gr}}\,\tnorm\varphi\Big(1+\big\|\sup_{t\le T}|X_{t}|\big\|_{L_{2m}}\Big)
\Big(\int_{s}^{t}\big\|\delta X_{s,r}\big\|^{2}_{L_{2m}}\dr\Big)^{1/2}\\
\ &\le\ C_{m}C_{\mathrm{gr}}\,\tnorm\varphi\,\sup_{r\in[s,t]}\big\|\delta X_{s,r}\big\|_{L_{2m}}
\Big(1+\big\|\sup_{t\le T}|X_{t}|\big\|_{L_{2m}}\Big)\,|t-s|^{1/2} ,
\end{align*}
the inequality of \eqref{eq:Ntildebound}. It remains to identify the rate. By \eqref{eq:davie},
Minkowski's integral inequality for the $\dr$-integral, Burkholder--Davis--Gundy for the
$\d B$-integral, and \eqref{eq:Egrowth} together with \eqref{eq:MProughgrowth} for the coefficients,
\begin{align*}
\big\|\delta X_{s,r}\big\|_{L_{2m}}
\ &\le\ \Big\|\int_{s}^{r}b_{u}(X_{u})\du\Big\|_{L_{2m}}
+\Big\|\int_{s}^{r}\sigma_{u}(X_{u})\,\d B_{u}\Big\|_{L_{2m}}\\
&\phantom{xx}{}+\big\|f_{s}(X_{s})\big\|_{L_{2m}}\big|\delta Z_{s,r}\big|
+\big\|(Df\,f+f')(X_{s})\big\|_{L_{2m}}\big|\ZZ_{s,r}\big|
+\big\|J_{s,r}\big\|_{L_{2m}}\\
\ &\le\ C\Big(1+\big\|\sup_{t\le T}|X_{t}|\big\|_{L_{2m}}\Big)
\Big(|r-s|+|r-s|^{1/2}+|r-s|^{\gamma}+|r-s|^{2\gamma}\Big)
+\big\|J_{s,r}\big\|_{L_{2m}} ,
\end{align*}
with $C$ depending on $C_{\mathrm{gr}}$, on the H\"older norms of $\RZ$, on $T$ and on $m$. Of the four
exponents $1,\tfrac12,\gamma,2\gamma$ the third is the smallest, by $\gamma<\tfrac12<1$ and
$\gamma<2\gamma$, so that the third term dominates once the powers $T^{a-\gamma}$ are absorbed into
$C$, and we obtain
\[
\sup_{r\in[s,t]}\|\delta X_{s,r}\|_{L_{2m}}\le C(1+\|\sup_{t\le T}|X_{t}|\|_{L_{2m}})|t-s|^{\gamma}
+\sup_{r\in[s,t]}\|J_{s,r}\|_{L_{2m}}=O(|t-s|^{\gamma}),
\]
the last step by \eqref{eq:daviemoduli} at order $2m\le m_{0}$, at a modulus taken
non-decreasing near the diagonal, and, away from it, by
$\sup_{s\le r}\|J_{s,r}\|_{L_{2m}}<\infty$, which \eqref{eq:davie}, \eqref{eq:Egrowth},
\eqref{eq:MProughgrowth}, Burkholder--Davis--Gundy and the hypothesis
$\|\sup_{t\le T}|X_{t}|\|_{L_{m'}}<\infty$ of \ref{mpe1} give at $2m\le3m\le m'$. Whence the rate
$|t-s|^{\gamma}\cdot|t-s|^{1/2}=|t-s|^{\gamma+1/2}$. Minkowski's inequality does the same work for
$\tilde P^{\varphi}$ and $\tilde Q^{\varphi}$: taking the
$L_{m}$-norm inside the $\dr$-integral replaces $\sup_{r\in[s,t]}|\delta X_{s,r}|$ by
$\sup_{r\in[s,t]}\|\delta X_{s,r}\|_{L_{3m}}$, a supremum of norms at fixed pairs,
subject to the same restriction near the diagonal, now at exponent $3m$. With
$\|D^{2}\varphi\|_{\infty}\vee\|D^{3}\varphi\|_{\infty}\le\tnorm\varphi$, with \eqref{eq:Egrowth} in the
two forms $|b_{r}(X_{r})|\le C_{\mathrm{gr}}(1+|X_{r}|)$ and
$\operatorname{tr}
a_{r}(X_{r})=\Frob{\sigma_{r}(X_{r})}^{2}\le C_{\mathrm{gr}}^{2}(1+|X_{r}|)^{2}$, using
$|S\!:\!a|=|\operatorname{tr}(Sa)|\le\Op{S}\operatorname{tr}a$ for $a\succeq0$, and with H\"older at the conjugate pair
$(3,\tfrac32)$ inside the norm, we obtain
\begin{align*}
\big\|\tilde P^{\varphi}_{s,t}\big\|_{L_{m}}
\ &\le\ \int_{s}^{t}\big\|\big(\nabla\varphi(X_{r})-\nabla\varphi(X_{s})\big)\cdot b_{r}(X_{r})\big\|_{L_{m}}\dr
\ \le\ C_{\mathrm{gr}}\tnorm\varphi\int_{s}^{t}\big\|\,\big|\delta X_{s,r}\big|
\big(1+|X_{r}|\big)\,\big\|_{L_{m}}\dr\\
\ &\le\ C_{\mathrm{gr}}\tnorm\varphi\int_{s}^{t}\big\|\delta X_{s,r}\big\|_{L_{3m}}
\big\|1+|X_{r}|\big\|_{L_{3m/2}}\dr\\
\ &\le\ C_{\mathrm{gr}}\tnorm\varphi
\Big(1+\big\|\sup_{t\le T}|X_{t}|\big\|_{L_{3m}}\Big)^{2}
\sup_{r\in[s,t]}\big\|\delta X_{s,r}\big\|_{L_{3m}}\,|t-s| ,
\end{align*}
and similarly we obtain, the factor $\tfrac12$ of \eqref{eq:PQtilde} being kept,
\begin{align*}
\big\|\tilde Q^{\varphi}_{s,t}\big\|_{L_{m}}
\ &\le\ \tfrac12\int_{s}^{t}\big\|\big(D^{2}\varphi(X_{r})-D^{2}\varphi(X_{s})\big)\!:\!a_{r}(X_{r})\big\|_{L_{m}}\dr\\
\ &\le\ \tfrac12C_{\mathrm{gr}}^{2}\tnorm\varphi\int_{s}^{t}\big\|\,\big|\delta X_{s,r}\big|
\big(1+|X_{r}|\big)^{2}\,\big\|_{L_{m}}\dr\\
\ &\le\ \tfrac12C_{\mathrm{gr}}^{2}\tnorm\varphi\int_{s}^{t}\big\|\delta X_{s,r}\big\|_{L_{3m}}
\big\|1+|X_{r}|\big\|^{2}_{L_{3m}}\dr\\
\ &\le\ \tfrac12C_{\mathrm{gr}}^{2}\tnorm\varphi
\Big(1+\big\|\sup_{t\le T}|X_{t}|\big\|_{L_{3m}}\Big)^{2}
\sup_{r\in[s,t]}\big\|\delta X_{s,r}\big\|_{L_{3m}}\,|t-s| .
\end{align*}
Adding the two yields the constant $C_{\mathrm{gr}}+\tfrac12C_{\mathrm{gr}}^{2}
=C_{\mathrm{gr}}(1+\tfrac12C_{\mathrm{gr}})$ of \eqref{eq:PQtildebound}. Further,
$\sup_{r\in[s,t]}\|\delta X_{s,r}\|_{L_{3m}}=O(|t-s|^{\gamma})$ by the computation just performed, read
at exponent $3m$ instead of $2m$, whence the rate $|t-s|^{1+\gamma}$. Since $\gamma>0$,
$\delta\tilde N^{\varphi}$ is $o(|t-s|^{1/2})$ and \ref{MP2} is untouched; but
$\gamma+\tfrac12<1$, so it is \emph{not} $o(|t-s|)$ and \ref{MP1} would be false without (ii). By
contrast $\tilde P^{\varphi}$ and $\tilde Q^{\varphi}$, at rate $|t-s|^{1+\gamma}$, are $o(|t-s|)$
outright and need no cancellation at all.

\emph{(ii) Their conditional means.}\stepnum{(ii)}{stp:MP9ii}
$\delta\tilde N^{\varphi}_{s,t}$ denotes a stochastic integral over $[s,t]$ against a Brownian motion
of the ambient filtration with adapted, square-integrable integrand, so
$\E[\delta\tilde N^{\varphi}_{s,t}\mid\mathcal F_{s}]=0$. Hence, by the tower property and
$\mathcal G_{s}\subseteq\mathcal F_{s}$,
\[
\E\big[\delta\tilde N^{\varphi}_{s,t}\mid\mathcal G_{s}\big]
\ =\ \E\Big[\,\E\big[\delta\tilde N^{\varphi}_{s,t}\mid\mathcal F_{s}\big]\ \Big|\ \mathcal G_{s}\Big]
\ =\ 0 :
\]
it joins $\mathcal M^{\varphi}$ in \eqref{eq:Dphisplit} and does not affect \ref{MP1}. We do not claim this for $\tilde P^{\varphi}$ and $\tilde Q^{\varphi}$, Lebesgue integrals of non-centred
integrands. They join $\Upsilon^{\varphi}$, and by \ref{stp:MP9i} require no
cancellation.
\emph{(iii) Their degrees.}\stepnum{(iii)}{stp:MP9iii} The total degree in the state is
\[
\delta\tilde N^{\varphi}:\ \underbrace{1}_{\delta X}+\underbrace{1}_{\sigma}=2,
\qquad
\tilde P^{\varphi}:\ \underbrace{1}_{\delta X}+\underbrace{1}_{b}=2,
\qquad
\tilde Q^{\varphi}:\ \underbrace{1}_{\delta X}+\underbrace{2}_{a=\sigma\sigma^{\top}}=3,
\]
the last entry because \eqref{eq:Egrowth} makes $a=\sigma\sigma^{\top}$ quadratic and not linear in the
state. Accordingly \eqref{eq:Ntildebound} controls $\delta\tilde N^{\varphi}$ by two factors in
$L_{2m}$ and \eqref{eq:PQtildebound} controls $\tilde P^{\varphi},\tilde Q^{\varphi}$ by three factors
in $L_{3m}$,
\[
\frac1m=\frac1{2m}+\frac1{2m},
\qquad
\frac1m=\frac1{3m}+\frac1{3m}+\frac1{3m},
\]
using $\|\sup_{t\le T}|X_{t}|\|_{L_{m'}}<\infty$ at $m'\ge3m$ and the Davie moduli at order $3m$; the $2m$ that $\tilde P^{\varphi}$ alone
would need is written as $3m$ only so that one display serves both. All three lie at or below the top
degree $3$, $\tilde Q^{\varphi}$ attaining it and the other two below it, and $3m\le m_{0}$ with
$3m\le m'$, so the printed hypotheses of \ref{mpe1} are enough for the whole family.

\emph{Step~10: the exact decomposition \eqref{eq:Dphisplit}.}\stepnum{10}{stp:MP10}
A pointwise bound cannot see cancellation, so we also need the decomposition behind
\eqref{eq:Dphibound}; it is \eqref{eq:MPleftover}, regrouped. From \eqref{eq:MPleftover} we
subtract $\int_{s}^{t}\mathcal L_{r}\varphi(X_{r})\dr+\delta N^{\varphi}_{s,t}$,
which produces $-\tilde P^{\varphi}$, $-\delta\tilde N^{\varphi}$ and $-\tilde Q^{\varphi}$ from the
first three summands, the last of these through It\^o's isometry in the form
\[
\tfrac12D^{2}\varphi(X_{s})\!:\!M^{\otimes2}_{s,t}
=\tfrac12D^{2}\varphi(X_{s})\!:\!\int_{s}^{t}\!a_{r}(X_{r})\dr
+\tfrac12D^{2}\varphi(X_{s})\!:\!\Big(M^{\otimes2}_{s,t}-\int_{s}^{t}\!a_{r}(X_{r})\dr\Big) .
\]
Sorting the remaining summands of \eqref{eq:MPleftover} by whether they are conditionally centred, we
obtain
\begin{equation}\label{eq:Dphisplit}
D^{\varphi}_{s,t}=W^{\varphi}_{s,t}\,J_{s,t}+\mathcal M^{\varphi}_{s,t}+\Upsilon^{\varphi}_{s,t},
\qquad
W^{\varphi}_{s,t}\ \mathcal G_{s}\text{-measurable},\qquad
\E\big[\mathcal M^{\varphi}_{s,t}\mid\mathcal G_{s}\big]=0 ,
\end{equation}
with $\mathcal M^{\varphi}_{s,t}=-\delta\tilde N^{\varphi}_{s,t}+\mathcal M^{\varphi,\circ}_{s,t}$,
$\Upsilon^{\varphi}_{s,t}=-\tilde P^{\varphi}_{s,t}-\tilde Q^{\varphi}_{s,t}
+\Upsilon^{\varphi,\circ}_{s,t}$, and with the two remainders given explicitly by
\begin{equation}\label{eq:MUcirc}
\begin{aligned}
\mathcal M^{\varphi,\circ}_{s,t}
:={}&\tfrac12D^{2}\varphi(X_{s})\!:\!\Big(M^{\otimes2}_{s,t}-\int_{s}^{t}\!a_{r}(X_{r})\dr\Big)
\ +\ D^{2}\varphi(X_{s})\big(G^{1}_{s,t}+G^{2}_{s,t},\,M_{s,t}\big),\\[2pt]
\Upsilon^{\varphi,\circ}_{s,t}
:={}&D^{2}\varphi(X_{s})\big(G^{1}_{s,t},G^{2}_{s,t}\big)
+\tfrac12D^{2}\varphi(X_{s})\!:\!\big(G^{2}_{s,t}\big)^{\otimes2}
+D^{2}\varphi(X_{s})\big(M_{s,t},J_{s,t}\big)\\
&\phantom{xx}{}+\tfrac12D^{2}\varphi(X_{s})\!:\!J^{\otimes2}_{s,t}
+D^{2}\varphi(X_{s})\big(A_{s,t},\delta X_{s,t}-\tfrac12A_{s,t}\big)
+\mathcal R^{\varphi,(3)}_{s,t} ,
\end{aligned}
\end{equation}
in the notation \eqref{eq:MPpieces}. Both terms of $\mathcal M^{\varphi,\circ}$ are conditionally
centred: the first is the It\^o compensator of $M^{\otimes2}$, and in the second the factor
$G^{1}_{s,t}+G^{2}_{s,t}$ is $\mathcal G_{s}$-measurable while $\E[M_{s,t}\mid\mathcal G_{s}]=0$;
together with $\E[\delta\tilde N^{\varphi}_{s,t}\mid\mathcal G_{s}]=0$ this is the centring asserted in
\eqref{eq:Dphisplit}. We further obtain
$|\mathcal M^{\varphi,\circ}_{s,t}|\vee|\Upsilon^{\varphi,\circ}_{s,t}|\le\tnorm\varphi\cdot\Xi_{s,t}$
pointwise, the family $\mathcal K^{\varphi}$ being controlled instead by \eqref{eq:Ntildebound} and
\eqref{eq:PQtildebound}, and for $\Upsilon^{\varphi,\circ}$ the sharper
$|\Upsilon^{\varphi,\circ}_{s,t}|\le\tnorm\varphi\cdot\Xi^{\circ}_{s,t}$ by reading
$\|D^{2}\varphi\|_{\infty}\vee\|D^{3}\varphi\|_{\infty}\le\tnorm\varphi$ off each summand of
\eqref{eq:MUcirc}, where
\begin{equation}\label{eq:Xicirc}
\Xi^{\circ}_{s,t}:=\big|G^{1}_{s,t}\big|\big|G^{2}_{s,t}\big|
+\tfrac12\big|G^{2}_{s,t}\big|^{2}
+\big|M_{s,t}\big|\big|J_{s,t}\big|
+\tfrac12\big|J_{s,t}\big|^{2}
+\big|A_{s,t}\big|\Big(\big|\delta X_{s,t}\big|+\tfrac12\big|A_{s,t}\big|\Big)
+\tfrac16\big|\delta X_{s,t}\big|^{3} .
\end{equation}
Applying H\"older's inequality to each summand of \eqref{eq:Xicirc} at the splittings
$\tfrac1m=\tfrac1{2m}+\tfrac1{2m}$ and $\tfrac1m=3\cdot\tfrac1{3m}$, we obtain
\[
\begin{gathered}
\big\|\Xi^{\circ}_{s,t}\big\|_{L_{m}}
\ \le\ C\Big(\underbrace{|t-s|^{3\gamma}}_{G^{1}G^{2}}
+\underbrace{|t-s|^{4\gamma}}_{(G^{2})^{2}}
+\underbrace{|t-s|^{1/2}o\big(|t-s|^{1/2}\big)}_{MJ}
+\underbrace{o\big(|t-s|\big)}_{J^{2}}
\\
+\underbrace{|t-s|^{1+\gamma}}_{A(\delta X-\frac12A)}
+\underbrace{|t-s|^{3\gamma}}_{|\delta X|^{3}}\Big)
\ =\ o\big(|t-s|\big) ,
\end{gathered}
\]
every exponent exceeding $1$ because $\gamma>\tfrac13$, save the two $J$-entries, which are
$o(|t-s|)$ by \eqref{eq:daviemoduli} at order $2m\le3m$; the constant $C$ collects
$(1+\|\sup_{t\le T}|X_{t}|\|_{L_{3m}})^{3}$, the H\"older norms of $\RZ$ and $C_{\mathrm{gr}}$. Of the three members of
$\mathcal K^{\varphi}$ only $\delta\tilde N^{\varphi}$ is conditionally centred, so it alone belongs in
$\mathcal M^{\varphi}$ and $\tilde P^{\varphi},\tilde Q^{\varphi}$ belong in $\Upsilon^{\varphi}$; it
is $\Upsilon^{\varphi,\circ}$ and not $\Upsilon^{\varphi}$ that carries a pointwise $\varphi$-free
majorant. Here $W^{\varphi}$ collects every term in which $J_{s,t}$ is multiplied by
$\mathcal G_{s}$-measurable factors alone,
\[
W^{\varphi}_{s,t}=\nabla\varphi(X_{s})+D^{2}\varphi(X_{s})
\big(\,\cdot\,,f_{s}(X_{s})\delta Z_{s,t}+(Df\,f+f')(X_{s})\ZZ_{s,t}\big),
\]
the expansion being explicit to second order and every higher occurrence of $J$ sitting inside the
third-order remainder of \eqref{eq:MPtaylor}. Since $|\delta Z_{s,t}|=O(|t-s|^{\gamma})$ and
$|\ZZ_{s,t}|=O(|t-s|^{2\gamma})=O(|t-s|^{\gamma})$ on $[0,T]$,
\begin{align*}
\big\|W^{\varphi}_{s,t}-\nabla\varphi(X_{s})\big\|_{L_{m}}
\ &\le\ \|D^{2}\varphi\|_{\infty}\Big(\big\|f_{s}(X_{s})\big\|_{L_{m}}\big|\delta Z_{s,t}\big|
+\big\|(Df\,f+f')(X_{s})\big\|_{L_{m}}\big|\ZZ_{s,t}\big|\Big)\\
\ &\le\ C\tnorm\varphi\big(1+\|\sup_{t\le T}|X_{t}|\|_{L_{m}}\big)|t-s|^{\gamma}.
\end{align*}
The cross term
$D^{2}\varphi(X_{s})(f_{s}(X_{s})\delta Z_{s,t},J_{s,t})$ has to stay inside $W^{\varphi}J$ and can neither go
into $\mathcal M^{\varphi}$ nor $\Upsilon^{\varphi}$; Remark~\ref{rem:MPgrouping} gives
the reasoning.

\emph{Step~11: the four facts the estimates rest on.}\stepnum{11}{stp:MP11} The
estimates of \eqref{eq:Dphibound} and of the three groups of \eqref{eq:Dphisplit} rest on the
following four facts.
\begin{enumerate}
    \item No germ factor occurs alone. The germ
of \eqref{eq:davie} is $f_{s}(X_{s})\delta Z_{s,t}+(Df\,f+f')(X_{s})\ZZ_{s,t}$, and $\CZ$ joins it
through the carr\'e du champ \eqref{eq:carre}; each of $|\delta Z_{s,t}|$, $|\ZZ_{s,t}|$ and
$|\CZ_{s,t}|$ is subtracted in \eqref{eq:roughMP} and survives only inside a product with a second
factor, so no germ factor stands alone. The one factor that would otherwise have stood alone is
$\nabla\varphi(X_{s})\cdot\int_{s}^{t}\sigma_{r}(X_{r})\d B_{r}$, and it does not: against
$-\delta N^{\varphi}_{s,t}$ it forms $-\delta\tilde N^{\varphi}_{s,t}$, which is set aside, is no
longer part of $R^{\varphi}$ and is governed by \eqref{eq:Ntildebound} instead.
\color{blue}Remark~\ref{rem:MPfactsfail} records what a lone germ factor, and this factor left in
place, would have done to \ref{MP1} and \ref{MP2}.\color{black}

What the first fact delivers is that every term
of $\Theta^{(0)}$ of rate below $1$
survives only by cancellation, those terms being exactly the ones carrying one factor
$\int_{s}^{t}\sigma\d B$ beside a germ factor, of rate $|t-s|^{\gamma+1/2}$, which the third fact centres.
\item The surviving products of two or more rough increments have rate above $1$, and the
qualifier cannot be dropped.  The surviving products of
$\mathcal G_{s}$-measurable factors are $\delta Z\otimes\ZZ$, $\ZZ^{\otimes2}$ and their companions,
together with everything the third-order Taylor remainder of \eqref{eq:MPtaylor} carries. Since
$\gamma>\tfrac13$ and $\CZ$ is Lipschitz by \eqref{eq:bracketLip}, the rates are
\[
\begin{aligned}
\delta Z\otimes\ZZ&:\ \gamma+2\gamma=3\gamma>1,
&\quad \ZZ^{\otimes2}&:\ 4\gamma>\tfrac43>1,\\
\text{third-order Taylor remainder}&:\ 3\gamma>1,
&\quad \CZ\ \text{beside a germ factor}&:\ 1+\gamma>1,\\
\textstyle\int_{s}^{t}b_{r}(X_{r})\dr\ \text{beside a germ factor}&:\ 1+\gamma>1,
&\quad \delta Z^{\otimes2}&:\ 2\gamma<1 .
\end{aligned}
\]
No $\CZ$ survives into $R^{\varphi}$, by
Step~\ref{stp:MP7}. It is listed because carrying $|\CZ_{s,t}|$ there would change no rate.
The first five go into $\Upsilon^{\varphi}$. The sixth, $\delta Z^{\otimes2}$, is the reason for the
qualifier. By \eqref{eq:brackettaylor} the second-order Taylor coefficient
$\tfrac12D^{2}\varphi(X_{s})\big(f_{s}(X_{s})\delta Z_{s,t}\big)^{\otimes2}$ is precisely what the $\ZZ$-
and $\CZ$-terms of \eqref{eq:roughMP} subtract. It is also the only product of two or more germ
factors of rate $\le1$: every other one carries either three germ factors, and so sits inside the
third-order remainder at rate $3\gamma$, or a factor $\ZZ$ beside a further germ factor (rate
$3\gamma$ or better), or a factor $|t-s|$, $\CZ$ or $\int_{s}^{t}b_{r}(X_{r})\dr$ (rate $1+\gamma$ or
better).
\item The one surviving family of rate below $1$ is conditionally centred.  It consists of the terms
carrying exactly one factor $\int_{s}^{t}\sigma_{r}(X_{r})\d B_{r}$ against otherwise
$\mathcal G_{s}$-measurable factors ($\delta Z$ and $\ZZ$ are deterministic, and $f_{s}(X_{s})$,
$(Df\,f+f')(X_{s})$ and $D^{j}\varphi(X_{s})$ are $\mathcal G_{s}$-measurable), of rate
$|t-s|^{\gamma+1/2}$, together with the compensated square
$(\int_{s}^{t}\sigma\d B)^{\otimes2}-\int_{s}^{t}\sigma\sigma^{\top}\dr$, of rate $|t-s|$. Since
$\gamma+\tfrac12<1$ these are \emph{not} $o(|t-s|)$ and \ref{MP1} is false without the centring; with
it they contribute $0$ to $\E[\,\cdot\mid\mathcal G_{s}]$ and $o(|t-s|^{1/2})$ to \ref{MP2}.
\item The remaining terms containing $J$ are either small in $L_{m}$ or absorbed by $W^{\varphi}$. By H\"older at the conjugate pair $(2,2)$, by \eqref{eq:Egrowth} with Burkholder--Davis--Gundy for
$\int_{s}^{t}\sigma\d B$, and by \eqref{eq:daviemoduli} at order $2m$,
\begin{align*}
\big\|\,|J_{s,t}|^{2}\,\big\|_{L_{m}}
\ &=\ \big\|J_{s,t}\big\|^{2}_{L_{2m}}\ =\ o\big(|t-s|\big),\\
\Big\|J_{s,t}\otimes\int_{s}^{t}\sigma_{r}(X_{r})\d B_{r}\Big\|_{L_{m}}
\ &\le\ \big\|J_{s,t}\big\|_{L_{2m}}\Big\|\int_{s}^{t}\sigma_{r}(X_{r})\d B_{r}\Big\|_{L_{2m}}\\
\ &=\ o\big(|t-s|^{1/2}\big)\cdot O\big(|t-s|^{1/2}\big)\ =\ o\big(|t-s|\big),\\
\Big\|J_{s,t}\otimes\int_{s}^{t}b_{r}(X_{r})\dr\Big\|_{L_{m}}
\ &\le\ \big\|J_{s,t}\big\|_{L_{2m}}\Big\|\int_{s}^{t}b_{r}(X_{r})\dr\Big\|_{L_{2m}}\\
\ &=\ o\big(|t-s|^{1/2}\big)\cdot O\big(|t-s|\big)\ =\ o\big(|t-s|\big) ,
\end{align*}
so $J^{\otimes2}$, $J\otimes\int_{s}^{t}\sigma\d B$ and $J\otimes\int_{s}^{t}b\dr$ are $o(|t-s|)$ in
$L_{m}$, the first using $\|J_{s,t}\|_{L_{2m}}=o(|t-s|^{1/2})$. Similarly, denoting
$\mathsf G:=(Df\,f+f')(X_{s})$,
\[
\big\|J_{s,t}\otimes\mathsf G\,\ZZ_{s,t}\big\|_{L_{m}}
\ \le\ \big\|J_{s,t}\big\|_{L_{2m}}\big\|\mathsf G\big\|_{L_{2m}}\big|\ZZ_{s,t}\big|
\ =\ o\big(|t-s|^{1/2+2\gamma}\big),
\qquad
\tfrac12+2\gamma>\tfrac12+\tfrac23=\tfrac76>1 ,
\]
so $J\otimes\mathsf G\ZZ$ has rate above $1$; and every occurrence of $J$ inside the third-order Taylor
remainder is carried by the pointwise bound $|\delta X_{s,t}|^{3}$, of rate $3\gamma>1$. The sole
exception is $J\otimes f_{s}(X_{s})\delta Z_{s,t}$, which is exactly what $W^{\varphi}$ was introduced
for.
\end{enumerate}

\emph{Step~12: the H\"older split for $W^{\varphi}J$.}\stepnum{12}{stp:MP12} $W^{\varphi}_{s,t}$ is
$\mathcal G_{s}$-measurable, so one would like to condition on $\mathcal G_{s}$ and take it out of
the norm, which would leave $W^{\varphi}$ at the exponent $m$ and cost no H\"older loss.
Remark~\ref{rem:MPWsplit} shows that this needs the uniform conditional modulus
\eqref{eq:daviemoduliinfty}, which \eqref{eq:Egrowth} does not give. We therefore run the step as a
H\"older split at two exponents,
\[
\big\|W^{\varphi}_{s,t}J_{s,t}\big\|_{L_{m}}\ \le\ \big\|W^{\varphi}_{s,t}\big\|_{L_{r}}
\big\|J_{s,t}\big\|_{L_{r'}},
\qquad \tfrac1m=\tfrac1r+\tfrac1{r'} .
\]
Three inputs fix $(r,r')$. First, both appeals need $J$ at order $r'$: \ref{MP1} needs
$\|\E[J\mid\mathcal G_{s}]\|_{L_{r'}}=o(|t-s|)$, the first modulus of \eqref{eq:daviemoduli} at
$r'$, and \ref{MP2} needs $\|J\|_{L_{r'}}=o(|t-s|^{1/2})$, which by the tower property is the
second modulus at $r'$. Both are available exactly when $r'\le m_{0}$. Second, the state
degree of $W^{\varphi}$ is $1$, it being $\nabla\varphi(X_{s})$, of degree $0$, plus the single
second-order term displayed after \eqref{eq:Dphisplit}, whose one state factor is the germ coefficient
and whose bound was recorded there and holds at any exponent $r\le m'$. Hence
\[
\big\|W^{\varphi}_{s,t}\big\|_{L_{r}}\ \le\ C\tnorm\varphi\Big(1+\big\|\sup_{t\le T}|X_{t}|\big\|_{L_{r}}\Big)
\]
and what this requires is $r\le m'$. Third, we can take $r'=m_{0}$ as large as the Davie moduli
allow, which makes $r$ as small as possible, $\tfrac1r=\tfrac1m-\tfrac1{m_{0}}$, so a split exists if
and only if
\begin{equation}\label{eq:Wsplit}
\frac{mm_{0}}{m_{0}-m}\ \le\ m' .
\end{equation}
At the printed Davie order $m_{0}=3m$ the left side is
\[
\frac{mm_{0}}{m_{0}-m}\bigg|_{m_{0}=3m}\ =\ \frac{3m^{2}}{2m}\ =\ \tfrac32m ,
\]
and part~\ref{mpe1} asks $m'\ge3m$, so \eqref{eq:Wsplit} holds with room to spare. The exponents are
$r'=3m$ and $r=\tfrac32m$, that is $\tfrac1m=\tfrac2{3m}+\tfrac1{3m}$, with $r=\tfrac32m<3m\le m'$; and
$r=\tfrac32m\ge\tfrac32\cdot2=3>1$
at every $m\ge2$, so the split is legitimate.

\emph{Step~13: taking norms, and \ref{MP1}--\ref{MP2}.}\stepnum{13}{stp:MP13}
The remaining norms are taken by generalised H\"older: a term of
\eqref{eq:Dphibound} of total degree
$N\le3$ is bounded in $L_{m}$ by a product of $N$ factors in $L_{Nm}$,
\[
\Big\|\prod_{k=1}^{N}U_{k}\Big\|_{L_{m}}\ \le\ \prod_{k=1}^{N}\big\|U_{k}\big\|_{L_{Nm}},
\qquad
\frac1m=\underbrace{\frac1{Nm}+\dots+\frac1{Nm}}_{N\ \text{summands}} ,
\]
and by \eqref{eq:davie}, \eqref{eq:Egrowth} and Burkholder--Davis--Gundy we obtain
\[
\|\delta X_{s,t}\|_{L_{3m}}\le C\big(1+\|\sup_{t\le T}|X_{t}|\|_{L_{3m}}\big)|t-s|^{\gamma}
+\|J_{s,t}\|_{L_{3m}} .
\]
What is needed is therefore $\|\sup_{t\le T}|X_{t}|\|_{L_{3m}}<\infty$, together with the Davie moduli
at order $3m$. The first of these is $m'\ge3m$, and it is the degree count; the second is
$m_{0}\ge3m$. It is what the
generalised H\"older at $N=3$ requires of the Davie modulus, $\delta X$ being one of the
three factors placed in $L_{3m}$ and carrying $J$ with it. Both are used at the top
degree $3$, which after the clearing below is every term of the sum in $\Xi$. The family set aside,
of total degree at most $3$, needs no more.

The count cannot be applied to \eqref{eq:Dphibound} as it stands. The factor
$\mathsf m_{s,t}$ of \eqref{eq:Thetadef} is of state degree $0$ but has no pathwise bound, so it
cannot be taken out of the norm. We clear the normalisation first. Write $\#$ for the
number of factors $\mathsf m_{s,t}$ in a term of $\Theta^{(j)}$ and $\theta_{s,t}$ for the product
of its remaining, deterministic, factors. Every term of $\Xi_{s,t}$ other than $|J_{s,t}|$ is then
\[
\mathsf S^{3-j}\big|\delta X_{s,t}\big|^{j}\,\mathsf m^{\#}_{s,t}\,\theta_{s,t}
\ =\ \mathsf S^{3-j-\#}\big|\delta X_{s,t}\big|^{j}\big|M_{s,t}\big|^{\#}\,\theta_{s,t},
\qquad \mathsf S:=\mathsf S(X,s,t),
\]
by \eqref{eq:Thetaatoms}. By \eqref{eq:Thetadef}, $\#\le2$ at $j=0$, $\#\le1$ at $j=1$ and
$\#=0$ at $j=2,3$, so $j+\#\le3$ and the exponent $3-j-\#$ is non-negative. The right side is
therefore a product of three random factors of state degree $1$, which is what the count requires, and
\[
\big\|\mathsf S\big\|_{L_{3m}}=O(1),
\qquad
\big\|\delta X_{s,t}\big\|_{L_{3m}}=O\big(|t-s|^{\gamma}\big),
\qquad
\big\|M_{s,t}\big\|_{L_{3m}}=O\big(|t-s|^{1/2}\big),
\]
the second by the term above, by $\|J_{s,t}\|_{L_{3m}}=o(|t-s|^{1/2})$ and by $\gamma<\tfrac12$,
the third by Burkholder--Davis--Gundy against \eqref{eq:Egrowth}. Apart from $\mathsf S$, which is
$O(1)$, the factors available are $|\delta X_{s,t}|$, $|M_{s,t}|$, $|t-s|$, $|\delta Z_{s,t}|$ and
$|\ZZ_{s,t}|$, at $\gamma$, at $\tfrac12$, at $1$, at $\gamma$ and at $2\gamma$. Every
term carries two of them, except $|\delta X_{s,t}|^{3}$ at $j=3$, which carries three, and the lone
$|t-s|$ of \eqref{eq:Thetadef}, which carries one. The slowest runs at $\gamma$, so every term is
$O(|t-s|^{2\gamma})$, the lone one too because $2\gamma<1$, and $2\gamma>\tfrac23$. With
$\|J_{s,t}\|_{L_{m}}=o(|t-s|^{1/2})$ we obtain
$\|\Xi_{s,t}\|_{L_{m}}=o(|t-s|^{1/2})$, while
$\|\Xi^{\circ}_{s,t}\|_{L_{m}}=o(|t-s|)$ is the bound stated just after \eqref{eq:Xicirc}; the term $\delta\tilde N^{\varphi}$ set aside above is
$O(|t-s|^{\gamma+1/2})=o(|t-s|^{1/2})$ in $L_{m}$ by \eqref{eq:Ntildebound} and contributes $0$ to the
conditional mean, while $\tilde P^{\varphi}$ and $\tilde Q^{\varphi}$ are
$O(|t-s|^{1+\gamma})=o(|t-s|)$ in $L_{m}$ by \eqref{eq:PQtildebound} and enter the conditional mean
through the contraction $\|\E[\,\cdot\mid\mathcal G_{s}]\|_{L_{m}}\le\|\cdot\|_{L_{m}}$ alone.
Collecting the estimates, we obtain from \eqref{eq:Dphisplit}
\[
\begin{gathered}
\big\|\E[D^{\varphi}_{s,t}\mid\mathcal G_{s}]\big\|_{L_{m}}
\le\big\|W^{\varphi}_{s,t}\big\|_{L_{3m/2}}\big\|\E[J_{s,t}\mid\mathcal G_{s}]\big\|_{L_{3m}}
+\big\|\tnorm\varphi\,\Xi^{\circ}_{s,t}\big\|_{L_{m}}
+\big\|\tilde P^{\varphi}_{s,t}\big\|_{L_{m}}+\big\|\tilde Q^{\varphi}_{s,t}\big\|_{L_{m}}\\
=\ \tnorm\varphi\cdot o(|t-s|),
\qquad
\big\|D^{\varphi}_{s,t}\big\|_{L_{m}}\le\tnorm\varphi\cdot o(|t-s|^{1/2}),
\end{gathered}
\]
the first because $W^{\varphi}$ is $\mathcal G_{s}$-measurable, so that
$\E[W^{\varphi}_{s,t}J_{s,t}\mid\mathcal G_{s}]=W^{\varphi}_{s,t}\E[J_{s,t}\mid\mathcal G_{s}]$, and the
second because $\|W^{\varphi}_{s,t}J_{s,t}\|_{L_{m}}\le\|W^{\varphi}_{s,t}\|_{L_{3m/2}}
\|J_{s,t}\|_{L_{3m}}$ by \eqref{eq:Wsplit}. The second bound reads
\[
\big\|D^{\varphi}_{s,t}\big\|_{L_{m}}
\le\big\|W^{\varphi}_{s,t}\big\|_{L_{3m/2}}\big\|J_{s,t}\big\|_{L_{3m}}
+\big\|\delta\tilde N^{\varphi}_{s,t}\big\|_{L_{m}}
+2\,\tnorm\varphi\,\big\|\Xi_{s,t}\big\|_{L_{m}}
+\big\|\tilde P^{\varphi}_{s,t}\big\|_{L_{m}}+\big\|\tilde Q^{\varphi}_{s,t}\big\|_{L_{m}} ,
\]
the five summands being $O(1)\cdot o(|t-s|^{1/2})$, $O(|t-s|^{\gamma+1/2})$,
$\tnorm\varphi\cdot o(|t-s|^{1/2})$ and $O(|t-s|^{1+\gamma})$ twice over. Each is also
finite uniformly over $s\le t$: every factor is a norm of a power of $\sup_{t}|X_{t}|$ or of a term of
\eqref{eq:davie}, bounded under the hypothesis $m'\ge3m$ of \ref{mpe1} together with
\eqref{eq:Egrowth} and \eqref{eq:MProughgrowth}, and every
rate is bounded on $[0,T]$. That is the first half of \eqref{eq:epsdelta}, which \ref{MP1} and
\ref{MP2} needs alongside the rates. These are \ref{MP1} and
\ref{MP2}. Uniformity over $\tnorm\varphi\le\Lambda$ is automatic here; see
Remark~\ref{rem:MPstrong}(i).

\begin{lemma}[Two facts about the mixed norm]\label{lem:mixednorm}
Let $\mathcal F\subseteq\mathcal A$ be a sub-$\sigma$-field and, following \cite[\S2]{FHL}, denote
$\|\xi\mid\mathcal F\|_{q}:=\big(\E[|\xi|^{q}\mid\mathcal F]\big)^{1/q}$ and
$\|\xi\|_{q,r}:=\big\|\,\|\xi\mid\mathcal F\|_{q}\,\big\|_{L_{r}}$ for $1\le q\le r$, with $q<\infty$ and $r\le\infty$, the outer norm at $r=\infty$ being the essential supremum. Then
the following hold:
\begin{enumerate}[label=\normalfont(\roman*),leftmargin=2.4em]
\item\label{mx1} for every $\theta\in[1,\infty)$ and $p\in[1,\infty)$,
\begin{equation}\label{eq:mixedidentity}
\big\|\,\E\big[|\xi|^{\theta}\mid\mathcal F\big]\,\big\|_{L_{p}}\ =\ \big\|\xi\big\|^{\theta}_{\theta,\,\theta p} ;
\end{equation}
\item\label{mx2} $q\mapsto\|\xi\|_{q,r}$ is non-decreasing on $[1,r]$.
\end{enumerate}
\end{lemma}

\begin{proof}
$\theta\ge1$ and $p\ge1$ is the range in
which $\|\cdot\|_{\theta,\theta p}$ is defined above. By definition,
\[
\big\|\xi\big\|^{\theta}_{\theta,\theta p}
=\big(\E\big[\big(\E[|\xi|^{\theta}\mid\mathcal F]^{1/\theta}\big)^{\theta p}\big]\big)^{\theta/(\theta p)}
=\big(\E\big[\E[|\xi|^{\theta}\mid\mathcal F]^{p}\big]\big)^{1/p} ,
\]
and the right-hand side is the left-hand side of \eqref{eq:mixedidentity}.
\ref{mx2} follows from conditional Jensen, $\|\xi\mid\mathcal F\|_{q}\le\|\xi\mid\mathcal F\|_{q'}$ for
$q\le q'$, together with monotonicity of the $L_{r}$-norm. Neither is printed in
\cite{FHL}.
\end{proof}

\color{blue}
\begin{proposition}[The Taylor step, without a Riemann sum]\label{prop:Taylorsew}
Let $\gamma\in(\tfrac13,\tfrac12)$, $\varkappa_{\varphi}\in(1/\gamma,3]$ and
$\mfn >\max\{4,2/\gamma\}$. Let $\varphi$ have $D^{j}\varphi$ bounded for
$1\le j\le N:=\lceil\varkappa_{\varphi}\rceil-1$ with $D^{N}\varphi$
globally H\"older of order $\varkappa_{\varphi}-N\in(0,1]$, and let $Y$
satisfy the hypotheses of Lemma~\ref{lem:roughItoLn}, so that in particular
\begin{equation}\label{eq:dYholder}
\|\delta Y\|_{\gamma;4,\mfn }\ :=\ \sup_{s<t}\frac{\|\delta Y_{s,t}\|_{4,\mfn }}{|t-s|^{\gamma}}\ <\ \infty .
\end{equation}
Let $A_{s,t}$ denote the germ of \cite[proof of Thm.~4.13]{FHL} and denote
$\Phi_{t}:=\varphi(Y_{t})-\varphi(Y_{0})$. Set $\theta:=\min\{\varkappa_{\varphi},\mfn /2\}$. Then
$\theta\in(1/\gamma,3]$, $2\theta\le \mfn $, and the following estimates are satisfied,
\begin{equation}\label{eq:Taylorsew}
\big\|\Phi_{t}-\Phi_{s}-A_{s,t}\big\|_{L_{2}}\le\|D^{2}\varphi\|_{\infty}\|\delta Y\|^{2}_{\gamma;4,\mfn }|t-s|^{2\gamma},
\qquad
\big\|\E_{s}\big[\Phi_{t}-\Phi_{s}-A_{s,t}\big]\big\|_{L_{2}}\le C_{\varphi}\|\delta Y\|^{\theta}_{\gamma;4,\mfn }|t-s|^{\theta\gamma},
\end{equation}
with $2\gamma>\tfrac12$ and $\theta\gamma>1$. Hence $\Phi$ lies in the class in which
Lemma~\ref{lem:SSL} states uniqueness for the germ $A$ at $m=2$; so does the right-hand side of the
rough It\^o formula, by the four defect estimates of \cite[proof of Thm.~4.13]{FHL} read with
Step~\ref{stp:rIto2}'s interpolation, and the two coincide.
\end{proposition}

\begin{proof}
\emph{Step 1:}\stepnum{1}{stp:Tsew1} $\Phi_{t}-\Phi_{s}-A_{s,t}$ is the second-order Taylor remainder. The germ of
\cite[proof of Thm.~4.13]{FHL} is
$A_{s,t}=\langle D\varphi(Y_{s}),\delta Y_{s,t}\rangle+\langle D^{2}\varphi(Y_{s}),Y'^{\otimes2}_{s}\ZZ_{s,t}\rangle
+\tfrac12\langle D^{2}\varphi(Y_{s}),\delta Y^{\otimes2}_{s,t}-2Y'^{\otimes2}_{s}\operatorname{Sym}\ZZ_{s,t}\rangle$.
Since $D^{2}\varphi$ is symmetric and $Y'^{\otimes2}_{s}\big(\ZZ_{s,t}-\operatorname{Sym}\ZZ_{s,t}\big)$
antisymmetric, the two $\ZZ$-terms cancel and $A_{s,t}$ is exactly the second-order Taylor
polynomial of $\varphi$ at $Y_{s}$ evaluated at $Y_{t}$. Therefore $R_{s,t}:=\Phi_{t}-\Phi_{s}-A_{s,t}$
is that remainder, and we obtain
\begin{equation}\label{eq:Rbranches}
|R_{s,t}|\ \le\ \|D^{2}\varphi\|_{\infty}\,|\delta Y_{s,t}|^{2},
\qquad
|R_{s,t}|\ \le\ C_{\varphi}\,|\delta Y_{s,t}|^{\varkappa_{\varphi}},
\qquad\text{hence}\qquad
|R_{s,t}|\ \le\ C_{\varphi}'\,|\delta Y_{s,t}|^{\theta}
\end{equation}
for every $\theta\in[2,\varkappa_{\varphi}]$, by $\min\{a^{2},a^{\varkappa}\}\le a^{\theta}$ for
$a\ge0$. No sup-norm of $\varphi$ appears, allowing Lemma~\ref{lem:roughItoLn}'s unbounded test
functions.

\emph{Step 2: the unconditional bound, from the quadratic branch, at $L_{4}$ only.}\stepnum{2}{stp:Tsew2}
Applying the first branch of \eqref{eq:Rbranches}, the identity
$\|\xi^{2}\|_{L_{2}}=\|\xi\|^{2}_{L_{4}}$ and $\|\cdot\|_{L_{4}}\le\|\cdot\|_{4,\mfn }$
\textup{(\cite[Prop.~2.3(ii)]{FHL})}, we obtain
\[
\|R_{s,t}\|_{L_{2}}\ \le\ \|D^{2}\varphi\|_{\infty}\big\|\,|\delta Y_{s,t}|^{2}\big\|_{L_{2}}
=\|D^{2}\varphi\|_{\infty}\|\delta Y_{s,t}\|^{2}_{L_{4}}
\ \le\ \|D^{2}\varphi\|_{\infty}\|\delta Y_{s,t}\|^{2}_{4,\mfn },
\]
and with \eqref{eq:dYholder} this is the first display of \eqref{eq:Taylorsew}. The exponent
is $2\gamma$, and $2\gamma>\tfrac12$ since $\gamma>\tfrac13$. Nothing above $L_{4}$ was required.

\emph{Step 3: the conditional bound.}\stepnum{3}{stp:Tsew3} Applying conditional Jensen,
$|\E_{s}R_{s,t}|\le\E_{s}|R_{s,t}|$, then the third branch of \eqref{eq:Rbranches},
\eqref{eq:mixedidentity} at $p=2$, Lemma~\ref{lem:mixednorm}\ref{mx2} and $2\theta\le \mfn $, we obtain
\[
\big\|\E_{s}R_{s,t}\big\|_{L_{2}}
\ \le\ C'_{\varphi}\big\|\E_{s}\big[|\delta Y_{s,t}|^{\theta}\big]\big\|_{L_{2}}
\ =\ C'_{\varphi}\big\|\delta Y_{s,t}\big\|^{\theta}_{\theta,\,2\theta}
\ \le\ C'_{\varphi}\big\|\delta Y_{s,t}\big\|^{\theta}_{4,\,2\theta}
\ \le\ C'_{\varphi}\big\|\delta Y_{s,t}\big\|^{\theta}_{4,\,\mfn },
\]
the third step legitimate because $\theta\le\varkappa_{\varphi}\le3<4$, the fourth by
monotonicity of $L_{r}$ in $r$. With \eqref{eq:dYholder} this is the second term of
\eqref{eq:Taylorsew}. The $\theta$-th power demands inner index $\theta\le3$, below
the class's inner index $4$ and so free, and outer index $2\theta$, its own parameter $\mfn $.
The step one must not take is
\[
\big\|\E_{s}|\xi|^{\theta}\big\|_{L_{2}}\le\big\|\,|\xi|^{\theta}\big\|_{L_{2}}=\|\xi\|^{\theta}_{L_{2\theta}} .
\]
It discards the conditioning and reintroduces $L_{2\varkappa_{\varphi}}$, as the $L_{2}$ reading of a
Riemann sum does.

\emph{Step 4: the exponent.}\stepnum{4}{stp:Tsew4} Since $\theta=\min\{\varkappa_{\varphi},\mfn /2\}$, we obtain
\[
\theta\gamma>1
\quad\Longleftrightarrow\quad
\theta>1/\gamma
\quad\Longleftrightarrow\quad
\varkappa_{\varphi}>1/\gamma\ \text{ and }\ \mfn /2>1/\gamma,\ \text{ i.e.~}\mfn >2/\gamma .
\]
The first is a standing hypothesis and the second is required of $\mfn $. Both hold.

\emph{Step 5: uniqueness.}\stepnum{5}{stp:Tsew5} $\Phi$ is adapted with $\Phi_{0}=0$, and
$\Phi_{t}-A_{0,t}=R_{0,t}\in L_{2}$ by Step~\ref{stp:Tsew2}, using only $\delta Y\in L_{4}$. In particular
$\varphi(Y_{t})\in L_{2}$ is not required. Denote by $\mathcal I$ the rough It\^o formula's
right-hand side. By the four estimates on $J^{1},\dots,J^{4}$ in \cite[proof of Thm.~4.13]{FHL},
read with the interpolation of Step~\ref{stp:rIto2} at the one summand that needs it and therefore at
$1+\gamma\lambda$ in place of $3\gamma$, $\mathcal I$ satisfies the same two
bounds at $\gamma+\beta$ and
$\mu:=\min\{\gamma+\beta+\beta'',1+\gamma\lambda,1+\gamma\}$, both $>\tfrac12$ and
$>1$ by \eqref{eq:paramsIto} and stated in $L_{2}$ directly, the index needed. Those four estimates
are independent of the Taylor step of \cite{FHL}, which is the step this proposition replaces, so no
circle arises. We can place
both in the class of Lemma~\ref{lem:SSL} at the common pair
\[
\epsilon_{2}:=\min\{2\gamma,\gamma+\beta\}=\gamma+\beta>\tfrac12,
\qquad
\epsilon_{1}:=\min\{\theta\gamma,\mu\}>1 ,
\]
the weakening being harmless on $|t-s|\le T$. The germ $A$
meets Lemma~\ref{lem:SSL}'s two hypotheses, by $\delta A_{s,u,t}=-D_{s,t}+D_{s,u}+D_{u,t}$ with
$D_{s,t}:=\mathcal I_{t}-\mathcal I_{s}-A_{s,t}$ and
$\|\E_{s}D_{u,t}\|_{L_{2}}=\|\E_{s}\E_{u}D_{u,t}\|_{L_{2}}\le\|\E_{u}D_{u,t}\|_{L_{2}}$. Uniqueness
yields $\Phi=\mathcal I$, both continuous, so the identity is up to indistinguishability.
\end{proof}
\color{black}

\begin{proposition}\label{prop:Athreshuses}
The hypotheses of Lemma~\ref{lem:roughItoLn} and the
parameter conditions \eqref{eq:paramsIto} hold at each of the following two places, which are the
places at which this paper applies that lemma.
\begin{enumerate}[label=\normalfont(\arabic*),leftmargin=2.4em]
\item\label{app:Athresh1} \emph{Step~\ref{step:Emoment}\ref{proof:Eproof:step2:a2} of
Theorem~\ref{thm:weakex}}, at $V(X^{k,R})$ for fixed $(k,R)$, where the truncated coefficients are
bounded. No demand on $m$ arises, the integrability index being unconstrained.
\item\label{app:Athresh2} \emph{Steps~\ref{proof:Eproof:step2:step2} and
\ref{proof:Eproof:step2:step3} of Lemma~\ref{lem:Escalefree}'s proof}, at the same
$V(X^{k,R})$ and the same fixed $(k,R)$. No demand on $m$ arises.
\end{enumerate}
\end{proposition}

\begin{proof}
At both places the test function admits $\varkappa_{\varphi}=3$, so the first entry of the minimum in
\eqref{eq:paramsIto} is $\gamma$. The pair fed to the formula is not $(f,f')$ but
$\big(f_{R}(X^{k,R}),(Df_{R}f_{R}+f'_{R})(X^{k,R})\big)$, placed by \eqref{eq:Jbarbeta} in
$\mathscr D^{\gamma,\bar\beta}_{\RZ}L_{m,\infty}$ with
$\bar\beta=\min\{\beta+\beta'-\gamma,\gamma\}$: first index $\gamma$, third entry $\bar\beta$, so
$\beta''=\min\{\gamma,\gamma(\tfrac \mfn 4-1),\bar\beta\}$. That class embeds in the $L_{4,\mfn}$ the
lemma requires, at every $\mfn>4$, once $m$ is taken at $4$ or above, by
$\|\cdot\|_{L_{\mfn}}\le\|\cdot\|_{L_{\infty}}$ and then Lemma~\ref{lem:mixednorm}\ref{mx2} at outer
index $\infty$. This choice of $m$ costs nothing, Step~\ref{step:Ereg} producing the solution simultaneously at
every $m\ge2$. The first condition in \eqref{eq:paramsIto} therefore states $2\gamma>\tfrac12$,
automatic, and the second
\begin{equation}\label{eq:paramsdischarge}
2\gamma+\min\big\{\gamma,\gamma(\tfrac \mfn 4-1),\bar\beta\big\}\ >\ 1 ,
\qquad\text{which for }\ \bar\beta>1-2\gamma\ \text{ reads }\qquad
\mfn \ >\ \frac4\gamma-4 .
\end{equation}
Both cases of that minimum are attained on an open set, and the two feed different rates.
$\bar\beta=\gamma$ exactly when $\beta+\beta'\ge2\gamma$, as e.g.~at
$\gamma=\tfrac{11}{30}$, $\beta=\tfrac7{20}$, $\beta'=\tfrac25$; otherwise
$\bar\beta=\beta+\beta'-\gamma<\gamma$, as e.g.~at $\gamma=\beta=0.45$, $\beta'=0.35$, where
$\bar\beta=0.35$. So \eqref{eq:paramsdischarge} is not a bound on $\mfn$ here. The
solution is $L_{m,\infty}$, so $\mfn$ may be chosen as large as we want, and the minimum is then
attained at $\bar\beta$. On the branch $\bar\beta=\beta+\beta'-\gamma$ the condition follows from
\eqref{eq:rhoexp}, and on the branch $\bar\beta=\gamma$ from $\gamma>\tfrac13$. Taking $\varkappa_{\varphi}=3$ requires the least, not the most, since lowering
$\varkappa_{\varphi}$ lowers $\beta''$ through the first entry, and $\varkappa_{\varphi}>1/\gamma$
keeps that entry from being restrictive.

\emph{\ref{app:Athresh1}: absorbed.} The formula is applied to $V(X^{k,R})$ at fixed $(k,R)$, where the
truncated $b^{k}_{R},\sigma^{k}_{R},f_{R},f'_{R}$ are bounded, so \eqref{eq:rItoLn} holds immediately and
only the test-function condition of Lemma~\ref{lem:roughItoLn} binds: $V(x)=(1+|x|^{2})^{1/2}$ is unbounded, while $V\in C^{\infty}$ is such that
$\|D^{j}V\|_{\infty}<\infty$ for $j=1,2,3$ by \eqref{eq:Vjet}, admitting $\varkappa_{\varphi}=3$. The
integrability index is unconstrained, the pair
$\big(f_{R}(X^{k,R}),(Df_{R}f_{R}+f'_{R})(X^{k,R})\big)$ being bounded at fixed $(k,R)$, hence in
$L_{4,\mfn }$ for every $\mfn <\infty$. At $\mfn =8$ the second entry in \eqref{eq:paramsIto} is $\gamma$.
The third-seminorm cap that \cite{HuberII} carries is not felt either: at fixed $(k,R)$ the truncated
coefficients are bounded and $X^{k,R}$ is an $L_{m,\infty}$-solution simultaneously at every $m\ge2$,
so every summand of the Davie expansion \eqref{eq:davie} of $\delta X^{k,R}_{u,w}$, the remainder $J$
included, lies in $L_{q}$ at every finite $q$ with order $\gamma$, and the outer index
$(1+\theta)\mfn $ of the interpolation of \cite{HuberII} is available without further hypothesis. Its H\"older index is not unconstrained: boundedness in $L_{4,\mfn }$ says nothing about
the H\"older exponent, and \eqref{eq:Jbarbeta} fixes it at $\beta=\gamma$, $\beta'=\bar\beta$,
$\beta''=\min\{\gamma,\gamma,\bar\beta\}=\bar\beta$. The conclusion survives by \eqref{eq:rhoexp},
not \eqref{eq:paramsdischarge}: in \eqref{eq:paramsIto} the first condition reads
$\gamma+\gamma=2\gamma>\tfrac12$, automatic, the second $2\gamma+\bar\beta>1$, holding
in both cases of $\bar\beta$, since
\[
2\gamma+\bar\beta=\gamma+\beta+\beta'\ \ge\ 2\beta+\beta'\ >\ 1
\quad\text{if }\bar\beta=\beta+\beta'-\gamma,
\qquad
2\gamma+\bar\beta=3\gamma\ >\ 1
\quad\text{if }\bar\beta=\gamma ,
\]
the first by $\beta\le\gamma$ and the second by $\gamma>\tfrac13$. Nothing moves here.

\emph{\ref{app:Athresh2}: absorbed, by the sentence absorbing \ref{app:Athresh1}.} The formula is
applied to the same $V(X^{k,R})$, at the same fixed $(k,R)$, on the same truncated, mollified system
of Step~\ref{step:Ereg}, now for its conclusion's integral form at
Step~\ref{proof:Eproof:step2:step2} and for the continuity in the upper variable that same
conclusion gives at Step~\ref{proof:Eproof:step2:step3}, not a
Gr\"onwall. Every hypothesis, \eqref{eq:rItoLn}, the index $\mfn =8$, the test-function condition \eqref{eq:Vjet}, the H\"older index of \eqref{eq:Jbarbeta}, is verified as at \ref{app:Athresh1}, and
\eqref{eq:paramsIto} again by \eqref{eq:rhoexp}.
\end{proof}

\subsection{What \texorpdfstring{\cite[Thm.~4.6]{FHL}}{[FHL, Thm.~4.6]} asks, and what is provided}\label{app:FHL46}

On the rough data \cite[Thm.~4.6]{FHL} asks
\begin{equation}\label{eq:FHL46hyp}
(f,f')\in\mathscr D^{2\alpha}_{\RZ}L_{m,\infty}\mathcal C^{\varkappa}_{b},\quad
(Df,Df')\in\mathscr D^{\alpha,\alpha''}_{\RZ}L_{m,\infty}\mathcal C^{\varkappa-1}_{b},
\qquad \varkappa>\tfrac1\alpha,\quad 2\alpha+\alpha''>1 ,
\end{equation}
with $\mathscr D^{2\alpha}:=\mathscr D^{\alpha,\alpha}$, $\alpha$ an index at which the driver is a
rough path, $b,\sigma$ bounded and Lipschitz \cite[Def.~4.1(b)]{FHL}. Both memberships can fail at the driver's own index, and \eqref{eq:FHL46alphamax} below says when. We apply \eqref{eq:FHL46hyp} at
\begin{equation}\label{eq:FHL46alpha}
\alpha\ :=\ \beta\wedge\beta'\ \in\ (\tfrac13,\gamma] ,
\end{equation}
a $\gamma$-rough path being an $\alpha$-rough path for $\alpha\le\gamma$ on a
bounded interval, Chen's relation carrying no index and
$|t-s|^{\gamma}\le T^{\gamma-\alpha}|t-s|^{\alpha}$. The
rows read
\begin{center}
\begin{tabular}{lcc} {condition} & {order required} & {order available}\\\hline
$[\![\delta f]\!]_{\alpha}$ & $\alpha$ & $\gamma\wedge(\beta+\beta')$\\
$[\![\delta f']\!]_{\alpha}$ & $\alpha$ & $\beta'$\\
$[\![\delta Df]\!]_{\alpha}$ & $\alpha$ & $\gamma\wedge(\beta+\beta')$\\
$\E_{s}R^{f}_{s,t}$ & $2\alpha$ & $\beta+\beta'$\\
$\E_{s}R^{Df}_{s,t}$ & $2\alpha$ & $\beta+\beta'$\\
$\delta Df'$ & $\alpha$ & $\beta'$\\
$\delta D^{2}f$ & $\alpha$ & $\beta'$\\
$\sup_{s}|f_{s}|_{\varkappa}$, $\sup_{s}|f'_{s}|_{\varkappa-1}$ & & Step~\ref{step:Ereg}, at a constant depending on $(k,R)$\\
$\sup_{s}|Df_{s}|_{\varkappa-1}$, $\sup_{s}|Df'_{s}|_{\varkappa-2}$ & & Step~\ref{step:Ereg}, likewise
\end{tabular}
\end{center}
The first seven rows are the union over the two memberships of
\eqref{eq:FHL46hyp}. Condition~(b) of \cite[Def.~3.7]{FHL} names three quantities,
$[\![\delta f]\!]_{\beta}$, $[\![\delta f']\!]_{\beta'}$ and $[\![\delta Df]\!]_{\beta'}$, all of
them plain increments, and its condition~(c) names one compensated defect, $\E_{s}R^{f}$. The first
membership therefore gives rows one to four. The second, which is \eqref{eq:FHL46hyp} applied to
$(Df,Df')$, gives $\E_{s}R^{Df}$ from~(c) and $\delta Df'$ and $\delta D^{2}f=\delta D(Df)$
from~(b), which are rows five to seven. A compensated defect of a second derivative would need a
third membership, and \cite[Def.~3.7]{FHL} states none.

The last two rows are condition~(a) of \cite[Def.~3.7]{FHL}, applied once to $(f,f')$ at
$(\varkappa,\varkappa-1)$ and once to $(Df,Df')$ at $(\varkappa-1,\varkappa-2)$. They are a real
restriction, since the fields of Assumption~\ref{ass:Erough} are unbounded. They hold for the
truncated field $f_{R}$ of Step~\ref{step:Ereg} and not for $f$, at the same $(k,R)$-dependent
constant as $b$ and $\sigma$ below. The entry $\sup_{s}|Df_{s}|_{\varkappa-1}$ follows from
$\sup_{s}|f_{s}|_{\varkappa}$ and is listed only for completeness.

Two consequences of the table are worth recording. First, it verifies
$(Df,Df')\in\mathscr D^{\alpha,\alpha''}$ at $\alpha''=\alpha$, which is more than
\cite[Thm.~4.6]{FHL} requires. Then $2\alpha+\alpha''=3\alpha>1$ holds automatically, and the second
parameter condition of \eqref{eq:FHL46hyp} never has to be discussed. Second, $\varkappa>1/\alpha$
at $\alpha=\beta\wedge\beta'$, which may lie close to $\tfrac13$, forces $\varkappa=3$; this is
admissible only because $\alpha>\tfrac13$ strictly.
The rows permit exactly
\begin{equation}\label{eq:FHL46alphamax}
\alpha\ \le\ \min\Big\{\beta',\ \tfrac{\beta+\beta'}2,\ \gamma\Big\} ,
\end{equation}
a bound $\ge\beta\wedge\beta'$, so $\alpha=\beta\wedge\beta'$ is available at every triple, the whole
demand being $\beta\wedge\beta'>\tfrac13$. Step~\ref{step:Ereg} supplies the It\^o data and the
weights.

The rows are checked for the truncated pair $(f_{R},f'_{R})$ of
Step~\ref{step:Ereg}, not for $(f,f')$. For the truncated pair the $\delta D^{2}f$ row holds at
order $\gamma\wedge\beta'$, still at least $\alpha=\beta\wedge\beta'$, and every row but the third and
the fifth carries the weight $1+2R$ of the truncation, which is harmless at fixed $R$. Those two are
unweighted, by the first bound of \eqref{eq:Eroughderiv} and the bound on $Df'$ of
\eqref{eq:Eroughgrowth}. We use \cite[Thm.~4.6]{FHL} for three things: the
existence of $X^{k,R}$; its uniqueness in the $L_{m,\infty}$ class, which makes the family well
defined in Step~\ref{step:Etight}; and the two moduli at every $m\ge2$ at once, which lets
\S\ref{subsec:Eproof} run at $m=p$. The moment bound
$\|\sup_{t\le T}|X^{k,R}_{t}|\|_{L_{p}}<\infty$ is not taken from the theorem.
Step~\ref{step:Emoment}\ref{proof:Eproof:step2:b} proves it from the second modulus of
\eqref{eq:daviemoduli} and the truncation. \eqref{eq:FHL46alphamax} says at which indices $\alpha$
the solution is integrable, and says nothing about that moment bound.

 The table covers only the rough data
\eqref{eq:FHL46hyp}. Of the two remaining hypotheses of \cite[Thm.~4.6]{FHL}, the boundedness
and Lipschitz continuity of $b,\sigma$ required by \cite[Def.~4.1(b)]{FHL} is supplied by the truncation
of Step~\ref{step:Ereg}, at a constant depending on $(k,R)$, which is harmless: the theorem is
applied at fixed $(k,R)$, and the limit $(k,R)\to\infty$ is taken only after the a priori estimate of
Step~\ref{step:Emoment}. On the initial condition, the source asks only
$\xi\in L_{0}(\mathcal F_{0})$, so the
$\mathcal F_{0}$-measurable $X_{0}\sim\nu$ of Assumption~\ref{ass:Einit} qualifies.

\bibliographystyle{alpha}
\bibliography{paper_A}

\end{document}